\documentclass{amsart}
\usepackage{amssymb,bm,amsmath, mathrsfs, amsxtra,amsthm,amscd }
\usepackage[all]{xy}
\usepackage{latexsym}
\usepackage{amsmath}
\usepackage{tabu}
\usepackage[T1]{fontenc}
\usepackage{amssymb,amscd}
\usepackage{setspace}
\usepackage{mathdots}
\usepackage[mathscr]{eucal}
\usepackage{bbm}
\usepackage{stmaryrd}
\usepackage{color}
\usepackage{hyperref}
\usepackage{booktabs}

\newtheorem{theorem}{Theorem}[section] 
\newtheorem{lemma}[theorem]{Lemma}     
\newtheorem{corollary}[theorem]{Corollary}
\newtheorem{proposition}[theorem]{Proposition}
\newtheorem{remark}[theorem]{Remark}
\newtheorem{definition}[theorem]{Definition}
\newtheorem{conjecture}[theorem]{Conjecture}
\newtheorem{example}[theorem]{Example}

\numberwithin{equation}{section}

\newcommand{\fa}{\mathfrak{a}}
\newcommand{\fb}{\mathfrak{b}}
\newcommand{\p}{\mathfrak{p}}
\newcommand{\fm}{\mathfrak{m}}

\newcommand{\cC}{\mathcal{C}}

\newcommand{\cO}{\mathcal{O}}

\newcommand{\cM}{\mathcal{M}}
\newcommand{\cN}{\mathcal{N}}

\newcommand{\cS}{\mathcal{S}}
\newcommand{\cJ}{\mathcal{J}}
\newcommand{\cI}{\mathcal{I}}
\newcommand{\F}{\mathbb F}
\newcommand{\EE}{\mathrm{E}}

\newcommand{\GKdim}{\mathrm{dim}_G}
\newcommand{\N}{\mathbb N}
\newcommand{\Q}{\mathbb Q}

\newcommand{\Z}{\mathbb Z}

\newcommand{\W}{\mathbb W}
\newcommand{\A}{\mathbb A}

\newcommand{\rRep}{\mathrm{Rep}}
\newcommand{\fg}{\mathfrak{g}}

\newcommand{\tGamma}{\widetilde{\Gamma}}
\newcommand{\tD}{\widetilde{D}}
\newcommand{\filt}{\textrm{-}\mathrm{filt}}

\newcommand{\wh}{\widehat}

\newcommand{\wt}{\widetilde}

\newcommand{\Fil}{\mathrm{Fil}}

\newcommand{\Ord}{\mathrm{Ord}}

\newcommand{\ra}{\rightarrow}
\newcommand{\lra}{\longrightarrow}

\newcommand{\GL}{\mathrm{GL}}
\newcommand{\SL}{\mathrm{SL}}

\newcommand{\id}{\mathrm{Id}}
\newcommand{\bFp}{\overline{\F}_p}
\newcommand{\bQp}{\overline{\Q}_p}

\newcommand{\Sym}{\mathrm{Sym}}

\providecommand{\cInd}{\mathrm{c}\textrm{-}\mathrm{Ind}}

\newcommand{\cris}{\mathrm{cris}}

\newcommand{\brho}{\overline{\rho}}
\newcommand{\ide}{\mathbf{1}}

\newcommand{\plim}{\varprojlim}
\newcommand{\ilim}{\varinjlim}

\newcommand{\xto}[1][]{\xrightarrow{#1}}
\newcommand{\simto}{
\xto[\sim]} 

\providecommand{\ligne}{\textbf{---}}

\newcommand{\matr}[4]{\begin{pmatrix}{#1}&{#2}\\ {#3}&{#4}\end{pmatrix}}
\newcommand{\smatr}[4]{\bigl(\begin{smallmatrix} {#1}& {#2}\\ {#3}&{#4}\end{smallmatrix}\bigl)}

\newcommand{\un}[1]{\underline{#1}}

\newcommand{\gr}{\mathrm{gr}}

\def\QM{{\mathbb{Q}}}
\def\OC{{\mathcal{O}}}
\def\FM{{\mathbb{F}}}

\def\AM{{\mathbb{A}}}
\def\TM{{\mathbb{T}}}
\def\Fov{{\overline{F}}}
\def\into{\hookrightarrow}
\def\onto{\twoheadrightarrow}

\def\To#1{\buildrel\hbox{\tiny{$#1$}}\over\longrightarrow}

\DeclareMathOperator{\Ann}{{\mathrm{Ann}}}
\DeclareMathOperator{\Art}{{\mathrm{Art}}}

\DeclareMathOperator{\Coker}{{\mathrm{Coker}}}

\DeclareMathOperator{\End}{{\mathrm{End}}}
\DeclareMathOperator{\Ext}{{\mathrm{Ext}}}
\DeclareMathOperator{\Fr}{{\mathrm{Fr}}}
\DeclareMathOperator{\Frob}{{\mathrm{Frob}}}
\DeclareMathOperator{\Gal}{{\mathrm{Gal}}}
\DeclareMathOperator{\Hom}{{\mathrm{Hom}}}

\DeclareMathOperator{\im}{{\mathrm{Im}}}

\DeclareMathOperator{\Ind}{{\mathrm{Ind}}}
\DeclareMathOperator{\rInj}{{\mathrm{Inj}}}

\DeclareMathOperator{\JH}{{\mathrm{JH}}}

\DeclareMathOperator{\Ker}{{\mathrm{Ker}}}
\DeclareMathOperator{\loc}{{\mathrm{loc}}}

\DeclareMathOperator{\Mod}{\mathrm{Mod}}

\DeclareMathOperator{\Proj}{{\mathrm{Proj}}}
\DeclareMathOperator{\rad}{{\mathrm{rad}}}

\DeclareMathOperator{\Res}{{\mathrm{Res}}}
\DeclareMathOperator{\Rep}{{\mathrm{Rep}}}

\DeclareMathOperator{\soc}{{\mathrm{soc}}}

\DeclareMathOperator{\Spec}{{\mathrm{Spec}}}

\DeclareMathOperator{\tr}{{\mathrm{tr}}}

\DeclareMathOperator{\Tor}{{\mathrm{Tor}}}

\DeclareMathOperator{\rsoc}{{\mathrm{soc}}}

\DeclareMathOperator{\rProj}{{\mathrm{Proj}}}

\def\a{\alpha}

\def\G{\Gamma}

\def\e{\varepsilon}

\def\l{\lambda}
\def\L{\Lambda}

\def\o{\omega}

\def\s{\sigma}

\newcommand{\Serre}{\mathscr{D}(\brho)}

\newcommand{\PtG}{\rProj_{\tGamma}}
\newcommand{\PG}{\rProj_{\Gamma}}

\newcommand{\Ug}{U(\overline{\mathfrak{g}})}
\newcommand{\Ugi}{U(\overline{\mathfrak{g}}_i)}

\newcommand{\ord}{\mathrm{ord}}
\newcommand{\defn}{\overset{\rm{def}}{=}}
\newcommand{\FKZ}{\F[\![K/Z_1]\!]}
\newcommand{\FIwZ}{\F[\![I/Z_1]\!]}

\newcommand{\quash}[1]{}

\makeatletter
\def\thm@space@setup{%
  \thm@preskip=0.3cm plus 0.1 cm minus 0.1 cm
  \thm@postskip=\thm@preskip
}
\makeatother

\begin{document}

\title
{On the mod $p$ cohomology for $\GL_2$: the non-semisimple case}

\author{Yongquan HU \and Haoran WANG }
\thanks{Morningside Center of Mathematics, Academy of Mathematics and Systems Science,
 Chinese Academy of Sciences,  Beijing 100190, China; University of the Chinese Academy of Sciences, Beijing 100049,
China\\
{\it E-mail:} {\ttfamily yhu@amss.ac.cn}\\
\noindent  Yau Mathematical Sciences Center, Tsinghua University, Beijing, 100084, China\\
{\it E-mail:} {\ttfamily haoranwang@mail.tsinghua.edu.cn}
}
\thanks{Mathematics Subject Classification 2010: 22E50, 11F70}
\date{}
\maketitle

\begin{abstract}
Let $F$ be a totally real field unramified at all places above $p$  and $D$ be a quaternion algebra which splits at either none,  or exactly one,  of the infinite places. Let $\overline{r}:\mathrm{Gal}(\overline{F}/F)\ra \GL_2(\overline{\F}_p)$ be a continuous irreducible representation which, when restricted to a fixed place $v|p$, is non-semisimple and sufficiently generic. Under some mild assumptions, we prove that the admissible smooth representations of $\mathrm{GL}_2(F_v)$ occurring in the corresponding Hecke eigenspaces of the mod $p$ cohomology of Shimura varieties associated to $D$ have Gelfand-Kirillov dimension $[F_v:\mathbb{Q}_p]$.
We also prove that any such representation can be generated as a $\mathrm{GL}_2(F_v)$-representation by its subspace of invariants under  the  first principal congruence subgroup. If moreover $[F_v:\mathbb{Q}_p]=2$, we prove that such representations have length $3$, confirming a speculation of Breuil and Pa\v{s}k\=unas.
\end{abstract}

\setcounter{tocdepth}{1}
\tableofcontents

\section{Introduction}

Let $p$ be a prime number. The mod $p$ (and also $p$-adic) Langlands program has been emerged starting from the fundamental work  of Breuil \cite{Br03}. Up to present, the  (mod $p$)  correspondence in the case of $\GL_2(\Q_p)$ has been well-understood in various aspects, by  the works of \cite{Br03}, \cite{Co}, \cite{Em3}, and \cite{Pa13}.
 However, the situation is much more complicated if $\GL_2(\Q_p)$ is replaced by a higher dimensional group, and a large part of the theory remains mysterious. One of the main obstacles is that we don't have a satisfactory understanding of supersingular representations of $p$-adic reductive groups.

The aim of this paper is to study  the mod $p$ Langlands correspondence for $\GL_2$ of a finite \emph{unramified} extension of $\Q_p$, in the context of local-global compatibility following \cite{BDJ}.  By the work of Emerton \cite{Em3}, the mod $p$ correspondence for $\GL_2(\Q_p)$ can be realized in the mod $p$ cohomology of modular curves. It is thus natural to search for this hypothetical correspondence for $\GL_2$  in the cohomology of  Shimura curves.  
To explain this we  fix the global setup.

Let $F$ be a totally real extension of $\QM$ in which $p$ is unramified. Let $D$ be a quaternion algebra with center $F.$ We assume that $D$ splits at exactly one infinite place in the introduction. For $U$ a compact open subgroup of $ ( D\otimes_F \A_{F,f})^{\times}$ let $X_U$ be the associated smooth projective Shimura curve over $F,$ (in the case  $(F,D) = (\Q, \GL_2)$, $X_U$ is the compactified modular curve). We fix a place $v$ above $p$ and let $f \defn [F_v:\Q_p].$ Let $\F$ be a sufficiently large finite extension of $\F_p$ (served as the coefficient field). Let  $\overline{r}:\Gal(\overline{F}/F)  \ra \GL_2(\F)$ be an absolutely irreducible continuous  Galois representation. Fixing $U^v $ a compact open subgroup of $ ( D\otimes_F \A^{\{v\}}_{F,f})^{\times}$ and letting  $U_v$ run over compact open subgroups of $(D\otimes_F F_v)^{\times} \cong \GL_2(F_v)$, we consider the $\F$-vector space
\begin{equation}\label{equ::def-of-pi}
\ilim_{U_v} \Hom_{\Gal(\overline{F}/F)} \left(\overline{r} , H^1_{\textrm{\'et}} (X_{U^v U_v}\times_F \overline{F} , \F) \right)
\end{equation}
which is an  admissible smooth representation of $\GL_2(F_v)$ over $\F.$

By carefully choosing the ``away from $v$ data'' as in \cite{BreuilDiamond} and \cite{EGS}, we land  in the so-called \emph{minimal} case, and denote the resulting representation by $\pi^D_v (\overline{r})$ (see \cite[Eq. (28)]{BreuilDiamond}). As suggested by \cite[Conj.~4.7]{BDJ} and \cite[Cor.~3.7.4]{BreuilDiamond}, $\pi_v^D(\overline{r})$ is expected to realize a mod $p$ Langlands correspondence.
\emph{A priori}, $\pi_v^D(\overline{r})$ might depend on the various global choices but,  conjecturally, $\pi^D_v (\overline{r})$ depends only on $\brho \defn \overline{r}^{\vee} |_{\Gal(\overline{F}_v/ F_v)},$ the restriction of $\overline{r}^{\vee}$ to $\Gal(\overline{F}_v/ F_v).$  For this reason, in the following we write  \[
\pi (\brho) \defn \pi^D_v (\overline{r}).
\]

There have been a lot of works  studying the representation-theoretic properties of $\pi(\brho)$, see \cite{BDJ}, \cite{Gee}, \cite{Gee-Kisin}, \cite{Br14}, \cite{BreuilDiamond}, \cite{EGS}, \cite{HuJLMS}, \cite{HW}, \cite{LMS}, \cite{Le}, \cite{Dotto-Le}, etc.  These works often have the common aim to determine certain invariants attached to the restriction of $\pi(\brho)$ to $K\defn\GL_2(\cO_{F_v})$, like the socle, the  subspace of invariants under the first principal subgroup $K_1\defn 1+p\mathrm{M}_2(\cO_{F_v})$ or the pro-$p$ Iwahori subgroup $I_1$, and also some local-global compatibility related to these subspaces. For example, it is known that (under various mild assumptions) 
\begin{enumerate}
\item[(i)] $\soc_{K}\pi(\brho)\cong \oplus_{\sigma\in\mathscr{D}(\brho)}\sigma$, where $\mathscr{D}(\brho)$ is an explicit set of Serre weights (i.e. irreducible $\F$-representations of $K$) associated to $\brho$ in \cite[\S9]{BP}, see \cite{Gee-Kisin}, \cite{EGS};
\item[(ii)] $\pi(\brho)^{K_1}\cong D_0(\brho)$, where $D_0(\brho)$ is a representation of  $\GL_2(\F_{p^f})$  constructed in \cite[\S13]{BP}, see \cite{HW}, \cite{LMS}, \cite{Le}.
\end{enumerate}
Nonetheless, when $F_v\neq \Q_p$, a complete description of $\pi(\brho)$ still seems to be out of reach.
\vspace{1mm}

From now on, we make the following assumptions on $\overline{r}:$
\begin{enumerate}
\item[(a)] $\overline{r}|_{\Gal \left(\overline{F} / F(\sqrt[p]{1}) \right)}$ is absolutely irreducible, and modular (i.e. $\pi(\brho)$ is nonzero);

\item[(b)] for $w\nmid p$ such that either $D$ or $\overline{r}$ ramifies, the framed deformation ring of $\overline{r}|_{\Gal (\overline{F}_w / F_w )}$ over the ring of  Witt vectors $W(\F)$ is formally smooth;

\item[(c)] for $w|p,$ $w\neq v,$ $\overline{r}|_{I_{ F_w }}$ is generic in the sense of \cite[Def.~11.7]{BP}, where $I_{ F_w }$ is the inertia subgroup at $w;$

\item[(d)] $\brho$ is \emph{reducible nonsplit} and, when restricted to $I_{F_v}$,  is of the following form up to twist:
\[
\matr{\omega_f^{\sum_{i=0}^{f-1}p^i(r_i+1)}}*01
\]
where $\omega_f$ denotes Serre's fundamental character of $I_{F_v}$ of level $f$.  We assume $\brho$ is \emph{strongly generic} in the sense that  $2 \leq r_i\leq p-5$ for each $i.$   In particular, this implies $p \geq 7.$
\end{enumerate}

The following is  our first main result.
  \begin{theorem}[Theorem \ref{thm:main-flat}]
  \label{thm:intro-GK}
Keep the above assumptions on $F$, $D$ and $\overline{r}$. We have \[\dim_{\GL_2(F_v)}(\pi(\brho))=f.\]
 \end{theorem}
Here, $\dim_{\GL_2(F_v)}(\pi(\brho))$  denotes the Gelfand-Kirillov dimension of $\pi(\brho)$ which, roughly speaking, measures the ``size'' of $\pi(\brho)$, see \S\ref{subsection-duality}.
 The importance of controlling the Gelfand-Kirillov dimension of $\pi(\brho)$ was first pointed out in \cite{Gee-Newton}.  Namely, Theorem \ref{thm:intro-GK}  implies that the patched modules  constructed in \cite{CEGGPS1}, commonly denoted by $M_{\infty}$, are \emph{flat} over the corresponding patched deformation rings $R_{\infty}$, which are power series rings over  $W(\F)$ by (b). Consequently, as explained in \cite[\S1.1]{CEGGPS1},  this allows to define a candidate for the $p$-adic Langlands correspondence, see Theorem \ref{thm:intro-pLL} below for a precise statement.

 The patched modules  also play an important role in the proof of Theorem \ref{thm:intro-GK}. Using them, it is proved in \cite[Appendix A]{Gee-Newton} that we always have $\dim_{\GL_2(F_v)}(\pi(\brho))\geq f$. Hence, it is enough to prove the upper bound $\dim_{\GL_2(F_v)} (\pi(\brho))\leq f$, whose proof relies on the following key criterion proved in \cite{BHHMS}.  To state it we introduce some more notation. Let   $I$ be the (upper triangular) Iwahori subgroup of $K$ and $Z_1$ the center of $K_1$. 
Let $\F[\![K_1/Z_1]\!]$ (resp. $\F[\![I_1/Z_1]\!]$) denote the Iwasawa algebra of $K_1/Z_1$ (resp. $I_1/Z_1$) with maximal ideal $\frak{m}_{K_1/Z_1}$ (resp. $\frak{m}_{I_1/Z_1}$). Also let $\tGamma\defn \F[\![K/Z_1]\!]/\frak{m}_{K_1/Z_1}^2$ and $\Gamma \defn \F[\GL_2(\F_{p^f})]\cong \F[\![K/Z_1]\!]/\frak{m}_{K_1/Z_1}$. 

\begin{theorem} $($\cite[Cor.~5.3.5]{BHHMS}$)$
\label{thm:intro-BHHMS}
Let $\pi$ be an admissible smooth representation of $\GL_2(F_v)$ over $\F$ with a central character. Assume that we have an equality of multiplicities
\[\big[\pi[\fm_{I_1/Z_1}^{3}]:\chi\big]=\big[\pi[\fm_{I_1/Z_1}]:\chi\big]\] for each character $\chi$ such that $\big[\pi[\fm_{I_1/Z_1}]:\chi\big]\neq0$. Then $\dim_{\GL_2(F_v)}(\pi)\leq f$.
\end{theorem}

Using the above criterion, proving Theorem \ref{thm:intro-GK} is reduced to proving the following multiplicity one property of $\pi(\brho)$ (recall that we are considering  the minimal case).

\begin{theorem}[Corollary \ref{cor:multione-Iwahori}]
\label{thm:intro-multione}
(i) For any $\sigma\in \mathscr{D}(\brho)$, we have $\big[\pi(\brho)[\fm_{K_1/Z_1}^2]:\sigma\big]=1$.

(ii) For any $\chi:I\ra \F^{\times}$ such that $\big[\pi(\brho)[\fm_{I_1/Z_1}]:\chi\big]\neq0$, we have $\big[\pi(\brho)[\fm_{I_1/Z_1}^3]:\chi\big]=1$.
 \end{theorem}

It is clear that, to apply Theorem \ref{thm:intro-BHHMS}, we only need to study the subspace $\pi(\brho)[\fm_{I_1/Z_1}^3]$. However, the information  we could gain  from $\brho$  using $p$-adic Hodge theory, like the structure of  certain Galois deformation rings of $\brho$, is about the restriction of $\pi(\brho)$ to $K$. In \S\ref{section-BP}, we study the relation between $\pi(\brho)[\fm_{K_1/Z_1}^2]$ and $\pi(\brho)[\fm_{I_1/Z_1}^3]$, and prove a key result, Theorem \ref{thm-criterion}, which reduces the proof of Theorem \ref{thm:intro-multione} to proving (i) for one  single Serre weight $\sigma_0$. Namely, Theorem  \ref{thm:intro-multione} will follow if we can  prove 
\begin{equation}\label{eq:intro-multione}
\exists\ \sigma_0\in\mathscr{D}(\brho)\ \text{such that} \ \big[\pi(\brho)[\fm_{K_1/Z_1}^2]:\sigma_0\big]=1.\end{equation}
There is a distinguished element in $\mathscr{D}(\brho)$, namely the Serre weight $\sigma_0\defn (r_0,\cdots,r_{f-1})$ (see \S\ref{section-finiteRep} for the notation), which we call the \emph{ordinary} one.  We are thus left to verify \eqref{eq:intro-multione} for this $\sigma_0$, equivalently
\[\dim_{\F}\Hom_{K}(\Proj_{\tGamma}\sigma_0,\pi(\brho))=1, \]
where $\Proj_{\tGamma}\sigma_0$ denotes a projective envelope of $\sigma_0$ in the category of $\tGamma$-modules.

The common strategy to prove such a statement is provided by the Taylor-Wiles patching method, initially due to  Emerton, Gee and Savitt (\cite{EGS}) and later on generalized in  \cite{HW}, \cite{LMS}, \cite{Le}.
A patched module $M_{\infty}$ (in \cite{CEGGPS1} or \cite{Dotto-Le}) carries simultaneously a continuous action of $\GL_2(F_v)$ commuting with the action of  $R_{\infty}$, and is projective when viewed as a pseudo-compact $\F[\![K/Z_1]\!]$-module.  By setting $M_{\infty}(-)\defn \Hom_{K}^{\rm cont}(M_{\infty},-^{\vee})^{\vee}$ where $(-)^{\vee}$ denotes Pontryagin dual,  we obtain an exact covariant functor from the category of continuous representations of $K$ on finitely generated $W(\F)$-modules to the category of finitely generated $R_{\infty}$-modules.  By construction, we have an isomorphism of $\F[\GL_2(F_v)]$-modules $M_{\infty}/\fm_{R_{\infty}}\cong \pi(\brho)^{\vee}$, which implies an isomorphism
\[M_{\infty}(\Proj_{\tGamma}\sigma_0)/\fm_{R_{\infty}}\cong\Hom_{K}(\Proj_{\tGamma}\sigma_0,\pi(\brho))^{\vee}.\] 
Therefore,  proving \eqref{eq:intro-multione} is equivalent to  proving that $M_{\infty}(\Proj_{\tGamma}\sigma_0)$ is cyclic over $R_{\infty}$. Finally, we prove  this cyclicity by combining the  result of \cite{Le} on the cyclicity of $M_{\infty}(\Proj_{\Gamma}\sigma_0)$  and the semisimplicity of  the ordinary part of $\pi(\brho)$ (a result of \cite{HuJLMS} in the indefinite case and of \cite{Breuil-Ding} in the definite case).

A general construction  in \cite[\S13]{BP} shows that there exists a largest subrepresentation $\widetilde{D}_0(\brho)$ of $\bigoplus_{\sigma\in\mathscr{D}(\brho)}\rInj_{\tGamma}\sigma$ such that $[\widetilde{D}_0(\brho):\sigma]=1$ for any $\sigma\in\mathscr{D}(\brho)$. Here $\rInj_{\tGamma}\sigma$ denotes an injective envelope of $\sigma$ in the category of $\tGamma$-modules. In Theorem \ref{thm-tD=multione}, we prove  that $\widetilde{D}_0(\brho)$  is multiplicity free.
Combining this result with Theorem \ref{thm:intro-multione}, we obtain the following description of $\pi(\brho)[\fm_{K_1/Z_1}^2]$.
\begin{theorem}[Corollary \ref{cor::m^2-torsion}]\label{thm:intro-tD0}
There is an isomorphism $\pi(\brho)[\fm_{K_1/Z_1}^2]\cong \widetilde{D}_0(\brho)$. In particular, $\pi(\brho)[\fm_{K_1/Z_1}^2]$ is multiplicity free.
\end{theorem}

We state the following application of Theorem \ref{thm:intro-GK} mentioned above, which might be thought of as a candidate for the $p$-adic Langlands correspondence in this setting (cf. \cite[\S1.1]{CEGGPS1}).  Let $E'$ be a finite extension of $W(\F)[1/p]$ with ring of integers $\cO'$ and residue field $\F'$.
\begin{theorem}[Corollary \ref{cor:Pi(x)}]
\label{thm:intro-pLL}
 Let $x:R_{\infty}\ra \cO'$  be a local morphism of $W(\F)$-algebras.   Set
\[\Pi(x)^{0}\defn\Hom_{\cO'}^{\rm cont}(M_{\infty}\otimes_{R_{\infty},x}\cO',\cO')\]
and $\Pi(x)\defn\Pi(x)^{0}\otimes_{\cO'} E'$. Then
$\Pi(x)$ is a nonzero admissible unitary Banach representation of $G$ over $E'$ with $G$-invariant unit ball $\Pi(x)^{0}$ which lifts $\pi(\brho)\otimes_{\F}\F'$. 
\end{theorem}

In a companion paper \cite{BHHMS},   an analog of Theorem \ref{thm:intro-GK}, and consequently analogs of Theorem \ref{thm:intro-tD0} and  Theorem \ref{thm:intro-pLL}, are proved in the case $\brho$ is \emph{semisimple} and sufficiently generic (but with slightly stronger genericity assumptions than ours). The  proofs in both the semisimple and non-semisimple cases follow the same strategy, by first proving Theorem \ref{thm:intro-multione} and then applying the criterion of Theorem \ref{thm:intro-BHHMS}. However, the corresponding proofs of Theorem \ref{thm:intro-multione}  are very different. In \cite{BHHMS}, (the analog of) Theorem \ref{thm:intro-multione} is proved by complicated computations of potentially crystalline deformation rings, along the line in previous works of \cite{EGS} and \cite{Le}.  Our proof of Theorem \ref{thm:intro-multione} relies more on combinatorial properties of representations of $\tGamma$  together with a little computation based on \cite{Le}. Namely, we use Theorem \ref{thm-criterion} to reduce the computation to a minimal level. Another bonus of this treatment is that we require weaker genericity assumptions on $\brho$.  However, we should point out that our method only applies to the \emph{non-semisimple} case, while the method of \cite{BHHMS} should apply in all cases once the corresponding Galois deformation rings are worked out.

\vspace{1mm}

Next, we turn to the subtler question of determining  the  structure of $\pi(\brho)$ as a representation of $\GL_2(F_v)$. The conjectural shape of $\pi(\brho)$ was formulated in the fundamental  work of Breuil--Pa\v{s}k\={u}nas \cite{BP} (by local means), as follows.\footnote{Strictly speaking, the authors of \cite{BP} did not state it as a conjecture, and the family of admissible smooth representations of $\GL_2(F_v)$ constructed there is much richer than the one considered in this paper. However, the philosophy is  clearly due to them, see the discussion on page 107 of \emph{loc. cit.}.} Under the genericity condition in the sense of \cite[Def.~11.7]{BP}, which is weaker than our  genericity in (d), $\pi(\brho)$ should have  length $f+1$,  with a unique Jordan--H\"older filtration of the form:
 \[\pi_0\ \ligne\ \pi_1\ \ligne\ \cdots\ \ligne\  \pi_{f-1} \ \ligne\ \pi_f, \]
 where $\pi_0$ and $\pi_f$ are (irreducible) principal series  explicitly determined by $\brho$, and $\pi_i$ are supersingular representations for $1\leq i\leq f-1$.

As an application of  Theorem \ref{thm:intro-multione} and the flatness of $M_{\infty}$ over $R_{\infty}$, we first prove the following.
\begin{theorem}[Theorem \ref{thm-generation-D0}]\label{thm:intro-finite}
$\pi(\brho)$ is generated by $D_0(\brho)$ as a $\GL_2(F_v)$-representation.
\end{theorem}

Let us sketch the proof of Theorem \ref{thm:intro-finite} which is somewhat lengthy. Recall that  $R_{\infty}$ is formally smooth over $W(\F)$ by our assumption (b) and $M_{\infty}$ is flat  over $R_{\infty}$ satisfying $M_{\infty}/\fm_{R_{\infty}}\cong \pi(\brho)^{\vee}$. By choosing a regular system of parameters of $R_{\infty}$, we  obtain a Koszul type resolution  of $\pi(\brho)^{\vee}$ by projective $\F[\![K/Z_1]\!]$-modules. However, this resolution is not minimal and $M_{\infty}$ is not even finitely generated over $\F[\![K/Z_1]\!]$. For this reason  we further do a base change from $R_{\infty}$ to a suitable quotient, denoted by $R_v$ in the context. The resulting Koszul type resolution of $\pi(\brho)^{\vee}$, denoted by $P_{\bullet}$, is \emph{minimal} when viewed as a complex of $\F[\![K/Z_1]\!]$-modules, and \emph{partially minimal} if further restricted to $\F[\![I/Z_1]\!]$ (but we ignore this issue in the introduction).

A consequence of the above resolution is that $\pi(\brho)^{\vee}$ is essentially self-dual (see Definition \ref{def:selfdual} for this notion). This implies the crucial fact  that the $\GL_2(F_v)$-cosocle of $\pi(\brho)$ is isomorphic to $\pi_f$, because the socle of $\pi(\brho)$ is isomorphic to $\pi_0$ (this follows from the description of $\soc_K\pi(\brho)$ and a certain mod $p$ local-global compatibility). It follows that an $I$-subrepresentation $W$ of $\pi(\brho)$ generates $\pi(\brho)$ as a $\GL_2(F_v)$-representation  if and only if
the composite morphism $W\hookrightarrow \pi(\brho)\twoheadrightarrow \pi_f$ is nonzero, for which it suffices to find some character $\chi$ of $I$ and some $i\geq 0$ such that the composite morphism
\[\Ext^i_{I/Z_1}(\chi,W)\ra \Ext^i_{I/Z_1}(\chi,\pi(\brho))\ra\Ext^i_{I/Z_1}(\chi,\pi_f)\]
is nonzero. Using the resolution $P_{\bullet}$ we can determine the derived ordinary parts of $\pi(\brho)$ and  show that the quotient $\pi(\brho)\twoheadrightarrow \pi_f$ induces  an  isomorphism
\[\Ext^{2f}_{I/Z_1}(\chi_0^s,\pi(\brho))\simto\Ext^{2f}_{I/Z_1}(\chi_0^s,\pi_f),\]
where $\chi_0^s$ denotes the character of $I$ acting on the space of coinvariants $(\sigma_0)_{I_1}$.
Hence, $\pi(\brho)$ can be generated by any $I$-subrepresentation $W$ such that the natural morphism
 \[\Ext^{2f}_{I/Z_1}(\chi_0^s,W)\ra \Ext^{2f}_{I/Z_1}(\chi_0^s,\pi(\brho))\]
is surjective (or even nonzero).

However, it is in general hard to calculate $\Ext^i_{I/Z_1}$ for higher degrees; for example, we don't know how to write down an injective resolution of $W$ if we take $W=D_0(\brho)|_{I}$. To solve this issue, we define  a certain finite dimensional $I$-subrepresentation  of $\pi(\brho)$ denoted by $\tau(\brho)$, in an artificial way, of which a minimal injective resolution can be easily written down. By definition $\tau(\brho)$ is a tensor product of suitable $I$-representations $\tau(\brho)_{\kappa}$ along all embeddings $\kappa:\F_{p^f}\hookrightarrow \F$.  The construction of $\tau(\brho)_{\kappa}$ is motivated by the case $f=1$, see Example \ref{example:f=1}. For example, if $f=1$ and $\brho$ is reducible nonsplit, then $\tau(\brho)$ is the largest subrepresentation of $(\rInj_{I/Z_1}\chi_0^s)[\fm_{I_1/Z_1}^3]$ in which $\chi_0^s$ occurs with multiplicity one.  The advantage to define $\tau(\brho)$ in such a way is that we may explicitly construct a \emph{minimal} projective resolution of  $\gr_{\fm_{I_1/Z_1}}(\tau(\brho)^{\vee})$, say $G_{\bullet}$, as a tensor product of  resolutions along each embedding; this uses the fact that the graded algebra of $\F[\![I_1/Z_1]\!]$ with respect to the $\fm_{I_1/Z_1}$-adic filtration  is isomorphic to the tensor product of the corresponding graded algebra in the case of $\GL_2(\Q_p)$. While it is direct to lift $G_{\bullet}$ to a filt-projective resolution $Q_{\bullet}$ of $\tau(\brho)^{\vee}$, $Q_{\bullet}$ need \emph{not} be minimal in general. Fortunately, in our case  any lift $Q_{\bullet}$ is automatically minimal. This allows us to calculate $\Ext^i_{I/Z_1}(\tau(\brho)^{\vee},(\chi_0^s)^{\vee})$ for any $i$.

Now, the projectivity of $P_{\bullet}$ allows to lift the quotient morphism $\pi(\brho)^{\vee}\twoheadrightarrow\tau(\brho)^{\vee}$ to a morphism of complexes $P_{\bullet}\ra Q_{\bullet}$. Here comes another complication:  although numerically we have
\[ \dim_{\F}\Ext^{i}_{I/Z_1}(\pi(\brho)^{\vee},(\chi_0^s)^{\vee})=\dim_{\F}\Ext^{i}_{I/Z_1}(\tau(\brho)^{\vee},(\chi_0^s)^{\vee}),\]
it is not at all obvious that the natural morphism \[\beta_i: \Ext^{i}_{I/Z_1}(\tau(\brho)^{\vee},(\chi_0^s)^{\vee})\ra \Ext^{i}_{I/Z_1}(\pi(\brho)^{\vee},(\chi_0^s)^{\vee})\] is surjective, or even nonzero!  To solve this, we use crucially the fact that   $P_{\bullet}$ is a \emph{Koszul}  complex, motivated by an old theorem of Serre \cite{Se56}.  Roughly speaking, since $P_{\bullet}$ is a Koszul complex, to prove $\beta_{i}$ is surjective,  it suffices to prove $\beta_1$ is surjective which itself is a consequence of Theorem \ref{thm:intro-multione}(ii) and the construction of $\tau(\brho)$. In all we obtain that $\pi(\brho)$ is generated by $\tau(\brho)$ as a $\GL_2(F_v)$-representation and Theorem \ref{thm:intro-finite} follows easily.

Combining  Theorem \ref{thm:intro-finite} with the explicit structure of $D_0(\brho)$ (\cite{Le}), and also the main result of \cite{HuJLMS}, we finally arrive at the following result, confirming the aforementioned speculation in \cite{BP} in the case $f=2$, and provides strong evidence for the  general case.
\begin{theorem}[Theorem  \ref{thm-main-f=2}]
\label{thm:intro-f=2}
Assume $f=2$. Then $\pi(\brho)$ is uniserial of length $3$, with a unique Jordan--H\"older filtration of the form
\[\pi_0\ \ligne\ \pi_1\ \ligne\ \pi_2\]
where $\pi_0$, $\pi_2$ are principal series and $\pi_1$ is a supersingular representation.
\end{theorem}

Although limited to the case of $\GL_2(\Q_{p^2})$, Theorem \ref{thm:intro-f=2} provides the first nontrivial result, beyond the case of $\GL_2(\Q_p)$ and some related groups like $\SL_2(\Q_p)$, showing that admissible smooth representations corresponding to Hecke eigenspaces of the mod $p$ cohomology of Shimura varieties can have finite length.\footnote{Recently, similar finiteness results are proved for semisimple $\brho$  in \cite{BHHMS2}.}

\medskip 

We now give a brief overview of the contents of each section.
From \S\ref{section-finiteRep} to \S\ref{section-BP}, we study modular  representation theory of  $\Gamma$, of $\tGamma$ and also of $I$. The main result is Theorem \ref{thm-criterion}.  In \S\ref{section:ordinary}, we recall Emerton's functor of ordinary parts and prove Proposition \ref{prop:Ord-semisimple} which reinterprets the semisimplicity of the ordinary part of $\pi(\brho)$ in terms of the restriction of $\pi(\brho)$ to $K$.
In \S\ref{Section::galois}, we study two classes of quotients of the universal deformation ring of $\brho$: one is the reducible deformation ring and the other is  the multi-type potentially Barsotti-Tate deformation rings  studied in \cite{Le}. In \S\ref{Section::automorphic}, we recall $P$-ordinary automorphic forms and its relation to reducible deformation rings. In \S\ref{section-patching}, we combine all the previous results to prove  our (first) main result  Theorem \ref{thm:main-flat}.
\S\ref{section-HA} and \S\ref{section-fg} are devoted to the proof of our second main results,  Theorem \ref{thm-generation-D0} and Theorem  \ref{thm-main-f=2}: \S\ref{section-HA} contains some preliminary results and the proofs of the theorems are presented in \S\ref{section-fg}. In Appendix \S\ref{section:appendix}, we collect some useful definitions and results in the theory of non-commutative Iwasawa algebra.

\subsection{Notation}
\label{Sec::notation}

If $F$ is a field, let $G_F \defn \Gal(\Fov/F)$ denote its absolute Galois group. Let $\e$ denote the $p$-adic cyclotomic character of $G_F$, and $\omega$ the mod $p$ cyclotomic character.

If $F$ is a number field and $v$ is a place of $F,$ let $F_v$ denote the completion of $F$ at $v.$ If $v$ is a finite place of $F,$ let $\cO_{F_v}$ denote the ring of integers of $F_v$ with uniformiser $\varpi_v$ and residue field $k_{F_v}.$ The cardinality of $k_{F_v}$ is denoted by $q_v.$ We fix an embedding $\Fov\into \Fov_v$ so that $G_{F_v}$ identifies with the decomposition group of $F$ at $v$, and write $I_{F_v}$ for the inertia subgroup of $G_{F_v}.$ Let $\Frob_v\in G_{F_v}$ denote a (lift of the) geometric Frobenius element, and let $\Art_{F_v}$ denote the local Artin reciprocity map, normalized so that it sends  $\varpi_v$ to $\Frob_v$. The global Artin map is compatible with the local Artin maps and is denoted by $\Art_F.$

We fix a prime number $p.$ Let $E/\Q_p$ be a finite extension in $\bQp,$ with ring of integers $\cO$ and residue field $\F \defn \cO/(\varpi)$ where $\varpi$ is a fixed uniformizer. We will often assume without further comment that $E$ and $\F$ are sufficiently large; they will serve as coefficient fields. Let $\mathrm{Art}(\cO)$ denote the category of local artinian $\cO$-algebras with residue field $\F$.

If $F$ is a number field and $v$ is a place of $F$ above $p,$ an {\em inertial type} for $F_v$ is a two-dimensional $E$-representation $\tau_v$ of the inertia group $I_{F_v}$ with open kernel, which can be extended to $G_{F_v}.$ Under Henniart's inertial local Langlands correspondence \cite[Appendice A]{BM}, a {\em non-scalar} tame inertial type $\tau$ corresponds to an irreducible $E$-representation $\s(\tau)$ of $\GL_2(\OC_{F_v})$ that arises by inflation from an irreducible $E$-representation of $\GL_2(k_{F_v})$ which is either a principal series or a cuspidal representation. Such $\s(\tau)$ is called a {\em tame type}.

If $F$ is a $p$-adic field, $V$ is a de Rham representation of $G_F$ over $E,$ and $\kappa: F\into E,$ then we will write ${\rm HT}_{\kappa} (V)$ for the multiset of Hodge-Tate weights of $V$ with respect to $\kappa.$ By definition, ${\rm HT}_{\kappa} (V)$ consists of $- i$ with multiplicity $\dim_E(V \otimes_{\kappa, F} \wh{\Fov}(i))^{G_F}$, e.g. ${\rm HT}_{\kappa} (\e) = \{1\}$ at all embedding $\kappa.$

Throughout this paper we fix $L$ a finite \emph{unramified} extension of $\Q_p$ of degree $f\defn [L:\Q_p]$. Denote by $\cO_L$ the ring of integers in $L$ and $\F_q$ the residual field of $\cO_L$ where $q \defn p^f$.  For $\lambda\in \F_q$, $[\lambda]\in\cO_L$ denotes its Teichm\"uller lift.

If $G$ is a $p$-adic analytic group, we denote by $\rRep_{\F}(G)$
(resp. $\rRep_{\F}^{\rm ladm}(G)$, resp. $\rRep_{\F}^{\rm adm}(G)$) the category of smooth (resp. locally admissible, resp. admissible) representations of $G$ on $\F$-vector spaces.
If $\zeta: Z(G)\to \F^{\times}$ is a continuous character of the center of $G$ then we add a subscript $\zeta$ to indicate that we consider only those representations on which the center $Z(G)$ acts via $\zeta.$ For example, $\rRep_{\F,\zeta}(G)$ is the full subcategory of $\rRep_{\F}(G)$  consisting of smooth representations on which $Z(G)$ acts by the character $\zeta.$

If $M$ is a linear-topological $\cO$-module (i.e. it has a topology for which both addition and the action of $\cO$ are continuous), then $M$ has a fundamental system of open neighborhoods of zero which are $\cO$-submodules. We write $M^{\vee}$ for its Pontryagin dual $\Hom^{\rm cont}_{\cO} (M, E/\cO),$ where $E/\cO$ is equipped with the discrete topology, and we give $M^{\vee}$ the compact open topology. The Pontryagin duality functor $M\mapsto M^{\vee} $ induces an anti-equivalence of categories between the category of discrete $\cO$-modules and the category of pseudo-compact $\cO$-modules.

If $M$ is a torsion free linear-topological $\cO$-module,  $M^{d}$ denotes its Schikhof dual $\Hom^{\rm cont}_{\cO} (M, \cO).$ The Schikhof duality functor $M \mapsto M^d$ induces an anti-equivalence of categories between the category of pseudo-compact torsion free  $\cO$-modules and the category of $\varpi$-adically complete and separated torsion free  $\cO$-modules.

If $R$ is a ring (e.g. $R=\F[G]$ for a group $G$) and $M$ is a left $R$-module, we denote by $\soc_R(M)$ (resp. $\mathrm{cosoc}_R(M)$) the socle (resp. cosocle) of $M$ (see Definition \ref{defn:app-socle}). Inductively we define the socle (resp. cosocle)  filtration   of $M$. 
If $M$ is of finite length, we denote by $\JH(M)$ the {\em multi-set} of composition factors (also called Jordan--H\"older factors) of $M$. If $\sigma$ is a simple $R$-module, we let  $[M:\sigma]$ denote the multiplicity of $\sigma$ in $\JH(M)$. \vspace{2mm}

Throughout the paper, we assume $p>2$.

\section{Finite representation theory I} \label{section-finiteRep}

In this section, we study smooth representation theory of $\GL_2(\cO_L)$.

First introduce some notation. Recall that $L$ is  a (fixed) finite unramified extension of $\Q_p$ of degree $f$. Let \[K\defn\GL_2(\cO_L)\] and $K_1$ be the first principal congruence subgroup, i.e. the kernel of the mod $p$ reduction morphism $K\twoheadrightarrow\GL_2(\F_q)$.  Let $Z$ be the center of $G$ and $Z_1=Z\cap K_1$. Let $\fm_{K_1}=\fm_{K_1/Z_1}$ denote the maximal ideal of the Iwasawa algebra $\F[\![K_1/Z_1]\!]$, which carries a conjugate action of $K$.  Denote
\[\Gamma\defn\F[\GL_2(\F_q)]\cong \F[\![K/Z_1]\!]/\fm_{K_1},\ \ \ \tGamma\defn \FKZ/\fm_{K_1}^2.\]
Note that $\tGamma$ is a finite dimensional $\F$-algebra but not a group algebra.
 Let $I\subset K$ denote the (upper) Iwahori subgroup,  $I_1\subset I$  the pro-$p$-Iwahori subgroup and
 \[H\defn\Big\{\matr{[\lambda]}00{[\mu]}, \lambda,\mu\in\F_q^{\times}\Big\}.\]
We have $I= H\ltimes I_1$.

We call a \emph{Serre weight} an isomorphism class of  irreducible representations of $\Gamma$ over $\bFp$. We take $\F$ large enough and fix an embedding $\F_q \into \F$, so that any Serre weight is defined over $\F$. Then a Serre weight is (up to isomorphism) of the form (\cite[Prop.~1]{BL})
\[\Sym^{r_0}\F^2\otimes_{\F} (\Sym^{r_1}\F^2)^{\rm Fr}\otimes_{\F} \cdots\otimes_{\F} (\Sym^{r_{f-1}}\F^2)^{{\rm Fr}^{f-1}}\otimes_{\F} \eta\circ\det\]
where $0\leq r_i \leq p-1$,  $\eta$ is a character of $\F_q^{\times}$ and ${\rm Fr}: \smatr{a}bcd \mapsto \smatr{a^p}{b^p}{c^p}{d^p}$ is the Frobenius on $\Gamma$. Following \cite{BP}, we denote this representation by $(r_0 ,\cdots, r_{f-1})\otimes\eta$.

For $n\geq0$, we say a Serre weight $(r_0,\cdots,r_{f-1})\otimes\eta$ is \emph{$n$-generic}, 
if
\[n\leq r_i\leq p-2-n,\ \ \forall  0\leq i\leq f-1.\]
Note that the existence of an $n$-generic Serre weight implies implicitly  $p\geq 2n+2$.

If $\sigma$ is a Serre weight, let $\rProj_{\Gamma}\sigma$ (resp. $\rProj_{\tGamma}\sigma$) be a projective envelope of $\sigma$ in the category of $\Gamma$-representations (resp. $\tGamma$-representations), and $\rInj_{\Gamma}\sigma$ (resp. $\rInj_{\tGamma}\sigma$) be an injective envelope of $\sigma$ in the category of $\Gamma$-representations (resp. $\tGamma$-representations). Note that $\Proj_{\Gamma}\sigma$ is isomorphic to $\rInj_{\Gamma}\sigma$, but  $\Proj_{\tGamma}\sigma$ is not isomorphic to $\rInj_{\tGamma}\sigma$. However, we have the following fact: if $\zeta$ denotes the central character of $\sigma$ then \begin{equation}\label{eq:Proj-Inj}(\rInj_{\tGamma}\sigma)^{\vee}\cong \Proj_{\tGamma}\sigma^{\vee}\cong (\Proj_{\tGamma}\sigma)\otimes\zeta^{-1}\circ\det\end{equation}
where the second isomorphism holds as $\sigma^{\vee}\cong \sigma\otimes\zeta^{-1}\circ\det$. 

If $\lambda=(\lambda_i(x_i))_{0\leq i\leq f-1}$ is an $f$-tuple with $\lambda_i(x_i)\in\Z\pm x_i$, we define (following \cite[\S2]{BP})
\[e(\lambda)\defn \frac{1}{2}\Big(\sum_{i=0}^{f-1}p^i(x_i-\lambda_i(x_i))\Big)\ \ \mathrm{if}\ \lambda_{f-1}(x_{f-1})\in\Z+x_{f-1}\]
\[e(\lambda)\defn\frac{1}{2}\Big(p^f-1+\sum_{i=0}^{f-1}p^i(x_i-\lambda_i(x_i))\Big) \ \ \mathrm{otherwise}.\ \ \ \ \ \ \ \ \]
One checks that $e(\lambda)\in\Z\oplus (\oplus_{i=0}^{f-1}\Z x_i)$, see \cite[Lem.~2.1]{BP}.
If $\sigma=(r_0,\cdots,r_{f-1})\otimes\eta$ is a Serre weight, we define
\begin{equation}\label{eq:defn-lambda(sigma)}\lambda(\sigma):=(\lambda_0(r_0),\cdots,\lambda_{f-1}(r_{f-1}))\otimes {\det}^{e(\lambda)(r_0,\cdots,r_{f-1})}\eta\end{equation}
provided that  $0\leq \lambda_i(r_i)\leq p-1$ for all $i$, in which case we call $\lambda(\sigma)$ the \emph{evaluation} of $\lambda$ at $\sigma$; otherwise we leave $\lambda(\sigma)$ undefined.

The following lemma  is an exercise left  
in the proof of \cite[Lem.~12.8]{BP}. 
\begin{lemma}\label{lemma:genericity}
Let $\sigma=(r_0,\cdots,r_{f-1})\otimes\eta$. Let $\lambda,\lambda'$ be two $f$-tuples with $\lambda_i(x_i),\lambda_i'(x_i)\in\Z\pm x_i$ and such that
\[0\leq \lambda_i(r_i),\lambda'_i(r_i)\leq p-1,\ \ \forall 0\leq i\leq f-1.\]
Assume that $\lambda,\lambda'$ satisfy the following condition:
\begin{equation}\label{eq:cond-lambda}
 \lambda_i(x_i)=\lambda_i'(x_i)\ \Longrightarrow\ \lambda_{i-1}(x_{i-1})-\lambda_{i-1}'(x_{i-1})\in\Z,\ \ \forall 0\leq i\leq f-1.\end{equation}
Then $\lambda(\sigma)\cong\lambda'(\sigma)$ as Serre weights if and only if $\lambda=\lambda'$ as $f$-tuples.
\end{lemma}

\begin{proof}
The sufficiency is trivial. To prove the necessity, assume $\lambda(\sigma)\cong\lambda'(\sigma)$. This is equivalent to
\begin{equation}\label{eq:equal-lambda}\lambda_i(r_i)=\lambda_i'(r_i),\ \ \forall 0\leq i\leq f-1\end{equation}
and
\begin{equation}\label{eq:equal-det}e(\lambda)(r_0,\cdots,r_{f-1})\equiv e(\lambda')(r_0,\cdots,r_{f-1})\ \ (\mathrm{mod}\ p^f-1).\end{equation}

If  $\lambda_{f-1}(x_{f-1})-\lambda'_{f-1}(x_{f-1})\in \Z$, equivalently, either both $\lambda_{f-1}(x_{f-1})$ and $\lambda'_{f-1}(x_{f-1})$ lie in $\Z+x_{f-1}$ or both in $\Z-x_{f-1}$, then \eqref{eq:equal-lambda} implies the equality $\lambda_{f-1}(x_{f-1})=\lambda'_{f-1}(x_{f-1})$. Hence, the assumption \eqref{eq:cond-lambda} allows one to show inductively that $\lambda_i(x_i)=\lambda'_{i}(x_i)$ for all $0\leq i\leq f-1$.

If $\lambda_{f-1}(x_{f-1})-\lambda'_{f-1}(x_{f-1})\notin \Z$, equivalently, one of $\lambda_{f-1}(x_{f-1})$ and $\lambda'_{f-1}(x_{f-1})$ lies in $\Z+x_{f-1}$ and the other lies in $\Z-x_{f-1}$, then by definition of $e(\lambda)$ and \eqref{eq:equal-lambda} we have
\[e(\lambda)(r_0,\cdots,r_{f-1})-e(\lambda')(r_0,\cdots,r_{f-1})=\pm \frac{1}{2}(p^f-1),\]
a contradiction to \eqref{eq:equal-det}. This finishes the proof.
\end{proof}

Given a Serre weight $\sigma$, it is well-known that $\sigma^{I_1}$ is one-dimensional over $\F$, and  we denote by $\chi_{\sigma}$ the character of $H$, also of $I$, acting on $\sigma^{I_1}$.  Given a character $\chi:H\ra \F^{\times}$, denote by $\chi^s$  the character  $\chi^s(h)\defn\chi(shs)$ for $s=\smatr{0}110$ and $h\in H$. If $\chi\neq \chi^s$, then there exists a \emph{unique} Serre weight denoted by $\sigma_{\chi}$ such that $H$ acts on $\sigma_{\chi}^{I_1}$ via $\chi$.

For convenience, we often write $\cS$ for the set $\Z/f\Z$, identified with $\{0,\cdots,f-1\}$.

\subsection{The structure of $\rInj_{\Gamma}\sigma$  } \label{subsection:Gamma}

Let $\sigma$ be a Serre weight. The structure of $\rInj_{\Gamma}\sigma$ is studied in detail in \cite[\S3-\S4]{BP} and generalized in \cite[\S5]{EGS}.   We first recall the following useful result.

\begin{proposition}\label{prop-multione-BP}
Let $\tau$ be a Jordan--H\"older factor of $\rInj_{\Gamma}\sigma$. Among those subrepresentations of $\rInj_{\Gamma}\sigma$ whose cosocle is isomorphic to $\tau$, there is a unique one, denoted by $I(\sigma,\tau)$, such that $\sigma$ occurs with multiplicity $1$. Moreover, $I(\sigma,\tau)$ is multiplicity free.
\end{proposition}
\begin{proof}
This is \cite[Cor.~3.12]{BP} if $\sigma$ is $0$-generic, and is \cite[Lem.~5.1.9]{EGS} in general.
\end{proof}

\begin{corollary}\label{cor-multi-transfer}
Let   $V$ be a subrepresentation of  $(\rInj_{\Gamma}\sigma)^{\oplus s}$ for some $s\geq 1$. Then for any Serre weight $\tau$, we have  $[V:\sigma]\geq[V:\tau]$. If, moreover, $\mathrm{cosoc}_{\Gamma}(V)$ is isomorphic to $\tau^{\oplus r}$ for some $r\geq 1$, then $[V:\sigma]=[V:\tau]$.
\end{corollary}
\begin{proof}
Using that $\mathrm{soc}_{\Gamma}(V)$ has the form $\sigma^{\oplus s'}$ for some $s'\leq s$, we can construct a finite filtration of $V$  such that  each graded piece  has socle isomorphic to $\sigma$ and $\sigma$ occurs only once there. Hence we are reduced to the situation in which $\mathrm{soc}_{\Gamma}(V)=\sigma$ and $[V:\sigma]=1$, and the result follows from Proposition \ref{prop-multione-BP}. The second assertion is clear by duality using \eqref{eq:Proj-Inj}.
\end{proof}

Following \cite[\S3]{BP}, the Jordan--H\"older factors of $\rInj_{\Gamma}\sigma$ can be described as follows.
Let $x_0,\dots,x_{f-1}$ be variables, and define the set $\cI(x_0,\cdots,x_{f-1})$ of  $f$-tuples $\lambda\defn (\lambda_0(x_0),\cdots,\lambda_{f-1}(x_{f-1}))$, where $\lambda_0(x_0)\in\{x_0,p-2-x_0\pm1\}$ if $f=1$, and if $f>1$ then
\begin{enumerate}
\item[(i)] $\lambda_i(x_i)\in\{x_i,x_i\pm1,p-2-x_i,p-2-x_i\pm1\}$ for $i\in\cS$
\item[(ii)] if $\lambda_i(x_i)\in\{x_i,x_i\pm1\}$, then $\lambda_{i+1}(x_{i+1})\in\{x_{i+1},p-2-x_{i+1}\}$
\item[(iii)] if $\lambda_i(x_i)\in\{p-2-x_i,p-2-x_i\pm1\}$, then $\lambda_{i+1}(x_{i+1})\in\{x_{i+1}\pm1,p-2-x_{i+1}\pm1\}$
\end{enumerate}
with the convention $x_{f}\defn x_0$ and $\lambda_{f}(x_f)\defn\lambda_0(x_0)$. By \cite[Lem.~3.2]{BP}, each Jordan--H\"older factor of $\rInj_{\Gamma}\sigma$ is isomorphic to $\lambda(\sigma)$ (see \eqref{eq:defn-lambda(sigma)}) for a uniquely determined $\lambda\in\cI(x_0,\cdots,x_{f-1})$. If $\sigma$ is $1$-generic, then $\lambda(\sigma)$ is a genuine Serre weight for any $\lambda\in\cI(x_0,\cdots,x_{f-1})$.

 For $\lambda\in \cI(x_0,\cdots,x_{f-1})$, set
\[\cS(\lambda)\defn\big\{i\in\cS:\ \lambda_i(x_i)\in\{x_i\pm1,p-2-x_i\pm1\}\big\}.\]
By abuse of notation, we also write $\cS(\tau)=\cS(\lambda)$ if $\tau=\lambda(\sigma)$.

Recall from \cite[Def.~4.10]{BP} that, given $\lambda,\lambda'\in\cI(x_0,\cdots,x_{f-1})$, we say that $\lambda$ and $\lambda'$ are \emph{compatible} if, whenever $\lambda_i(x_i),\lambda_i'(x_i)\in\{x_i\pm1, p-2-x_i-\pm1\}$,   the signs of the $\pm 1$ are the same in $\lambda_i(x_i)$ and $\lambda'_i(x_i)$.   Note that, if $\cS(\lambda)\cap \cS(\lambda')=\emptyset$, then $\lambda $ and $\lambda'$ are always compatible.
The following result  determines the structure of $I(\sigma,\tau)$, see \cite[Cor.~4.11]{BP} and \cite[\S5]{EGS}.
\begin{proposition}\label{prop-BP-4.11}
Let $\tau,\tau'\in\JH(\rInj_{\Gamma}\sigma)$ which correspond to $\lambda,\lambda'\in\cI(x_0,\break\cdots,x_{f-1})$, respectively. Then $\tau'$ occurs in $I(\sigma,\tau)$ as a subquotient if and only if $\lambda'\leq \lambda$, meaning that $\cS(\lambda')\subseteq\cS(\lambda)$ and $\lambda$, $\lambda'$ are compatible.
\end{proposition}

The notion of compatibility can be defined for more general $f$-tuples $\nu=(\nu_i(x_i))_i$ with $\nu_i(x_i)\in\Z\pm x_i$ in an obvious way: given $\nu,\nu'$ and $i\in\cS$, we say $\nu$ and $\nu'$  are \emph{compatible at $i$} if, 
whenever $\nu_i(x_i),\nu_i'(x_i)\in\{x_i\pm1,p-2-x_i-\pm1\}$, the signs of the $\pm1$ are the same. We say $\nu$ and $\nu'$ are \emph{compatible} if they are compatible at all $i\in\cS$. Also set
\[\cS(\nu)\defn\big\{i\in\cS:\nu_i(x_i)\in\{x_i\pm1,p-2-x_i\pm1\}\big\}.\]
We say $\nu\leq \nu'$ if $\cS(\nu)\subseteq \cS(\nu')$ and $\nu,\nu'$ are compatible. Note that if $\nu_1,\nu_2\leq \nu'$ for some common $\nu'$, then $\nu_1$ and $\nu_2$ are automatically compatible.
\vspace{1mm}

We establish some combinatorial lemmas on $\cI(x_0,\cdots,x_{f-1})$ which will be used in \S\ref{section-BP}.

\begin{lemma}\label{lemma-HW-2.19}
Let $\lambda,\lambda'\in\cI(x_0,\cdots,x_{f-1})$. Let $\cS''$ be a subset of $\cS(\lambda)\cap\cS(\lambda')$ such that $\lambda$ and $\lambda'$ are compatible at any $i\in\cS''$. Then there exists a unique $\lambda''\in\cI(x_0,\cdots,x_{f-1})$ with $\cS(\lambda'')=\cS''$ and such that $\lambda''$ is compatible with both $\lambda$ and $\lambda'$.
\end{lemma}
\begin{proof}
This is a direct check; a similar check can be found in \cite[Lem.~2.19]{HW}. Note that, in the special case $\cS''=\cS(\lambda)\cap\cS(\lambda')$, i.e. $\lambda$ and $\lambda'$ are compatible,  $\lambda''$ is given by the intersection $\lambda\cap\lambda'$, see  \cite[Lem.~12.5]{BP} and the construction before it.
\end{proof}

\begin{lemma}\label{lemma-compose-lambda}
Let $\lambda,\lambda'$ be $f$-tuples with
 \[\lambda_i(x_i),\lambda_i'(x_i)\in\big\{x_i,x_i\pm1, p-2-x_i,p-2-x_i\pm1\big\},\ \ \forall i\in\cS.\]

(i) $\lambda\circ\lambda'$ is compatible with $\lambda'$ and $ \cS(\lambda\circ\lambda')=\cS(\lambda)\Delta\cS(\lambda')$.

(ii) If $i\notin\cS(\lambda)\cap\cS(\lambda')$, then
\[(\lambda\circ\lambda')_i(x_i)\in\{x_i,x_i\pm1, p-2-x_i,p-2-x_i\pm1\}.\]
\end{lemma}

 \begin{proof}
 This is a direct check using the following table.
{\tiny{ 
 \begin{table}[h]
 \tabcolsep=3pt
\begin{tabu}{ |c|[1pt]c|c|c|c|c|c|}
\hline
$\lambda_i(\lambda'_i(x_i))$& $\lambda'_i(x_i)=x_i$ &  $x_i-1$ & $x_i+1$ &$p-2-x_i$&$p-1-x_i$&$p-3-x_i$   \\
\tabucline[1pt]{-}
$\lambda_i(x_i)=x_i$&$x_i$ & $x_i-1$&$x_i+1$&$p-2-x_i$&$p-1-x_i$&$p-3-x_i$  \\
\hline
$x_i-1$&$x_i-1$&$x_i-2$&$x_i$&$p-3-x_i$&$p-2-x_i$&$p-4-x_i$\\
\hline
$x_i+1$&$x_i+1$&$x_i$&$x_i+2$&$p-1-x_i$&$p-x_i$&$p-2-x_i$\\
\hline
$p-2-x_i$&$p-2-x_i$&$p-1-x_i$&$p-3-x_i$ & $x_i$& $x_i-1$&$x_i+1$\\
\hline
$p-1-x_i$& $p-1-x_i$& $p-x_i$& $p-2-x_i$& $x_i+1$ & $x_i $& $x_i+2$\\
\hline
$p-3-x_i$ & $p-3-x_i$& $p-2-x_i$ & $p-4-x_i$& $x_i-1$& $x_i-2$& $x_i$\\
\hline
\end{tabu}
\end{table}
}}
\end{proof}

\begin{lemma}\label{lemma:cond-for-cI}
Given $\lambda,\lambda'\in\cI(x_0,\cdots,x_{f-1})$, the condition \eqref{eq:cond-lambda} of Lemma \ref{lemma:genericity} is satisfied for $\lambda$ and $\lambda'$. Moreover, given $\lambda,\lambda',\mu,\mu'\in\cI(x_0,\cdots,x_{f-1})$, \eqref{eq:cond-lambda} is satisfied for $\lambda\circ\mu$ and $\lambda'\circ\mu'$.
\end{lemma}
\begin{proof}
The first assertion is immediate by definition of $\cI(x_0,\cdots,x_{f-1})$. The second one is a direct check using the table in the proof of Lemma \ref{lemma-compose-lambda}.
\end{proof}

\begin{lemma}\label{lemma-HW-2.20}
Let $\tau,\tau'\in\JH(\rInj_{\Gamma}\sigma)$ and assume $\cS(\tau)\cap\cS(\tau')=\emptyset$. Then $\tau'$ is a subquotient of $\rInj_{\Gamma}\tau$ and $I(\tau,\tau')$ contains $\sigma$ as a subquotient.
\end{lemma}

\begin{proof}
Let $\lambda,\lambda'\in\cI(x_0,\cdots,x_{f-1})$ be the elements corresponding to $\tau,\tau'$ respectively. Let $\nu=\lambda'\circ\lambda^{-1}$, where $\lambda^{-1}$ is the unique element in $\cI(x_0,\cdots,x_{f-1})$ defined by demanding the formal identities $\lambda_i^{-1}(\lambda_i(x_i))=x_i$  for all $i\in\cS$. Then $\nu$ is an $f$-tuple  with $\nu_i(x_i)\in\Z\pm x_i$ and such that $\lambda'=\nu\circ\lambda$. It is clear that $\cS(\lambda^{-1})=\cS(\lambda)$, so  $\cS(\lambda')\cap \cS(\lambda^{-1})=\emptyset$ by assumption and Lemma \ref{lemma-compose-lambda}(ii) implies that
  \[\nu_i(x_i)\in\{x_i,x_i\pm1,p-2-x_i,p-2-x_i\pm1\}\]
  for all $i\in\cS$. Moreover, using the fact $\cS(\lambda')\cap\cS(\lambda^{-1})=\emptyset$, one checks as in the proof of \cite[Lem.~2.20(iii)]{HW} that $\nu$ is actually an element of $\cI(x_0,\cdots,x_{f-1})$ and by construction $\tau'=\nu(\tau)$. To see that $\sigma$  occurs in $I(\tau,\tau')$,  by Proposition \ref{prop-BP-4.11} it is equivalent to check that $\lambda^{-1}\leq \nu$, but this follows from Lemma \ref{lemma-compose-lambda}(i) as $\cS(\lambda')\cap \cS(\lambda^{-1})=\cS(\lambda')\cap \cS(\lambda)=\emptyset$.
\end{proof}

\subsection{The structure of $\rInj_{\tGamma}{\sigma}$}
Let $\sigma$ be a Serre weight. In this subsection, we study the structure of $\rInj_{\tGamma}\sigma$ under some genericity condition on $\sigma$.

We have a short exact sequence (e.g. \cite[Prop.~18.4]{Al})
\begin{equation}\label{equation-P/J^2}0\ra \rInj_{\Gamma}\sigma\ra {\rInj_{\tGamma}\sigma}\ra\rInj_{\Gamma}\sigma \otimes_{\F} H^1(K_1/Z_1,\F)  \ra0\end{equation}
with $H^1(K_1/Z_1,\F)$ being equipped with the conjugate action of $K$. This action of $K$ on $H^1(K_1/Z_1,\F)$ factors through $\Gamma$. By  \cite[Prop.~5.1]{BP} we have a decomposition
\[H^1(K_1/Z_1,\F)\cong  \bigoplus_{i\in\cS}V_{2,i},\]
where $V_{2,i}$ denotes the $\Gamma$-representation $(\Sym^2\F^2\otimes{\det}^{-1})^{\mathrm{Fr}^i}$ for $i\in\cS$. Remark that $V_{2,i}$ is self-dual in the sense that $(V_{2,i})^{\vee}\cong V_{2,i}$.

\begin{definition}\label{def:delta}
For $*\in\{+,-\}$ and $i\in\cS$, we define two $f$-tuples $\delta_i^*$ and $\mu_{i}^*$  as follows.
\begin{enumerate}
\item[$\bullet$]  $(\delta_i^*)_i(x_i)=x_i*2$ and $(\delta_i^{*})_j(x_j)=x_j$ if $j\neq i$. 
\item[$\bullet$]  If $f=1$, $\mu_0^+(x_0)=p-3-x_0$ and $\mu_0^-(x_0)=p-1-x_0$.\footnote{We caution that the definition in the case $f=1$ is different from the one of \cite[Def.~2.8]{HW}.}
 If $f\geq 2$, $(\mu_i^*)_i(x_i)=x_i*1$, $(\mu_i^*)_{i-1}(x_{i-1})=p-2-x_{i-1}$, and $(\mu_i^*)_j(x_j)=x_j$ if $j\notin \{i,i-1\}$. It is direct to check that $\mu_i^{*}\in\cI(x_0,\cdots,x_{f-1})$.
\end{enumerate}
\end{definition}
We make the convention that $-*=-$ if $*=+$, and $-*=+$ if $*=-$. By definition, we have
\begin{equation}\label{eq:delta=mumu}\delta_i^*=\left\{\begin{array}{ll}\mu_i^{-*}\circ\mu_i^{*} & f=1\\
\mu_i^*\circ\mu_i^*& f\geq 2\end{array}\right. \end{equation}
\begin{equation}\label{eq:id=mumu}
(x_0,\cdots,x_{f-1})=\left\{\begin{array}{ll}\mu_i^*\circ\mu_i^{*} & f=1\\
\mu_i^{-*}\circ\mu_i^{*}& f\geq 2.\end{array}\right.
\end{equation}
Due to these facts, we sometimes need to  discuss separately  these two cases.

By \cite[Cor.~5.6]{BP}, $\mu_i^*(\sigma)$ are exactly the set of Serre weights which have non-trivial $\Gamma$-extensions with $\sigma$. More precisely,
$\dim_{\F}\Ext^1_{\Gamma}(\tau,\sigma)=1$
 if and only if $\tau=\mu_i^*(\sigma)$ for some pair $(i,*)$. Denote
 \[ \mathscr{E}(\sigma)\defn \big\{\mu_i^{*}(\sigma): i\in\cS, *\in\{+,-\}\big\},\] forgetting the undefined ones. It is clear that $\tau\in\mathscr{E}(\sigma)$ if and only if $\sigma\in\mathscr{E}(\tau)$.
 The following result will be frequently used later on.

\begin{lemma}\label{lemma-Hu10-2.21}
(i) Let $\sigma$ be a Serre weight. Assume $\sigma$ is $0$-generic if $f\geq 2$, and $\sigma\cong\Sym^r\F^2$ (up to twist) for $0\leq r\leq p-3$ if $f=1$. Then $\Ext^1_{K/Z_1}(\sigma,\sigma)=0$.

(ii) Let   $\sigma,\sigma'$ be $0$-generic Serre weights and assume $\sigma\neq\sigma'$. We have isomorphisms
\[\Ext^1_{K/Z_1}(\sigma,\sigma')\cong \Ext^1_{\tGamma}(\sigma,\sigma')\cong\Ext^1_{\Gamma}(\sigma,\sigma'),\]
which are nonzero (hence have dimension $1$ over $\F$) if and only if $\sigma'\in\mathscr{E}(\sigma)$.
\end{lemma}
\begin{proof}
(i) For $f\geq 2$,  it is proved in \cite[Prop.~2.21]{Hu10}. 
Below we give a simplified  proof (based on  \emph{loc.~cit.}) which treats both cases. 
If $f=1$, remark that $\Ext^1_{K/Z_1}(\Sym^{p-2}\F^2,\Sym^{p-2}\F^2)\neq0 $ and $\Ext^{1}_{K/Z_1}(\Sym^{p-1}\F^2,\Sym^{p-3}\F^2\otimes{\det})\neq0 $.
 
 For a contradiction, let $0\ra \sigma\ra V\ra \sigma\ra0$ be a nonsplit $K$-extension on which $Z_1$ acts trivially. Let $w\in V$ be an $H$-eigenvector of character $\chi_{\sigma}$ such that its image in the quotient $\sigma$ is nonzero and lies in $\sigma^{I_1}$. We will prove that $w$ is fixed by $I_1$, thus by Frobenius reciprocity we obtain a $K$-equivariant surjection $\Ind_I^K\chi_{\sigma}\twoheadrightarrow V$ which is impossible. 
Firstly, as in the proof of \cite[Prop.~2.21]{Hu10}, $w$ is fixed by $\smatr{1}{\cO_L}01$ because none of the $H$-characters $\{\chi\alpha_i,i\in\cS\}$ can occur in $V|_H$ by the genericity of $\sigma$ (this needs the assumption $r\leq p-3$ when $f=1$).  Secondly, let $\overline{N}_k=\smatr{1}0{p^k\cO_L}1$ for $k\geq 0$ and  consider the following operators  (recall that we have fixed an embedding $\F_{q}\hookrightarrow \F$)
\[X_i:=\sum_{\lambda\in\F_q}\lambda^{-p^i}\matr{1}0{[\lambda]}1\in\F[\![\overline{N}_0]\!],\ \ i\in\cS.\]   
It is easy to see that $X_iw$ has $H$-eigencharacter $\chi\alpha_i^{-1}$. If we write $\sigma=(r_0,\dots,r_{f-1})\otimes\eta$, then none of  $\{\chi\alpha_i^{-(r_i+1)}, i\in\cS\}$ can occur in $\sigma|_H$  by the genericity of $\sigma$, see \cite[Lem.~2.7]{BP}. We deduce that  $X_i^{r_i+1}w=0$, and so $X_i^{p}w=0$ for all $i\in\cS$. Since $\{X_i^p, i\in\cS\}$ topologically generate the maximal ideal of $\F[\![\overline{N}_1]\!]$, $w$ is fixed by $\overline{N}_1$. Since $I_1/Z_1$ is generated by $\smatr{1}{\cO_K}01$ and $\overline{N}_1$,  $w$ is fixed by $I_1$ as claimed.  

(ii) It is a consequence of \cite[Cor.~5.6]{BP}.
\end{proof}

 On the other hand,  denote  \[  \Delta(\sigma) \defn\big\{\delta_{i}^{*}(\sigma): i\in\cS, *\in\{+,-\}\big\}, \]
again forgetting the undefined ones.

\begin{lemma}\label{lemma-Delta-nocommonfactor}
Let $\sigma_1,\sigma_2\in \JH(\rInj_{\Gamma}\sigma)$ be compatible. If $\sigma_1\neq\sigma_2$,  then
\[\big(\{\sigma_1\}\cup \Delta(\sigma_1)\big) \cap \big(\{\sigma_2\}\cup\Delta(\sigma_2)\big)=\emptyset.\]
\end{lemma}
\begin{proof}
We only give the proof of the assertion $\Delta(\sigma_1)\cap \Delta(\sigma_2)=\emptyset$; the other cases can be treated  similarly. It is equivalent to showing that the equation $(\delta_{i_1}^{*_1}\circ \lambda_1)(\sigma)=(\delta_{i_2}^{*_2}\circ\lambda_2)(\sigma)$, where $\lambda_1,\lambda_2\in\cI(x_0,\cdots,x_{f-1})$ are compatible, has no solution other than $(i_1,*_1)=(i_2,*_2)$ and $\lambda_1=\lambda_2$. By definition of $\cI(x_0,\cdots,x_{f-1})$, it is clear that the condition \eqref{eq:cond-lambda} holds for the pair $(\delta_{i_1}^{*_1}\circ\lambda_1, \delta_{i_2}^{*_2}\circ\lambda_2)$,  so Lemma \ref{lemma:genericity} applies and implies $\delta_{i_1}^{*_1}\circ\lambda_1=\delta_{i_2}^{*_2}\circ\lambda_2$ as $f$-tuples. If $i\notin\{i_1,i_2\}$, then it is obvious that $(\lambda_1)_i(x_i)=(\lambda_2)_i(x_i)$. If $i\in\{i_1,i_2\}$, then we must have $(\lambda_1)_i(x_i)-(\lambda_2)_i(x_i)\in\{0,\pm2,\pm4\}$, and the definition of $\cI(x_0,\cdots,x_{f-1})$ and the compatibility between $\lambda_1$ and $\lambda_2$ force that $(\lambda_1)_i(x_i)-(\lambda_2)_i(x_i)=0$. Hence $\lambda_1=\lambda_2$, and consequently $(i_1,*_1)=(i_2,*_2)$.
\end{proof}

 By \cite[Prop.~5.1, Prop.~5.4]{BP}, if $\sigma\cong (r_0,\cdots,r_{f-1})$ up to twist with $0\leq r_i\leq p-3$ for all $i$, then
\begin{equation}\label{eq:sigma-H1}\sigma\otimes_{\F} H^1(K_1/Z_1,\F)\cong \sigma^{\oplus f}\oplus \big(\oplus_{\delta\in\Delta(\sigma)}\delta\big).\end{equation}
In general, for \emph{any} Serre weight $\sigma$, we have by \cite[Cor.~5.5]{BP}
\begin{equation}\label{eq:sigma-H1-socle}\soc_{\Gamma}\big(\sigma\otimes_{\F}H^1(K_1/Z_1,\F)\big)\cong \sigma^{\oplus f}\oplus \big(\oplus_{\delta\in\Delta(\sigma)}\delta\big).\end{equation}

\begin{proposition}\label{prop-structure-JProj}
(i) Assume $\sigma$ is $2$-generic and, if $f=1$, $\sigma\ncong\Sym^{2}\F^2\otimes{\det}^a$.
Then there is an isomorphism of $\Gamma$-representations
\[  (\rInj_{\Gamma}\sigma)\otimes_{\F} H^1(K_1/Z_1,\F)\cong \big(\rInj_{\Gamma}\sigma\big)^{\oplus f}\oplus  \big(\oplus_{\delta\in\Delta(\sigma)} \rInj_{\Gamma}\delta\big).\]

(ii) If $f=1$ and $\sigma\cong\Sym^2\F^2\otimes{\det}^a$, then
\[(\rInj_{\Gamma}\sigma)\otimes_{\F} H^1(K_1/Z_1,\F)\cong (\rInj_{\Gamma}\sigma)^{\oplus f}\oplus  (\oplus_{\delta\in\Delta(\sigma)} \rInj_{\Gamma}\delta)\oplus (\Sym^{p-1}\F^2\otimes{\det}^{a+1}). \]
\end{proposition}
\begin{proof}
(i) It is a general fact that $(\rInj_{\Gamma}\sigma)\otimes_{\F} H^1(K_1/Z_1,\F)$ is again an injective $\Gamma$-representation, see \cite[Lem.~7.4]{Al} (combined with \cite[Thm.~6.4]{Al}). Hence the natural embedding $\sigma\otimes_{\F} H^1(K_1/Z_1,\F)\hookrightarrow (\rInj_{\Gamma}\sigma)\otimes_{\F} H^1(K_1/Z_1,\F)$ extends to an embedding
\begin{equation}\label{eq:inj-embed}\rInj_{\Gamma}\big(\sigma\otimes_{\F} H^1(K_1/Z_1,\F)\big)\hookrightarrow (\rInj_{\Gamma}\sigma)\otimes_{\F} H^1(K_1/Z_1,\F).\end{equation}
By the genericity assumption on $\sigma$, the isomorphism \eqref{eq:sigma-H1} holds.  Moreover,  if $\delta\in\Delta(\sigma)$ and if $\delta=(s_0,\cdots,s_{f-1})$ up to twist, then $0\leq s_i\leq p-2$ for all $i\in\cS$ if $f\geq 2$ (resp. $1\leq s_0\leq p-2$ if $f=1$) and not all of $s_i$ are equal to $0$ so that $\dim_{\F}\delta\geq 2$. This implies $\dim_{\F}\rInj_{\Gamma}\delta=(2p)^f$, see e.g. \cite[\S3]{BP}.
Hence, \eqref{eq:inj-embed} is an isomorphism for the reason of dimensions.

(ii) It is a direct check using \cite[Prop.~5.4]{BP}.
\end{proof}

For the rest of this subsection, \emph{we assume that  $\sigma$ is $2$-generic}.
 We deduce from   \eqref{equation-P/J^2} and Proposition \ref{prop-structure-JProj} that, ignoring multiplicities,
 \begin{equation}\label{eq:JH-Inj}\JH(\rInj_{\tGamma}\sigma)=\JH(\rInj_{\Gamma}\sigma)\cup \big(\cup_{\delta\in\Delta(\sigma)} \JH(\rInj_{\Gamma}\delta)\big) \end{equation}
where, if $f=1$ and $\sigma\cong\Sym^{2}\F^2\otimes{\det}^a$, there is an extra Jordan--H\"older factor $\Sym^{p-1}\F^2\otimes{\det}^{a+1}$.
As a consequence, any Jordan--H\"older factor of $\rInj_{\tGamma}\sigma$ has the form $\lambda(\sigma)$ or $(\lambda\circ\delta_i^{*})(\sigma)$ for some $\lambda\in\cI(x_0,\cdots,x_{f-1})$ and $(i,*)\in\cS\times\{+,-\}$. Conversely, any Serre weight  of one of the two forms is actually a Jordan--H\"older factor of $\rInj_{\tGamma}\sigma$. Indeed, this follows from \cite[Lem 3.2(i)]{BP}, noting that  the $2$-genericity of $\sigma$ implies that $ \dim_{\F}\sigma,\dim_{\F}\delta_i^{*}(\sigma)\notin\{1, q\}$ if $f\geq 2$ or if $f=1$ and $\dim_{\F}\sigma\neq 3$; if $f=1$ and $\sigma\cong\Sym^2\F^2\otimes{\det}^a$ (so $\dim_{\F}\delta_0^{-}(\sigma)=1$), although the extra Serre weight $\Sym^{p-1}\F^2\otimes{\det}^{a+1}$ is not a Jordan--H\"older factor of $\rInj_{\Gamma}\delta_0^-(\sigma)$ (see \cite[Lem.~3.2(ii)]{BP}), it is a Jordan--H\"older factor of $\rInj_{\tGamma}\sigma$.

\textbf{Convention.} To give a uniform treatment, in the case $f=1$ and $\sigma\cong\Sym^{2}\F^2\otimes{\det}^a$ it is convenient to express $\Sym^{p-1}\F^2\otimes{\det}^{a+1}$ as $(\mu_0^{-}\circ\delta_0^-)(\sigma)$.

\begin{definition}\label{defn-newweight}
Let $\tau\in\JH(\rInj_{\tGamma}\sigma)$.
We say that $\tau$ is a \emph{new} (resp. \emph{old}) Serre weight, if $\tau$ does not occur in $\rInj_{\Gamma}\sigma$ (resp. occurs in $\rInj_{\Gamma}\sigma$) as a subquotient.
\end{definition}

For example, Serre weights in $\Delta(\sigma)$ are all new.

\begin{lemma}\label{lemma-newweight}
Let $\delta=\delta_i^{*}(\sigma)$ for some pair $(i,*)\in\cS\times\{+,-\}$ and $\lambda\in\cI(x_0,\cdots,x_{f-1})$. Then  $\lambda(\delta)$ is new (in $\rInj_{\tGamma}\sigma$) if and only if \begin{equation}\label{equation-condition-newweight}\lambda_i(x_i)\in\big\{x_i,x_i*1,p-2-x_i,p-2-x_i-(*1)\big\}.\end{equation}
As a consequence, if $\lambda(\delta)$ is new, then so is any Jordan--H\"older factor of $I(\delta,\lambda(\delta))$.
\end{lemma}
\begin{proof}
We assume $*=+$, the case $*=-$ being similar. Write $\nu=\lambda\circ\delta_i^{+}$. Assuming \eqref{equation-condition-newweight} holds, we have $\lambda_i(x_i)\in\{x_i,x_i+1,p-2-x_i,p-3-x_i\}$ and
\[\nu_i(x_i)\in\{x_i+2, x_i+3, p-4-x_i,p-5-x_i\},\]
so $\nu(\sigma)$ is not a subquotient of $\rInj_{\Gamma}\sigma$ by Lemma \ref{lemma:genericity} (cf. the proof of Lemma \ref{lemma-Delta-nocommonfactor}), namely $\nu(\sigma)$ is new.
For the converse, assume \eqref{equation-condition-newweight} does not hold, i.e.
$\lambda_i(x_i)=x_i-1$ or $p-1-x_i$, then $\nu_i(x_i)=x_i+1$ or $p-3-x_i$,  respectively. On the other hand,  $\nu_j(x_j)=\lambda_j(x_j)$ for $j\neq i$. It is then direct to check that $\nu$ defines an element in $\cI(x_0,\cdots,x_{f-1})$, i.e. $\lambda(\delta)$ is old.  This proves the first assertion
and the second one follows from this combined with Proposition \ref{prop-BP-4.11}.
\end{proof}

\begin{lemma}\label{lemma-unique-delta}
Let $\tau$ be a Jordan--H\"older factor of $\rInj_{\tGamma}\sigma$ which is new.
 Then there exists a unique $\delta\in\Delta(\sigma)$ such that $\tau$ occurs in $\rInj_{\Gamma}\delta$.
\end{lemma}

 \begin{proof}
A Jordan--H\"older factor of $\rInj_{\Gamma}\delta_i^{*}(\sigma)$ has the form $\nu(\sigma)$ with $\nu=\lambda\circ\delta_i^{*}$, for $\lambda\in\cI(x_0,\cdots,x_{f-1})$.  If $*=+$, then
\begin{equation}\label{equation-nu-case+}\nu_i(x_i)\in\{p-3-x_i,p-4-x_i,p-5-x_i,x_i+1,x_i+2,x_i+3\}, \end{equation}
while if $*=-$,  then
\begin{equation}\label{equation-nu-case-}\nu_i(x_i)\in\{p+1-x_i,p-x_i,p-1-x_i,x_i-3,x_i-2,x_i-1\}.\end{equation}
Therefore, for  any $i\in\cS$,   $\rInj_{\Gamma}\delta_i^+(\sigma)$ and $\rInj_{\Gamma}\delta_i^-(\sigma)$ can not have common Jordan--H\"older factors by Lemma \ref{lemma:genericity} (cf. the proof of Lemma \ref{lemma-Delta-nocommonfactor}).

Next, we show that if $i\neq i'$ and if $\tau$ is a common  subquotient of  $\rInj_{\Gamma}\delta_i^{*}(\sigma)$ and $\rInj_{\Gamma}\delta_{i'}^{*'}(\sigma)$, then $\tau$ is  an {\it old} Serre weight, i.e. $\tau\in\JH(\rInj_{\Gamma}\sigma)$. We only check this for $(*,*')=(+,+)$, and the other cases can be treated similarly.  Let $\lambda\in\cI(x_0,\cdots,x_{f-1})$ (resp. $\lambda'$) be the element corresponding to $\tau$ when viewed as a subquotient of $\rInj_{\Gamma}\delta_i^{+}(\sigma)$ (resp. $\rInj_{\Gamma}\delta_{i'}^{+}(\sigma)$), so that $\tau=\nu(\sigma)=\nu'(\sigma)$ where  $\nu=\lambda\circ\delta_i^+$ and $\nu'=\lambda'\circ\delta_{i'}^+$. Then  $\nu=\nu'$ in view of Lemma \ref{lemma:genericity}.  Since $i\neq i'$, we have   \begin{equation}\label{equation-nu-j}\nu_{j}(x_{j})=\nu_{j}'(x_j)\in\{x_j,x_j\pm1,p-2-x_j,p-2-x_j\pm1\} \end{equation}
for \emph{any} $j\in\cS$; indeed, this relation holds for $\nu_j(x_j)$ if $j\neq i$ and for $\nu'_j(x_j)$ if $j\neq i'$, and we recall $i\neq i'$.
One then checks that  $\nu$ defines an element of $\cI(x_0,\cdots,x_{f-1})$, so that the corresponding Serre weight $\tau$ is old.
\end{proof}
\begin{remark}\label{remark-unique-delta}
(i) Here is an example of an old Serre weight which occurs  in $\rInj_{\Gamma}\delta$ for distinct $\delta\in\Delta(\sigma)$. Take $f=2$ and $\sigma=(r_0,r_1)$, then $(p-3-r_0,p-3-r_1)\otimes{\det}^{r_0+1+p(r_1+1)}$ is old and occurs in both $\rInj_{\Gamma}\delta_1$ and $\rInj_{\Gamma}\delta_2$, where $\delta_1=(r_0+2,r_1)\otimes{\det}^{-1}$ and $\delta_2=(r_0,r_1+2)\otimes{\det}^{-p}$.

(ii) The proof of Lemma \ref{lemma-unique-delta} shows that if $\tau$ is a common subquotient of $\rInj_{\Gamma}\delta_i^{*}(\sigma)$ and $\rInj_{\Gamma}\delta_{i'}^{*'}(\sigma)$ with $i\neq i'$, then $i,i'\in\cS(\nu)$ and so $|\cS(\nu)|\geq 2$. For example, in the case  $(*,*')=(+,+)$, this follows from \eqref{equation-nu-case+} and \eqref{equation-nu-j}. As a consequence, if $\tau$ is a Serre weight such that $\Ext^1_{\Gamma}(\tau,\sigma)\neq0$ so that $|\cS(\tau)|=1$,  then there is  at most one $\delta\in\Delta(\sigma)$ such that $\tau$ occurs in $\rInj_{\Gamma}\delta$.  Actually, $\delta$ does exist: if $\tau=\mu_{i}^*(\sigma)$, then $\delta=\delta_i^*(\sigma)$.
\end{remark}

The next  auxiliary lemma  will be used in \S\ref{section-BP}.

\begin{lemma}\label{lemma-lambda'=mu-lambda}
Let $i\in\cS$ and $*\in\{+,-\}$. Let $\lambda_i, \lambda'_i,\mu_i$ be functions of the form $\Z\pm x_i$.  Assume
\begin{enumerate}
\item[(a)] $\lambda_i(x_i)\in\{x_i,x_i*1,p-2-x_i,p-2-x_i-(*1)\}$;
\item[(b)] $\lambda'_i(x_i),\mu_i(x_i)\in\{x_i,x_i\pm1,p-2-x_i,p-2-x_i\pm1\}$;
\item[(c)] the relation $\lambda_i(x_i*2)=\lambda_i'(\mu_i(x_i))$ holds.
\end{enumerate}
Then $i\notin\cS(\lambda)$ and $i\in\cS(\lambda')\cap \cS(\mu)$. Moreover, $\mu_i(x_i)\in \{x_i,x_i*1,p-2-x_i,p-2-x_i-(*1)\}$.
\end{lemma}

\begin{proof}
Without loss of generality, we may assume $*=+$, and therefore
\[\lambda_i(x_i+2)\in\{x_i+2,x_i+3,p-4-x_i,p-5-x_i\}.\]
By (b), the table in the proof of Lemma \ref{lemma-compose-lambda} lists all the possible values of $\lambda_i'(\mu_i(x_i))$. Together with  (c), we deduce that either
\[\lambda_i(x_i+2)=x_i+2,\ \
 \lambda_i'(x_i)=\left\{\begin{array}{ll}x_i+1\\
 p-1-x_i\end{array}\right.\ \ \mathrm{resp.}\ \ \mu_i(x_i)=\left\{\begin{array}{ll}x_i+1\\
 p-3-x_i\end{array}\right.\]
in which case $\lambda_i(x_i)=x_i$,
or
\[\lambda_i(x_i+2)=p-4-x_i,\ \   \lambda_i'(x_i)=\left\{\begin{array}{ll}x_i-1\\
 p-3-x_i\end{array}\right. \mathrm{resp.}\ \ \mu_i(x_i)=\left\{\begin{array}{ll}p-3-x_i\\
x_i+1\end{array}\right.\]
in which case $\lambda_i(x_i)=p-2-x_i$.
The result follows from this.
\end{proof}

 \subsection{An extension  lemma}

If $\sigma,\sigma'$ are two distinct $0$-generic Serre weights such that $\Ext^1_{\tGamma}(\sigma',\sigma)\neq0$, then this space has dimension $1$ by Lemma \ref{lemma-Hu10-2.21}(ii).
We denote by $E_{\sigma,\sigma'}$ the unique up to isomorphism nonsplit $\tGamma$-extension (actually $\Gamma$-extension)
\[0\ra \sigma\ra E_{\sigma,\sigma'}\ra \sigma'\ra0.\]

The aim of this subsection is to prove the following (easy) fact about the structure of the tensor product $E_{\sigma,\mu_i^{\pm}(\sigma)}\otimes_{\F} V_{2,i}$ for $i\in\cS$, where  $V_{2,i}$ denotes the $\Gamma$-representation $(\Sym^{2}\F^2\otimes{\det}^{-1})^{\Fr^i}$. It will be used in the proof of Theorem \ref{thm-I(sigmatau)-tGamma}.

 \begin{lemma}\label{lemma-structure-EotimesV2}
Assume    $\sigma$ is $2$-generic.
 Let $\mu=\mu_{i}^{*}(\sigma)$  for some $i\in\cS$ and $*\in\{+,-\}$. Then $E_{\sigma,\mu}\otimes_{\F} V_{2,i}$ admits a  quotient isomorphic to $E_{\delta_i^{-*}(\sigma),\mu_i^{-*}(\sigma)}$.
\end{lemma}
\begin{proof}
 The genericity condition on $\sigma$ implies that $\mu$ is $1$-generic. By \eqref{eq:sigma-H1} and Lemma \ref{lemma-Delta-nocommonfactor},  $E_{\sigma,\mu}\otimes_{\F}V_{2,i}$ has Loewy length  $2$ and is multiplicity free, with Jordan--H\"older factors  $\{\sigma, \delta_i^{\pm}(\sigma),\mu,\delta_i^{\pm}(\mu)\}$.
We claim that
\[\Hom_{\tGamma}(E_{\sigma,\mu}\otimes_{\F}V_{2,i},\delta_i^{-*}(\sigma))=0.\]
By \cite[Lem.~7.3]{Al}, this is equivalent to showing $\Hom_{\tGamma}(E_{\sigma,\mu},\delta_i^{-*}(\sigma)\otimes_{\F}V_{2,i})=0$ (as $V_{2,i}$ is self-dual). It suffices to show that $\mu$ is not a Jordan--H\"older factor of $\delta_i^{-*}(\sigma)\otimes_{\F}V_{2,i}$. Write $\delta_i^{-*}(\sigma)=(s_0,\cdots,s_{f-1})$ up to twist. Since $\sigma$ is $2$-generic, we have $0\leq s_i\leq p-2$ and $2\leq s_j\leq p-4$ for $j\neq i$. If $0\leq s_i\leq p-3$, then \cite[Lem.~3.5, Prop.~5.4(i)]{BP} implies that
\[\delta_i^{-*}(\sigma)\otimes_{\F}V_{2,i}\cong \delta_i^{-*}(\sigma)\oplus \sigma\oplus \delta_i^{-*}(\delta_i^{-*}(\sigma)),\]
forgetting the undefined ones.
Noting that  the condition \eqref{eq:cond-lambda} holds for the pair $(\mu, \delta_i^{-*})$ and also for $(\mu,\delta_i^{-*}\circ\delta_i^{-*})$,  Lemma \ref{lemma:genericity} implies that $\mu$ does not occur in the above decomposition, proving the claim. If $s_i=p-2$ and if $f\geq 2$, then using \cite[Prop.~5.4(ii)]{BP}  we have
\[\JH(\delta_i^{-*}(\sigma)\otimes_{\F}V_{2,i})=\big\{\sigma,\ \delta_i^{-*}(\sigma)^{\oplus 2},\ \mu_{i+1}^{\pm}(\delta_i^{-*}(\sigma))\big\}\]
and the claim follows as above.  Finally, the case $f=1$ and $s_0=p-2$ can be checked directly and we leave it to the reader.

There exists a unique quotient of  $E_{\sigma,\mu}\otimes_{\F}V_{2,i}$ with socle $\delta_i^{-*}(\sigma)$, say $Q$.
Since $E_{\sigma,\mu}\otimes_{\F} V_{2,i}$ has Loewy length $2$ and since $\delta_i^{-*}(\sigma)$ does not occur in its cosocle by the claim,  $Q$ also has Loewy length $2$ and only Serre weights in $\mathscr{E}(\delta_i^{-*}(\sigma))$ can occur in $\mathrm{cosoc}(Q)$.  Comparing  Jordan--H\"older factors, we find $\mathrm{cosoc}(Q)\subseteq \delta_i^{-*}(\mu)$ if $f\geq 2$ or $\mathrm{cosoc}(Q)\subseteq \delta_i^{*}(\mu)$ if $f=1$, which has to be an equality because $\mathrm{cosoc}(Q)$ is nonzero. In both cases, one checks $\mathrm{cosoc}(Q)=\mu_i^{-*}(\sigma)$, which finishes the proof.
\end{proof}
   Remark that, under a slightly stronger genericity condition on $\sigma$, the precise structure of $E_{\sigma,\mu_i^{\pm}(\sigma)}\otimes_{\F} V_{2,i}$ is determined in \cite[\S6.3]{BHHMS}.

\subsection{The representation $I(\sigma,\tau)$} \label{subsection:I(sigmatau)}

  The aim of this subsection is to generalize Proposition \ref{prop-multione-BP} to $\tGamma$-representations.

\begin{definition}\label{defn-sigma!}
Fix $(i,*)\in \cS\times\{+,-\}$. If $\lambda\in\cI(x_0,\cdots,x_{f-1})$ satisfies \eqref{equation-condition-newweight}, i.e.
\[\lambda_i(x_i)\in\big\{x_i,x_i*1,p-2-x_i,p-2-x_i-(*1)\big\},\]
we define $\lambda_{!}\in\cI(x_0,\cdots,x_{f-1})$ to be the unique element satisfying \eqref{equation-condition-newweight}, compatible with $\lambda$,  and such that $\cS(\lambda_{!})=\cS(\lambda)\cup \{i\}$.
\end{definition}

The uniqueness of $\lambda_{!}$ in Definition \ref{defn-sigma!} is  clear, because for each $j\in\cS(\lambda_{!})=\cS(\lambda)\cup\{i\}$ we have chosen a sign $\pm1$ for $(\lambda_!)_j(x_j)$. The existence of $\lambda_{!}$ is also easily verified. Remark that the definition of $\lambda_{!}$ depends on the fixed pair $(i,*)$, although this is not indicated in the notation. Note that $\lambda\leq \lambda_{!}$ and $\lambda$ is compatible with any $\lambda'\in\cI(x_0,\cdots,x_{f-1})$ satisfying $\lambda'\leq \lambda_{!}$, by the discussion after Proposition \ref{prop-BP-4.11}.

\begin{remark}\label{remark-sigma!}
Fix $(i,*)\in \cS\times\{+,-\}$ and let  $\lambda\in\cI(x_0,\cdots,x_{f-1})$ satisfying \eqref{equation-condition-newweight}.
 It is direct to  check that:  if $i\in \cS(\lambda)$ then $\lambda_{!}=\lambda$; if $i\notin\cS(\lambda)$, then \[\lambda_{!}=\left\{\begin{array}{rl}
\mu_i^{*}\circ\lambda &\ \mathrm{if}\ \lambda_i(x_i)=x_i\\
\mu_i^{-*}\circ\lambda &\ \mathrm{if}\ \lambda_i(x_i)=p-2-x_i.\end{array}\right.\]
 \end{remark}

\begin{lemma}\label{lemma-ext1-cSdifferbyi}
Let $\sigma$ be a Serre weight and $\delta=\delta_{i}^{*}(\sigma)$ be well-defined for some $(i,*)\in\cS\times\{+,-\}$. Let $\lambda\in\cI(x_0,\cdots,x_{f-1})$ satisfying \eqref{equation-condition-newweight} and assume $(\lambda\circ \delta_i^{*})(\sigma)$ is well-defined.
\begin{enumerate}
\item[(i)] If $i\in\cS(\lambda)$, then $\Ext^1_{\Gamma}(\lambda(\delta),\sigma')=0$ for any $\sigma'\in \JH(\rInj_{\Gamma}\sigma)$.
\item[(ii)] If $i\notin\cS(\lambda)$, then  $\lambda_{!}(\sigma)$ is the unique Jordan--H\"older factor of $\rInj_{\Gamma}\sigma$ which has nontrivial $\Gamma$-extensions with $\lambda(\delta)$.
\end{enumerate}
 \end{lemma}
\begin{proof}
(i) This is  a direct check. We assume $*=+$ without loss of generality. By \eqref{equation-condition-newweight}, the condition $i\in\cS(\lambda)$ is equivalent to
$\lambda_i(x_i)\in\{x_i+1,p-3-x_i\}$.
 As a consequence, we have
\[(\lambda\circ\delta_i^*)_i(x_i)\in\{x_i+3,p-5-x_i\},\]
  so $\lambda(\delta)$ can not lie in $\mathscr{E}(\sigma')$ for any $\sigma'\in\JH(\rInj_{\Gamma}\sigma)$ by   Lemma \ref{lemma:genericity} together with (a variant of) Lemma \ref{lemma:cond-for-cI}. This proves (i).

(ii) We have $\lambda\circ\delta_i^*=\delta_i^*\circ\lambda$ if $\lambda\in \Z+x_i$ and $\lambda\circ\delta_i^*=\delta_i^{-*}\circ\lambda$ if $\lambda_i(x_i)\in\Z-x_i$. Hence, when $i\notin\cS(\lambda)$, using \eqref{eq:delta=mumu} and Remark \ref{remark-sigma!}, we have if $f\geq 2$ \[\lambda(\delta)=(\lambda\circ\delta_i^*)(\sigma)=\left\{\begin{array}{cl}(\delta_i^*\circ\lambda)(\sigma)=(\mu_i^*\circ\lambda_{!})(\sigma)&\mathrm{if}\ \lambda_i(x_i)=x_i\\
(\delta_i^{-*}\circ\lambda)(\sigma)=(\mu_i^{-*}\circ\lambda_{!})(\sigma)&\mathrm{if}\ \lambda_i(x_i)=p-2-x_i.\end{array}\right.\]
If $f=1$, we need to replace $\mu_i^*$ by $\mu_i^{-*}$ (only the case $\lambda_{i}(x_i)=x_i$ can happen).
This implies $\Ext^1_{\Gamma}(\lambda(\delta),\lambda_{!}(\sigma))\neq0$. The rest can be checked as in (i).
\end{proof}

The main result of this subsection is Theorem \ref{thm-I(sigmatau)-tGamma} below. Before stating it, we first prove the following general fact without any genericity assumption on $\sigma$.

\begin{lemma}\label{lemma-I(sigmatau)-exist}
For any Serre weight $\sigma$ and any $\tau\in\JH(\rInj_{\tGamma}\sigma)$, there exists a subrepresentation $V$ of $\rInj_{\tGamma}\sigma$ such that $\mathrm{cosoc}_{\tGamma}(V)=\tau$ and $[V:\sigma]=1$ (hence $\sigma$ occurs in $V$ as subobject).
\end{lemma}
\begin{proof}
 The assumption $\tau\in\JH(\rInj_{\tGamma}\sigma)$ implies that there exists a nonzero morphism $\Proj_{\tGamma}\tau\ra \rInj_{\tGamma}\sigma$. Using \eqref{eq:Proj-Inj} we obtain a nonzero morphism $\Proj_{\tGamma}\sigma\ra \rInj_{\tGamma}\tau$.
If we take such a nonzero morphism $f: \Proj_{\tGamma}\sigma\ra \rInj_{\tGamma}\tau$ with $[\im(f):\sigma]$ minimal, then $Q =\im(f)$ is a quotient of $\Proj_{\tGamma}\sigma$  with socle $\tau$ and such that $[Q:\sigma]=1$ (otherwise, $1\leq [\rad(Q):\sigma]<[Q:\sigma]$, and we could construct a nonzero morphism $f':\Proj_{\tGamma}\sigma\ra \rad(Q)\hookrightarrow\rInj_{\tGamma}\tau$ which  contradicts  the choice of $f$). Taking dual and twisting suitably, we obtain a subrepresentation $V$ of $\rInj_{\tGamma}\sigma$ as required.

As a byproduct, we see that $\tau\in\JH(\rInj_{\tGamma}\sigma)$ if and only if $\sigma\in\JH(\rInj_{\tGamma}\tau)$.
\end{proof}

\begin{theorem}\label{thm-I(sigmatau)-tGamma}
Let $\sigma$ be a $2$-generic Serre weight and $\tau\in\JH(\rInj_{\tGamma}\sigma)$. Let $V$ be a subrepresentation of $\rInj_{\tGamma}\sigma$ such that $\mathrm{cosoc}_{\tGamma}(V)=\tau$ and $[V:\sigma]=1$, as in Lemma \ref{lemma-I(sigmatau)-exist}.  

\begin{enumerate}
\item[(i)] If $\tau$ is a new Serre weight (cf. Definition \ref{defn-newweight}), and if $\tau=\lambda(\delta)$ for  (uniquely determined) $\delta=\delta_i^*(\sigma)\in\Delta(\sigma)$ and $\lambda\in\cI(x_0,\cdots,x_{f-1})$ satisfying \eqref{equation-condition-newweight} (cf. Lemmas \ref{lemma-newweight}, \ref{lemma-unique-delta}), then $V^{K_1}=I(\sigma,\tau_{!})$ and  there exists a short exact sequence
\begin{equation}\label{equation-I(sigmatau)-new}
0\ra I(\sigma,\tau_{!})\ra V\ra I(\delta,\tau)\ra0,\end{equation}
where $\tau_{!}\defn \lambda_{!}(\sigma)$ and $I(\sigma,\tau_{!})$, $I(\delta,\tau)$ are $\Gamma$-representations constructed in Proposition \ref{prop-multione-BP}. In the case $f=1$, $\dim_{\F}\sigma=3$ and $\tau=(\mu_0^{-}\circ\delta_0^{-})(\sigma)$, the sequence should be replaced by \[
0\ra I(\sigma,\tau_{!})\ra V\ra \tau\ra0.\] 
Moreover, such a representation $V$ is unique (up to isomorphism);  we denote it by $I(\sigma,\tau)$.\vspace{1mm}

\item[(i)] If $\tau$ is an old Serre weight, then $V$ is actually a $\Gamma$-representation and coincides with the representation $I(\sigma,\tau)$ constructed in Proposition \ref{prop-multione-BP}; in particular, such a representation $V$ is unique (up to isomorphism).
\end{enumerate}

\end{theorem}
\begin{remark}\label{remark-conter-f=1}
The genericity condition in Theorem \ref{thm-I(sigmatau)-tGamma} may not be optimal. But, the following example shows that the result is false without any genericity condition:  when $f=1$, there exists a uniserial $\tGamma$-representation of length $3$, with (the graded pieces of) the socle filtration given by $\Sym^{1}\F^2$, $\Sym^{p-2}\F^2\otimes{\det}$, $\Sym^{p-2}\F^2\otimes{\det}$.
\end{remark}

\
Before giving the proof of Theorem \ref{thm-I(sigmatau)-tGamma}, we first establish some consequences.

\begin{corollary}\label{cor-I(sigmatau)-multione}
Keep the notation in Theorem \ref{thm-I(sigmatau)-tGamma}. The representation $I(\sigma,\tau)$ is multiplicity free.
\end{corollary}
\begin{proof}
 If $\tau$ is   a subquotient of $\rInj_{\Gamma}\sigma$, then it follows from Proposition \ref{prop-multione-BP}. Otherwise, we use the exact sequence \eqref{equation-I(sigmatau)-new}: on the one hand, both $I(\sigma,\tau_{!})$ and $I(\delta,\tau)$ are multiplicity free, on the other hand,
since $\tau$ is a new Serre weight  any Jordan--H\"older factor of $I(\delta,\tau)$ is also new by Lemma \ref{lemma-newweight}.
\end{proof}

\begin{corollary}\label{cor-I(sigmatau)-geq}
If  $V$ is a subrepresentation of  $(\rInj_{\tGamma}\sigma)^{\oplus s}$ for some $s\geq 1$, then $[V:\sigma]\geq [V:\tau]$ for any Serre weight $\tau$. If, moreover, $\mathrm{cosoc}(V)$ is isomorphic to $\tau^{\oplus r}$ for some $2$-generic Serre weight $\tau$ and some $r\geq 1$, then $[V:\sigma]=[V:\tau]$.
\end{corollary}
\begin{proof}
The proof is the same as that of Corollary \ref{cor-multi-transfer}.
\end{proof}

We divide the proof of  Theorem  \ref {thm-I(sigmatau)-tGamma}  into two lemmas: Lemma \ref{lemma-I(sigmatau)-case1} for Case (i) and Lemma \ref{lemma-I(sigmatau)-case2} for Case (ii).   

\begin{lemma}\label{lemma-I(sigmatau)-case1}
Theorem \ref{thm-I(sigmatau)-tGamma}(i) is true.
\end{lemma}
\begin{proof}
The exceptional case $f=1$, $\dim_{\F}\sigma=3$ and $\tau=(\mu_0^-\circ\delta_0^{-})(\sigma)$ can be checked directly, so we omit this case for the rest of the proof.

Let $V$ be a subrepresentation of $\rInj_{\tGamma}\sigma$ as in the statement, i.e. $\mathrm{cosoc}(V)=\tau$ and $[V:\sigma]=1$.  We identify $\rInj_{\Gamma}\sigma$ with the subspace of $K_1$-invariants of $\rInj_{\tGamma}\sigma$. Since $V^{K_1}=V\cap \rInj_{\Gamma}\sigma$, there is an embedding by Proposition \ref{prop-structure-JProj}(i)
\begin{equation}\label{equation-C-embeds}C\defn V/V^{K_1}\hookrightarrow \rInj_{\tGamma}\sigma/\rInj_{\Gamma}\sigma\cong(\rInj_{\Gamma}\sigma)^{\oplus f}\oplus \big(\oplus_{\delta'\in\Delta(\sigma)}\rInj_{\Gamma}\delta'\big).\end{equation}
 Using the fact  $[C:\sigma]=0$, \eqref{equation-C-embeds} factors through an embedding $C\hookrightarrow \oplus_{\delta'}\rInj_{\Gamma}\delta'$ and finally
\[C \hookrightarrow \rInj_{\Gamma}\delta,\]
where $\delta=\delta_i^*(\sigma)$ is as in the statement of the theorem.
In particular, $C$ has  socle $\delta$ and cosocle  $\tau$.

\textbf{Step 1}. Prove that $C\cong I(\delta,\tau)$.
 By Proposition \ref{prop-multione-BP} it suffices to prove $[C:\delta]=1$.
Taking $K_1$-invariants of the short exact sequence $0\ra V^{K_1}\ra V\ra C\ra0$ gives
an injection
\begin{equation}\label{eq:partial-C}C\hookrightarrow H^1(K_1/Z_1,V^{K_1})\cong V^{K_1}\otimes_{\F} H^1(K_1/Z_1,\F). \end{equation}
Knowing that $[C:\delta]\geq 1$,  it suffices to prove
\[[V^{K_1}\otimes_{\F} H^1(K_1/Z_1,\F):\delta]= 1.\]
However,   any  $\sigma'\in \JH(V^{K_1})$  is $1$-generic,  hence by  \eqref{eq:sigma-H1} $[\sigma'\otimes_{\F} H^1(K_1/Z_1,\F):\delta]=1$ if and only if $\delta\in\Delta(\sigma')$, if and only if $\sigma'\cong \sigma$ by Lemma \ref{lemma-Delta-nocommonfactor} (note that $\sigma$ is always compatible with $\sigma'$). Since $[V^{K_1}:\sigma]=1$ and $[\sigma\otimes_{\F} H^1(K_1/Z_1,\F):\delta]=1$, the result follows.
As a consequence, $C$ is multiplicity free. On the other hand, since $[V^{K_1}:\sigma]=1$, $V^{K_1}$ is multiplicity free by Corollary \ref{cor-multi-transfer}, hence $V$ is also multiplicity free as in the proof of Corollary \ref{cor-I(sigmatau)-multione}.

\textbf{Step 2.} Prove that $V^{K_1}\cong I(\sigma,\tau_{!})$. First, if $\tau'\in\JH(I(\delta,\tau))$ with $\tau'=\lambda'(\delta)$ for $\lambda'\in\cI(x_0,\cdots,x_{f-1})$, then $\tau'$ is also a new Serre weight by Lemma \ref{lemma-newweight}. By definition it is easy to see  that
$\cS(\lambda_{!}')\subseteq \cS(\lambda_{!})$ and $\lambda_{!}'$ is compatible with $\lambda_{!}$, hence  $\tau'_{!}\defn\lambda'_{!}(\sigma)$ occurs in $I(\sigma,\tau_{!})$ by Proposition \ref{prop-multione-BP}.  Using Lemma
\ref{lemma-ext1-cSdifferbyi} we deduce by d\'evissage that if $\sigma'\in\JH(\rInj_{\Gamma}\sigma)$ which does not occur in $I(\sigma,\tau_{!})$, then
\[\Ext^1_{\tGamma}\big(I(\delta,\tau),\sigma'\big)=0,\]
and consequently  $\Hom_{\tGamma}(V^{K_1},\sigma')\cong\Hom_{\tGamma}(V,\sigma')$ by Step 1.   However,  since $V$ has irreducible cosocle $\tau$ which is new,  $\Hom_{\tGamma}(V,\sigma')=0$ for any $\sigma'$ as above and so $\Hom_{\tGamma}(V^{K_1},\sigma')=0$.  Using the fact that $V^{K_1}$ is multiplicity free,  we deduce $V^{K_1}\subseteq I(\sigma,\tau_{!})$.

It remains to show the inclusion  $I(\sigma,\tau_{!})\subseteq V^{K_1}$, for which it suffices to show that $\tau_{!}$ occurs in $V^{K_1}$ as a subquotient. First assume $i\in\cS(\lambda)$, so that $\lambda_{!}=\lambda$  by Remark \ref{remark-sigma!}. Using the embedding \eqref{eq:partial-C}, we know that
\[[V^{K_1}\otimes_{\F} H^1(K_1/Z_1,\F):\tau]\geq 1.\]
However, since $\tau=\lambda(\delta)\in\Delta(\tau_{!})$ (as $\lambda_{!}=\lambda$),  by Lemma \ref{lemma-Delta-nocommonfactor} $\sigma'=\tau_{!}$ is the unique subquotient of $I(\sigma,\tau_{!})$ with the property $[\sigma'\otimes_{\F} H^1(K_1/Z_1,\F):\tau]\neq0$. We deduce that   $\tau_{!}$ occurs in $V^{K_1}$ as a subquotient.
 Assume  now $i\notin\cS(\lambda)$.
 Then $\tau_{!}=\mu_i^*(\lambda(\sigma))$ if $\lambda_i(x_i)=x_i$, or $\tau_{!}=\mu_i^{-*}(\lambda(\sigma))$ if $\lambda_i(x_i)=p-2-x_i$, see Remark \ref{remark-sigma!}. We prove the assertion in two steps.
\begin{enumerate}
\item[(a)] The special case $\cS(\lambda)=\emptyset$, i.e. $\tau=\delta$.  Then $\tau_!=\mu_i^*(\sigma)$,  and $I(\sigma,\tau_{!})$ is just the nonsplit extension of $\tau_{!}$ by $\sigma$. Since we already know $V^{K_1}\subset I(\sigma,\tau_{!})$, it suffices to prove $V^{K_1}\neq \sigma$, which is obvious as $\Ext^1_{\tGamma}(\delta,\sigma)=0$ by Lemma \ref{lemma-Hu10-2.21}(ii) (both $\sigma$ and $\delta$ are $0$-generic).
As a consequence, any $\tGamma$-representation with socle $\sigma$ and cosocle $\delta$ contains $\tau_{!}$ as a subquotient. \vspace{1mm}
\item[(b)] The general case $i\notin\cS(\lambda)$.   In this case, the same argument as in the case $i\in\cS(\lambda)$ shows that the Serre weight  $\lambda(\sigma)$ occurs in $V^{K_1}$. Hence, $V$ admits a quotient, say $\overline{V}$, whose socle is $\lambda(\sigma)$ and cosocle is $\tau$. On the other hand,  the proof of Lemma \ref{lemma-ext1-cSdifferbyi} shows $\tau=\delta_i^{*}(\lambda(\sigma))$ if $\lambda_i(x_i)=x_i$ or $\tau=\delta_i^{-*}(\lambda(\sigma))$ if $\lambda_i(x_i)=p-2-x_i$. Hence, applying (a) to $\overline{V}$ we obtain  $[\overline{V}:\tau_{!}]\neq 0$, and therefore $[V:\tau_{!}]=[V^{K_1}:\tau_{!}]\neq0$ (as $\tau_{!}$ is an old Serre weight of $\rInj_{\tGamma}\sigma$).
\end{enumerate}

\textbf{Step 3.} Prove the uniqueness of $V$. By  Step 1 and Step  2, it suffices to prove
\[\dim_{\F}\Ext^1_{\tGamma}(I(\delta,\tau),I(\sigma,\tau_{!}))\leq 1,\]
and the equality would then follow by the existence of $V$.
The Hochschild-Serre spectral sequence gives an exact sequence
\begin{multline*}0\ra \Ext^1_{\Gamma}(I(\delta,\tau),I(\sigma,\tau_{!}))\ra \Ext^1_{\tGamma}(I(\delta,\tau),I(\sigma,\tau_{!})) \ra \Hom_{\Gamma}\big(I(\delta,\tau),I(\sigma,\tau_{!})\otimes_{\F} H^1(K_1/Z_1,\F)\big).\end{multline*}
On the one hand, since the Jordan--H\"older factors of $I(\delta,\tau)$ are all new by Lemma \ref{lemma-newweight}, a standard d\'evissage argument shows that $\Ext^1_{\Gamma}(I(\delta,\tau),I(\sigma,\tau_{!}))=0$.  On the other hand,  Lemma \ref{lemma-Delta-nocommonfactor} implies
\[\JH(I(\delta,\tau))\cap\big(\{\sigma\}\cup \Delta(\sigma)\big)=\{\delta\}\]
because any Jordan--H\"older factor of $I(\delta,\tau)$ has the form $\delta_i^{\pm}(\sigma')$ for some $\sigma'\in\JH(\rInj_{\Gamma}\sigma)$. Since the socle of $I(\sigma,\tau_{!})\otimes_{\F} H^1(K_1/Z_1,\F)$ is equal to $\sigma\otimes_{\F} H^1(K_1/Z_1,\F)\cong \sigma^{\oplus f}\oplus (\oplus_{\delta'\in\Delta(\sigma)}\delta')$, we deduce that
\begin{multline*}\dim_{\F}\Hom_{\Gamma}\big(I(\delta,\tau),I(\sigma,\tau_{!})\otimes_{\F} H^1(K_1/Z_1,\F)\big) \leq\dim_{\F}\Hom_{\Gamma}\big(\delta,I(\sigma,\tau_{!})\otimes_{\F} H^1(K_1/Z_1,\F)\big)=1.\end{multline*}
This proves the uniqueness of $V$ and finishes the proof.
 \end{proof}

We have the following direct consequence of Lemma \ref{lemma-I(sigmatau)-case1}, which will be used in the proof of  Theorem \ref{thm-I(sigmatau)-tGamma}(ii) (i.e. Lemma \ref{lemma-I(sigmatau)-case2}) below.
\begin{corollary} \label{cor-I(sigmatau)-case1}
Keep the notation of Theorem \ref{thm-I(sigmatau)-tGamma}. Assume $\tau=\delta_i^{*}(\sigma)$ for some $(i,*)\in\cS\times\{+,-\}$. Then $I(\sigma,\tau)$ is uniserial of length $3$, and (the graded pieces of) the socle filtration is given by $\sigma$, $\tau_{!}$, $\tau$, with $\tau_{!}=\mu_i^*(\sigma)$. For any subrepresentation $V$ of $\rInj_{\tGamma}\sigma$ with $[V:\tau]\geq 1$, there exists an embedding $I(\sigma,\tau)\hookrightarrow V$. Moreover, we have $[V:\sigma]\geq [V:\tau]$.
\end{corollary}

\begin{lemma}\label{lemma-I(sigmatau)-case2}
Theorem \ref{thm-I(sigmatau)-tGamma}(ii) is true.
\end{lemma}

\begin{proof}
Let $V$ be a subrepresentation of $\rInj_{\tGamma}\sigma$ as in the statement, i.e. $\mathrm{cosoc}(V)=\tau$ and $[V:\sigma]=1$. It is enough to prove $V=V^{K_1}$, because then  $V$ is actually a $\Gamma$-representation, and  Proposition \ref{prop-multione-BP} applies.
Again, as in the proof of Lemma \ref{lemma-I(sigmatau)-case1}, there is an embedding
\begin{equation}\label{equation-case2-C}
C\defn V/V^{K_1}\hookrightarrow  \bigoplus_{\delta\in\Delta(\sigma)}\rInj_{\Gamma}\delta.\end{equation}
We need to prove $C=0$. Assume this is not the case for a contradiction.
\vspace{2mm}

\textbf{Case 1}. Assume  $\Ext^1_{\Gamma}(\tau,\sigma)\neq0$, i.e. $\tau=\mu_{i}^{*}(\sigma)$ for some $i\in\cS$ and $*\in\{+,-\}$. In this case, Remark \ref{remark-unique-delta}(ii) implies that
$C\hookrightarrow \rInj_{\Gamma}\delta$
where $\delta\defn\delta_i^{*}(\sigma)$. Moreover, we have $C=I(\delta,\tau)$: indeed, this is equivalent to $[C:\delta]=1$ by Proposition \ref{prop-multione-BP}; but if we had $[C:\delta]\geq 2$, then $[V:\delta]\geq2$ and Corollary \ref{cor-I(sigmatau)-case1}
would imply $[V:\sigma]\geq 2$, a contradiction.
Since $\Ext^1_{\Gamma}(\tau,\delta)\neq0$, $C=I(\delta,\tau)$ is exactly the  nonsplit extension $0\ra \delta\ra E_{\delta,\tau}\ra\tau\ra0$.    Consider $V$ as a nonzero extension class in $\Ext^1_{\tGamma}(E_{\delta,\tau},V^{K_1})$. As in the proof of Lemma \ref{lemma-I(sigmatau)-case1},  it induces an embedding
$E_{\delta,\tau}\hookrightarrow V^{K_1}\otimes_{\F} H^1(K_1/Z_1,\F)$,
and further an embedding \[E_{\delta,\tau}\hookrightarrow V^{K_1}\otimes_{\F} V_{2,i}\] because $\tau$ is not a subquotient of $V^{K_1}\otimes_{\F} V_{2,j}$ for  any $j\neq i$ by Lemma \ref{lemma-unique-delta}. By \cite[Lem.~7.3]{Al} and the self-duality of $V_{2,i}$, we finally obtain a nonzero morphism \[\partial: E_{\delta,\tau}\otimes_{\F} V_{2,i}\ra V^{K_1}.\]

On the other hand, letting $\mu\defn\mu_{i}^{-*}(\sigma)$,
 Lemma \ref{lemma-structure-EotimesV2}  implies a surjection
$E_{\sigma,\mu}\otimes_{\F}V_{2,i}\twoheadrightarrow E_{\delta,\tau}$.
By \cite[Lem.~7.3]{Al} and the self-duality of $V_{2,i}$, it induces a  morphism
\[\iota: E_{\sigma,\mu}\ra E_{\delta,\tau}\otimes_{\F}V_{2,i}\]
which is  injective by examining the socles.  For the same reason the composition $\partial\circ \iota$ is also injective.  Hence,  $E_{\sigma,\mu}$ embeds in $V$ and $V/\sigma$ admits a quotient $Q$ with socle $\mu$ (and cosocle $\tau$).  But, one checks that  $\tau=\delta_i^{*}(\mu)$ if $f\geq 2$, resp. $\tau=\delta_i^{-*}(\mu)$ if $f=1$, so Corollary \ref{cor-I(sigmatau)-case1} applies and implies that  $\sigma$ occurs in $V/\sigma$ as a subquotient. Remark that  $\mu$ need not be $2$-generic in which case Corollary \ref{cor-I(sigmatau)-case1} does not apply,  but if this happens,  then $\tau$ has to be $2$-generic and we may apply Corollary \ref{cor-I(sigmatau)-case1} to the dual of $Q$. Therefore, we obtain  $[V:\sigma]\geq 2$, a contradiction. In conclusion, we deduce that $C=0$ and $V=V^{K_1}$.

 \vspace{2mm}

\textbf{Case 2}. Now treat the general case. As observed in Remark \ref{remark-unique-delta}(i), $\tau$ may occur in $\rInj_{\Gamma}\delta$ for distinct $\delta\in\Delta(\sigma)$.  We choose one $\delta$ in such  a way that \eqref{equation-case2-C} induces a nonzero morphism
$C\ra \rInj_{\Gamma}\delta$ when composed with the natural projection to $\rInj_{\Gamma}\delta$; let $C_{\delta}$ denote the image. As in Case 1, we have $C_{\delta}\cong I(\delta,\tau)$.    Write $\delta=\delta_i^{*}(\sigma)$ for some $(i,*)\in\cS\times\{+,-\}$ and set $\tau'\defn\mu_{i}^{*}(\sigma)=\mu_i^{-*}(\delta)$.  Then $\tau'$ has nontrivial extensions with both $\sigma$ and $\delta$. We claim that $\tau'$ is a subquotient of  $I(\delta,\tau)$.  Indeed,  writing $\tau=\lambda(\delta)$ for $\lambda\in\cI(x_0,\cdots,x_{f-1})$, then   Lemma \ref{lemma-newweight} implies
$\lambda_i(x_i)\in\{x_i-(*1),p-2-x_i+(*1)\}$
because $\tau$ is old by assumption; in particular $i\in \cS(\lambda)$.
On the other hand, viewing $\tau'$ as a subquotient of $\rInj_{\Gamma}\delta$, it corresponds to $\mu_i^{-*}$ which is compatible with $\lambda$ at $\{i\}=\cS(\mu_i^{-*})$. The claim  follows from this using Proposition \ref{prop-BP-4.11}. By the claim, we may construct a subrepresentation of $V$ with cosocle $\tau'$ which is not fixed by $K_1$, but this contradicts
Case 1.
\end{proof}

 Note that $I(\sigma,\tau)$ can be viewed as a quotient of $\Proj_{\tGamma}\tau$. Using Corollary \ref{cor-I(sigmatau)-multione}, we have the following  dual version of Theorem \ref{thm-I(sigmatau)-tGamma}.
\begin{theorem}\label{thm-I(sigmatau)-dual}
Let $\tau$ be a $2$-generic Serre weight. Among the quotients of $\Proj_{\tGamma}\tau$ whose socle is isomorphic to $\sigma$ (not necessarily $2$-generic), there exists a unique one, denoted by $I(\sigma,\tau)$,  in which $\tau$ occurs with multiplicity $1$. If moreover $\sigma$ is $2$-generic, then this representation coincides with the one constructed in Theorem \ref{thm-I(sigmatau)-tGamma}.
\end{theorem}

Combining Theorems \ref{thm-I(sigmatau)-tGamma} and \ref{thm-I(sigmatau)-dual}, we see that $I(\sigma,\tau)$ is well-defined provided that either $\sigma$ or $\tau$ is $2$-generic.

\begin{corollary}\label{coro:Q-killedbyJ}
Let $\tau$ be a $2$-generic Serre weight. Let $Q$ be a quotient of $\Proj_{\tGamma}\tau$ satisfying the following conditions:
\begin{enumerate}
\item[(a)] $[Q:\tau]=1$;
\item[(b)] for any Serre weight $\sigma$ in $\mathrm{soc}_{\tGamma}(Q)$, $\sigma$ is a subquotient of $\Proj_{\Gamma}\tau$.
\end{enumerate}
Then $Q$ is multiplicity free and  a quotient of $\Proj_{\Gamma}\tau$, i.e. $Q$ is annihilated by $\fm_{K_1}$.
\end{corollary}

\begin{proof}
 First note that $Q$ is multiplicity free by (the dual version) of Corollary \ref{cor-I(sigmatau)-geq}. In particular, $\rsoc_{\tGamma}(Q)$ is multiplicity free. If $\sigma$ is a Serre weight occurring in $\soc_{\tGamma}(Q)$, then $Q$ admits $I(\sigma,\tau)$ as a quotient by Theorem \ref{thm-I(sigmatau)-dual}. The morphism
$Q \ra \oplus_{\sigma}I(\sigma,\tau)$,
where $\sigma$ runs over all Serre weights in $\soc_{\tGamma}(Q)$, is injective as it is injective on $\soc_{\tGamma}(Q)$. By (b), each $I(\sigma,\tau)$ is annihilated by $\fm_{K_1}$, hence  so is  $Q$, i.e. $Q$ is  a quotient of $\Proj_{\Gamma}\tau$.
\end{proof}

\begin{corollary}\label{coro-Q-length3}
Let $\tau$ be a $2$-generic Serre weight. Let $Q$ be a quotient of $\Proj_{\tGamma}\tau$ satisfying the following conditions:
\begin{enumerate}
\item[(a)] $\rsoc_{\tGamma}(Q)\cong \tau^{\oplus r}$ for some $r\geq 1$;
\item[(b)] $\rad_{\tGamma}(Q)/\rsoc_{\tGamma}(Q)$ is nonzero and does not admit $\tau$ as a subquotient.
\end{enumerate}
Then $\rad_{\tGamma}(Q)/\rsoc_{\tGamma}(Q)$ is semisimple and there is an embedding
 \[\rad_{\tGamma}(Q)/\rsoc_{\tGamma}(Q)\hookrightarrow \bigoplus_{\sigma\in\mathscr{E}(\tau)}\sigma.\]
Moreover, the length of $\rad_{\tGamma}(Q)/\rsoc_{\tGamma}(Q)$ is greater than or equal to $r$.
\end{corollary}
\begin{proof}
Consider the quotient $Q/\rsoc_{\tGamma}(Q)$, which is a quotient of $\Proj_{\tGamma}\tau$ and in which $\tau$ occurs once by condition (b). Then $Q/\soc_{\tGamma}(Q)$ is multiplicity free by (the dual version) of Corollary \ref{cor-I(sigmatau)-geq}. We denote by $\rsoc_1(Q)$ the socle of $Q/\rsoc_{\tGamma}(Q)$. If $\sigma\hookrightarrow \rsoc_1(Q)$, then $\Ext^1_{\tGamma}(\sigma,\rsoc_{\tGamma}(Q))\neq0$, and therefore $\sigma\in\mathscr{E}(\tau)$ by  (a). As in the proof of
 Corollary \ref{coro:Q-killedbyJ}, we obtain an embedding
 $Q/\rsoc_{\tGamma}(Q)\hookrightarrow \oplus_{\sigma}I(\sigma,\tau)$
 where $\sigma$ runs over the Serre weights in $\rsoc_{1}(Q)$. Note that   $I(\sigma,\tau)$  is just the nonsplit extension of $\tau$ by $\sigma$, so $Q/\rsoc_{\tGamma}(Q)$ fits in a short exact sequence
 \[0\ra \rsoc_1(Q)\ra Q/\rsoc_{\tGamma}(Q)\ra \tau\ra0. \]
 Thus, we may identify $\soc_1(Q)$ with $\rad_{\tGamma}(Q)/\soc_{\tGamma}(Q)$, proving the first assertion.

It remains to show $\rsoc_1(Q)$ has length $\geq r$. In fact, this follows from (a), which implies
\[\dim_{\F}\Ext^1_{\tGamma}(\rsoc_1(Q),\tau)\geq r,\]
while $\dim_{\F}\Ext^1_{\tGamma}(\sigma,\tau)=1$ for any $\sigma\in\mathscr{E}(\tau)$ by Lemma \ref{lemma-Hu10-2.21}(ii).
\end{proof}

\begin{remark}
It will be proved in Proposition \ref{prop-Theta} that there exists a (unique) representation $Q$ as in Corollary \ref{coro-Q-length3} such that $\rad_{\tGamma}(Q)/\rsoc_{\tGamma}(Q)\cong \bigoplus_{\sigma\in\mathscr{E}(\tau)}\sigma$.
\end{remark}

\subsection{The structure of $I(\sigma,\tau)$} Let $\sigma $ be a $2$-generic Serre weight. It will be useful to have an explicit description of the lattice structure of subrepresentations of $I(\sigma,\tau)$ for $\tau\in\JH(\rInj_{\tGamma}\sigma)$. The case when $\tau\in\JH(\rInj_{\Gamma}\sigma)$ is treated in \cite[Cor.~4.11]{BP}.

\begin{definition}\label{defn:tildeI}
Fix $(i,*)\in\cS\times \{+,-\}$. Set
\begin{multline}\label{eq:def-tildeI}\widetilde{\cI}_{(i,*)}\defn \big\{(\lambda,\mathcal{T})\in\cI(x_0,\cdots,x_{f-1})\times \{\emptyset, (i,*)\}:  \lambda_i(x_i)\ \mathrm{satisfies}\ \eqref{equation-condition-newweight}\ \mathrm{if}\ \mathcal{T}=(i,*) \big\}.\end{multline}
We define a partial  order  $\widetilde{\cI}_{(i,*)}$ as follows.  Given  two elements $\widetilde{\lambda}=(\lambda,\mathcal{T})$, $\widetilde{\lambda}'=(\lambda',\mathcal{T}')$, we say $\widetilde{\lambda}'\leq \widetilde{\lambda}$ if and only if one of the following holds:
\begin{enumerate}
\item[$\bullet$] $\mathcal{T}'=\mathcal{T}$ and $\lambda'\leq \lambda$, meaning that $\lambda'$, $\lambda$ are compatible and $\cS(\lambda')\subset\cS(\lambda)$;
\item[$\bullet$]    $\mathcal{T}'=\emptyset$, $\mathcal{T}=(i,*)$, and $\lambda'\leq \lambda_{!}$, where $\lambda_{!}\in\cI(x_0,\cdots,x_{f-1})$ is as in Definition \ref{defn-sigma!}.
\end{enumerate}
\end{definition}

 We define a length function on $ \widetilde{\cI}_{(i,*)}$ by setting
\begin{equation}\label{equation-length-tildeIambda}\ell(\widetilde{\lambda})=\ell(\lambda,\mathcal{T})\defn\left\{\begin{array}{lll}|\cS(\lambda)| & \mathrm{if}\ \mathcal{T}=\emptyset\\
|\cS(\lambda)|+2& \mathrm{if}\ \mathcal{T}=(i,*).\end{array}\right. \end{equation}
To $\widetilde{\lambda}=(\lambda,\mathcal{T})\in\widetilde{\cI}_{(i,*)}$, we associate a Jordan--H\"older factor of $\rInj_{\tGamma}\sigma$, as follows:
\begin{equation}\label{equation-tildelambda-sigma}\widetilde{\lambda}(\sigma)=\left\{\begin{array}{lll} \lambda(\sigma)&\mathrm{if}\ \mathcal{T}=\emptyset\\
\lambda(\delta_i^{*}(\sigma))&\mathrm{if}\ \mathcal{T}=(i,*).\end{array}\right.\end{equation}
It is a direct consequence of Proposition \ref{prop-structure-JProj} and Lemma \ref{lemma-newweight} that  any Jordan--H\"older factor of $\rInj_{\tGamma}\sigma$ is isomorphic to $\widetilde{\lambda}(\sigma)$ for some pair $(i,*)$ and some $\widetilde{\lambda}\in\widetilde{\cI}_{(i,*)}$.

\begin{corollary}\label{cor-I(sigmatau)-JH}
Fix $(i,*)\in\cS\times\{+,-\}$ and let $\widetilde{\lambda}\in\widetilde{\cI}_{(i,*)}$. Then the Jordan--H\"older factors of $I(\sigma,\widetilde{\lambda}(\sigma))$ are
given by
\[\big\{\widetilde{\lambda}'(\sigma): \  \widetilde{\lambda}'\in\widetilde{\cI}_{(i,*)},\ \widetilde{\lambda}'\leq \widetilde{\lambda}\big\}\] and the graded pieces of its socle filtration\footnote{If $V$ is a $\tGamma$-module, $(V_k)_{k\geq 0}$ denotes the graded pieces of its socle filtration with convention $V_0=\soc(V)$.} are given by:
\[I(\sigma,\widetilde{\lambda}(\sigma))_k=\bigoplus_{\widetilde{\lambda}'\leq \widetilde{\lambda},\  \ell(\widetilde{\lambda}')=k }\widetilde{\lambda}'(\sigma).\]
In the exceptional case $f=1$, $\dim_{\F}\sigma=3$ and $\widetilde{\lambda}=(\mu_0^-,(0,-))$, we forget the Serre weight $\delta_0^-(\sigma)$ which corresponds to $\widetilde{\lambda}'=(x_0, (0,-))$ in $\JH\big(I(\sigma,\widetilde{\lambda}(\sigma))\big)$, and set $\ell(\widetilde{\lambda})=2$ in the   formula of the socle filtration.  \end{corollary}

\begin{proof}
It is a reformulation of Theorem \ref{thm-I(sigmatau)-tGamma} using Proposition \ref{prop-BP-4.11}.
\end{proof}

\section{Finite representation theory II}

\label{section:Rep-II}

In this section, we study the smooth representation theory of the Iwahori subgroup $I$ (over $\F$) and its relation to representation theory of $K=\GL_2(\cO_L)$ studied in \S\ref{section-finiteRep}.

\subsection{$I$-Extensions}

Let $\alpha:H\ra \F^{\times}$ be the character sending $\smatr{[a]}00{[d]}$ to $ad^{-1}$, where $a,d\in\F_q^{\times}$. Let $\alpha_i:=\alpha^{p^i}$ for $i\in\cS$ (recall $\cS=\{0,\dots,f-1\}$).

\begin{lemma}\label{lemma:Ext1-chi}
If $\chi,\chi':I\ra \F^{\times}$ are smooth characters such that $\Ext^1_{I/Z_1}(\chi,\chi')\neq0$, then $\chi'\in\{\chi\alpha_i^{\pm1},\ i\in \cS\}$. Moreover, in this case we have $\dim_{\F}\Ext^1_{I/Z_1}(\chi,\chi')=1$.
\end{lemma}
\begin{proof}
See \cite[Lem.~2.4(i)]{Hu10}, which is based on \cite[Prop.~5.2]{Pa10}.
\end{proof}

We denote by $\mathscr{E}(\chi)$ the set of characters $\chi'$ such that $\Ext^1_{I/Z_1}(\chi,\chi')\neq0$. For $\chi'\in\mathscr{E}(\chi)$, we denote by $E_{\chi',\chi}$ the unique nonsplit $I$-extension
\[0\ra \chi'\ra E_{\chi',\chi}\ra \chi\ra0.\]
Remark that  $K_1$ acts trivially on $E_{\chi',\chi}$ if and only if $\chi'=\chi\alpha_i$ for some $i\in\cS$, see \cite[Lem.~2.4(ii)]{Hu10}.

For a character $\chi:I\ra \F^{\times}$,  $\Proj_{I/Z_1}\chi$ denotes a projective envelope of $\chi$ in the category of pseudo-compact $\F[\![I/Z_1]\!]$-modules. For $n\geq 1$, define
\[W_{\chi,n}\defn\Proj_{I/Z_1}\chi/\fm^n,\]
where $\fm=\fm_{I_1/Z_1}$ denotes the maximal ideal of $\F[\![I_1/Z_1]\!]$.
Clearly, the Loewy length of $W_{\chi,n}$ is equal to $n$. We will mainly be concerned with the cases $n=2,3$. For example, $W_{\chi,2}$ fits in  a short exact sequence
\[0\ra \bigoplus_{i\in\cS}\chi\alpha_{i}^{\pm1}\ra W_{\chi,2}\ra \chi\ra0. \]
For $W_{\chi,3}$, we have $0\ra \soc_I(W_{\chi,3})\ra W_{\chi,3}\ra W_{\chi,2}\ra0$, with
 $\soc_{I}(W_{\chi,3})$ isomorphic to
\[\chi^{\oplus 2f}\oplus\bigoplus_{0\leq i\leq j\leq f-1}\chi\alpha_i\alpha_j\oplus\bigoplus_{0\leq i\leq j\leq f-1}\chi\alpha_i^{-1}\alpha_j^{-1}\oplus\bigoplus_{0\leq i\neq j\leq f-1}\chi\alpha_i\alpha_j^{-1},\]
see \cite[\S5.3]{BHHMS}.
Let $X''$ denote the direct sum of characters in $\soc_I(W_{\chi,3})$ which are not isomorphic to $\chi$  and set
\begin{equation}\label{eq:bar-W}\overline{W}_{\chi,3}\defn W_{\chi,3}/X''. \end{equation}
This representation will play a prominent role in the whole paper. By definition, $\overline{W}_{\chi,3}$ fits in the following exact sequence
\begin{equation}\label{eq:barW-sq1}0\ra \chi^{\oplus 2f}\ra \overline{W}_{\chi,3}\ra W_{\chi,2}\ra0.\end{equation}

\begin{lemma}\label{lemma-bar-W}
We have $\soc_{I}(\overline{W}_{\chi,3})\cong \chi^{\oplus 2f}$ and there exists a short exact sequence
\begin{equation}\label{eq:barW-sq2}0\ra \bigoplus_{\chi'\in\mathscr{E}(\chi)}E_{\chi,\chi'}\ra \overline{W}_{\chi,3}\ra\chi\ra0.
\end{equation}
\end{lemma}
\begin{proof}
We know that $\chi^{\oplus 2f}$ embeds in $\overline{W}_{\chi,3}$. Let $\overline{W}'_{\chi,3}$ be the largest quotient of $\overline{W}_{\chi,3}$ whose socle is isomorphic to $\chi^{\oplus  2f}$. We claim that $\overline{W}_{\chi,3}=\overline{W}_{\chi,3}'$ from which the first assertion follows.  By a similar argument as in the proof of Corollary \ref{coro-Q-length3}, we have
\[\dim_{\F}\rad(\overline{W}'_{\chi,3})/\soc(\overline{W}'_{\chi,3})\geq 2f.\]
Since $\rad(\overline{W}'_{\chi,3})/\soc(\overline{W}'_{\chi,3})\hookrightarrow \oplus_{\chi'\in\mathscr{E}(\chi)}\chi'$ and $|\mathscr{E}(\chi)|=2f$, the above inequality is an equality and the embedding is an isomorphism. Comparing Jordan--H\"older factors,  we get $\overline{W}_{\chi,3}=\overline{W}_{\chi,3}'$.

Prove \eqref{eq:barW-sq2}. Let $\chi'\in\mathscr{E}(\chi)$ which is a Jordan--H\"older factor of $\overline{W}_{\chi,3}$. Then as seen above $\chi'$ occurs in $\rad(\overline{W}_{\chi,3})/\soc(\overline{W}_{\chi,3})$.  Consequently, the extension $E_{\chi,\chi'}$ embeds in $\overline{W}_{\chi,3}$. Taking sum, we obtain an embedding $\sum_{\chi'\in\mathscr{E}(\chi)}E_{\chi,\chi'}\hookrightarrow \overline{W}_{\chi,3}$. To conclude it suffices to check that
\[\sum_{\chi'\in\mathscr{E}(\chi)}E_{\chi,\chi'}=\bigoplus_{\chi'\in\mathscr{E}(\chi)}E_{\chi,\chi'}.\]
It is equivalent to show $[\sum_{\chi'}E_{\chi,\chi'}:\chi]=2f$,  equivalently $[Q:\chi]=1$, where $Q$ denotes the quotient of $\overline{W}_{\chi,3}$ by $\sum_{\chi'}E_{\chi,\chi'}$.
  Since $[\overline{W}_{\chi,3}:\chi']=1$ for any $\chi'\in\mathscr{E}(\chi)$, we have  $[Q:\chi']=0$, and $Q$ admits only $\chi$ as subquotients.  Since the cosocle of $Q$ is $\chi$ and   $\Ext^1_{I/Z_1}(\chi,\chi)=0$, we must have $Q=\chi$. This finishes the proof.
 \end{proof}

\begin{corollary}
The $I$-representation $\overline{W}_{\chi,3}$ is annihilated by $\fm_{K_1}^2$.
\end{corollary}
\begin{proof}
Let $W'$ be the subrepresentation of $\overline{W}_{\chi,3}$ defined by
\[0\ra \soc(\overline{W}_{\chi,3})\ra W'\ra \bigoplus_{j\in\cS}\chi\alpha_j^{-1}\ra0\]
 and $W''$ the corresponding quotient.
It is easy to see that $W'$ is isomorphic to $\chi^{\oplus f}\oplus (\oplus_{j\in\cS}E_{\chi,\chi\alpha_j^{-1}})$, thus $W'$ is  annihilated  by $\fm_{K_1}$, as each $E_{\chi,\chi\alpha_j^{-1}}$ is.  On the other hand, by \eqref{eq:barW-sq2}, $W''$ embeds in $\oplus_{j\in\cS}E_{\chi\alpha_j,\chi}$, hence is also annihilated by $\fm_{K_1}$. The result follows.
\end{proof}

\subsection{Induced representations}
 \label{subsection:PS}

In this subsection, we  study the structure of $\Ind_{I}^KW_{\chi,2}$.

Let $r$ be the unique integer in $\{0,\dots,q-2\}$ such that $\chi(\smatr{a}00{d})=a^r\eta(ad)$ for some character $\eta:\F_q^{\times}\ra \F^{\times}$. Write $r=\sum_{i\in\cS}p^ir_i$ with $0\leq i\leq p-1$. For $n\geq 0$, we say $\chi$ is \emph{$n$-generic}, if $n\leq r_i\leq p-2-n$ for all $i$.

Following \cite[\S2]{BP}, we let $\mathscr{P}(x_0,\cdots,x_{f-1})$  be the subset of $\cI(x_0,\cdots,x_{f-1})$ consisting of  $\lambda$ such that $\lambda_0(x_0)\in\{x_0,p-1-x_0\}$ if $f=1$, and if $f\geq 2$,
\[\lambda_i(x_i)\in\{x_i,x_i-1,p-2-x_i,p-1-x_i\},\ \ \forall i\in\cS.\]
For $\lambda\in\mathscr{P}(x_0,\cdots,x_{f-1})$, set
\[J(\lambda)\defn\{i\in\cS,\ \lambda_i(x_i)\in\{p-2-x_i,p-1-x_i\}\}\subseteq \cS.\]
In this way, one checks that $\mathscr{P}(x_0,\cdots,x_{f-1})$ is in bijection to the set of subsets of $\cS$. By \cite[Lem.~2.2]{BP}, $\Ind_I^K\chi$ is multiplicity free with Jordan--H\"older factors  \[(\lambda_0(r_0),\cdots,\lambda_{f-1}(r_{f-1}))\otimes{\det}^{e(\lambda)(r_0,\cdots,r_{f-1})}\eta \]
for  $\lambda\in\mathscr{P}(x_0,\cdots,x_{f-1})$. For notational convenience, if $\tau\in\JH(\Ind_I^K\chi)$ and corresponds to $\lambda$, we also write $J(\tau)$ for $J(\lambda)$.

Since we prefer to use $\Ind_I^K\chi$ rather than $\Ind_I^K\chi^s$, to be compatible with the notation in \cite[\S2]{BP}, we introduce the following notation. Let $\sigma_{\emptyset}=(p-1-r_0,\cdots,p-1-r_{f-1})\otimes{\det}^{r}\eta$. For $J\subset \cS$,  let $\lambda_J\in\mathscr{P}(x_0,\cdots,x_{f-1})$ be the unique element with $J(\lambda_J)=J$ and set  (as in \eqref{eq:defn-lambda(sigma)})
\[\sigma_{J}\defn \lambda_J(\sigma_{\emptyset}).\]
Note that in the case $r\neq0$, the socle (resp. cosocle) of $\Ind_I^K\chi$ is irreducible and isomorphic to $\sigma_{\emptyset}$ (resp. $\sigma_{\cS}$).  With the notation introduced at the beginning of \S\ref{section-finiteRep}, we have $\sigma_{\emptyset}=\sigma_{\chi^s}$ and $\sigma_{\cS}=\sigma_{\chi}$. Moreover, $\chi$ is $n$-generic if and only if $\sigma_{\chi}$ is $n$-generic.

\begin{lemma}\label{lemma:IndW2-multione}
Let $\chi$ be a $2$-generic character. Then $\Ind_{I}^KW_{\chi,2}$ is multiplicity free.
\end{lemma}
\begin{proof}
This is a direct check using the $2$-genericity of $\chi$.
\end{proof}

\begin{definition}
Let $\tau$ (resp. $\tau'$) be a Jordan--H\"older factor of $\Ind_I^K\chi$ (resp. $\Ind_I^K\chi'$) such that $\Ext^1_{K/Z_1}(\tau,\tau')\neq 0$. We say that the extension $E_{\tau',\tau}$ occurs in $\Ind_I^KE_{\chi',\chi}$ if $\Ind_{I}^KE_{\chi',\chi}$ admits a subquotient isomorphic to $E_{\tau',\tau}$.
\end{definition}

\begin{lemma}\label{lemma:J-J'}
Let $\lambda,\lambda'\in \mathscr{P}(x_0,\cdots,x_{f-1})$. Assume that for some $j\in\cS$,
\[j-1\notin J(\lambda),\ \ J(\lambda')=J(\lambda)\cup \{j-1\}.\]
If $f=1$, then $\lambda_0(x_0)=x_0$, $\lambda'_0(x_0)=p-1-x_0$, and $\lambda'=\mu_0^{-}\circ\lambda$.  If $f\geq 2$, then 
\[\lambda'=\left\{\begin{array}{ll}
\mu_j^{-}\circ\lambda & \mathrm{if}\ \lambda_j(x_j)=x_j\ (\Leftrightarrow \lambda_j'(x_j)=x_j-1) \\
\mu_j^{+}\circ\lambda &\mathrm{if}\ \lambda_j(x_j)=p-2-x_j \ (\Leftrightarrow \lambda_j'(x_j)=p-1-x_j).
\end{array}\right. \]
Conversely, given $\lambda$ such that $j-1\notin J(\lambda)$ and define $\lambda'$ by the above formula, then $J(\lambda')=J(\lambda)\cup\{j-1\}$.
\end{lemma}
\begin{proof}
This is a direct check by definition of $\mathscr{P}(x_0,\cdots,x_{f-1})$.
\end{proof}

\emph{From now on, let $\chi$ be a $2$-generic character of $I$  and $\chi'\in\mathscr{E}(\chi)$.}

\begin{lemma}\label{lemma-Esigma'-occur-plus}

Assume $\chi'=\chi\alpha_j^{-1}$ for some $j\in\cS$. Let $\tau$ (resp. $\tau'$) be a  Jordan--H\"older factor of $\Ind_I^K\chi$ (resp. $\Ind_I^K\chi'$), with  parametrizing subset $J(\tau)$ (resp. $J(\tau')$). Then the following statements are equivalent:
\begin{enumerate}
\item[(i)] $\Ext^1_{\Gamma}(\tau',\tau)\neq0$;
\item[(ii)]  $j-1\notin J(\tau)$ and $J(\tau')=J(\tau)\cup \{j-1\}$.\end{enumerate}
If these conditions hold, then $E_{\tau',\tau}$ occurs  in  $\Ind_{I}^KE_{\chi',\chi}$.
\end{lemma}

\begin{proof}
We only treat the case $f\geq 2$ (the case $f=1$ can be treated similarly).   Let $\sigma_{\emptyset}$ (resp. $\sigma_{\emptyset}'$) be the socle of $\Ind_I^K\chi$ (resp. $\Ind_I^K\chi'$), and $\lambda$ (resp. $\lambda'$) be the element of $\mathscr{P}(x_0,\cdots,x_{f-1})$ such that $\tau=\lambda(\sigma_{\emptyset})$ (resp. $\tau'=\lambda'(\sigma_{\emptyset}')$). One checks that
$\sigma_{\emptyset}=\delta_j^{-}(\sigma_{\emptyset}')$.

Assume (ii) holds.  Using Lemma \ref{lemma:J-J'},  we have  (since $f\geq 2$)
\[\tau=\lambda(\sigma_{\emptyset})=(\lambda\circ \delta_j^{-})(\sigma_{\emptyset}')=(\mu_j^{*}\circ\lambda'\circ\delta_j^{-})(\sigma_{\emptyset}')\]
where $*=+$ (resp. $*=-$) if $\lambda_j(x_j)=x_j$ (resp. if $\lambda_j(x_j)=p-2-x_j$). As noted in the proof of Lemma \ref{lemma-ext1-cSdifferbyi}(ii), we have correspondingly $\lambda'\circ\delta_j^{-}=\delta_j^{-}\circ\lambda'$ (resp. $\lambda'\circ\delta_j^-=\delta_j^+\circ\lambda'$). Hence, we finally obtain $\tau=(\mu_j^{\mp}\circ\lambda')(\sigma_{\emptyset}')$ and proves (ii) $\Rightarrow$ (i). To prove (i) $\Rightarrow$ (ii), running back the above argument and using Lemma \ref{lemma:genericity}, we need to show that the equation $\lambda\circ\delta_j^-=\mu_i^*\circ\lambda'$ for $(i,*)\in\cS\times\{+,-\}$ admits a unique solution, and we may conclude by Lemma \ref{lemma:J-J'}. We leave the details to the reader.

The last statement is a consequence of \cite[Lem.~18.4]{BP}, which says that either $E_{\tau',\tau}$ or $E_{\tau,\tau'}$ occurs in  $\Ind_I^KE_{\chi',\chi}$,  but  it is clear that $E_{\tau,\tau'}$ can not occur.
\end{proof}

\begin{lemma}\label{lemma-Esigma'-occur-minus}
Assume $\chi'=\chi\alpha_j$ for some $j\in\cS$. Let $\tau$ (resp. $\tau'$) be a Jordan--H\"older factor of $\Ind_I^K\chi$ (resp. $\Ind_I^K\chi'$), with parametrizing subset $J(\tau)$ (resp. $J(\tau')$). Then the following statements are equivalent:
\begin{enumerate}
\item[(i)] $\Ext^1_{\Gamma}(\tau',\tau)\neq0$;
\item[(ii)]   $j-1\notin J(\tau')$ and $J(\tau)=J(\tau')\cup \{j-1\}$.
\end{enumerate}
If these conditions hold, then  $E_{\tau',\tau}$ occurs  in  $\Ind_{I}^KE_{\chi',\chi}$.
\end{lemma}
\begin{proof}
The equivalence (i) $\Leftrightarrow$ (ii) is checked as in Lemma \ref{lemma-Esigma'-occur-plus}. In particular, if we let $\sigma_{\emptyset}$, $\sigma_{\emptyset}'$, $\lambda$,  $\lambda'$ be as in the proof of \emph{loc. cit.}, then  $\tau=(\mu_{j}^{*}\circ\lambda')(\sigma_{\emptyset}')$. Note, however, that the assumption $\chi'=\chi\alpha_j$ implies $\sigma_{\emptyset}=\delta_j^{+}(\sigma_{\emptyset}')$.

Since $\Ind_I^KE_{\chi',\chi}$ is multiplicity free by Lemma \ref{lemma:IndW2-multione}, there exists a unique subrepresentation, say $V_{\tau}\subset \Ind_I^KE_{\chi',\chi}$, with cosocle $\tau$, and $E_{\tau',\tau}$ occurs in $\Ind_I^KE_{\chi',\chi}$ if and only if $V_{\tau}$ admits $E_{\tau',\tau}$ as a quotient. It is clear that $V_{\tau}$ fits in a short exact sequence
\begin{equation}\label{equation-Vtau}0\ra V_{\tau}\cap\Ind_I^K\chi'\ra V_{\tau}\ra I(\sigma_{\emptyset},\tau)\ra0.\end{equation}
Here, note that since $\Ind_I^KE_{\chi',\chi}$ is a $\Gamma$-representation as $E_{\chi',\chi}$ is by \cite[Lem.~2.4(ii)]{Hu10}, the representation $I(\sigma_{\emptyset},\tau)$ is well-defined by \S\ref{subsection:Gamma}.
We claim that $V_{\tau}\cap \Ind_I^K\chi'\neq0$. Otherwise, we would obtain a $K$-equivariant embedding $I(\sigma_{\emptyset},\tau)\hookrightarrow \Ind_I^KE_{\chi',\chi}$, hence a nonzero $I$-equivariant morphism $I(\sigma_{\emptyset},\tau)\ra E_{\chi',\chi}$ by Frobenius reciprocity. However,  using the explicit basis given in \cite[Lem.~2.7(ii)]{BP} and the assumption  $j-1\in J(\tau)$, one checks that $I(\sigma_{\emptyset},\tau)|_{I}$ does not admit $E_{\chi',\chi}$ as a quotient (this holds true even when $f=1$ in which case $I(\sigma_{\emptyset},\tau)$ is equal to $\Ind_I^K\chi$).

The claim implies that $\sigma_{\emptyset}'\hookrightarrow V_{\tau}$, hence $V_{\tau}$ admits $I(\sigma_{\emptyset}',\tau)$ as a quotient. It suffices to prove that $E_{\tau',\tau}$ occurs in $I(\sigma_{\emptyset}',\tau)$ (as a quotient),  or equivalently $\tau'$ is a subquotient of $I(\sigma_{\emptyset}',\tau)$. Note that $I(\sigma_{\emptyset}',\tau)$ is a $\Gamma$-representation, because $\Ind_{I}^KE_{\chi',\chi}$ is.   Viewing both $\tau$ and $\tau'$ as subquotients of $\rInj_{\Gamma}\sigma_{\emptyset}'$ and using  Proposition \ref{prop-BP-4.11}, it suffices to check that $\tau$ and $\tau'$ are compatible and $\cS(\tau')\subset \cS(\tau)$.
We have seen that $\tau=(\mu_{i}^{*}\circ\lambda')(\sigma_{\emptyset}')$ at the beginning of the proof. By Lemma \ref{lemma-compose-lambda}(ii), we have $\mu_i^{*}\circ\lambda'$ and $\lambda'$ are always compatible and
\[\cS(\mu_i^{*}\circ\lambda')=\{j\}\Delta\cS(\lambda')=\cS(\lambda')\cup \{j\};\]
here the last equality holds as
$j\notin\cS(\lambda')$ (equivalent to $j-1\notin J(\lambda')$, see \cite[\S2]{BP}). This completes the proof.
\end{proof}

Let $\tau$ be a Jordan--H\"older factor of $\Ind_I^KW_{\chi,2}$.
Since $\Ind_I^KW_{\chi,2}$ is multiplicity free, there exists a unique (up to scalar) nonzero $K$-equivariant morphism
\begin{equation}\label{equation-defn-phi2}\phi_{\tau,\chi,2}:\Proj_{\tGamma}\tau \ra \Ind_I^KW_{\chi,2}.\end{equation}
For our purposes,  $\chi$ will be fixed while $\tau$ may vary among subquotients of $\Ind_I^K\chi$, so we omit $\chi$ in the notation and write simply $\phi_{\tau,2}\defn\phi_{\tau,\chi,2}$. It is clear that  $[\Coker(\phi_{\tau,2}):\tau]=0$.

\begin{proposition}\label{prop-W2-coker-hasnoextension}
Assume $\tau$ is a Jordan--H\"older factor of $\Ind_I^K\chi$.
 Then $\Ext^1_{K/Z_1}(\tau',\tau)=0$ for any $\tau'\in \JH(\Ind_I^K\soc(W_{\chi,2}))\cap \JH(\Coker(\phi_{\tau,2}))$.
\end{proposition}

\begin{proof}
There exists a unique $\chi'\in\soc_I(W_{\chi,2})$ such that $\tau'\in \JH(\Ind_I^K\chi')$. Thus, by composing $\phi_{\tau,2}$ with the natural projection $\Ind_I^KW_{\chi,2}\twoheadrightarrow \Ind_I^KE_{\chi',\chi}$, we are reduced to the case of $\Proj_{\tGamma}\tau\ra \Ind_I^KE_{\chi',\chi}$ and we conclude by  Lemmas \ref{lemma-Esigma'-occur-plus} and \ref{lemma-Esigma'-occur-minus}.
\end{proof}

\subsubsection{Generalization}
In this subsection, we prove a generalization of Proposition \ref{prop-W2-coker-hasnoextension}. Let $\chi$ be a $2$-generic character of $I$.

\begin{proposition}\label{prop-coker-no-sigma}
Let $\tau$ be a Jordan--H\"older factor of $\Ind_I^K\overline{W}_{\chi,3}$.
\begin{enumerate}
\item[(i)] There exists a  morphism
 $\phi_{\tau}:\Proj_{\tGamma}\tau\ra \Ind_I^K\overline{W}_{\chi,3}$ such that
 $[\Coker(\phi_{\tau}):\tau]=0$.
\item[(ii)] For morphisms $\phi_{\tau}$ in (i),  $\im(\phi_{\tau})$ does not depend on  the choice of $\phi_{\tau}$.
 \end{enumerate}
\end{proposition}

\begin{proof}
(i) Since $\Ind_I^KW_{\chi,2}$  and $\Ind_{I}^K\overline{W}_{\chi,3}$ have the same Jordan--H\"older factors up to multiplicity, $\tau$ also occurs in $\Ind_I^KW_{\chi,2}$.
 By projectivity of $\Proj_{\tGamma}\tau$, we may lift $\phi_{\tau,2}$ to a morphism $\phi_{\tau}$, making the following diagram commutative
\[\xymatrix{\Proj_{\tGamma}\tau\ar@{-->}^{\phi_{\tau}\ }[r]\ar_{\phi_{\tau,2}}[rd]& \Ind_I^K\overline{W}_{\chi,3}\ar@{->>}[d] \\
&\Ind_I^KW_{\chi,2}.}\]
We need to prove that $[\mathrm{Coker}(\phi_{\tau}):\tau]=0$. The case $\tau\in\JH(\Ind_I^K\chi')$ for $\chi'\in\mathscr{E}(\chi)$ is obvious, so we  assume $\tau\in\JH(\Ind_I^K\chi)$ for the rest.

First treat the case $\tau=\sigma_{\cS}$, the cosocle of $\Ind_I^K\chi$. This case is essentially proved in \cite[Prop.~6.4.1]{BHHMS}. Assume $[\mathrm{Coker}(\phi_{\sigma_{\cS}}):\sigma_{\cS}]\geq 1$ for a contradiction. We may find a quotient of $\mathrm{Coker}(\phi_{\sigma_{\cS}})$, say $Q$, such that $[Q:\sigma_{\cS}]=1$. Consider  the induced $I$-equivariant morphism $f:\overline{W}_{\chi,3}\ra Q|_I$. Since $[Q:\sigma_{\cS}]=1$, we have $\dim_{\F}\Hom_I(\chi,Q)\leq 1$. Thus, $f$ must factor through a quotient of $\overline{W}_{\chi,3}$, say $W$, which satisfies the assumptions  of  \cite[Prop.~6.4.1]{BHHMS} and $Q$ is a quotient of $\Ind_I^KW$. By \emph{loc.~cit.}, $Q$ is a quotient of $\Ind_I^K\overline{W}_{\chi,2}$, hence of $\mathrm{Coker}(\phi_{\sigma_{\cS},2})$, but this is impossible as $[\mathrm{Coker}(\phi_{\sigma_{\cS},2}):\sigma_{\cS}]=0$.    Note that the genericity condition on $\chi$ of \emph{loc. cit.} is slightly stronger than ours, but it is caused by the use of \cite[Prop.~6.3.5]{BHHMS} which can be replaced by  our Proposition \ref{prop-W2-coker-hasnoextension}. 

Now we treat the general case $\tau\in \JH(\Ind_I^K\chi)$. 
The case for $\sigma_{\cS}$ treated above implies $[\mathrm{Coker}(\phi_{\sigma_{\cS}}):\tau]=0$ as well, as $\sigma_{\cS}$ is the cosocle of $\Ind_I^K\chi$. Thus $\mathrm{Im}(\phi_{\tau})$ is contained in $\mathrm{Im}(\phi_{\sigma_{\cS}})$, and by the projectivity of $\Proj_{\tGamma}\tau$ there exists a morphism $h:\Proj_{\tGamma}\tau\ra \Proj_{\tGamma}\sigma_{\cS}$ such that $\phi_{\tau}=\phi_{\sigma_{\cS}}\circ h$. By the construction, one checks that the composition 
\[\Proj_{\tGamma}\tau\overset{h}{\ra} \Proj_{\tGamma}\sigma_{\cS}\twoheadrightarrow I(\tau,\sigma_{\cS})\]
is nonzero, where $I(\tau,\sigma_{\cS})$ is as in Theorem \ref{thm-I(sigmatau)-dual}. We deduce that $[\mathrm{Coker}(h):\tau]=0$, because any quotient of $\Proj_{\tGamma}\sigma_{\cS}$ in which $\tau$ occurs  must admit $I(\tau,\sigma_{\cS})$ as a quotient by Theorem \ref{thm-I(sigmatau)-dual}.  By the snake lemma, we have an exact sequence
\[\mathrm{Coker}(h)\ra \mathrm{Coker}(\phi_{\tau})\ra \mathrm{Coker}(\phi_{\sigma_{\cS}})\ra0, \] 
from which the result follows. 

(ii) Fix a morphism $\phi_{\tau}$ as in (i). It suffices to show that if $\varphi:\Proj_{\tGamma}\tau\ra\Ind_I^K\overline{W}_{\chi,3}$ is any $K$-equivariant morphism, then  $\im(\varphi)\subset \im(\phi_{\tau})$. But, if it were not the case, the composite morphism $\im(\varphi)\hookrightarrow \Ind_I^KW_{\chi,3}\twoheadrightarrow \Coker(\phi_{\tau})$ would be nonzero. Since $\im(\varphi)$  has cosocle $\tau$, we get $[\Coker(\phi_{\tau}):\tau]\neq0$, a contradiction to (i). 
\end{proof}

\begin{corollary}\label{cor-imageoffi-inclusion}
Let $\tau_1,\tau_2$ be Jordan--H\"older factors of $\Ind_I^K\overline{W}_{\chi,3}$. Let $\phi_{\tau_i}:\Proj_{K/Z_1}\tau_i\ra \Ind_I^K\overline{W}_{\chi,3}$ be a morphism such that $[\Coker(\phi_{\tau_i}):\tau_i]=0$.  Then $\im(\phi_{\tau_1})\subseteq\im(\phi_{\tau_2})$ if one of the following cases happens:
\begin{enumerate}
\item[(a)] both $\tau_1$ and $\tau_2$ are subquotients of $\Ind_I^K\chi$ and $J(\tau_1)\subseteq J(\tau_2)$;
\item[(b)] both $\tau_1$ and $\tau_2$ are subquotients of $\Ind_I^K\chi'$ for some $\chi'\in\mathscr{E}(\chi)$, and $J(\tau_1)\subseteq J(\tau_2)$;
\item[(c)] $\tau_1$ (resp. $\tau_2$) is a subquotient of $\Ind_I^K\chi'$ with $\chi'\in\mathscr{E}(\chi)$ (resp. of\ $\Ind_I^K\chi$), and $\Ext^1_{\tGamma}(\tau_1,\tau_2)\neq0$.
\end{enumerate}
\end{corollary}
\begin{proof}
For $i\in\{1,2\}$, consider the  morphism \eqref{equation-defn-phi2}
\[\phi_{\tau_i, 2}: \Proj_{\tGamma}\tau_i\ra \Ind_I^KW_{\chi,2}.\]
We first prove \begin{equation}\label{equation-image-inclusion}\im(\phi_{\tau_1,2})\subset \im(\phi_{\tau_2,2})\end{equation} in the three cases listed in the corollary.  Since $\Ind_I^KW_{\chi,2}$ is multiplicity free and $\im(\phi_{\tau_1,2})$ has cosocle isomorphic to $\tau_1$, it suffices to prove  $[\Coker(\phi_{\tau_2,2}):\tau_1]= 0$.  This is clear in Case (a) by further projecting  to $\Ind_I^K\chi$, and also in Case (b) because $\im(\phi_{\tau_i,2})$ is contained in $\Ind_I^K\chi'$. In Case (c), it follows from Proposition \ref{prop-W2-coker-hasnoextension}.

Next, by projectivity of $\Proj_{\tGamma}\tau$  and \eqref{equation-image-inclusion}, we may lift $\phi_{\tau_1,2}$ to $\psi: \Proj_{\tGamma}\tau_1\ra \Proj_{\tGamma}\tau_2$, making the following diagram commutative
\[\xymatrix{\Proj_{\tGamma}\tau_1 \ar@{-->}^{\psi}[d]\ar^{\phi_{\tau_1,2}}[drr]&&\\
\Proj_{\tGamma}\tau_2\ar^{\phi_{\tau_2}}[r]&\Ind_I^K\overline{W}_{\chi,3}\ar[r]&\Ind_I^KW_{\chi,2}.}\]
By Proposition \ref{prop-coker-no-sigma}(ii), we have $\im(\phi_{\tau_2}\circ\psi)=\im(\phi_{\tau_1})$, hence $\im(\phi_{\tau_1})\subset\im(\phi_{\tau_2})$ as desired.
\end{proof}

\subsection{The representation $\Theta_{\tau}$}
\label{subsection:Theta}
We construct a certain $\tGamma$-representation which has an analogous submodule structure as $\overline{W}_{\chi,3}$.  Let $\chi$ be a $2$-generic character of $I$.

\begin{proposition}\label{prop-Theta}
For any Jordan--H\"older factor $\tau$ of $\Ind_I^K\chi$, $\Ind_I^K\overline{W}_{\chi,3}$ admits a subquotient, denoted by $\Theta_{\tau}$, satisfying the following properties:
\begin{enumerate}
\item[(i)] $\mathrm{cosoc}_{\tGamma}(\Theta_{\tau})$ is isomorphic to $\tau$ and $\rsoc_{\tGamma}(\Theta_{\tau})$  isomorphic to $\tau^{\oplus 2f}$;
\item[(ii)] $\rad_{\tGamma}(\Theta_{\tau})/\rsoc_{\tGamma}(\Theta_{\tau})$ is semisimple and isomorphic to $\bigoplus_{\tau'\in\mathscr{E}(\tau)}\tau'$.
\end{enumerate}
\end{proposition}

\begin{proof}
First note that $\tau$ is  $1$-generic as $\chi$ is $2$-generic. Hence, all the $\mu_{i}^{\pm}(\tau)$ are well-defined and  so $|\mathscr{E}(\tau)|=2f$.

Let $\phi_{\tau}:\Proj_{\tGamma}\tau\ra \Ind_I^K\overline{W}_{\chi,3}$ be a morphism as in Proposition \ref{prop-coker-no-sigma}. We will construct $\Theta_{\tau}$ as a certain quotient of $\im(\phi_{\tau})$; actually we just take $\Theta_{\tau}$ to be the quotient of $\im(\phi_{\tau})$ by the largest subrepresentation in which $\tau$ does not occur. But, to verify condition (i)  we divide this process into two steps.

First, by Proposition \ref{prop-coker-no-sigma} we have
\begin{equation}\label{equation-multi=2f+1}[\im(\phi_{\tau}):\tau]=[\Ind_I^K\overline{W}_{\chi,3}:\tau]=2f+1.\end{equation}
Let $\sigma_{\emptyset}$ denote the socle of $\Ind_I^K\chi$. Since $\chi^{\oplus 2f}\hookrightarrow \rsoc_I(\overline{W}_{\chi,3})$, we obtain an embedding
\[I(\sigma_{\emptyset},\tau)^{\oplus 2f}\hookrightarrow \Ind_I^K\chi^{\oplus 2f}\hookrightarrow\Ind_I^K\overline{W}_{\chi,3}\]
whose image is contained in $\im(\phi_{\tau})$ by Proposition \ref{prop-coker-no-sigma}. In particular, modulo $\rad(I(\sigma_{\emptyset},\tau))^{\oplus 2f}$, we obtain a quotient  of $\im(\phi_{\tau})$, say $Q$, such that  $\tau^{\oplus 2f}$ embeds in $Q$.  Moreover, by \eqref{equation-multi=2f+1}, it is easy to see that $\dim_{\F}\Hom_{\tGamma}(\tau,Q)=2f$.  Next, we define $\Theta_{\tau}$ to be the quotient of $Q$  by the  largest subrepresentation of $Q$ in which $\tau$ does not occur. It is then clear that Condition (i) is satisfied, and (ii) follows from Corollary \ref{coro-Q-length3} (as $|\mathscr{E}(\tau)|=2f$).  
\end{proof}

Note that $\Theta_{\tau}$ can be defined for any $2$-generic Serre weight $\tau$,  taking $\chi=\chi_{\tau}$ in Proposition \ref{prop-Theta}.

\begin{corollary}\label{cor-Theta-noextension}
Let $\tau$ be a $2$-generic Serre weight. Then $\Ext^1_{\tGamma}(\Theta_{\tau},\tau)=0$.
\end{corollary}
\begin{proof}
From the exact sequence
$0\ra \rsoc(\Theta_{\tau})\ra \Theta_{\tau}\ra \Theta_{\tau}/\rsoc(\Theta_{\tau})\ra0$, we obtain
\begin{multline*}0\ra \Hom_{\tGamma}(\rsoc(\Theta_{\tau}),\tau)\ra \Ext^1_{\tGamma}(\Theta_{\tau}/\rsoc(\Theta_{\tau}),\tau)
\ra \Ext^1_{\tGamma}(\Theta_{\tau},\tau)\ra \Ext^1_{\tGamma}(\soc(\Theta_{\tau}),\tau)=0\end{multline*}
where the vanishing of the last term follows from   Proposition \ref{prop-Theta}(i) and the fact  $\Ext^1_{\tGamma}(\tau,\tau)=0$  by Lemma \ref{lemma-Hu10-2.21}(i) (note that $\tau$ is $1$-generic as $\chi$ is $2$-generic). By Proposition \ref{prop-Theta}(i), it suffices to show $\dim_{\F}\Ext^1_{\tGamma}(\Theta_{\tau}/\rsoc(\Theta_{\tau}),\tau)\leq2f$. Using again the fact $\Ext^1_{\tGamma}(\tau,\tau)=0$, the exact sequence $0\ra \oplus_{\tau'\in\mathscr{E}(\tau)}\tau'\ra
\Theta_{\tau}/\rsoc(\Theta_{\tau})\ra \tau\ra0$ induces an injection
\[\Ext^1_{\tGamma}(\Theta_{\tau}/\rsoc(\Theta_{\tau}),\tau)\hookrightarrow \bigoplus_{\tau'\in\mathscr{E}(\tau)}\Ext^1_{\tGamma}(\tau',\tau),\]
and the result follows because $\dim\Ext^1_{\tGamma}(\tau',\tau)=1$ for any $\tau'\in\mathscr{E}(\tau)$.
\end{proof}

 \begin{corollary}\label{cor:Theta-univ}
Keep the notation of Proposition \ref{prop-Theta}. Let $Q$ be a quotient of $\Proj_{\tGamma}\tau$. Assume that there exists an injection $\tau^{\oplus m}\hookrightarrow Q$ for some  $m\geq 0$, such that $Q/\tau^{\oplus m}$  has Loewy length $2$ and fits in a short exact sequence
\[0\ra S\ra Q/\tau^{\oplus m}\ra \tau\ra0\]
where $S$ is a subrepresentation of $\bigoplus_{\tau'\in\mathscr{E}(\tau)}\tau'$.
Then $Q$ is a quotient of $\Theta_{\tau}$ (in particular $m\leq 2f$).  \end{corollary}
 \begin{proof}
By Proposition \ref{prop-Theta}, there is a short exact sequence  $0\ra \bigoplus_{\tau'\in\mathscr{E}(\tau)}\tau'\ra \Theta_{\tau}/\tau^{\oplus 2f}\ra \tau\ra0$. As a consequence, the assumption on $Q$ implies a surjection $\Theta_{\tau}/\tau^{\oplus 2f}\twoheadrightarrow Q/\tau^{\oplus m}$, thus a surjection $\iota:\Theta_{\tau}\twoheadrightarrow Q/\tau^{\oplus m}$.  Corollary \ref{cor-Theta-noextension} implies that the natural morphism 
 $\Hom_{\tGamma}(\Theta_{\tau},Q)\twoheadrightarrow \Hom_{\tGamma}(\Theta_{\tau},Q/\tau^{\oplus m})$ is surjective. Therefore,  $\iota$ can be lifted to a morphism $\Theta_{\tau}\ra Q$, which is surjective (being surjective on cosocles), as required.
 \end{proof}

 \begin{corollary}\label{cor:Theta-indep}
The representation $\Theta_{\tau}$ constructed in Proposition \ref{prop-Theta} does not depend on the choice of $\chi$.
 \end{corollary}
 \begin{proof}
 It is a direct consequence of Corollary \ref{cor:Theta-univ}.
 \end{proof}

By a similar proof of Lemma \ref{lemma-bar-W}, we have the following result showing that $\Theta_{\tau}$ has an analogous structure as $\overline{W}_{\chi,3}$.

\begin{corollary}\label{cor:Theta-tau-structure}
For any $\tau'\in\mathscr{E}(\tau)$,  there exist both an embedding
$E_{\tau,\tau'}\hookrightarrow \Theta_{\tau}$
and a quotient $\Theta_{\tau}\twoheadrightarrow E_{\tau',\tau}$. Moreover, $\Theta_{\tau}$ fits in a short exact sequence
\begin{equation}\label{eq:seq-Theta-tau}0\ra\bigoplus_{\tau'\in\mathscr{E}(\tau)}E_{\tau,\tau'}\ra \Theta_{\tau}\ra\tau\ra0.\end{equation}
\end{corollary}

\subsection{The representation $\Theta_{\tau}^{\rm ord}$}
\label{subsection:Theta-ord}

Keep the notation of the last subsection. In this subsection, we define a certain quotient of $\Theta_{\tau}$ which is related to ordinary parts of representations of $\GL_2(L)$  studied in \S\ref{section:ordinary}.

Fix a Serre weight $\tau$ which is a Jordan--H\"older factor of $\Ind_{I}^K\chi$ for some $2$-generic character $\chi$. For example, $\tau$ can be any $2$-generic Serre weight.

\begin{lemma}\label{lemma:Theta-ord}
There exists a unique quotient of  $\Theta_{\tau}$, denoted by $\Theta_{\tau}^{\rm ord}$, which has Loewy length $3$ and satisfying the following properties:
\begin{enumerate}
\item[(i)] $\soc (\Theta_{\tau}^{\rm ord})$ is isomorphic to $\tau^{\oplus f}$;
\item[(ii)] $\rad(\Theta_{\tau}^{\rm ord})/\soc(\Theta_{\tau}^{\rm ord})$ is semisimple and isomorphic to $\bigoplus_{i\in\cS}\mu_i^-(\tau)$.
\end{enumerate}
\end{lemma}

\begin{proof}
With the notation of Corollary \ref{cor:Theta-tau-structure}, it suffices to take $\Theta_{\tau}^{\rm ord}$ to be the quotient of $\Theta_{\tau}$ by $\oplus_{\tau'}E_{\tau,\tau'}$, where $\tau'$ runs over the Serre weights $\{\mu_{i}^+(\tau), i\in\cS\}$.
\end{proof}
The proof of Lemma \ref{lemma:Theta-ord} shows that $\Theta_{\tau}^{\rm ord}$ fits in a short exact sequence
\begin{equation}\label{eq:seq-Theta-ord}0\ra\bigoplus_{i\in\cS}E_{\tau,\mu_i^-(\tau)}\ra \Theta_{\tau}^{\rm ord}\ra\tau\ra0.\end{equation}

For a  smooth  representation $V$ of $K$, denote by $V_{K_1}$ the space of $K_1$-coinvariants of $V$; it  is equal to the largest quotient of $V$ on which $K_1$ acts trivially.

\begin{lemma}\label{lemma:Theta-ord-K1}
We have  $(\Theta_{\tau}^{\rm ord})_{K_1}=\Theta_{\tau}^{\rm ord}/\soc(\Theta_{\tau}^{\rm ord})$. Moreover, it is a quotient of $\Ind_I^K\chi_{\tau}$.
\end{lemma}
\begin{proof}
Using Lemma \ref{lemma-Hu10-2.21}, it is easy to check that a $\tGamma$-representation $V$ satisfying $0\ra \bigoplus_{i\in\cS}\mu_i^-(\tau) \ra V\ra \tau\ra 0$ and $\mathrm{cosoc}_{\tGamma}(V)=\tau$ is unique (up to isomorphism) and is actually a $\Gamma$-representation. Lemma \ref{lemma:Theta-ord}(ii) implies that  $ \Theta_{\tau}^{\rm ord}/\soc(\Theta_{\tau}^{\rm ord})$ is such a representation. On the other hand, it  follows from  \cite[Thm.~2.4]{BP} that $\Ind_I^K\chi_{\tau}$ also admits such a representation as a quotient.   This proves the  second assertion and that $\Theta_{\tau}^{\rm ord}/\soc(\Theta_{\tau}^{\rm ord})$
is a quotient of $(\Theta_{\tau}^{\rm ord})_{K_1}$.

Recall that $\soc(\Theta_{\tau}^{\rm ord})\cong \tau^{\oplus f}$ by Lemma \ref{lemma:Theta-ord}(i). To prove the first assertion, it suffices to prove
\[\Ext^1_{\Gamma}\big(\Theta_{\tau}^{\rm ord}/\soc(\Theta_{\tau}^{\rm ord}),\tau\big)=0.\]
To simplify the notation, write $A$ for $\Theta_{\tau}^{\rm ord}/\soc(\Theta_{\tau}^{\rm ord})$, and $B$ for the corresponding kernel fitting in the exact sequence
\[0\ra B\ra \Ind_I^K\chi_{\tau}\ra A\ra0.\]
We  may identify $\Ind_I^K\chi_{\tau}$ with $\Ind_{P(\F_q)}^{\Gamma}\chi_{\tau}$.   Using  the fact that  $\Ind_{P(\F_q)}^{\Gamma}\chi_{\tau}$ is multiplicity free, we get  $\Hom_{\Gamma}(B,\tau)=0$.  Hence, it suffices to prove $\Ext^1_{\Gamma}(\Ind_{P(\F_q)}^{\Gamma}\chi_{\tau},\tau)=0$, equivalently $\Ext^1_{P(\F_q)}(\chi_{\tau},\tau)=0$ by Shapiro's lemma. But this follows from (a variant of) \cite[Prop.~2.5]{HuJLMS}. Note that, if $f=1$ then we need $\dim_{\F}\tau\neq p-2$ to ensure the vanishing of $\Ext^1_{P(\F_p)}(\chi_{\tau},\tau)$, whereas no genericity condition is needed if $f\geq 2$.
\end{proof}

\begin{lemma}\label{lemma:Theta-K1}
There exists an exact sequence
\[0\ra \tau^{\oplus f}\ra \Theta_{\tau}\ra (\Theta_{\tau})_{K_1}\ra0.\]
\end{lemma}
\begin{proof}
In fact, we can determine $(\Theta_{\tau})_{K_1}$ explicitly. First observe that there exists a unique quotient $Q$ of $\Proj_{\Gamma}\tau$, which has Loewy length $3$ and such that
\begin{enumerate}
\item[$\bullet$] $\soc(Q)=\tau^{\oplus f}$;
\item[$\bullet$] $\rad(Q)/\soc(Q)\cong\bigoplus_{\tau'\in\mathscr{E}(\tau)}\tau$.
\end{enumerate}
Indeed, it suffices to take $Q$ to be the dual of $A_{\sigma}'$ defined in \cite[Def.~2.5]{HW} with $\sigma=\tau^{\vee}$.
Corollary \ref{cor:Theta-univ}  shows
 that $Q$ is a quotient of $\Theta_{\tau}$, hence of $(\Theta_{\tau})_{K_1}$.  In particular, $[(\Theta_{\tau})_{K_1}:\tau]\geq f+1$.

By construction, there is a short exact sequence
\[0\ra \bigoplus_{i\in \cS}E_{\tau,\mu_i^+(\tau)}\ra \Theta_{\tau}\ra\Theta_{\tau}^{\rm ord}\ra0\]
which induces
\begin{equation}\label{eq:Theta-inproof}\bigoplus_{i\in \cS}E_{\tau,\mu_i^+(\tau)}\overset{\iota}{\ra} (\Theta_{\tau})_{K_1}\ra(\Theta_{\tau}^{\rm ord})_{K_1}\ra0.\end{equation}
By Lemma \ref{lemma:Theta-ord-K1}, $[(\Theta_{\tau}^{\rm ord})_{K_1}:\tau]=1$. Comparing the multiplicity of $\tau$,  we see that $\iota$ has to be injective because it is injective on socle. This implies that $[(\Theta_{\tau})_{K_1}:\tau]=f+1$ and a comparison of Jordan--H\"older factors using \eqref{eq:Theta-inproof} shows $Q=(\Theta_{\tau})_{K_1}$.
\end{proof}
 The next result will be used in \S\ref{section-patching}.

\begin{proposition} \label{prop:Theta-twoparts}
 There exists a  short exact sequence
\begin{equation}\label{eq:Theta-twoparts}0\lra \Theta_{\tau}\lra \Theta_{\tau}^{\rm ord}\oplus(\Theta_{\tau})_{K_1} \overset{(q_1,q_2)}{\lra} (\Theta_{\tau}^{\rm ord})_{K_1}\lra0\end{equation}
where $q_1, q_2$ are natural projections.
\end{proposition}

\begin{proof}
Let $V$ be the kernel of the map $(q_1,q_2)$; we need to prove $V\cong \Theta_{\tau}$. Taking $K_1$-coinvariants of
$ 
0\ra V\ra \Theta_{\tau}^{\rm ord} \oplus (\Theta_{\tau})_{K_1}\ra (\Theta_{\tau}^{\rm ord})_{K_1}\ra0$  
 induces a sequence
\[0\ra V_{K_1}\ra (\Theta_{\tau}^{\rm ord})_{K_1}\oplus (\Theta_{\tau})_{K_1}\ra (\Theta_{\tau}^{\rm ord})_{K_1}\ra0,\]
which is  exact because the morphism
 \[\big[H_1(K_1,(\Theta_{\tau})_{K_1})\ra H_1(K_1,(\Theta_{\tau}^{\rm ord})_{K_1})\big]=H_1(K_1,\F)\otimes\big[(\Theta_{\tau})_{K_1}\ra (\Theta_{\tau}^{\rm ord})_{K_1}\big]\]
is (automatically) surjective.  This implies
\begin{enumerate}
\item[(1)] an isomorphism $V_{K_1}\cong (\Theta_{\tau})_{K_1}$ and 
\item[(2)] a short exact sequence using  Lemmas \ref{lemma:Theta-ord} and \ref{lemma:Theta-ord-K1}
\[0\ra \tau^{\oplus f}\ra V\ra V_{K_1}\ra0. \]
\end{enumerate}
From (1) we deduce that $\mathrm{cosoc}(V)\cong \tau$ and so $V$ is a quotient of $\Proj_{\tGamma}\tau$.  Using Corollary \ref{cor:Theta-univ} and Lemma \ref{lemma:Theta-K1}, it is easy to check that  $V=\Theta_{\tau}$.
  \end{proof}

\section{Combinatorics \`a la Breuil-Pa\v{s}k\=unas}\label{section-BP}

In this section, we recall and generalize a construction of Breuil and Pa\v{s}k\=unas (\cite[\S13]{BP}). Keep the notation in previous sections. In particular, $K=\GL_2(\cO_L)$, $\Gamma=\F[\GL_2(\F_q)]$ and $\tGamma= \FKZ/\fm_{K_1}^2. $

Fix a continuous representation $\brho:G_L\ra \GL_2(\F)$, which  is \emph{generic} in the sense of \cite[\S11]{BP}, that is, $\brho|_{I(\bQp/L)}$ is isomorphic to one of the following two forms (always possible up to twist)
\begin{enumerate}
\item $\matr{\omega_f^{\sum_{i=0}^{f-1}p^i(r_i+1)}}*01$ with $0\leq r_i\leq p-3$ for each $i$, and not all $r_i$ equal to $0$ or equal to $p-3$;
\item $\matr{\omega_{2f}^{\sum_{i=0}^{f-1}p^i(r_i+1)}}00{\omega_{2f}^{p^f\sum_{i=0}^{f-1}p^i(r_i+1)}}$ with $1\leq r_0\leq p-2$, and $0\leq r_i\leq p-3$ for $i>0$.
\end{enumerate}
where $\omega_{f'}$ is Serre's fundamental character of $I(\bQp/L)$ of level $f'$ for $f'\in\{f,2f\}$.

To $\brho$ is associated a set of Serre weights,  denoted by $\mathscr{D}(\brho)$ (see \cite[\S11]{BP}).    The genericity of $\brho$ implies that the cardinality of $\mathscr{D}(\brho)$ is $2^f$ if $\brho$ is semisimple, and is $2^d$ for some $0\leq d\leq f-1$ if $\brho$ is reducible nonsplit. If $\brho$ is reducible and if $\brho^{\rm ss}$ denotes the semisimplification of $\brho$, then we always have $\mathscr{D}(\brho)\subseteq \mathscr{D}(\brho^{\rm ss})$. In fact, by \cite[\S4]{Br14}, the set $\mathscr{D}(\brho^{\rm ss})$ is parametrized by a certain set $\mathscr{RD}(x_0,\cdots,x_{f-1})$ of $f$-tuples $\lambda=(\lambda_j(x_j))_{j\in\cS}$ satisfying $\lambda_j(x_j)\in\{x_j,x_j+1,p-3-x_j,p-2-x_j\}$ and some other conditions, in the sense that
\begin{multline}\label{eq:lambda-rho}\mathscr{D}(\brho^{\rm ss})=\big\{(\lambda_0(r_0),\cdots,\lambda_{f-1}(r_{f-1}))\otimes {\det}^{e(\lambda)(r_0,\cdots,r_{f-1})}: \lambda\in\mathscr{RD}(x_0,\cdots,x_{f-1}) \big\}.\end{multline}
Then  $\mathscr{D}(\brho)$ corresponds  exactly to the subset  of $\mathscr{RD}(x_0,\cdots,x_{f-1})$ consisting of $\lambda$ such that
$\lambda_j(x_j)\in\{p-3-x_j,x_j+1\}$ implies $j\in J_{\brho}$, where $J_{\brho}$ is a certain subset of $\cS$ uniquely determined by the Fontaine-Laffaille module of $\brho$ (see \cite[\S4]{Br14}).

It is constructed in \cite[\S13]{BP} a finite dimensional representation $D_0(\brho)$ of $\Gamma$ such that
\begin{enumerate}
\item[(i)] $\rsoc_{\Gamma}D_0(\brho)=\oplus_{\sigma\in\mathscr{D}(\brho)}\sigma$; 
\item[(ii)] any Serre weight of $\mathscr{D}(\brho)$ occurs at most once as a subquotient in $D_0(\brho)$; 
\item[(iii)] $D_0(\brho)$ is maximal with respect to properties (i), (ii).
\end{enumerate}
By \cite[Prop.~13.1]{BP}, we have a decomposition of  $\Gamma$-representations
\[D_0(\brho)=\bigoplus_{\sigma\in\mathscr{D}(\brho)}D_{0,\sigma}(\brho)\]
 with each $D_{0,\sigma}(\brho)$ satisfying $\rsoc_{\Gamma}D_{0,\sigma}(\brho)=\sigma$. Moreover, $D_0(\brho)$ is multiplicity free by \cite[Cor.~13.5]{BP}.

The aim of this section is to generalize the above construction to $\tGamma$-representations and relate it to a certain class of admissible smooth $\GL_2(L)$-representations over $\F$.

\subsection{The representation $\widetilde{D}_0(\brho)$}
\label{subsection:tD0}

\begin{proposition}\label{prop-BP-13.1}
Let $\mathscr{D}$ be a finite set of distinct Serre weights. Then there exists a unique (up to isomorphism) finite dimensional  representation $\widetilde{D}_0$ of $\tGamma$ (over $\F$) such that
\begin{enumerate}
\item[(i)] $\rsoc_{\tGamma}\widetilde{D}_0=\bigoplus_{\sigma\in\mathscr{D}}\sigma$;  
\item[(ii)] any Serre weight of $\mathscr{D}$ occurs at most once as a subquotient in $\widetilde{D}_0$; 
\item[(iii)] $\widetilde{D}_0$ is maximal with respect to properties (i), (ii).
\end{enumerate}
Moreover, there is an isomorphism of $\tGamma$-representations
$\widetilde{D}_0=\bigoplus_{\sigma\in\mathscr{D}}\widetilde{D}_{0,\sigma}$
with $\rsoc_{\tGamma}\widetilde{D}_{0,\sigma}=\sigma$.
\end{proposition}
\begin{proof}
The proof is the same as \cite[Prop.~13.1]{BP}.\footnote{At the end of the proof of \cite[Prop.~13.1]{BP}, the idempotent $e_{\sigma}\in \End_{\Gamma}(\rInj_{\Gamma}\tau)$ need not be unique; but we can certainly choose $e_{\sigma}\in\End_{\Gamma}(\rInj\tau)$ for each $\sigma$ such that $\sum_{\sigma}e_{\sigma}=1$, and the rest of the proof goes through.}
\end{proof}

 \begin{corollary}\label{cor:tD-sigma}
 With the notation of Proposition \ref{prop-BP-13.1}, for $\sigma\in\mathscr{D}$, $\widetilde{D}_{0,\sigma}$ is the largest subrepresentation of $\rInj_{\tGamma}\sigma$ such that $[\widetilde{D}_{0,\sigma}:\sigma]=1$ and $[\widetilde{D}_{0,\sigma}:\tau]=0$ for any $\tau\in \mathscr{D}$ with $\tau \neq \sigma$.
 \end{corollary}
\begin{proof}
If $\widetilde{D}_{0,\sigma}'\subseteq\rInj_{\tGamma}\sigma$ is another subrepresentation satisfying the conditions in the corollary, then the sum $\widetilde{D}_{0,\sigma}+\widetilde{D}_{0,\sigma}'$ also satisfies these conditions by the proof of \cite[Prop.~13.1]{BP}. But, the direct sum $(\widetilde{D}_{0,\sigma}+\widetilde{D}_{0,\sigma}')\oplus(\oplus_{ \tau\neq \sigma}\widetilde{D}_{0,\tau})$   also satisfies the conditions of Proposition \ref{prop-BP-13.1}, so we must have  $\widetilde{D}_{0,\sigma}+\widetilde{D}_{0,\sigma}'=\widetilde{D}_{0,\sigma}$, i.e. $\widetilde{D}_{0,\sigma}'\subseteq\widetilde{D}_{0,\sigma}$.\end{proof}
\begin{definition}
Define $\tD_0(\brho)$ to be the representation attached to $\mathscr{D}=\mathscr{D}(\brho)$ by Proposition \ref{prop-BP-13.1}.
\end{definition}

By Proposition \ref{prop-BP-13.1}, there is a direct sum decomposition
\begin{equation}\label{eq:tD-decomp}
\tD_0(\brho)=\bigoplus_{\sigma\in\mathscr{D}(\brho)}\tD_{0,\sigma}(\brho)\end{equation}
with $\rsoc_{\tGamma} \tD_{0,\sigma}(\brho)=\sigma$.

\begin{definition}\label{defn:strong-generic}
We say $\brho$ is \emph{strongly generic} if, in Case (1), $2\leq r_i\leq p-5$ for each $i$, or in Case (2), $3\leq r_0\leq p-4$ and $2\leq r_i\leq p-5$ for $i>0$.
\end{definition}
 \begin{lemma}\label{lem:Serre=3generic}
Assume $\brho$ is strongly generic. Then any $\sigma\in\mathscr{D}(\brho)$ is $2$-generic.
\end{lemma}
\begin{proof}
It is a direct check using the explicit description of $\mathscr{D}(\brho)$ in \cite[\S11]{BP}.
\end{proof}

 The main result of this subsection is the following.

\begin{theorem}\label{thm-tD=multione}
Assume $\brho$ is strongly generic. The representation $\widetilde{D}_0(\brho)$ is multiplicity free. Moreover, for any $\sigma\in\mathscr{D}(\brho)$, we have $D_{0,\sigma}(\brho)\subset \tD_{0,\sigma}(\brho)$ and $\tD_{0,\sigma}(\brho)^{K_1}\cong D_{0,\sigma}(\brho)$.
\end{theorem}
\begin{proof}
Using Lemma \ref{lemma-BP-12.8} below, the first assertion is proved by the same argument as in \cite[Cor.~13.5]{BP}. The second assertion is clear from the construction.
\end{proof}
 \begin{remark}
 A similar result is proved in \cite[\S6.3]{BHHMS} when $\brho$ is semisimple; moreover, the set of Jordan--H\"older factors and the submodule structure of $\widetilde{D}_0(\brho)$ are  determined.
 \end{remark}

For a Serre weight $\tau$, define
\[\ell(\brho,\tau)\defn\mathrm{min}\{\ell(\sigma,\tau),\ \sigma\in\mathscr{D}(\brho)\}\in\Z_{>0}\cup\{+\infty\},\]
where $\ell(\sigma,\tau)\defn +\infty$ if $\tau$ does not occur in $\rInj_{\tGamma}\sigma$, and  is the Loewy length of $I(\sigma,\tau)$ otherwise. Here, $I(\sigma,\tau)$ is the representation of $\tGamma$ constructed in Theorem \ref{thm-I(sigmatau)-tGamma}, well-defined thanks to Lemma \ref{lem:Serre=3generic}. The following result is an analog of \cite[Lem.~12.8]{BP} in our setting.

\begin{lemma}\label{lemma-BP-12.8}
Assume $\brho$ is strongly generic. Let $\tau$ be any Serre weight such that $\ell(\brho,\tau)<+\infty$.

(i) There exists a unique $\sigma\in\mathscr{D}(\brho)$ such that $\ell(\sigma,\tau)=\ell(\brho,\tau)$.

(ii) Let $\sigma'\in\mathscr{D}(\brho)$ such that $I(\sigma',\tau)\neq0$. If $\sigma'\neq\sigma$ with $\sigma$ as in (i), then  $I(\sigma',\tau)$ contains $\sigma$ as a subquotient.
\end{lemma}
\begin{proof}
 Let $\sigma\in\mathscr{D}(\brho)$ be a Serre weight such that
\begin{equation}\label{equation-ell-sigma}\ell(\sigma,\tau)= \ell(\brho,\tau).\end{equation}
Also let $\sigma'\in\mathscr{D}(\brho)$ be a Serre weight distinct with $\sigma$ such that $I(\sigma',\tau)\neq 0$. We will prove that $I(\sigma',\tau)$ contains $\sigma$ as a subquotient, which will prove (i) and (ii) simultaneously. In the exceptional case $f=1$, $\sigma=\Sym^2\F^2\otimes{\det}^a$ and $\tau=\Sym^{p-1}\F^2\otimes{\det}^{a+1}$, one checks that $\sigma$ is the unique Serre weight in $\mathscr{D}(\brho)$ such that $\ell(\sigma,\tau)<+\infty$, so the result is obvious and we exclude this case in the rest.

Since $I(\sigma,\tau)\neq 0$ and $I(\sigma',\tau)\neq0$,
 we have the following possibilities:
\begin{enumerate}
\item[(a)] $\tau$ is an old Serre weight in both $\rInj_{\tGamma}\sigma$ and $\rInj_{\tGamma}\sigma'$, i.e. $\tau\in \JH(\rInj_{\Gamma}\sigma)\cap\JH(\rInj_{\Gamma}\sigma')$;
\item[(b)] $\tau$ is a new Serre weights in both $\rInj_{\tGamma}\sigma$ and $\rInj_{\tGamma}\sigma'$, i.e. $\tau\notin \JH(\rInj_{\Gamma}\sigma)\cup \JH(\rInj_{\Gamma}\sigma')$;
\item[(c)] $\tau$ is an old  Serre weight in $\rInj_{\tGamma}\sigma$, but  new in $\rInj_{\tGamma}\sigma'$, i.e. $\tau\in\JH(\rInj_{\Gamma}\sigma)\backslash \JH(\rInj_{\Gamma}\sigma');$
\item[(c')] $\tau$ is a new Serre weight in $\rInj_{\tGamma}\sigma$, but  old in $\rInj_{\tGamma}\sigma'$, i.e. $\tau\in\JH(\rInj_{\Gamma}\sigma')\backslash \JH(\rInj_{\Gamma}\sigma).$
\end{enumerate}
\vspace{1mm}

In Case (a), we may view both $\sigma$ and $\sigma'$ as subquotients of $\rInj_{\Gamma}\tau$.  If $\lambda,\lambda'\in\cI(x_0,\cdots,x_{f-1})$  correspond to $\sigma,\sigma'$ respectively, then \cite[Lem.~12.6]{BP} implies that the intersection $\lambda\cap\lambda'\in\cI(x_0,\cdots,x_{f-1})$ (see \cite[\S12]{BP} for the definition of $\lambda\cap \lambda'$) corresponds again to a Serre weight  in $\mathscr{D}(\brho)$, say $\sigma''$. It is clear that $\ell(\sigma'',\tau)\leq\ell(\sigma,\tau)$, with equality if and only if $\sigma''=\sigma$. By \eqref{equation-ell-sigma} we indeed have $\sigma''=\sigma$ and $\sigma$ occurs in $I(\tau,\sigma')$ by \cite[Lem.~12.5]{BP}, hence also in $I(\sigma',\tau)$.

In Case (b), there exist uniquely determined $(i,*)$ and $(i',*')$ in $\cS\times\{+,-\}$ such that \[\tau\in\JH(\rInj_{\Gamma}\delta_i^*(\sigma))\cap \JH(\rInj_{\Gamma}\delta_{i'}^{*'}(\sigma')).\]
Since $\sigma'$ occurs in $\rInj_{\Gamma}\sigma$ by \cite[Prop.~2.24]{HW}, we may write $\sigma'=\mu(\sigma)$ for $\mu\in\cI(x_0,\cdots,x_{f-1})$.
Let  $\lambda$ (resp. $\lambda'$) be the element of $\cI(x_0,\cdots,x_{f-1})$ such that $\tau=\lambda(\delta_i^*(\sigma))$ (resp. $\tau=\lambda'(\delta_{i'}^{*'}(\sigma'))$). Using Lemma \ref{lemma:genericity} together with (a variant of) Lemma \ref{lemma:cond-for-cI}, we have
\begin{equation}\label{eq:tuple-caseb}\lambda'\circ\delta_{i'}^{*'}\circ\mu=\lambda\circ\delta_i^{*}. \end{equation}
We have two possibilities: $i=i'$ or $i\neq i'$.
\begin{enumerate}
\item[(b1)] Assume $i=i'$. Then by the proof of Lemma \ref{lemma-unique-delta}, precisely by \eqref{equation-nu-case+} and \eqref{equation-nu-case-}, we must have $*=*'$.  Moreover, using  \eqref{equation-nu-case+}  (or \eqref{equation-nu-case-}, depending on $*$), the equality $\lambda'\circ\delta_i^*\circ\mu=\lambda\circ\delta_i^*$ forces $\mu_i(x_i)\in\{x_i,x_i\pm 1\}$ and so $\delta_i^*\circ\mu=\mu\circ\delta_i^*$.  Hence, \eqref{eq:tuple-caseb} becomes $\lambda'\circ\mu\circ\delta_i^*=\lambda\circ\delta_i^*$, equivalently, $\lambda'\circ\mu=\lambda$.  If we define $\tau'\defn \lambda(\sigma)=\lambda'(\sigma')$, then $\tau'$ is a common subquotient of $\rInj_{\Gamma}\sigma$ and $\rInj_{\Gamma}\sigma'$. Consequently, we may view $\sigma$ and $\sigma'$ as subquotients of $\rInj_{\Gamma}\tau'$.
As in Case (a), applying \cite[Lem.~12.6]{BP}, we obtain a Serre weight $\sigma''\in\mathscr{D}(\brho)$ which  occurs in  $\JH(I(\sigma,\tau'))\cap\JH(I(\sigma',\tau'))$. On the other hand,  by  Corollary \ref{cor-I(sigmatau)-JH}, $\tau'$ is a common subquotient of $I(\sigma,\tau)$ and $I(\sigma',\tau)$, hence so is $\sigma''$.  By \eqref{equation-ell-sigma}, this forces  $\sigma''=\sigma$, and so $\sigma$ occurs in $I(\sigma',\tau)$.

\item[(b2)] Assume $i\neq i'$ (so $f\geq 2$). Using \eqref{eq:tuple-caseb} at $i$, we deduce from Lemma \ref{lemma-lambda'=mu-lambda} (condition (a) in \emph{loc. cit.} holds by Lemma \ref{lemma-newweight})  the following facts
\begin{enumerate}
\item[$\bullet$] $i\notin\cS(\lambda)$, $i\in\cS(\mu)$;
\item[$\bullet$] $\mu_i(x_i)\in\{x_i,x_i*1,p-2-x_i,p-2-x_i-(*1)\}$.
\end{enumerate}
Let $\lambda''\in \cI(x_0,\cdots,x_{f-1})$ be the unique element with $\cS(\lambda'')=\{i\}$ and satisfying \eqref{equation-condition-newweight}, i.e. $(\lambda'',\emptyset)\in\cI_{(i,*)}$ (see Definition \ref{defn:tildeI}) and let $\sigma''=\lambda''(\sigma)$. On the one hand,  we have  $\sigma''\in\JH(I(\sigma,\tau))$ by Corollary \ref{cor-I(sigmatau)-JH}.  On the other hand,  we have $\lambda''\leq \mu$, hence $\sigma''\in \JH(I(\sigma,\sigma'))$ by Proposition \ref{prop-BP-4.11}. By \cite[Prop.~2.24]{HW}, this implies $\sigma''\in \mathscr{D}(\brho)$. However, it is clear that $\ell(\sigma'',\tau)<\ell(\sigma,\tau)$, which  contradicts the choice of $\sigma$, see \eqref{equation-ell-sigma}.
 \end{enumerate}

In Case (c), we may view both $\tau$ and $\sigma'$ as  subquotients  of $\rInj_{\Gamma}\sigma$ (use \cite[Prop.~2.24]{HW} for $\sigma'$); let $\lambda, \lambda'\in \cI(x_0,\cdots,x_{f-1})$ be the corresponding element, respectively.   By Lemma  \ref{lemma-HW-2.20}, we have $\cS(\lambda)\cap \cS(\lambda')\neq\emptyset$, otherwise $\sigma'$ would occur in $\rInj_{\Gamma}\tau$, contradicting the assumption.  We claim that \[|\cS(\lambda)\cap \cS(\lambda')|=1.\] Let $i\in\cS(\lambda)\cap \cS(\lambda')$. Then Lemma \ref{lemma-HW-2.19} implies that  $\lambda$ and $\lambda'$  are \emph{not} compatible at $i$; otherwise $I(\sigma,\sigma')$ and $I(\sigma,\tau)$ would contain a common irreducible subquotient distinct with $\sigma$, say $\sigma''$, and by \cite[Prop.~2.24]{HW} $\sigma''$  must lie in $\mathscr{D}(\brho)$,   contradicting \eqref{equation-ell-sigma}.   Set $\nu\defn \lambda\circ\lambda'^{-1}$ so that
\[\tau=\lambda(\sigma)=\nu(\sigma').\]
Using the table in the proof of Lemma \ref{lemma-compose-lambda}, a case-by-case check shows that
\[\nu_i(x_i)\in\{p-x_i,p-4-x_i,x_i+2,x_i-2\}.\]
For example, if $\lambda'_i(x_i)=p-1-x_i$, then $\lambda_i'^{-1}(x_i)=p-1-x_i$ and $\lambda_i(x_i)\in\{p-3-x_i,x_i+1\}$ (as $\lambda$ and $\lambda'$ are not compatible at $i$), so finally $\nu_i(x_i)\in\{x_i-2,p-x_i\}$.
Hence,  when viewing $\tau$ as a subquotient of $\rInj_{\tGamma}\sigma'$, $\tau$ is a new Serre weight and must occur in $\rInj_{\Gamma}\delta_i^{*}(\sigma')$ for
\begin{equation}\label{equation-sign-*}*=\left\{\begin{array}{rll}+& \mathrm{if}\ \nu_i(x_i)\in\{p-4-x_i,x_i+2\}\\ -& \mathrm{if}\ \nu_i(x_i)\in\{p-x_i,x_i-2\}. \end{array}\right.\end{equation}
 By Lemma \ref{lemma-unique-delta}, this property determines uniquely $i$ and the claim follows.

In summary, $\tau$ occurs in $\rInj_{\Gamma}\delta_i^*(\sigma')$, where $i$ is the unique index in $\cS(\lambda)\cap\cS(\lambda')$ and $*$ is as in \eqref{equation-sign-*}. Write $\tau=\nu(\sigma')=\mu(\delta_i^*(\sigma'))$ with $\mu\in\cI(x_0,\cdots,x_{f-1})$. Using \eqref{equation-sign-*}, one checks that $\mu_i(x_i)\in\{p-2-x_i,x_i\}$,  i.e. $i\notin\cS(\mu)$. To prove  that $\sigma$ occurs in $I(\sigma',\tau)$, by Corollary \ref{cor-I(sigmatau)-JH} it suffices to check $(\lambda'^{-1},\emptyset)\leq (\mu,(i,*))$ or, equivalently, the following two conditions
\begin{enumerate}
\item[(c1)] $\lambda'^{-1}$  is compatible with $\mu_{!}$;  

\item[(c2)] $\cS(\lambda'^{-1})\subseteq \cS(\mu)\cup\{i\}$,  equivalently $\cS(\lambda'^{-1})\backslash\{i\}\subset \cS(\mu)\backslash\{i\}$.
\end{enumerate}
For (c1),  the compatibility at $j\neq i$ follows from Lemma \ref{lemma-compose-lambda} because $\mu_j=\nu_j=\lambda_j\circ\lambda_j'^{-1}$ for $j\neq i$; on the other hand, using Lemma \ref{lemma-lambda'=mu-lambda} the relation $\mu_i(x_i*2)=\lambda_i(\lambda_i'^{-1}(x_i))$ and the fact $\mu_i(x_i)\in\{p-2-x_i,x_i\}$ imply that $\lambda_i'^{-1}(x_i)$ satisfies \eqref{equation-condition-newweight}, hence $\lambda_i'^{-1}$ is compatible with $\mu_{!}$ at $i$ (cf. Definition \ref{defn-sigma!}).
 For (c2), we note that \[\cS(\lambda'^{-1})\backslash \{i\} \subseteq \cS(\nu)\backslash \{i\}=\cS(\mu)\backslash\{i\}\]
where the inclusion follows from $\cS(\lambda)\cap\cS(\lambda'^{-1})=\{i\}$ using Lemma \ref{lemma-compose-lambda}, and the equality from the fact $\mu_j=\nu_j$ for $j\neq i$.

Finally, we prove that Case (c') can \emph{not} happen, which will finish the proof of the lemma. Indeed, the same argument in Case (c) shows that $\sigma'$ (the old Serre weight) occurs in $I(\sigma,\tau)$ (where $\sigma$ is the new Serre weight), hence $\ell(\sigma',\tau)<\ell(\sigma,\tau)$, contradicting \eqref{equation-ell-sigma}.
\end{proof}

For the rest of this subsection, we assume $\brho$ is strongly generic.

\begin{corollary}\label{cor-Ext1-D=tD}
Given $\tau\in\mathscr{D}(\brho)$, the inclusion $D_0(\brho)\into \tD_0(\brho)$ induces an isomorphism
\begin{equation}\label{equation-Ext-D=tD}\Ext^1_{\tGamma}(\tau,D_0(\brho))\simto \Ext^1_{\tGamma}(\tau,\tD_0(\brho)).\end{equation}
\end{corollary}
\begin{proof}
We first note that, by the proof of \cite[Lem.~12.8]{BP},   $\oplus_{\tau\in\mathscr{D}(\brho)}\rInj_{\Gamma}\tau$ and $D_0(\brho)$ have the same set of Jordan--H\"older factors, ignoring multiplicities. Since $\tD_0(\brho)$ is multiplicity free by Theorem \ref{thm-tD=multione}, the quotient $\tD_0(\brho)/D_0(\brho)$ does not have common Jordan--H\"older factors with $\oplus_{\tau\in\mathscr{D}(\brho)}\rInj_{\Gamma}\tau$. Using Lemma \ref{lemma-Hu10-2.21} we get for $\tau\in\mathscr{D}(\brho)$
\[\Hom_{\tGamma}(\tau,\tD_0(\brho)/D_0(\brho))=\Ext^1_{\tGamma}(\tau,\tD_0(\brho)/D_0(\brho))=0,\]
and the result follows.
\end{proof}

In fact, we have the following finer property. 

\begin{lemma}\label{lemma-HW-2.25}
Let $\sigma,\tau\in \mathscr{D}(\brho)$ and assume $\sigma\neq \tau$. Then for any nonzero subrepresentation $V_{\sigma}$ of $\widetilde{D}_{0,\sigma}(\brho)$ (hence  $\sigma\hookrightarrow V_{\sigma}$), the natural morphisms
\[\Ext^1_{\tGamma}(\tau,\sigma)\ra \Ext^1_{\tGamma}(\tau,V_{\sigma})\ra \Ext^1_{\tGamma}(\tau,\widetilde{D}_{0,\sigma}(\brho))\]
are isomorphisms.
\end{lemma}
\begin{proof}
First, the  morphisms in the lemma are
all injective, because $[\tD_{0,\sigma}(\brho):\tau]=0$ by Corollary \ref{cor:tD-sigma}. Hence, it suffices to prove that their composition
is an isomorphism.
Using Corollary \ref{cor-Ext1-D=tD}, it suffices to prove the natural morphism
\[\Ext^1_{\tGamma}(\tau,\sigma)\ra\Ext^1_{\tGamma}(\tau,D_{0,\sigma}(\brho))\]
is an isomorphism. It is proved in \cite[Lem.~2.25]{HW} that the last morphism is an isomorphism if we replace $\Ext^1_{\tGamma}$ by $\Ext^1_{\Gamma}$ (in \emph{loc. cit.} $\brho$ is only required to be generic in the sense of \cite[Def.~11.7]{BP}).  Hence, it suffices to show $\Ext^1_{\Gamma}(\tau,\sigma)\cong \Ext^1_{\tGamma}(\tau,\sigma)$ and similarly for $D_{0,\sigma}(\brho)$ in place of $\sigma$. Using the Hochschild-Serre spectral sequence and the assumption $\sigma\neq\tau$, this follows from the fact
\[\Hom_{\Gamma}\big(\tau,\sigma\otimes H^1(K_1/Z_1,\F)\big)=\Hom_{\Gamma}\big(\tau, D_{0,\sigma}(\brho)\otimes H^1(K_1/Z_1,\F)\big)=0,\]
see  Proposition \ref{prop-structure-JProj} (applicable thanks to Lemma \ref{lem:Serre=3generic}).
\end{proof}

\subsection{A combinatorial lemma}
In this subsection,  we assume $\brho$ is generic in the sense of \cite[Def.~11.7]{BP}.

Let $D_1(\brho)\defn D_0(\brho)^{I_1}$ and $D_{1,\tau}(\brho)\defn D_{0,\tau}(\brho)^{I_1}$ for $\tau\in\mathscr{D}(\brho)$. Given $\chi\in \JH(D_1(\brho))$, there exists a unique $\tau\in\mathscr{D}(\brho)$ such that $\chi$ occurs in $D_{1,\tau}(\brho)$.

\begin{lemma}\label{lemma-tau-maximal}
Keep the above notation. If $\sigma\in\mathscr{D}(\brho)$ is another Serre weight which is also a Jordan--H\"older factor of $\Ind_I^K\chi$,  then $J(\sigma)\subseteq J(\tau)$, viewing both $\sigma,\tau$ as subquotients of $\Ind_I^K\chi$ (cf. \S\ref{subsection:PS}).
\end{lemma}
\begin{proof}
The cosocle (resp. socle) of $\Ind_I^K\chi$ is isomorphic to $\sigma_{\chi}$ (resp. $\sigma_{\chi^s}$). By assumption,  $\sigma_{\chi}$ is a subquotient of $D_{0,\tau}(\brho)$, i.e. $\ell(\brho,\sigma_{\chi})=\ell(\tau,\sigma_{\chi})$.   Lemma \ref{lemma-BP-12.8}(ii)  implies that $\tau$ occurs in $I(\sigma,\sigma_{\chi})$  as a subquotient. Equivalently, $\sigma$ occurs in $I(\sigma_{\chi^s},\tau)$ as a subquotient, and so $J(\sigma)\subseteq J(\tau)$ by  \cite[Cor.~4.11]{BP}. \end{proof}

\begin{lemma}\label{lemma-PD-set}
Let  $\chi,\chi'$ be two characters such that $\Ext^1_{I/Z_1}(\chi,\chi')\neq0$ and assume $\chi,\chi'\in\JH(D_1(\brho))$. Let $\tau\in \mathscr{D}(\brho)$ (resp. $\tau'\in\mathscr{D}(\brho)$) be the Serre weight such that $\chi\in \JH(D_{1,\tau}(\brho))$ (resp. $\chi'\in \JH(D_{1,\tau'}(\brho))$). Let $J(\tau)\subset \cS$ (resp. $J(\tau')$) be the subset parametrizing the position of $\tau$ (resp. $\tau'$) inside $\Ind_I^K\chi$ (resp. in $\Ind_I^K\chi'$).
\begin{enumerate}
\item[(i)] If $\chi'=\chi\alpha_j^{-1}$ for some $j\in\cS$, then $j-1\notin J(\tau)$ and $J(\tau')=J(\tau)\cup\{j-1\}$.  
\item[(ii)] If $\chi'=\chi\alpha_j$ for some $j\in\cS$, then $j-1\notin J(\tau')$ and $J(\tau)=J(\tau')\cup \{j-1\}$.
\end{enumerate}
Moreover, we have $\Ext^1_{K/Z_1}(\tau,\tau')\cong\Ext^1_{\Gamma}(\tau,\tau')\neq0$.
\end{lemma}

\begin{proof}
First assume $\brho$ is reducible. Following  \cite[\S4]{Br14}, we define  $\mathscr{PD}(x_0,\cdots,x_{f-1})$ to be the set of $f$-tuples $\lambda=(\lambda_i(x_i))_{i\in\cS}$ such  that
\begin{enumerate}
\item[$\bullet$]$\lambda_i(x_i)\in\{x_i,x_i+1,x_i+2,p-3-x_i,p-2-x_i,p-1-x_i\}$, 
\item[$\bullet$] if $\lambda_i(x_i)\in\{x_i,x_i+1,x_i+2\}$, then $\lambda_{i+1}(x_i+1)\in\{x_{i+1},x_{i+1}+2,p-2-x_{i+1}\}$, 
\item[$\bullet$] if $\lambda_i(x_i)\in\{p-1-x_i,p-2-x_i,p-3-x_i\}$, then $\lambda_{i+1}(x_{i+1})\in\{p-1-x_{i+1},p-3-x_{i+1},x_{i+1}+1\}$, 
\item[$\bullet$]  $\lambda_i(x_i)\in \{p-3-x_i,x_i+2\}$ implies $i\in J_{\brho}$.
\end{enumerate}
By  \cite[Prop.~4.2]{Br14},    the set $\JH(D_1(\brho))$ consists of  the characters of $I$ acting on $\sigma^{I_1}$, where $\sigma$ runs over the set of Serre weights associated to $\lambda\in\mathscr{PD}(x_0,\cdots,x_{f-1})$  as in \eqref{eq:lambda-rho}.

 Given $\lambda\in\mathscr{PD}(x_0,\cdots,x_{f-1})$, we define
\[J_{\lambda}^{\rm max}=\delta\big(\{i\in\cS:\ \lambda_i(x_i)\notin\{p-3-x_i,x_i\}\ \mathrm{and}\ (i\in J_{\brho}\ \mathrm{if}\ \lambda_j(x_i)=p-2-x_i)\}\big)\]
where $\delta$ is the shift on $J$: $i-1\in \delta(J)$ if and only if $i\in J$. By \cite[Prop.~4.4]{Br14}, if $\psi^s\in D_{1,\tau}(\brho)$ for $\tau\in\mathscr{D}(\brho)$ then, when viewed as a subquotient of  $\Ind_I^K(\psi^s)$, $\tau$ is parametrized by $J^{\rm max}_{\lambda}$.   Since our setting differs from that of \cite{Br14} by a conjugation, we make a change of variables, by setting $\psi\defn\chi^s$ and $\psi'\defn\chi'^s$. Let $\lambda$, $\lambda'\in\mathscr{PD}(x_0,\cdots,x_{f-1})$ be the  elements corresponding to $\psi$, $\psi'$ respectively.

(i) The assumption $\chi'=\chi\alpha_j^{-1}$ translates to  $\psi'=\psi\alpha_{j}$. We have $\lambda_i(x_i)=\lambda_i'(x_i)$ if $i\neq j$, and two possibilities if $i=j$:
\[\left\{\begin{array}{ll}\lambda_j(x_j)=x_j\\
\lambda_j'(x_j)=x_j+2\end{array}\right.\ \ \ \mathrm{or}\ \ \ \left\{\begin{array}{ll}\lambda_j(x_j)=p-3-x_j\\
\lambda_j'(x_j)=p-1-x_j.\end{array}\right.\]
One checks that $j-1\notin J_{\lambda}^{\rm max}$ and $J^{\rm max}_{\lambda'}=J^{\rm max}_{\lambda}\cup\{j-1\}$, as desired.

(ii) The assumption $\chi'=\chi\alpha_j$ translates to  $\psi'=\psi\alpha_{j}^{-1}$. We have $\lambda_i(x_i)=\lambda_i'(x_i)$ if $i\neq j$, and two possibilities if $i= j$:
 \[\left\{\begin{array}{ll}\lambda_j(x_j)=x_j+2\\
\lambda_j'(x_j)=x_j\end{array}\right.\ \ \ \mathrm{or}\ \ \ \left\{\begin{array}{ll}\lambda_j(x_j)=p-1-x_j\\
\lambda_j'(x_j)=p-3-x_j.\end{array}\right.\]
One checks that $j-1\notin J_{\lambda'}^{\rm max}$ and $J^{\rm max}_{\lambda}=J^{\rm max}_{\lambda'}\cup\{j-1\}$, as desired.

The case $\brho$ is irreducible can be treated in a similar way, using the set $\mathscr{PID}(x_0,\cdots,x_{f-1})$ in place of $\mathscr{PD}(x_0,\cdots,x_{f-1})$.  Finally, the last assertion  follows from Lemmas \ref{lemma-Esigma'-occur-plus} and \ref{lemma-Esigma'-occur-minus}, together with Lemma \ref{lemma-Hu10-2.21}.
\end{proof}

\begin{remark} \label{remark-onlyone-chi'}
It follows from the conclusion of Lemma \ref{lemma-PD-set} that if $\chi\in \JH(D_1(\brho))$, then at most one of $\{\chi\alpha_{j}^{\pm 1}: j\in\cS\}$ can also occur in $D_1(\brho)$. Of course, this can also be deduced from the property of $\mathscr{PD}(x_0,\cdots,x_{f-1})$.
\end{remark}

\begin{corollary}\label{cor-coker-no-serreweight}
Let $\chi\in \JH(D_1(\brho))$ and  $\tau\in\mathscr{D}(\brho)$ be the Serre weight such that $\chi$ occurs in $D_{1,\tau}(\brho)$. Let
\[\phi_{\tau}:\Proj_{\tGamma}\tau\ra \Ind_I^K\overline{W}_{\chi,3},\]
be as in Proposition \ref{prop-coker-no-sigma}.
Then $\JH(\Coker(\phi_{\tau}))\cap \mathscr{D}(\brho)=\emptyset$.\end{corollary}

\begin{proof}
Let $\sigma\in\mathscr{D}(\brho)$. We need to show that $[\Coker(\phi_{\tau}):\sigma]=0$. Clearly we may assume $\sigma\in\JH(\Ind_I^K\overline{W}_{\chi,3})$. Letting $\phi_{\sigma}:\Proj_{\tGamma}\sigma\ra\Ind_I^K\overline{W}_{\chi,3}$ be a morphism as in Proposition \ref{prop-coker-no-sigma}(i), it is equivalent to show $\im(\phi_{\sigma})\subset \im(\phi_{\tau})$.   We have two possibilities: $\sigma$ is a subquotient of $\Ind_I^K\chi$, or of $\Ind_I^K\chi'$ for some $\chi'\in\mathscr{E}(\chi)$.

If $\sigma\in\JH(\Ind_I^K\chi)$,  Lemma \ref{lemma-tau-maximal} implies that  $J(\sigma)\subset J(\tau)$ if we view both $\sigma,\tau$ as subquotients of $\Ind_I^K\chi$, and we conclude by  Corollary \ref{cor-imageoffi-inclusion}(a).

If $\sigma\in\JH(\Ind_I^K\chi')$ for some $\chi'\in\mathscr{E}(\chi)$, then  $\chi'\in \JH(D_1(\brho))$; let $\tau'\in \mathscr{D}(\brho)$ be the unique Serre weight such that $\chi'$ occurs in $D_{1,\tau'}(\brho)$.  As above,   Lemma \ref{lemma-tau-maximal}  implies that $J(\sigma)\subset J(\tau')$ if we view $\sigma$, $\tau'$ as subquotients of $\Ind_I^K\chi'$, hence $\im(\phi_{\sigma})\subset \im(\phi_{\tau'})$ by Corollary \ref{cor-imageoffi-inclusion}(b). On the other hand,  we have $\Ext^1_{K/Z_1}(\tau',\tau)\neq0$  by Lemma \ref{lemma-PD-set}, hence $\im(\phi_{\tau'})\subset\im(\phi_{\tau})$ by  Corollary \ref{cor-imageoffi-inclusion}(c). This finishes the proof.
\end{proof}

\subsection{Multiplicity one}
Keep the notation of last subsections and assume $\brho$ is strongly generic.  Let $\pi$ be an admissible smooth $G$-representation  over $\F$ (with a central character) satisfying the following condition:
\begin{enumerate}
\item[(a)] $\pi^{K_1}\cong D_0(\brho)$, in particular $\soc_K\pi\cong\oplus_{\sigma\in\mathscr{D}(\brho)}\sigma$. 
\end{enumerate}

The aim of this subsection is to prove a   criterion for $\pi[\fm_{K_1}^2]$ to be multiplicity free, see Theorem \ref{thm-criterion} below. By Theorem \ref{thm-tD=multione}, this amounts to proving  that for \emph{any} $\sigma\in\mathscr{D}(\brho)$, \begin{equation}\label{eq:cond-dim=1}\dim_{\F} \Hom_{\tGamma}(\Proj_{\tGamma}\sigma,\pi)=1.\end{equation}
The main point of this criterion is that, when $\brho$ is \emph{indecomposable}, we only need to check \eqref{eq:cond-dim=1} for \emph{some} $\sigma\in\mathscr{D}(\brho)$. Correspondingly, for our application in \S\ref{section-patching} where $\brho$ will be  reducible nonsplit, the computation of various deformation rings  \emph{exactly} allows us to check this condition for  one special Serre weight in $\mathscr{D}(\brho)$, namely  the ``ordinary'' Serre weight denoted by $\sigma_{\emptyset}$ there.
Another crucial point of the criterion is that we deduce at the same time
\[\dim_{\F}\Hom_I(W_{\chi,3},\pi)=1,\ \ \forall \chi\in\JH(\pi^{I_1}),\]
which allows us to apply a result proved in \cite[\S5]{BHHMS},  to control the Gelfand-Kirillov dimension of $\pi$, see Theorem \ref{thm:main-flat} below.

\vspace{2mm}

 Recall from \S\ref{section:Rep-II} for the representations $W_{\chi,n}$ for $n\geq 1$.
\begin{proposition}\label{prop-W2topi=dim1}
For any $\chi\in\JH(\pi^{I_1})$, the natural morphism
\[\Hom_{I}(\chi,\pi)\ra \Hom_I(W_{\chi,2},\pi)\]
is an isomorphism.  
\end{proposition}

\begin{proof}
Let $h:W_{\chi,2}\ra \pi|_I$ be a nonzero  morphism and
 \[\widetilde{h}:\ \Ind_I^KW_{\chi,2}\ra \pi|_K\]
 be the induced morphism by Frobenius reciprocity. Assume that $h$ is nonzero when restricted to $\rsoc_I(W_{\chi,2})$, say $h|_{\chi'}\neq0$ for some $\chi'\hookrightarrow W_{\chi,2}$. Remark \ref{remark-onlyone-chi'} implies that $h$ must factor through
 \[W_{\chi,2}\twoheadrightarrow E_{\chi',\chi}\hookrightarrow \pi.\]
The image of $\Ind_I^K\chi'$ under $\widetilde{h}$ has the form $I(\tau',\sigma_{\chi'})$, where $\tau'$ is the unique Serre weight in $\mathscr{D}(\brho)$ such that $\chi'$ occurs in $D_{1,\tau'}(\brho)$ and $\sigma_{\chi'}$ denotes the cosocle of $\Ind_I^K\chi'$. In particular, $\tau'$ embeds in $\im(\widetilde{h})$. On the other hand,  if $\tau\in\mathscr{D}(\brho)$ denotes the Serre  weight such that $\chi$ occurs in $D_{1,\tau}(\brho)$,  then Lemma \ref{lemma-PD-set} implies that $\Ext^1_{\Gamma}(\tau',\tau)\neq0$, and so  the extension   $E_{\tau',\tau}$ is a subquotient of $\im(\widetilde{h})$ by Lemma \ref{lemma-Esigma'-occur-plus} and Lemma \ref{lemma-Esigma'-occur-minus}.

Since $W_{\chi,2}$ is annihilated by $\fm_{K_1}^2$, so are $\Ind_I^KW_{\chi,2}$ and $\im(\widetilde{h})$. Hence, there exists a nonzero morphism $\Proj_{\tGamma}\tau\ra \im(\widetilde{h})$
whose image we denote by $Q$. By the above discussion  $E_{\tau',\tau}$ occurs in $Q$  (actually as a quotient).
 We know  that $Q$ is annihilated by $\fm_{K_1}$ by Corollary \ref{coro:Q-killedbyJ} (condition (b) is satisfied by \cite[Prop.~2.24]{HW}), hence is contained in   $\pi^{K_1}=D_0(\brho)$ by (a). This gives a contradiction because $\tau$ only occurs in the socle of $D_0(\brho)$.
\end{proof}

\begin{corollary}\label{cor:no-embedding}
Let $\chi,\chi'\in\JH(\pi^{I_1})$ and assume $\Ext^1_{I/Z_1}(\chi,\chi')\neq0$. Then there exists no $I$-equivariant embedding $E_{\chi,\chi'}\hookrightarrow \pi|_I$.
\end{corollary}
\begin{proof}
This is a direct consequence of Proposition \ref{prop-W2topi=dim1}.
\end{proof}

\begin{corollary}\label{cor-factor-barW}
Let $\chi\in \JH(\pi^{I_1})$. Any $I$-equivariant morphism $W_{\chi,3}\ra \pi|_I$ factors through $\overline{W}_{\chi,3}$.
\end{corollary}

\begin{proof}
The proof is as in Step 1 of \cite[Prop.~6.4.6]{BHHMS}, using Proposition \ref{prop-W2topi=dim1} as a replacement of \cite[Lem.~6.4.4]{BHHMS}. We briefly recall the argument. Let $f:W_{\chi,3}\ra\pi$ be an $I$-equivariant morphism. By definition, $\overline{W}_{\chi,3}$ is the quotient of $W_{\chi,3}$ by the direct sum of $\chi''$ in $\soc_I(W_{\chi,3})$ which are distinct with $\chi$, and each of these characters occurs once in $\soc_I(W_{\chi,3})$.  Let $\chi''\neq \chi$ be such a character in $\soc_{I}(W_{\chi,3})$ such that $f|_{\chi''}$ is nonzero. Then by the connectedness of $D_1(\brho)$, see \cite[Def.~6.4.2, Lem.~6.4.3]{BHHMS},\footnote{In \cite[Lem.~6.4.3]{BHHMS}, the genericity assumption on $\brho$ is stronger than ours, but using the set $\mathscr{PD}(x_0,\cdots,x_{f-1})$ we may check  that $D_1(\brho)$  is still connected when $\brho$ is generic in the sense of \cite[Def.~11.7]{BP}.} we may find $\chi'\in \mathscr{E}(\chi)\cap \mathscr{E}(\chi'')$ with $\chi'\in\JH(\pi^{I_1})$. By \cite[Lem.~6.1.2]{BHHMS}, there exists an injection $W_{\chi',2}\hookrightarrow W_{\chi,3}$, hence $f$ restricts to a morphism $W_{\chi',2} \to \pi|_{I}$ which does not factor through the cosocle $\chi'$, because it is nonzero on $\chi''$ which embeds in $W_{\chi',2}$. This contradicts Proposition \ref{prop-W2topi=dim1}.
\end{proof}

Recall from Proposition \ref{prop-Theta} the representation $\Theta_{\tau}$ of $\tGamma$.

\begin{proposition}\label{prop-criteria-Theta=1}
Let $\tau\in\mathscr{D}(\brho)$.
 Then the following two conditions are equivalent:
\begin{enumerate}
\item[(i)] $\dim_{\F}\Hom_{K}(\Proj_{\tGamma}\tau,\pi)=1$;
\item[(ii)] $\dim_{\F}\Hom_{K}(\Theta_{\tau},\pi)=1$.
\end{enumerate}
\end{proposition}
\begin{proof}
Since $\Theta_{\tau}$ is a quotient of $\Proj_{\tGamma}\tau$ and $\tau\hookrightarrow \pi$, we have trivially (i)$\Rightarrow$(ii).

(ii)$\Rightarrow$(i).
Assume (i) does not hold. Then there exists a nonzero morphism $h:\Proj_{\tGamma}\tau\ra \pi$ which is \emph{not} a scalar of the composition $h_0: \Proj_{\tGamma}\tau\twoheadrightarrow \tau\hookrightarrow \pi$. We choose $h$ in such a way that the multiplicity $[\im(h):\tau]$ is minimal. Let $Q\subseteq \pi$ denote the image of $h$. We will prove that $Q$ is a quotient of $\Theta_{\tau}$, so that there exists a morphism $\Theta_{\tau}\ra \pi$ which does not factor through the cosocle $\Theta_{\tau}\twoheadrightarrow \tau$, contradicting (ii).

First assume that $\tau$ does not occur in $\soc(Q)$. Then the projectivity of $\Proj_{\tGamma}\tau$ and the choice of $h$ implies  $[Q:\tau]=1$ (with $\tau$ in the cosocle of $Q$);  otherwise, we could always construct a nonzero morphism $h':\Proj_{\tGamma}\tau\ra \pi$, with $\im(h')\subsetneq Q$ and $[\im(h'):\tau]<[Q:\tau]$,   which contradicts the choice of $h$.  Here, the assumption that $\tau\notin \JH(\soc (Q))$ ensures that $h'$ is still not a scalar of $h_0$.  By Corollary \ref{coro:Q-killedbyJ} (condition (b) in \emph{loc. cit.} is satisfied by \cite[Prop.~2.24]{HW}), we deduce that $Q$ is a multiplicity free $\Gamma$-representation, hence is contained in $D_0(\brho)\cong \pi^{K_1}$ by Condition (a) imposed on $\pi$. Recall that $D_0(\brho)$ is multiplicity free and  $\tau$ occurs in the socle of $D_0(\brho)$. Hence, if $\tau$ occurs in $Q$, it must occur in  $\rsoc(Q)$, giving a contradiction.

Assume that $\tau$ occurs in $\rsoc(Q)$ for the rest of the proof. A similar argument as in the above case shows that $[Q:\tau]=2$; more precisely, we have $\tau\hookrightarrow \soc(Q)$ and $\mathrm{cosoc}(Q)\cong \tau$.  We claim the following:

\begin{enumerate}
\item[(1)] $\mathrm{rad}(Q)$ is multiplicity free; 
\item[(2)] $\mathrm{rad}(Q)$ is a subrepresentation of $D_0(\brho)$.
\end{enumerate}

For (1), since $\mathrm{cosoc}(Q)\cong\tau$, it  is equivalent to show that $Q/\tau$ is multiplicity free. Note that $Q/\tau$ is a quotient of $\Proj_{\tGamma}\tau$ and $[Q/\tau:\tau]=[Q:\tau]-1=1$, so the assertion  follows from (the dual version of) Corollary \ref{cor-I(sigmatau)-geq}.

For (2), note that $\soc(Q) \subset \soc (D_0(\brho))$ is multiplicity free. For each $\sigma\in\JH(\soc(Q))$, let $Q_{\sigma}$ be the unique quotient of $Q$ with socle $\sigma$ and such that $\sigma\hookrightarrow Q\twoheadrightarrow Q_{\sigma}$ is nonzero, so that $Q$ embeds in $\bigoplus_{\sigma\in\JH(\soc(Q))}Q_{\sigma}$, and consequently
\begin{equation}\label{eq:radQ-embeds}\rad(Q)\hookrightarrow \bigoplus_{\sigma\in\JH(\soc(Q))}\rad(Q_{\sigma}).\end{equation}
It suffices to prove that $\rad(Q_{\sigma})$ is a $\Gamma$-representation for each $\sigma\in\JH(\soc(Q))$.
If $\sigma\neq \tau$, then $Q_{\sigma}$ itself is a $\Gamma$-representation by the same argument as in the above case. Assume $\sigma=\tau$. Since $Q_{\tau}$ has cosocle  $\tau$, the cosocle of $\rad(Q_{\tau})$ can be embedded in $\oplus_{\tau'\in\mathscr{E}(\tau)}\tau'$. On the other hand, by construction $\rad(Q_{\tau})$ is multiplicity free with socle $\tau$, so  $\rad(Q_{\tau})$ is also a $\Gamma$-representation by (the dual version of) Corollary \ref{coro:Q-killedbyJ}. This proves (2). As a consequence, the embedding \eqref{eq:radQ-embeds} is an isomorphism because, on the one hand, each projection $\rad(Q)\ra \rad(Q_{\sigma})$ is surjective, on the other hand, since $\rad(Q_{\sigma})\subset D_{0,\sigma}(\brho)$ by (2) and $D_0(\brho)$ is multiplicity free,  $\rad(Q_{\sigma_1})$ and $\rad(Q_{\sigma_2})$ don't have any common Jordan--H\"older factors for $\sigma_1\neq \sigma_2$. We also deduce that $\rad(Q_{\sigma})$ is a subrepresentation of $D_{0,\sigma}(\brho)$ for any $\sigma\in \JH(\soc(Q))$.

We now prove that $Q$ is  a quotient of $\Theta_{\tau}$. Note that $\tau$ is $2$-generic by Lemma \ref{lem:Serre=3generic}.   Let $\sigma\in\JH(\soc(Q))$ and assume $\sigma\neq\tau$.  Since  $Q_{\sigma}$ has cosocle isomorphic to $\tau$, $Q_{\sigma}$ gives rise to a nonzero class in $\Ext^1_{\tGamma}(\tau,\rad(Q_{\sigma}))$
which implies
$\sigma\in\mathscr{E}(\tau)$ by Lemma \ref{lemma-HW-2.25}. In other words, $\JH(\soc(Q))$ is contained in $\{\tau\}\cup\mathscr{E}(\tau)$, and for any $\sigma\in\JH(\soc(Q))$  we have  $\dim_{\F}\Hom_{\tGamma}(\sigma,Q)=1$. Let $C$ denote the quotient $Q/\tau$, so that we have a short exact sequence
\begin{equation}\label{eq:Q-C}0\ra \tau\ra Q\ra C\ra0.\end{equation}
We claim that $C$ has Loewy length $2$ and fits in a short exact sequence
\[0\ra S\ra C\ra \tau\ra0\]
for some subrepresentation $S$ of $\bigoplus_{\tau'\in\mathscr{E}(\tau)}\tau'$, and   the result will  follow by Corollary \ref{cor:Theta-univ}.
Using  \eqref{eq:Q-C} we see that if $\Hom_{\tGamma}(\tau',C)\neq 0$ for some Serre weight $\tau'$, then either $\Hom_{\tGamma}(\tau',Q)\neq0$ or $\Ext^1_{\tGamma}(\tau',\tau)\neq0$,   hence  $\tau'\in\{\tau\}\cup\mathscr{E}(\tau)$ by the above discussion. Moreover, since $\dim_{\F}\Hom_{\tGamma}(\tau,Q)=1$ and $\Ext^1_{\tGamma}(\tau,\tau)=0$ by Lemma \ref{lemma-Hu10-2.21}(i), we have  $\Hom_{\tGamma}(\tau,C)=0$ and consequently $\JH(\soc(C))\subset \mathscr{E}(\tau)$. As in the proof of Corollary \ref{coro:Q-killedbyJ}, we then have
\[C\hookrightarrow \bigoplus_{\tau'\in \JH(\soc(C))}I(\tau',\tau)\cong \bigoplus_{\tau'\in\JH(\soc(C))}E_{\tau',\tau}, \]
where the isomorphism holds as $\tau'\in\mathscr{E}(\tau)$. This proves the claim and finishes the proof of the proposition.
\end{proof}

Now we make an extra assumption on $\pi$:
\begin{enumerate}
\item[(b)] if $\Ext^1_{K/Z_1}(\sigma,\pi)\neq0$ for some Serre weight $\sigma$, then $\sigma\in\mathscr{D}(\brho)$.
\end{enumerate}

\begin{remark}
We will see examples of $G$-representations satisfying (a) and (b)  in \S\ref{section-patching}.
\end{remark}

\begin{proposition}\label{prop-dim=1-equivalent}
Let $\chi\in\JH(\pi^{I_1})$ and let $\tau\in\mathscr{D}(\brho)$ be the unique Serre  weight such that $\chi$ occurs in $D_{1,\tau}(\brho)$. Then the following statements are equivalent:
\begin{enumerate}
\item[(i)] $\dim_{\F}\Hom_K(\Proj_{\tGamma}\tau,\pi)=1$;
\item[(ii)] $\dim_{\F}\Hom_K(\Theta_{\tau},\pi)=1$;
\item[(iii)] $\dim_{\F}\Hom_{I}(W_{\chi,3},\pi)=1$.
\end{enumerate}
\end{proposition}
\begin{proof}
Using Proposition \ref{prop-criteria-Theta=1} it suffices to prove   the following inequalities:
\begin{equation}\label{eq:inequalities}\dim_{\F}\Hom_{K}(\Theta_{\tau},\pi)\leq \dim_{\F}\Hom_{I}(W_{\chi,3},\pi)\leq \dim_{\F}\Hom_{K}(\Proj_{\tGamma}\tau,\pi).\end{equation}
By Corollary \ref{cor-factor-barW} and Frobenius reciprocity, we may replace the middle term by \[\dim_{\F}
\Hom_K(\Ind_I^K\overline{W}_{\chi,3},\pi). \]
Let $\phi_{\tau}:\Proj_{\tGamma}\tau\ra \Ind_I^K\overline{W}_{\chi,3}$ be a morphism as in Proposition \ref{prop-coker-no-sigma}. On the one hand, by Corollary \ref{cor-coker-no-serreweight} and Conditions (a),(b) satisfied by $\pi$, the inclusion $\im(\phi_{\tau})\hookrightarrow \Ind_I^K\overline{W}_{\chi,3}$ induces an isomorphism
\[\Hom_{K}(\Ind_{I}^K\overline{W}_{\chi,3},\pi)\simto\Hom_K(\im(\phi_{\tau}),\pi).\]
On the other hand, there are surjections $\Proj_{\tGamma}\tau\twoheadrightarrow\im(\phi_{\tau})\twoheadrightarrow \Theta_{\tau}$, which induce
\[\Hom_{K}(\Theta_{\tau},\pi)\hookrightarrow\Hom_K(\im(\phi_{\tau}),\pi)\hookrightarrow\Hom_K(\Proj_{\tGamma}\tau,\pi).\]
Putting them together, we deduce \eqref{eq:inequalities}.
\end{proof}

Summarizing what has been proved, we obtain the following ``multiplicity one'' criterion, the main result of this section.

\begin{theorem}\label{thm-criterion}
Assume $\brho$ is indecomposable and strongly generic. Assume $\pi$ is an admissible smooth $G$-representation  over $\F$ (with a central character) satisfying the following conditions:
\begin{enumerate}
\item[(a)] $\pi^{K_1}\cong D_0(\brho)$ (in particular $\rsoc_{K}\pi\cong \bigoplus_{\sigma\in\mathscr{D}(\brho)}\sigma$);
\item[(b)]  if $\Ext^1_{K/Z_1}(\sigma,\pi)\neq0$ for some Serre weight $\sigma$, then $\sigma\in\mathscr{D}(\brho)$;
\item[(c)] there exists one   $\sigma_0\in\mathscr{D}(\brho)$ such that $\dim_{\F}\Hom_{K}(\Theta_{\sigma_0},\pi)=1$.
\end{enumerate}
Then the following statements hold:
\begin{enumerate}
\item[(i)] $\dim_{\F}\Hom_{K}(\Proj_{\tGamma}\sigma,\pi)=1$ for any $\sigma\in\mathscr{D}(\brho)$, or equivalently, $\pi[\fm_{K_1}^2]\subset \tD_0(\brho)$;
\item[(ii)] $\dim_{\F}\Hom_{I}(W_{\chi,3},\pi)=1$ for any $\chi\in \JH(\pi^{I_1})$.
\end{enumerate}
\end{theorem}

\begin{proof}
By (a),   the basic $0$-diagram $(\pi^{K_1},\pi^{I_1},\mathrm{can})$ attached to $\pi$ in \cite[\S9]{BP}, where $\mathrm{can}: \pi^{I_1}\hookrightarrow \pi^{K_1}$ is the canonical inclusion,  is just $(D_0(\brho),D_1(\brho),\mathrm{can})$.

We define two sets as follows:
\begin{align*}
&\Sigma_0\defn\{\sigma\in\mathscr{D}(\brho):\ \dim_{\F}\Hom_K(\Proj_{\tGamma}\sigma,\pi)=1\}\\
&\Sigma_1\defn \{\chi\in \JH(D_1(\brho)): \ \dim_{\F} \Hom_{I}(W_{\chi,3},\pi)=1\}.\end{align*}
It is clear that $\Sigma_1$ is stable under the action of $\smatr{0}1p0$ (the one induced from $D_1(\brho$)). By Proposition \ref{prop-dim=1-equivalent}, if $\chi\in  \JH(D_{1,\sigma}(\brho))$, then $\chi\in \Sigma_1$ if and only if $\sigma\in\Sigma_0$. Using (c), this implies that
\[\Big(\bigoplus_{\sigma\in\Sigma_0}D_{0,\s}(\brho), \bigoplus_{\chi\in\Sigma_1}\chi,\mathrm{can}\Big)\]
is a \emph{nonzero} subdiagram of $(D_0(\brho), D_1(\brho),\mathrm{can})$, and in fact a direct summand as diagrams. However, the diagram $(D_0(\brho), D_1(\brho),\mathrm{can})$ is indecomposable by \cite[Thm.~15.4(i)]{BP}, thus they must be equal, and so $\Sigma_0=\mathscr{D}(\brho)$ and $\Sigma_1=\JH(D_1(\brho))$.
\end{proof}

\begin{remark}
If $\brho$ is  reducible split, then Theorem \ref{thm-criterion} fails because the diagram $(D_0(\brho),D_1(\brho),\mathrm{can})$ is not indecomposable anymore. In fact, we have to impose a stronger hypothesis in (c) for the theorem to be true.
\end{remark}

\section{Ordinary parts }
\label{section:ordinary}

In this section we recall and prove some general results about   smooth $\F$-representations of $\GL_2(L)$, where $L=\Q_{p^f}$ as before. Let $G=\GL_2(L)$, $K=\GL_2(\cO_L)$, and define the following subgroups of $G:$
 \[P\defn\matr{*}*0*, \ \ \overline{P}\defn\matr*0**,\ \ T\defn\matr*00*,\ \ N\defn\matr{1}{*}01.\]  Let $T_0\defn T\cap K$ and $N_0\defn N\cap K$.  Recall that $Z$ denotes the center of $G$ and $Z_1\defn Z\cap K_1$.

In this section, we only consider representations  defined on $\F$-vector spaces and \emph{with a central character}. The latter assumption is not always necessary, but  we make it for convenience.  

\subsection{Ordinary parts}
\label{subsection:ordinary}

 Emerton has defined  a left exact covariant functor in \cite{EmertonOrd1}, called \emph{ordinary parts} and denoted by $\Ord_P$, from the category of  smooth $\F$-representations of $G$  to the category of  smooth $\F$-representations of $T$, which preserves admissibility, and more generally local admissibility. He also defined in \cite[Def.~3.3.1]{EmertonOrd2} a $\delta$-functor $\{H^i\Ord_P:i\geq 0\}$ such that $H^0\Ord_P=\Ord_P$.

On the other hand, let $R^i\Ord_P$ be  the right derived functors of $\Ord_P$ for $i\geq 0$. The  main result of  \cite{EmertonPaskunas} says that there is a natural equivalence $R^{i}\Ord_P\simto  H^i\Ord_P$.  Using \cite[Prop.~3.6.1]{EmertonOrd2}, we deduce that $R^i\Ord_P$ vanishes  for $i\geq f+1$.

Recall that $\omega:G_L\ra \F_p^{\times}$ is the mod $p$ cyclotomic character, viewed as a character of $L^{\times}$ via the local Artin map normalized  in the way that uniformizers of $L$ are sent to geometric Frobenii.  Denote by  $\alpha_P$ the character $\omega\otimes\omega^{-1}: T\ra \F_p^{\times}\hookrightarrow \F^{\times}$.

The following proposition summarizes some properties of $R^i\Ord_P$.

\begin{proposition}\label{prop-Ord}
Let $U$ be a locally admissible smooth representation of $T$ and $V$ be a smooth representation of $G$. The following statements hold.
\begin{enumerate}
\item[(i)] There is an adjunction isomorphism
\begin{equation}\label{equation-ord-adjunction}\Hom_G(\Ind_{\overline{P}}^GU,V)\cong \Hom_T(U,\Ord_PV).\end{equation}

\item[(ii)] There is a canonical isomorphism $R^f\Ord_PV\cong V_N\otimes \alpha^{-1}_P$, where $V_N$ is the space of coinvariants (i.e. the usual Jacquet module of $V$ with respect to $P$).

\item[(iii)] There are canonical isomorphisms \[\Ord_P(\Ind_{ \overline{P}}^GU)\cong U,\ \ R^f\Ord_P(\Ind_{\overline{P}}^GU)\cong U^s\alpha_P^{-1}.\] Here $U^s$ denotes the representation of $T$ obtained by conjugating $U$ by $s=\smatr0110$.

\item[(iv)] There is a natural isomorphism
\[\Hom_T(R^f\Ord_PV,U)\cong \Hom_G\big(V,\Ind_{P}^G(U\alpha_P)\big).\]
Moreover,  the isomorphism sends epimorphisms to epimorphisms.

\item[(v)] If $L=\Q_p$, then $R^1\Ord_P(\Ind_{\overline{P}}^GU)\cong U^s\alpha_P^{-1}$; otherwise $R^1\Ord_P(\Ind_{\overline{P}}^G\break U)=0$.
\end{enumerate}
\end{proposition}
\begin{proof}
(i)  is \cite[Thm.~4.4.6]{EmertonOrd1} together with \cite[Rem.~3.7.3]{EmertonOrd2}, and (ii)   is \cite[Prop.~3.6.2]{EmertonOrd2} using the main result of \cite{EmertonPaskunas}.

(iii) The first isomorphism follows from \cite[Prop.~4.3.4]{EmertonOrd1} and the second from (ii) noting that $(\Ind_{\overline{P}}^GU)_{N}\cong (\Ind_P^GU^s)_N\cong U^s$.

 (iv) Using (ii) and the usual adjunction formula
\[\Hom_G(V,\Ind_P^G-)\cong \Hom_T(V_N,-),\]
we obtain
\[\Hom_{T}(R^f\Ord_PV,U)\cong \Hom_T(V_N,U\alpha_P)\cong \Hom_G(V,\Ind_P^G(U\alpha_P)).\]
The last assertion is obvious.

 (v) The case $L=\Q_p$ is contained in (iii) and the case $L\neq \Q_p$ is a special case of \cite[Cor.~4.2.4]{Hauseux}.
\end{proof}

There is a useful spectral sequence proved in \cite{EmertonOrd1}:
\begin{equation}\label{equation-Emerton-SS}E_2^{i,j}=\Ext^i_{T,\zeta}(U,R^j\Ord_PV)\Rightarrow \Ext^{i+j}_{G,\zeta}(\Ind_{\overline{P}}^GU,V).\end{equation}
Here, $\zeta$ denotes the central character of $V$ and $\Ext^i_{G,\zeta}$ (resp. $\Ext^i_{T,\zeta}$) 
indicates that we  compute extensions in the category $\Rep_{\F,\zeta}(G)$ (resp. $\Rep_{\F,\zeta}(T)$).
 In particular, we have a long exact sequence
\begin{multline*}0\ra \Ext^1_{T,\zeta}(U,\Ord_PV)\ra \Ext^1_{G,\zeta}(\Ind_{\overline{P}}^GU,V) \ra \Hom_{T}(U,R^1\Ord_PV)\ra \Ext^2_{T,\zeta}(U,\Ord_PV).\end{multline*}

\begin{corollary}\label{cor-Ext-2f+1}
 We have  a natural isomorphism
\[\Ext^{2f+1}_{G,\zeta}(\Ind_{\overline{P}}^GU,V)\cong \Ext^{f+1}_{T,\zeta}(U,R^{f}\Ord_PV).\]
For $i>2f+1$, we have $\Ext^i_{G,\zeta}(\Ind_{\overline{P}}^GU,V)=0$.
\end{corollary}
\begin{proof}
This follows from the fact that $R^{i}\Ord_P$ vanishes for $i\geq f+1$ and that $T/Z$ has cohomological dimension $f+1$.
\end{proof}

\begin{lemma} \label{lemma:Ord-irreducible}
Let $\pi$ be an irreducible smooth  representation of $G$.
\begin{enumerate}
\item[(i)] Assume $\pi\cong \mathrm{Sp}\otimes\chi\circ\det$,  where $\mathrm{Sp}$ denotes the Steinberg representation of $G$. Then $\Ord_P\pi\cong \chi\otimes \chi$ and $R^1\Ord_P\pi=0$.

\item[(ii)] Assume $\pi\cong \chi\circ\det$ is one-dimensional. Then $\Ord_P\pi=0$. If $L=\Q_p$, then $R^1\Ord_P\pi=\chi\omega^{-1}\otimes\chi\omega$; otherwise $R^1\Ord_P\pi=0$.
\item[(iii)] Assume $\pi$ is supersingular.\footnote{See \cite[p.290]{BL} for the definition of ``supersingular'' representations.} Then $\Ord_P\pi=0$.
\end{enumerate}
\end{lemma}
\begin{proof}
(i) It follows from  \cite[Thm.~4.2.12(2)]{EmertonOrd2}; the proof in \emph{loc.~cit.} works for general $L$.

(ii) The first assertion follows from Proposition \ref{prop-Ord}(i) and \cite[Prop.~29]{BL}. For the second, the case of $\GL_2(\Q_p)$ is proved in \cite[Thm.~4.2.12(3)]{EmertonOrd2}. 
The case  $L\neq\Q_p$  is a consequence of Proposition \ref{prop-Ord}(v). Indeed, we have $R^1\Ord_P(\Ind_{P}^G\chi\otimes\chi)=0$ and we deduce the result using (i) together with the short exact sequence $0\ra \chi\circ\det\ra \Ind_P^G\chi\otimes\chi\ra \mathrm{Sp}\otimes\chi\circ\det\ra0$.

(iii) It is a consequence of Proposition \ref{prop-Ord}(i).
\end{proof}

\begin{lemma}\label{lemma:V=directsum}
Let $U$ be a locally admissible smooth representation of $T$ (with a central character) and $V$ be a subquotient of $\Ind_{\overline{P}}^GU$. If $\Ord_PV=0$, then $V$ is a direct sum of one-dimensional representations of $G$.
\end{lemma}
\begin{proof}
If $\psi,\psi':T\ra\F^{\times}$ are distinct characters, then $\Ext^1_{T}(\psi,\psi')=0$ by \cite[Lem.~4.3.10]{EmertonOrd2}. Hence, any locally admissible $T$-representation $U$ can be decomposed as a direct sum $U\cong\oplus_{\psi} U_{\psi}$, where $U_{\psi}$ is the largest subrepresentation of $U$ whose Jordan--H\"older factors are all isomorphic to $\psi$. This implies $\Ind_{\overline{P}}^GU\cong\oplus_{\psi}\Ind_{\overline{P}}^GU_{\psi}$. By \cite[Thm.~30(1)]{BL} combined with Proposition \ref{prop-Ord}(iii), for $\psi\neq\psi'$ we have \[\JH(\Ind_{\overline{P}}^G\psi)\cap \JH(\Ind_{\overline{P}}^G\psi')=\emptyset.\]
As a consequence, any subrepresentation $V$ of $\Ind_{\overline{P}}^GU$ has a decomposition $V\cong \oplus_{\psi}V_{\psi}$, where $V_{\psi}$ is the largest subrepresentation of $V$ whose Jordan--H\"older factors all lie in $\JH(\Ind_{\overline{P}}^G\psi)$; explicitly $V_{\psi}=V\cap  \Ind_{\overline{P}}^GU_{\psi}$. It is clear that this decomposition remains true for any \emph{subquotient} of $\Ind_{\overline{P}}^GU$.  Hence, to prove the lemma, we may assume $U=U_{\psi}$ for some $\psi$, and so $V=V_{\psi}$.

Assume $\Ord_P V  = 0.$ Let $\pi$ be an irreducible subrepresentation of $V$. Then $\pi$ is  non-supersingular and $\Ord_P\pi=0$. By Proposition \ref{prop-Ord}(iii) and Lemma \ref{lemma:Ord-irreducible}, $\pi$ has to be one-dimensional, say $\pi\cong\chi\circ\det$ for some character $\chi: L^{\times}\ra \F^{\times}$,  and the assumption $U=U_{\psi}$  implies  $\psi=\chi\otimes\chi$. We claim that $\Ord_P(V/\pi)=0$. Assuming the claim, we may continue the argument to deduce that all Jordan--H\"older factors of $V$ are one-dimensional. Since $p>2$, $V$ has to be semisimple by \cite[Lem.~4.3.20, Prop.~4.3.21]{EmertonOrd2}  and the result follows.

Now we prove the claim. If  $L\neq \Q_p$, then the claim is obvious using Lemma \ref{lemma:Ord-irreducible}(ii).
If $L=\Q_p$, then by Lemma \ref{lemma:Ord-irreducible}(ii) the sequence $0\ra \pi\ra V\ra V/\pi\ra0$ induces an injection \[\partial: \Ord_P(V/\pi)\hookrightarrow R^1\Ord_P\pi\cong \chi\omega^{-1}\otimes\chi\omega.\]
However, since $V=V_{\psi}$, $\Ord_{P}(V/\pi)$ admits only $\psi=\chi\otimes\chi$ as subquotients, so $\partial$ must be zero (as $p>2$) and the claim follows.
\end{proof}
\subsection{Ordinary parts of injectives}

We first recall the following result.

\begin{proposition}\label{prop-BD-ordinary=injective}
Let $\Omega$ be an admissible smooth representation of $G$ such that $\Omega|_K$ is an  injective object in the category $\Rep_{\F}(K/Z_1)$. Then
\begin{enumerate}
\item[(i)] $\Ord_P\Omega$ is an injective object in the category $\Rep_{\F}(T_0/Z_1)$ and

\item[(ii)] $R^i\Ord_P\Omega=0$ for $i\geq 1$.
\end{enumerate}
\end{proposition}
\begin{proof}
(i) It is a special case of \cite[Cor.~4.5]{Breuil-Ding}.

(ii) It follows directly from the definition that $H^i\Ord_P\Omega=0$ if $\Omega$ is injective. The result then follows from the main result of \cite{EmertonPaskunas} recalled at the beginning of  \S\ref{subsection:ordinary}.
\end{proof}

\begin{lemma}\label{lemma:gene-by-socle}
Let $U$ be a finite dimensional  representation of $T$. Assume that $U$ becomes semisimple when restricted to $T_0$. Then $\Ind_{\overline{P}}^GU$ is generated by its $K$-socle as a $G$-representation.
\end{lemma}
\begin{proof}
Note that the $K$-socle of $\Ind_{\overline{P}}^GU$ depends only on the restriction of $U$ to $T_0$.   The assumption on $U$ implies that \[\soc_K(\Ind_{\overline{P}}^GU)\cong\bigoplus_{\psi\in\JH(U)}\soc_K(\Ind_{\overline{P}}^G\psi).\]
By \cite[Thm.~30]{BL}, $\Ind_{\overline{P}}^G\psi$ is generated by its $K$-socle, namely the assertion holds if $U=\psi$ is one-dimensional. The general case follows from the above equality of socles.
\end{proof}

\begin{proposition} \label{prop:ordinary-injective}
Let  $\Omega$ be an admissible smooth representation of $G$ such that $\Omega|_K$ is injective in the category $\Rep_{\F}(K/Z_1)$. Let  $\iota:V\hookrightarrow \Omega$ be a subrepresentation with   $\rsoc_K(V)=\rsoc_K(\Omega)$.
Then the induced inclusion  $\Ord_P(\iota): \Ord_PV\hookrightarrow \Ord_P\Omega$ is essential  when restricted to $T_0$.
\end{proposition}

\begin{proof}
Assume that $\Ord_P(\iota)$ is not essential when restricted to $T_0$. Then there exists a smooth character $\psi_0:T_0\ra \F^{\times}$ together with a $T_0$-equivariant embedding
\begin{equation}\label{equation-tau-T0}
\psi_0\oplus \Ord_PV\hookrightarrow \Ord_P\Omega.
\end{equation}
Choose a basis $v$ for the underlying space of $\psi_0$, and let $U:=\langle T.v\rangle\subset \Ord_P\Omega$ be the $T$-representation generated by $v$.    Since $\Ord_P\Omega$ is admissible and $T$ is abelian, $U$ is finite dimensional over $\F$ (because if $v$ is fixed by some open compact subgroup of $T_0$ then so is $tv$ for any $t\in T$). Moreover, again using the fact $T$ is abelian, one checks   that  $U|_{T_0}$ is semisimple and $\psi_0$-isotypic, i.e. $U|_{T_0}\cong \psi_0^{\oplus r}$ where $r=\dim_{\F}U$. Lemma \ref{lemma:gene-by-socle} implies that $\Ind_{\overline{P}}^GU$ is generated by its $K$-socle.
Hence, the image of the morphism (provided by Proposition \ref{prop-Ord}(i))
\[\beta:\Ind_{\overline{P}}^GU\ra \Omega\]
is also generated by its $K$-socle. In particular, $\im(\beta)\subset \langle G.\soc_K(\Omega)\rangle$. However, by assumption $\soc_K(V)=\soc_K(\Omega)$, so we get $\im(\beta)\subset V$  and consequently $U\subset \Ord_PV$,  contradicting \eqref{equation-tau-T0}.
\end{proof}

\begin{corollary}\label{cor:End-T0}
Keep the notation of Proposition \ref{prop:ordinary-injective}. Assume moreover that $\Ord_PV\cong \chi$ is irreducible. Then there is a ring isomorphism
\[\End_{T_0}((\Ord_P\Omega)^{\vee}|_{T_0})\cong\F[\![S_1,\dots,S_f]\!].\]
\end{corollary}
\begin{proof}
Combining Proposition \ref{prop-BD-ordinary=injective} and Proposition \ref{prop:ordinary-injective}, $(\Ord_P\Omega)|_{T_0}$ is isomorphic to an injective envelope of $\chi$ in $\Rep_{\F}(T_0/Z_1)$, and so $(\Ord_P\Omega)^{\vee}|_{T_0}$ is isomorphic to $\Proj_{T_0/Z_1}\chi^{\vee}$.
Let $T_1$ denote the pro-$p$ Sylow subgroup of $T_0$. Endowed with  the trivial action of $H$, $\F[\![T_1/Z_1]\!]$ is isomorphic to $\Proj_{T_0/Z_1}\ide$.  The assertion follows from \cite[Lem.~3.32]{Pa13} which says that $\Proj_{T_0/Z_1}\chi^{\vee}\cong \chi^{\vee}\otimes \F[\![T_1/Z_1]\!]$ represents the universal deformation problem (with $\varpi$-torsion coefficients) of $\chi$ with the universal deformation ring  isomorphic to $\F[\![S_1,\dots,S_f]\!]\cong\F[\![T_1/Z_1]\!]$.
\end{proof}
 \begin{lemma}\label{lemma:Ext-Inj=0}
 Let $\tau$ be a $1$-generic Serre weight and $\psi$ be a character of $T$. Let $U$ be an admissible $T$-representation whose Jordan--H\"older factors are all isomorphic to  $\psi$. Assume that  $U|_{T_0}$ is injective in the category $\Rep_{\F}(T_0/Z_1)$. If $\Hom_{K}(\tau,\Ind_{\overline{P}}^GU)\neq 0$, then $\Ext^1_{K/Z_1}(\tau,\Ind_{\overline{P}}^GU)=0$.
 \end{lemma}
 \begin{proof}
 First note that  the assumptions imply that $\Hom_K(\tau,\Ind_{\overline{P}}^{G}\psi)\neq0$. Since $\tau$ is $1$-generic, in particular $1<\dim_{\F}\tau<q$, it follows from \cite[\S7]{BL} that $\Ind_{\overline{P}}^G\psi$ is irreducible with $K$-socle isomorphic to $\tau$. We deduce that $\Hom_K(\tau',\Ind_{\overline{P}}^GU)=0$ for any Serre weight $\tau'$ such that $\tau'\neq \tau$.

 By  Shapiro's lemma, it is equivalent to show
$\Ext^1_{(\overline{P}\cap K)/Z_1}(\tau,U)=0$.
Note that $(\overline{P}\cap K)/Z_1\cong (T_0/Z_1) \ltimes \overline{N}_0$. Since $\overline{N}_0$ acts trivially on $U$  and $U$ is injective as a $T_0/Z_1$-representation by assumption, the  Hochschild-Serre spectral sequence implies
\[\Ext^1_{(\overline{P}\cap K)/Z_1}(\tau,U)\cong H^1\big((\overline{P}\cap K)/Z_1,\tau^{\vee}\otimes_{\F}U\big)\cong H^0\big(T_0/Z_1,H^1(\overline{N}_0,\tau^{\vee})\otimes_{\F} U\big).\]
A similar computation as in \cite[Prop.~2.5]{HuJLMS} shows that, if we write $\tau=(s_0,\cdots,s_{f-1})\otimes\eta$, then
\[H^1(\overline{N}_0,\tau^{\vee})\cong \oplus_{j\in\cS}\chi_{\tau}^{-1}\alpha_j^{s_j+1}\]
as $T_0$-representations. Using the $1$-genericity of $\tau$, i.e. $1\leq s_j\leq p-3$ for all $j$, one checks that $(\chi_{\tau}^{-1}\alpha_j^{s_j+1})^{-1}=\chi_{\tau}\alpha_{j}^{-(s_j+1)}=\chi_{\mu_{j+1}^-(\tau)}$ if $f\geq 2$ (resp. $\chi_{\mu_0^{+}(\tau)}$ if $f=1$). 
Hence, to prove the result it is equivalent to prove
\[\Hom_{T_0}\big(\oplus_{j\in\cS}\chi_{\mu_j^-(\tau)},U|_{T_0}\big)=0\]
if $f\geq 2$ (resp. $\Hom_{T_0}\big(\chi_{\mu_0^+(\tau)},U|_{T_0}\big)=0$ if $f=1$).
Assume this is not the case and assume $f\geq 2$. Then there exists an embedding $\chi_{\mu_i^-(\tau)}\hookrightarrow U|_{T_0}$ for some $i\in\cS$, hence embeddings
\[\mu_i^{-}(\tau)\hookrightarrow \Ind_{\overline{P}\cap K}^K\chi_{\mu_i^-(\tau)}\hookrightarrow (\Ind_{\overline{P}}^GU)|_K\]
where the first one is obtained by Frobenius reciprocity and \cite[Lem.~2(2)]{BL}.
 This gives a contradiction to the conclusion in the last paragraph.  The case $f=1$ can be treated similarly with $\mu_i^-(\tau)$ replaced by $\mu_0^+(\tau)$.
 \end{proof}
\subsection{$\Theta_{\tau}^{\rm ord}$ and ordinary parts}

We discuss the relation of the representation $\Theta^{\rm ord}_{\tau}$ studied in \S\ref{subsection:Theta-ord} and the ordinary parts of a smooth representation of $G$.

Let $V$ be a locally admissible smooth representation of $G$. Proposition \ref{prop-Ord}(i) implies a natural map
\[\jmath: \Ind_{\overline{P}}^G\Ord_{P}V\ra V \]
 whose image we denote by $V^{\rm ord}$.\footnote{Note that this is \emph{different} from the notation used in \cite{BH15}, at least when $L=\Q_p$. }  By construction, we have $\Ord_PV^{\rm ord}=\Ord_PV$. 

\begin{lemma}\label{lemma:j-canonical}
Let $\phi:V\ra V$ be a $G$-equivariant endomorphism. Let $\Ord_P(\phi)$ be the induced endomorphism of $\Ord_PV$ and $\phi'$ be the induced one of $\Ind_{\overline{P}}^G\Ord_PV$. Then the following diagram is commutative:
\[\xymatrix{\Ind_{\overline{P}}^G\Ord_PV\ar^{\ \ \ \ \jmath}[r]\ar_{\phi'}[d]&V\ar^{\phi}[d]\\
\Ind_{\overline{P}}^G\Ord_PV\ar^{\ \ \ \ \jmath}[r]&V.}\]
\end{lemma}
\begin{proof}
Denote by $\iota$ the isomorphism \eqref{equation-ord-adjunction} of Proposition \ref{prop-Ord}(i). The assertion is equivalent to $\iota(\phi\circ\jmath)=\iota(\jmath\circ\phi').$ 
 It is clear that  $\iota(\jmath)=\mathrm{Id}$, and by Proposition \ref{prop-Ord}(iii)   $\iota(\phi')=\Ord_P(\phi)$. Thus, taking $\Ord_P(-)$ of the diagram in the statement gives
\[\xymatrix{\Ord_PV\ar^{\mathrm{Id}}[r]\ar_{\Ord_P(\phi)}[d]&\Ord_PV\ar^{\Ord_P(\phi)}[d]\\
\Ord_PV\ar^{\mathrm{Id}}[r]&\Ord_PV}\]
from which the result follows.
\end{proof}

\begin{lemma}\label{lemma:ker-j}
$\mathrm{(i)}$ $\Ker(\jmath)$ is a  direct sum  of one-dimensional representations of $G$.

$\mathrm{(ii)}$  If $V_1\subseteq V$ is a subrepresentation of $V$, then $V_1^{\rm ord}\subseteq V_1\cap V^{\rm ord}$ and the corresponding quotient is a direct sum of one-dimensional representations of $G$.
\end{lemma}
\begin{proof}
(i) By construction, we know that $\Ord_P(\Ker(\jmath))=0$, so we conclude by Lemma \ref{lemma:V=directsum}.

(ii) The inclusion  $V_1^{\rm ord}\subseteq V_1\cap V^{\rm ord}$ is obvious; let  $C$ denote the quotient.  It is easy to see that $C$  is a subquotient of $\Ind_{\overline{P}}^GU$ for some $T$-representation $U$ (e.g. we may take $U=\Ord_PV/\Ord_PV_1$). If $L\neq \Q_p$, then taking $\Ord_P$ of $0\ra V_1^{\rm ord}\ra V_1\cap V^{\rm ord}\ra C\ra0$ gives again a short exact sequence by Lemma \ref{lemma:Ord-irreducible}(ii), from which we deduce $\Ord_P(C)=0$ and we conclude by Lemma \ref{lemma:V=directsum}.

Assume $L=\Q_p$ for the rest of the proof. As in the proof of Lemma \ref{lemma:V=directsum}, we may decompose $V^{\rm ord}$ as $\oplus_{\psi}(V^{\rm ord})_{\psi}$, and consequently
$V_1\cap V^{\rm ord}=\oplus_{\psi} V_1\cap (V^{\rm ord})_{\psi}$.
It suffices to show that the cokernel of $(V_1^{\rm ord})_{\psi}\subseteq V_1\cap (V^{\rm ord})_{\psi}$, denoted by $C_{\psi}$, satisfies $\Ord_P(C_{\psi})=0$.  There are two cases:
\begin{enumerate}
\item[$\bullet$] $\psi\cong\chi\otimes\chi$ for some $\chi$. Then $\Ord_P(C_{\psi})=0$ by the same proof as in Lemma \ref{lemma:V=directsum}.
\item[$\bullet$] $\psi\ncong \chi\otimes\chi$ for any $\chi$. Then the morphisms $\Ind_{\overline{P}}^G(\Ord_PV)_{\psi}\ra (V^{\rm ord})_{\psi}$ and $\Ind_{\overline{P}}^G(\Ord_PV_1)_{\psi}\ra (V_1^{\rm ord})_{\psi}$ are isomorphisms using Lemma \ref{lemma:V=directsum}.  Proposition \ref{prop-Ord}(v) implies that
\[R^1\Ord_P(\Ind_{\overline{P}}^G(\Ord_PV_1)_{\psi})\ra R^1\Ord_P(\Ind_{\overline{P}}^G(\Ord_PV_1)_{\psi})\]
is equal to the natural morphism $(\Ord_PV_1)_{\psi}^s\ra (\Ord_PV)_{\psi}^s$ twisted by $\alpha_{P}^{-1}$, hence is injective. This means that the morphism $R^1\Ord_P(V_1^{\rm ord})_{\psi}\ra R^1\Ord_P(V^{\rm ord})_{\psi}$ is  injective, hence so is
\[R^1\Ord_P(V_1^{\rm ord})_{\psi}\ra R^1\Ord_P\big(V_1\cap (V^{\rm ord})_{\psi}\big).\]
This implies $\Ord_P(C_{\psi})=0$ as desired.
\end{enumerate}
\end{proof}
\begin{remark}
For our application in \S\ref{section-patching}, $V|_K$ will not admit one-dimensional Serre weights as subrepresentations, in which case  Lemma \ref{lemma:ker-j} is easy to show. However, we keep the generality  because the result might be useful elsewhere.
\end{remark}

 \begin{corollary}\label{cor:j=isom}
Let $V$ be a locally admissible smooth representation of $G$. Let $\lambda$ be a finite dimensional $K$-representation which does not admit any Jordan--H\"older factor  of dimension $1$ or $q$. Then $\jmath$ induces an isomorphism
 \[\Hom_{K}(\lambda,\Ind_{\overline{P}}^G\Ord_PV)\simto\Hom_{K}(\lambda,V^{\rm ord}).\]
 \end{corollary}
 \begin{proof}
 As in the proof of Lemma \ref{lemma:V=directsum}, we may assume $\Ord_PV=(\Ord_PV)_{\psi}$ for some $\psi$, so that all of  Jordan--H\"older factors of $V^{\rm ord}$ or $\Ind_{\overline{P}}^G\Ord_PV$ lie in $\JH(\Ind_{\overline{P}}^G\psi)$.

Write $\psi=\chi_1\otimes\chi_2$ for characters $\chi_1,\chi_2:L^{\times}\ra \F^{\times}$. If $\chi_1\neq\chi_2$, then $\Ind_{\overline{P}}^G\psi$ does not admit one-dimensional representations of $G$ as subquotients, thus $\jmath$ induces an isomorphism $\Ind_{\overline{P}}^G\Ord_PV\simto V^{\rm ord}$ by Lemma \ref{lemma:ker-j}(i) and the result is obvious. If $\chi_1= \chi_2$, then  any Serre weight occurring in $\soc_K(\pi)$, for $\pi\in \JH(\Ind_{\overline{P}}^G\psi)$  has dimension $1$ or $q$, thus the assumption on $\lambda$ implies that $\Hom_{K}(\lambda,V^{\rm ord})=\Hom_K(\lambda,\Ind_{\overline{P}}^G\Ord_PV)=0$. \end{proof}

\begin{lemma}\label{lemma:sigma-Vord}
Let $V$ be a locally admissible smooth representation of $G$. Let $\tau$ be a  Serre weight such that $1<\dim_{\F}\tau <q$. Assume that $\Hom_{K}(\tau',V)=0$ for any $\tau'\in \JH(\Ind_I^K\chi_{\tau})$ such that $\tau'\neq \tau$. Then $\jmath$ induces isomorphisms
 \begin{equation}\label{eq:tau-Vord} \Hom_K(\tau,\Ind_{\overline{P}}^G\Ord_PV)\simto \Hom_K(\tau, V^{\rm ord})\simto\Hom_K(\tau,V).\end{equation}
The same statement holds if we replace $\tau$ by any subrepresentation of $\Ind_I^K\chi_{\tau}^s$.
\end{lemma}
\begin{proof}
Since $\tau$ (resp. $\Ind_I^K\chi_{\tau}^s$) is finite dimensional and since $V$ is equal to the direct limit of its admissible subrepresentations, we may assume $V$ is admissible.

First prove the isomorphisms \eqref{eq:tau-Vord} for $\tau$. By Corollary \ref{cor:j=isom}, we are left to prove the second isomorphism. The proof is by a standard weight cycling argument.  Let $\mathfrak{R}_0\defn KZ$ and $I(\tau)\defn\cInd_{\mathfrak{R}_0}^G\tau$, where $Z$ acts on $\tau$ via the central character of $\pi$.  Since $V$ is admissible, $\Hom_{\mathfrak{R}_0}(\tau,V)$ is a finite dimensional $\F$-vector space.   It is well-known that $\Hom_{\mathfrak{R}_0}(\tau,V)$, which is isomorphic to $\Hom_{G}(I(\tau),V)$ via Frobenius reciprocity, carries  an action of the Hecke algebra $\End_{G}(I(\tau))\cong\F[T]$ (see \cite{BL}).  Up to enlarge  $\F$, we may assume all the eigenvalues of $T$ are contained in $\F$.

 We claim that $\lambda\neq 0$ for any eigenvalue $\lambda$ of $T$.  Otherwise, choose a non-zero eigenvector in $\Hom_K(\tau,V)$ on which $T$ acts by $0$, we then obtain a $G$-equivariant morphism
$I(\tau)/T\ra V$.
By considering the action of $\smatr{0}1p0$ on $\tau^{I_1}$, we obtain a nonzero $K$-equivariant morphism $\Ind_I^K\chi_{\tau}^s\ra V$, which factors through $(\Ind_I^K\chi_{\tau}^s)/\tau$ (this uses the explicit description of $T$, see \cite{BL}). Since $\Ind_I^K\chi_{\tau}^s$ is multiplicity free, we have $\Hom_K((\Ind_I^K\chi_{\tau}^s)/\tau,V)=0$ by assumption, a contradiction.

The claim implies that any morphism $ I(\tau) \ra V$ factors through $I(\tau)/f(T)$, for some polynomial $f(T)=\prod_{i}(T-\lambda_i)^{a_i}$, with $\lambda_i\neq0$.  By \cite[\S6]{BL},   if either $\dim_{\F}\tau\neq 1$ or $\lambda_i\neq \pm1$, then $ I(\tau)/(T-\lambda_i)$  is irreducible and we have $V_i\cong \Ind_{\overline{P}}^G\Ord_PV_i$  for any quotient $V_i$ of $I(\tau)/(T-\lambda_i)^{a_i}$.   Since $\tau$ has dimension $\geq 2$ by assumption, we deduce an isomorphism $\Hom_G(I(\tau),V^{\rm ord})\simto \Hom_G(I(\tau),V)$, and the result follows.

We now prove \eqref{eq:tau-Vord} with $\tau$ replaced by a subrepresentation $W$ of $\Ind_I^K\chi_{\tau}^s$. We may assume $W$ is nonzero and so $\tau\hookrightarrow W\hookrightarrow \Ind_I^K\chi_{\tau}^s$ (as $\tau$ is the socle of $\Ind_{I}^K\chi_{\tau}^s$). First note that for any admissible representation $V'$ of $G$,
we have the following injections and isomorphisms
\begin{equation}\label{eq:dim-tau-Ind}\Hom_{K}(\tau,V')\hookrightarrow \Hom_{I}(\chi_{\tau},V')\simto\Hom_{I}(\chi_{\tau}^s,V')\simto\Hom_K(\Ind_I^K\chi_{\tau}^s,V')\end{equation}
where the first map is the restriction map induced by $\tau^{I_1}\hookrightarrow \tau$ which is injective as $\tau$ is irreducible, the second map is obtained by taking conjugation by $\smatr{0}1p0$ (as $V'$ is a representation of $G$), and the third is induced by Frobenius reciprocity. If moreover $\Hom_{K}(\tau',V')=0$ for any $\tau'\in\JH(\Ind_I^K\chi_{\tau}^s)$ with $\tau'\neq \tau$, then we further have injections
\[\Hom_K(\Ind_I^K\chi_{\tau}^s,V')\hookrightarrow \Hom_K(W,V')\hookrightarrow \Hom_{K}(\tau,V')\]
which must be isomorphisms  by comparing their dimensions and using \eqref{eq:dim-tau-Ind}. The result then follows from this together with \eqref{eq:tau-Vord} for $\tau$ and Corollary \ref{cor:j=isom}.
\end{proof}

Recall that $\Theta_{\tau}^{\rm ord}$ is defined in Lemma  \ref{lemma:Theta-ord} for any $2$-generic Serre weight.

\begin{proposition}\label{prop:Hom-Theta-Ord}
Let $V$ be a locally admissible smooth representation of $G$ and $\tau$ be a $2$-generic Serre weight. Assume that $\Hom_{K}(\tau',V)=0$ for any $\tau'\in \JH(\Ind_I^K\chi_{\tau})$ such that $\tau'\neq \tau$. 
 Then $\jmath$ induces  isomorphisms 
 \begin{equation}\label{eq:Hom-Vord}
\Hom_K(\Theta_{\tau}^{\rm ord},\Ind_{\overline{P}}^G\Ord_PV) \simto \Hom_{K}(\Theta_{\tau}^{\rm ord}, V^{\rm ord})
\simto\Hom_K(\Theta_{\tau}^{\rm ord},V).\end{equation}
\end{proposition}
\begin{proof}

It is clear that we may assume $V$ is admissible. By \cite[Cor.~9.11]{BP},  there exists a $G$-equivariant embedding $V\hookrightarrow \Omega$, where $\Omega$ is a smooth $G$-representation such that  $\Omega|_K\cong \rInj_{K/Z_1}\soc_K(V)$. Note that although it is required that $p$ acts trivially on $V$ in \emph{loc. cit.}, up to twist the result applies to any admissible representation with a central character.   Assuming we have proven an isomorphism
\begin{equation}\label{eq:Hom-OmegaOrd}\Hom_{K}(\Theta_{\tau}^{\rm ord},\Omega^{\rm ord})\simto\Hom_K(\Theta_{\tau}^{\rm ord},\Omega), \end{equation}
 the desired isomorphism \eqref{eq:Hom-Vord}  will follow using  Lemma \ref{lemma:ker-j}(ii). Indeed, let $f\in\Hom_K(\Theta_{\tau}^{\rm ord},V)$; we need to prove $\im(f)\subset V^{\rm ord}$. By \eqref{eq:Hom-OmegaOrd}, $\im(f)\subset V \cap \Omega^{\rm ord}$. Since $\tau$ is $2$-generic by assumption,  no Jordan--H\"older factor  of $\Theta_{\tau}^{\rm ord}$ is one-dimensional, so we  actually have $\im(f)\subseteq V^{\rm ord}$ by Lemma \ref{lemma:ker-j}(ii).

So we may assume that $V=\Omega$ is injective when restricted to $K/Z_1$. 
Recall that $\Theta_{\tau}^{\rm ord}$ fits in a short exact sequence by \eqref{eq:seq-Theta-ord}
\[0\ra \bigoplus_{j\in\cS}E_{\tau,\mu_j^{-}(\tau)}\ra\Theta_{\tau}^{\rm ord}\ra \tau\ra0.\]
It induces a commutative diagram  
\[{\small\xymatrix{0\ar[r]&\Hom_K(\tau,\Omega')\ar[r]\ar^{\iota}[d]&\Hom_K(\Theta_{\tau}^{\rm ord}, \Omega') \ar[r]\ar[d]&\bigoplus_{j\in\cS}\Hom_K(E_j, \Omega')\ar^{\oplus_j\iota_j}[d]\ar^{\ \ \partial}[r]&\Ext^1_{K/Z_1}(\tau,\Omega')\\
0\ar[r]&\Hom_K(\tau,\Omega)\ar[r]&\Hom_K(\Theta_{\tau}^{\rm ord},\Omega)\ar[r]&\bigoplus_{j\in\cS}\Hom_K(E_j,\Omega)\ar[r]&0,}}\]
where  we have written $\Omega'=\Ind_{\overline{P}}^G\Ord_P\Omega$  and $E_j=E_{\tau,\mu_j^{-}(\tau)}$ to shorten the formulas. The bottom row is exact by the injectivity of $\Omega$.
Using Lemma \ref{lemma:sigma-Vord}, the assumption on $\soc_K(V)$ implies that $\iota$  and  all $\iota_j$ are isomorphisms (as $E_j$ is a subrepresentation of $\Ind_I^K\chi_{\tau}^s$). We claim that $\partial$ is the zero map, which will finish the proof   by the snake lemma.

Prove the claim.   Since  $\soc_K(V)=\soc_K(\Omega)$, the assumption  implies that $\Hom_{K}(\mu_j^-(\tau),\Omega)=0$ for all $j\in\cS$. As noted above, $\dim_{\F}\mu_j^{-}(\tau)>1$, hence by Lemma  \ref{lemma:ker-j}(i)
\[\Hom_K(\mu_j^{-}(\tau),\Ind_{\overline{P}}^G\Ord_P\Omega)=0.\]
Consequently,  we may assume $\Hom_K(\tau,\Ind_{\overline{P}}^G\Ord_P\Omega)\neq0$, otherwise the claim is trivial.
Decomposing $\Ord_P\Omega=\oplus_{\psi}(\Ord_P\Omega)_{\psi}$  as in the proof of Lemma \ref{lemma:V=directsum}, it suffices to prove the claim with $\Ord_P\Omega$ replaced by $(\Ord_P\Omega)_{\psi}$, for those $\psi$ such that $\Hom_K(\tau,\Ind_{\overline{P}}^G(\Ord_P\Omega)_{\psi})\neq0$.  But, Lemma \ref{lemma:Ext-Inj=0} implies that $\Ext^1_{K/Z_1}(\tau,\Ind_{\overline{P}}^G(\Ord_P\Omega)_{\psi})=0$, from which the claim follows.
\end{proof}

The next result gives an interpretation of  the semisimplicity of $(\Ord_PV)|_{T_0}$ in terms of $V|_K$.

\begin{proposition}\label{prop:Ord-semisimple}
Let $V$ be a locally admissible smooth representation of $G$ and $\tau$ be a $2$-generic Serre weight. Assume that $\Hom_{K}(\tau',V)=0$ for any $\tau'\in \JH(\Ind_I^K\chi_{\tau})$ such that $\tau'\neq \tau$.
 If $\Ord_PV$ is semisimple when restricted to $T_0$, then the quotient  $\Theta_{\tau}^{\rm ord}\twoheadrightarrow \tau$ induces an isomorphism
\[\Hom_K(\tau,V)\simto\Hom_K(\Theta_{\tau}^{\rm ord},V).\]
\end{proposition}
\begin{proof}
Again we may assume $V$ is admissible. Moreover, by Proposition \ref{prop:Hom-Theta-Ord} and its proof, we may assume $V=\Ind_{\overline{P}}^G\Ord_PV$. Since the assertion depends only on $V|_{K}$, hence only on $(\Ord_PV)|_{T_0}$ which by assumption is semisimple, we  may assume $ \Ord_PV=\psi$ is one-dimensional and so $V\cong\Ind_{\overline{P}}^G\psi$.
As in the proof of Proposition \ref{prop:Hom-Theta-Ord}, we may assume $\Hom_{K}(\tau,V)\neq0$ and consequently $\tau\cong \soc_K(V)$.

 Let $h:\Theta_{\tau}^{\rm ord}\ra V|_{K}$ be a nonzero morphism.
 We need to prove that $h$ factors through $ \Theta_{\tau}^{\rm ord}\twoheadrightarrow \tau$. It suffices to prove that $h$ is zero when restricted to $\soc(\Theta_{\tau}^{\rm ord})$. Assume this is not the case. Then  $[\im(h):\tau]\geq 2$ and Lemma \ref{lemma:Theta-ord-K1} implies that $\im(h)$ is not annihilated by $\fm_{K_1}$ (as $\im(h)_{K_1}$ is a quotient of $(\Theta_{\tau}^{\rm ord})_{K_1}$ which is multiplicity free). Moreover, it is easy to see that $\im(h)\cap V^{K_1}=\rad(\im(h))$, which induces an embedding
$\mathrm{cosoc}(\im(h))\cong \tau\hookrightarrow V/V^{K_1}$. This again gives a contradiction  by Lemma \ref{lemma:PS-K1} below.
\end{proof}

The following lemma is well-known; we include a proof for lack of a suitable reference.
\begin{lemma}\label{lemma:PS-K1}
Assume $p>2$. Let $\pi=\Ind_{P}^G\psi$ be a principal series of $G$. Let $\sigma$ be a Serre weight such that $\Hom_K(\sigma,\pi|_K)\neq0$.
Then $\Hom_K(\sigma,\pi/\pi^{K_1})=0$.
\end{lemma}
\begin{proof}
First observe that, since $\pi|_K\cong\Ind_{P\cap K}^K(\psi|_{T_0})$, 
the assumption  implies $\Hom_{P\cap K}(\sigma,\psi|_{T_0})\neq0$ by Frobenius reciprocity, hence  $\psi|_{T_0}=\chi_{\sigma}^s$ by \cite[Lem.~2]{BL}.

The exact sequence $0\ra \pi^{K_1}\ra\pi\ra \pi/\pi^{K_1}\ra0$ induces an exact sequence
\[0\ra \Hom_K(\sigma,\pi/\pi^{K_1})\ra \Ext^1_{K/Z_1}(\sigma,\pi^{K_1})\overset{\beta}{\ra}\Ext^1_{K/Z_1}(\sigma,\pi),\]
so it is enough to show $\beta$ is injective.  By Shapiro's lemma and using the fact that $\pi^{K_1}\cong\Ind_{I}^K(\psi|_{T_0})$, this is equivalent to show the injectivity of
\[\Ext^1_{I/Z_1}(\sigma,\chi_{\sigma}^s)\ra\Ext^1_{(P\cap K)/Z_1}(\sigma,\chi_{\sigma}^s),\]
or equivalently the injectivity of
\[\Ext^1_{I/Z_1}(\ide,\chi_{\sigma}^s\otimes\sigma^{\vee})\ra\Ext^1_{(P\cap K)/Z_1}(\ide,\chi_{\sigma}^s\otimes\sigma^{\vee})\]
where the $\ide$'s denote the trivial representations.

Consider an $I/Z_1$-extension $0\ra \chi_{\sigma}^s\otimes\sigma^{\vee}\ra \mathcal{E}\ra \F v\ra0$, where $I$ acts trivially on $v$, and assume that it  splits when restricted to $(P\cap K)/Z_1$. Then we may choose a lifting of $v$, say $w\in \mathcal{E}$, on which $P\cap K$ acts trivially. It is enough to prove that  $\overline{N}_1\defn\smatr{1}0{p\cO_L}1$ also acts trivially on $w$, because then $I$ will act trivially on $w$ and $\mathcal{E}$ splits. It is clear that $\smatr{1}0{p^2\cO_L}1$ acts trivially on $\mathcal{E}$. The matrix identity (for $b,c\in\cO_L$)
\[\matr{1}{b}01\matr{1}0{pc}1=\matr{1}0{pc(1+pbc)^{-1}}1\matr{1+pbc}{b}{0}{(1+pbc)^{-1}}\]
implies that $\smatr{1}0{pc}1w$ is again fixed by $N_0$. By \cite[Lem.~2]{BL}, $(\chi_{\sigma}^s\otimes \sigma^{\vee})^{N_0}$ is one-dimensional and it is easy to see that $H$ acts on it via the trivial character $\ide$. Hence, if $w$ were not fixed by $\overline{N}_1$, then we would obtain a nonsplit extension class in $\Ext^1_{H\overline{N}_1}(\ide,\ide)$. However, the same proof  of \cite[Lem.~5.6]{Pa10} shows that $\Ext^1_{H\overline{N}_1}(\ide,\ide)=0$ (this uses $p>2$), a contradiction. 
\end{proof}

\section{Galois deformation rings}
\label{Section::galois}

The aim of this section is to recall the results of \cite{Le} on multi-type potentially Barsotti-Tate deformation rings of two dimensional representations of $G_L$ over $\F$ (in the reducible nonsplit case), and prove Proposition \ref{prop-tangent} and Corollary \ref{cor::intersect-tang-space} which will be used in \S\ref{subsection:cyclic}. We first recall the notion of the universal (reducible) deformation rings.

\subsection{Universal deformation rings} \label{subsection-univdef}

Let $\brho=\smatr{\chi_1}{*}{0}{\chi_2}$ be a reducible nonsplit two-dimensional representation of $G_L$ over $\F$ satisfying
\begin{equation}\label{equation-condition-rhobar}
\chi_1\chi_2^{-1}\notin\{1,\omega,\omega^{-1}\}.
\end{equation}
Let $\mathrm{ad}(\brho)$ denote $\End_{\F}(\brho)$ with  the adjoint action of $G_L.$ The assumption (\ref{equation-condition-rhobar}) on $\brho$ implies that
\begin{equation}\label{equ-end-of-brho}
H^0(G_L , {\rm ad}(\brho)) = \End_{G_L} (\brho) = \F,~~ H^0 (G_L, {\rm ad}(\brho)(1))= \Hom_{G_L} (\brho, \brho(1)) = 0
\end{equation}
where $V(1)$ denotes the Tate twist of $V$ for any $G_L$-module $V.$

\begin{lemma}\label{lemma::extn-chars}
$\Ext^2_{G_L} (\chi_1, \chi_2) = \Ext^2_{G_L} (\chi_2, \chi_1) = 0.$
\end{lemma}
\begin{proof}
This follows from the assumption (\ref{equation-condition-rhobar}) and Tate local duality.
\end{proof}

\begin{lemma}\label{lemma-h^1}
$ H^2(G_L, {\rm ad}(\brho)) = 0$ and $\dim_{\F} H^1(G_L , {\rm ad}(\brho))=4f+1.$
\end{lemma}
\begin{proof}
This first equality follows from (\ref{equ-end-of-brho}) and Tate local duality. The second equality follows from the local Euler-Poincar\'e characteristic formula.
\end{proof}

Let $\mathrm{Art}(\cO)$ denote the category of local artinian $\cO$-algebras with residue field $\F$. A deformation of $\brho$ to $A\in\mathrm{Art}(\cO)$ is a representation $\rho_A : G_L\to \GL_2(A)$ of $G_L$ such that the composition of $\rho_A$ with the natural map $\GL_2(A) \to \GL_2(\F)$ is $\brho.$ Two deformations $\rho_A,~\rho_A'$ of $\brho$ to $A$ are strictly equivalent if there is $M \in \Ker(\GL_2(A) \to \GL_2(\F))$ such that $\rho_A = M^{-1} \rho_A' M.$ Let $\mathrm{Def}_{\brho}:\mathrm{Art}(\cO)\ra \mathrm{Sets}$ be the functor sending $A$ to the set of strictly equivalent classes of deformations of $\brho$ to $A.$ Since $\End_{G_L}(\brho)=\F,$ Mazur's theory \cite{Mazur} on the deformation of Galois representations shows that the deformation functor $\mathrm{Def}_{\brho}$ is pro-representable by a complete noetherian local $\cO$-algebra $R_{\brho}.$

\begin{corollary}
$R_{\brho}$ is formally smooth over $\cO$ of relative dimension $4f+1$.
\end{corollary}
\begin{proof}
Since $\dim_{\F} H^2(G_L, {\rm ad}(\brho)) = 0$, $R_{\brho}$ is formally smooth by \cite[Prop.~2]{Mazur}. The relative dimension of $R_{\brho}$ over $\cO$ follows from the corresponding dimension of $H^1(G_L , {\rm ad}(\brho))$ which is given in Lemma \ref{lemma-h^1}.
\end{proof}

Fix $\psi:G_{L}\ra \cO^{\times}$ a continuous character which lifts $\det\brho$. Let $\mathrm{Def}^{\psi}_{\brho}:\mathrm{Art}(\cO)\ra \mathrm{Sets}$ be the functor sending $A$ to the set of strictly equivalent classes of deformations $\rho_A$ of $\brho$ over $A$ such that $\det \rho_A = \psi_A,$ where $\psi_A$ is the composite $G_{L} \To{\psi} \cO^{\times} \to A^{\times}.$ The deformation functor $\mathrm{Def}^{\psi}_{\brho}$ is pro-representable by a complete noetherian local $\cO$-algebra $R^{\psi}_{\brho}.$ Let ${\rm ad}^0(\brho)$ be the subspace of ${\rm ad}(\brho)$ consisting of matrices of trace zero. It is stable under the action of $G_L.$
Similarly as in the proof of Lemma \ref{lemma-h^1}, one can show $\dim H^2(G_L, {\rm ad}^0(\brho)) = 0$ and $\dim H^1(G_L , {\rm ad}^0(\brho))=3f.$ We then deduce that $R^{\psi}_{\brho}$ is formally smooth over $\cO$ of relative dimension $3f$.

\subsection{Reducible deformation rings} \label{subsection-reducible-deformation}

Let $\brho$ be as in the last subsection. A deformation $\rho_A$ of $\brho$ to $A\in\mathrm{Art}(\cO)$ is said to be \emph{reducible} (or equivalently {\em $P$-ordinary} in \cite[\S5.1]{Breuil-Ding} where $P$ denotes the upper-triangular Borel subgroup of $\GL_2$) if $\rho_A$ has a free rank one direct summand over $A$ which is stable under $G_L.$ We define the functor $\mathrm{Def}^{\rm red}_{\brho}:\mathrm{Art}(\cO)\ra \mathrm{Sets}$ by sending $A$ to the set of strictly equivalent classes of reducible deformations of $\brho.$ By \cite[Prop.~3]{MazurFermat} (or by \cite[Lem.~5.3, Prop.~5.4]{Breuil-Ding} which is more adapted to our situation), $\mathrm{Def}^{\rm red}_{\brho}$ is a subfunctor of $\mathrm{Def}_{\brho}$ and is pro-representable by a complete  noetherian local $\cO$-algebra $R^{\rm red}_{\brho}$ with residue field $\F$.

Denote by $\mathrm{ad}(\brho)_{\rm red}\subset \mathrm{ad}(\brho)$ the subspace given by the following short exact sequence
\[0\ra \mathrm{ad}(\brho)_{\rm red}\ra \mathrm{ad}(\brho)\ra \Hom_{\F}(\chi_1,\chi_2)\ra0,\]
where the homomorphism $\mathrm{ad}(\brho)\ra \Hom_{\F}(\chi_1,\chi_2)$ is given by
\[
\phi \longmapsto (\chi_1\into \brho \To{\phi} \brho \onto \chi_2).
\]
One checks that $\mathrm{ad}(\brho)_{\rm red}$ is stable under the adjoint action of $G_L$.

\begin{lemma}\label{lemma::reducible-def-ring}

(i) $H^0 (G_L, \mathrm{ad}(\brho)_{\rm red}) = \F.$

(ii) $H^2 (G_L, \mathrm{ad}(\brho)_{\rm red}) = 0.$

(iii) $\dim H^1  (G_L, \mathrm{ad}(\brho)_{\rm red}) = 3f+1.$
\end{lemma}

\begin{proof}
(i) Since $H^0(G_L, \Hom_{\F}(\chi_1,\chi_2)) = \Hom_{G_L} (\chi_1, \chi_2) =0$ by the assumption \eqref{equation-condition-rhobar}, we get $H^0(G_L,\mathrm{ad}(\brho)_{\rm red} ) =  H^0 (G_L, \mathrm{ad}(\brho) ) = \F.$

For (ii), since $H^2(G_L,\mathrm{ad}(\brho))=0$, it suffices to show that the natural morphism
\begin{equation}\label{equ::ext1-surjection}
\Ext^1_{G_L} (\brho, \brho) \cong H^1(G_L,\mathrm{ad}(\brho))\ra H^1(G_L,\Hom_{\F}(\chi_1,\chi_2)) \cong \Ext^1_{G_L}(\chi_1, \chi_2)
\end{equation}
is surjective. First, applying $\Hom_{G_L} ( -, \chi_1)$ to the short exact sequence
\begin{equation}\label{eq:rhobar}
0 \to \chi_1 \to \brho \to \chi_2 \to 0
\end{equation}
we obtain an exact sequence
\[
\Ext^2_{G_L} (\chi_2, \chi_1) \to \Ext^2_{G_L} (\brho, \chi_1) \to \Ext^2_{G_L} (\chi_1, \chi_1).
\]
By Lemma \ref{lemma::extn-chars} and the fact that $\Ext^2_{G_L} (\chi_1, \chi_1) = 0,$ we have $\Ext^2_{G_L} (\brho, \chi_1)= 0$. As a consequence, applying $\Hom_{G_L} (\brho, - )$ to \eqref{eq:rhobar}
gives a surjection
\begin{equation}\label{sequence-Ext1-brho1}
\Ext^1_{G_L} (\brho, \brho) \to \Ext^1_{G_L} (\brho, \chi_2) \to 0.
\end{equation}
Similarly, since $\Ext^2_{G_L}(\chi_2,\chi_2)=0$, we have a surjection  \begin{equation}\label{sequence-Ext1-brho2}
\Ext^1_{G_L} (\brho, \chi_2) \to \Ext^1_{G_L} (\chi_1, \chi_2) \ra0.
\end{equation}
The surjectivity of (\ref{equ::ext1-surjection}) then follows from  (\ref{sequence-Ext1-brho1}) and (\ref{sequence-Ext1-brho2}).

(iii) follows from (i), (ii) and the local Euler-Poincar\'e characteristic formula.
\end{proof}

\begin{proposition}\label{prop-red-defring}
$R^{\rm red}_{\brho}$ is formally smooth over $\cO$ of relative dimension $3f+1$.  
\end{proposition}
\begin{proof}
This is a variant of \cite[\S30]{MazurFermat}. One checks that  a deformation $\brho'$ of $\brho$ to $\F[\epsilon]/\epsilon^2$ (the ring of dual numbers)  is reducible if and only if it takes its values in $\mathrm{ad}(\brho)_{\rm red}$ when viewed as an element in $H^1(G_L,\mathrm{ad}(\brho))$. Then the assertion follows from Lemma \ref{lemma::reducible-def-ring}.
\end{proof}

Set
\[
R^{\psi,\rm red}_{\brho} \defn R^{\psi}_{\brho}\otimes_{R_{\brho}}R^{\rm red}_{\brho}.
\]
We have the following variant of Proposition \ref{prop-red-defring}.
\begin{proposition}\label{prop-red-defring-det}
$R^{\psi, \rm red}_{\brho}$ is formally smooth over $\cO$ of relative dimension $2f$.  
\end{proposition}

\subsection{Serre weights}\label{Section::Serre-wts}

We recall some terminology used in \cite{Le}. Let $G$ be the algebraic group ${\rm Res}_{\F_{p^f}/\F_p}\GL_2.$ Let $T$ be the diagonal torus in $G.$ We identify the character group $X^*(T) = X^*(T\times_{\F_p}\F)$ with $(\Z^2)^f.$ We say that $\mu\in X^*(T)$ is {\em $p$-restricted} if $0\leq \langle \mu,\a\rangle <p$ for all positive coroots $\a.$ Let $\eta^{(i)}\in X^*(T)$ (resp. $\alpha^{(i)}\in X_*(T)$) be the dominant fundamental character (resp. the positive coroot) represented by $(1,0)$ (resp. $(1,-1)$) in the $i$-th coordinate and $0$ elsewhere, and $\eta \defn \sum_{i\in\Z/f\Z}\eta^{(i)}.$ Let $G^{\rm der}\defn {\rm Res}_{\F_{p^f}/\F_p}{\rm SL}_2$ and $T^{\rm der}\subset G^{\rm der}$ be the standard torus. Let $\omega^{(i)}$ be the restriction of $\eta^{(i)}$ to $T^{{\rm der}}.$

For a dominant character $\mu\in X^*(T),$ let $V(\mu)$ be the Weyl module defined in \cite[II.2.13(1)]{Jantzen}. It has a unique simple $G$-quotient $L(\mu).$ If $\mu=\sum_i\mu_i^{(i)}$ is $p$-restricted then $L(\mu)=\otimes_i L(\mu_i)^{(i)}$ by  Steinberg's tensor product theorem. Let $F(\mu)$ be the $\G$-representation $L(\mu)|_{\G},$ where $\G=G(\F_p)\cong \GL_2(\F_{p^f}).$ Then $F(\mu)$ is irreducible by \cite[A.1.3]{Herzig}.

Let $\mu\in X^*(T)$ be such that $1\leq \langle \mu-\eta,\a^{(i)} \rangle <p-2$ for all $i\in \Z/f\Z.$ Let $S \defn \{\pm\omega^{(i)}\}_{i\in \Z/f\Z}.$ For any subset $J$ of  $S,$ let $\s_J \defn F(\mathfrak{t}_\mu(\o_J))$ be the Serre weight defined in \cite[Def.~3.5]{LMS}, we refer the reader to \S 2 of {\em loc. cit.} for the notation used here.

Recall that $L$ denotes the fixed unramified extension of $\Q_p$ of degree $f.$ Write $P(v)=v+p$ for the minimal polynomial of $\pi_L = -p$ over $\Q_p.$ Let $L_{\infty}=L((-p)^{1/{p^{\infty}}})$ by choosing a compatible system of $p^n$-th roots of $-p$ in $\bQp$.
  Let $\brho:G_L\to \GL_2(\F)$ be a continuous {\em reducible nonsplit} representation, i.e. $\brho\cong\smatr{\chi_1}{*}0{\chi_2}$. We may write \[\chi_1 ={\rm nr}_{\a}\omega_f^{\sum_{i=0}^{f-1}\mu_{1,i}p^i},\ \  \chi_2 ={\rm nr}_{\a'}\omega_f^{\sum_{i=0}^{f-1}\mu_{2,i}p^i}\] for some dominant $p$-restricted character $\mu_{\brho} \defn (\mu_{1,i},\mu_{2,i})_{i\in \Z/f}.$ Up to twist, we may assume $(\mu_{1,i},\mu_{2,i})=(c_i,1).$ We further assume $4 \leq c_i \leq p-3$ for all $i\in \Z/f\Z,$ equivalently, $\brho$ is \emph{strongly generic} in the sense of Definition \ref{defn:strong-generic}. Note that this is the same genericity condition imposed in \cite{Le}. In particular this implies $p\geq 7.$ Moreover, $\overline{\rho}$ lies in the category of Galois representations defined by Fontaine-Laffaille (\cite{FL}), hence it can be written as
\[\overline{\rho}=\Hom_{\Fil^{\cdot},\varphi_{\cdot}}(M,A_{\cris}\otimes_{\Z_p}\F_p)\]
where $M$ is a filtered   $\varphi$-module of Fontaine-Laffaille uniquely (up to isomorphism) determined by $\overline{\rho}$, $A_{\cris}$ is Fontaine's ring of periods for integral crystalline representations, and   $\Hom_{\Fil^{\cdot},\varphi_{\cdot}}$ means that we consider the morphisms preserving the filtrations and commuting with $\varphi$. Explicitly, $M$ can be described as follows
$$M=M^{0}\times \cdots\times M^{f-1}, \quad \mathrm{with}\ M^{j}=\F e^{j}\oplus \F f^{j}$$
together with the filtration given by
\[
\begin{array}{llll}
\left\{\begin{array}{llllllll}
\Fil^iM^{j}&=&M^{j} &&\mathrm{if} &i\leq 1\\
\Fil^iM^j&=& \F f^j &&\mathrm{if}& 2\leq i\leq c_{f-j}\\
\Fil^iM^j&=&0&&\mathrm{if}&i\geq  c_{f-j}+1
\end{array}\right.
\end{array}
\]
and
\[
\begin{array}{llll}
\left\{\begin{array}{llllll}
\varphi(e^j)&=&e^{j+1}\\
\varphi_{c_{f-j}}(f^j)&=& f^{j+1}+a_{j-1}e^{j+1}
\end{array}\right.
\end{array}
\]
for $j\neq 1$ and
\[
\begin{array}{llll}
\left\{\begin{array}{llllll}
\varphi(e^1)&=& \a e^{2}\\
\varphi_{c_{f-1}}(f^1)&=& \a'( f^{2}+ a_{0}e^{2} )
\end{array}\right.
\end{array}
\]
where $a_j\in \F $ and $\a,\a'\in \F^{\times}.$ Set
\begin{equation}
\label{equation-J_rho}
S_{\brho} \defn \{\omega^{(i)}~|~a_{f-1-i}=0\} \subset S
\end{equation}
which depends only on $\overline{\rho}$. One checks directly
\begin{equation}\label{Serre-weights-relation}
S_{\brho}=\{\omega^{(i)}~|~i\in J_{\brho}\},
\end{equation}
where $J_{\brho}$ is the (proper) subset of $\cS=\Z/f\Z$ as in \S\ref{section-BP} (cf. \cite[Eq. (17)]{Br14}). Recall that $\Serre$ denotes the set of Serre weights associated to $\brho$ (see \S\ref{section-BP}). Then $\Serre = \{\s_J \defn F(\mathfrak{t}_{\mu_{\brho}}(\o_J)) ~|~J\subseteq S_{\brho}\}$ by \cite[Prop.~3.2]{Le} (where the set $\Serre$ is denoted by $W(\brho)$).

\subsection{Potentially Barsotti-Tate deformation rings}\label{section--pBT}

Let $\brho: G_L\to \GL_2(\F)$ be a strongly generic reducible nonsplit representation as above. Let $\cM = \prod_{i}\F(\!(v)\!)\mathfrak{e}^i\oplus \F(\!(v)\!)\mathfrak{f}^i$ denote the \'etale $\varphi$-module given by
\begin{align*}
i\neq 0:~~&\left\{ \begin{array}{ll}
\varphi_{\cM}(\mathfrak{e}^{i-1}) & = v^{c_{f-i}}(\mathfrak{e}^i+a_{i-1}\mathfrak{f}^i)\\
\varphi_{\cM}(\mathfrak{f}^{i-1}) & = v\mathfrak{f}^i
\end{array}\right.\\
i=0 :~~&\left\{ \begin{array}{ll}
\varphi_{\cM}(\mathfrak{e}^{f-1}) & = \alpha v^{c_{0}}(\mathfrak{e}^0+ a_{f-1}\mathfrak{f}^0)\\
\varphi_{\cM}(\mathfrak{f}^{f-1}) & =\alpha' v\mathfrak{f}^0.
\end{array}\right.
\end{align*}
\cite[Prop.~3.6]{Le} shows that $\mathbb{V}^*(\cM)\cong \brho|_{G_{L_{\infty}}},$ where $\mathbb{V}^* : \cM \mapsto (\cM\otimes (\cO_{\mathcal{E}^{\rm un},L})^{\varphi=1})^{\vee}$ is the anti-equivalent functor (defined by Fontaine) from the category of \'etale $\varphi$-modules over $\F(\!(v)\!)$ to the category of representations of $G_{L_{\infty}}$ over $\F$. Let $\cN$ denote the rank one \'etale $\varphi$-submodule $\prod_{j=0}^{f-1}\F(\!(v)\!) \frak{f}^j$ of $\cM.$ Let $\bar{\frak{e}}^i$ be the image of $\frak{e}^i$ in $\cM/\cN.$ Then $\{\bar{\frak{e}}^i\}_{i\in \Z/f\Z}$ forms a basis of $\cM/\cN$ over $\F(\!(v)\!).$ We have $\mathbb{V}^*(\cN)\cong \chi_2|_{G_{L_{\infty}}}$ and $\mathbb{V}^*(\cM/\cN)\cong \chi_1|_{G_{L_{\infty}}}.$

Let $\mathrm{Def}^{\Box}_{\brho}:\mathrm{Art}(\cO)\ra \mathrm{Sets}$ be the framed deformation functor ({\em \`a la} Kisin \cite{KisinAnnals}) which sends $A$ to the set of representations $\rho_A : G_L \to \GL_2(A) $ lifting $\brho.$ Then $\mathrm{Def}^{\Box}_{\brho}$ is pro-representable by a complete noetherian local $\cO$-algebra $R^{\Box}_{\brho}.$ If $\psi:G_L\to \cO^{\times}$ is a continuous character lifting $\det\brho$, let $R_{\brho}^{\Box,\psi}$ be the reduced $\varpi$-torsion free quotient ring of $R^{\Box}_{\brho}$ parametrizing framed deformations of $\brho$ with determinant $\psi.$ If $\tau$ is a tame inertial type and $\l= (a_{\kappa},b_{\kappa})_{\kappa\in \Hom(L,E)}$, where $a_{\kappa} > b_{\kappa}$ are integers, let $R^{\Box,\tau,\l}_{\brho}$ (resp. $R^{\Box, \psi,\tau,\l}_{\brho}$) be the quotient ring of $R_{\brho}^{\Box}$ (resp. $R_{\brho}^{\Box,\psi}$) which parametrizes framed potentially crystalline deformations of $\brho$ of inertial type $\tau$ and Hodge-Tate weights $(a_{\kappa},b_{\kappa})$ for the embedding $\kappa.$ If $\tau = \ide$ is trivial, we will write $R^{\Box,{\rm cris},\l}_{\brho}$ (resp. $R^{\Box, \psi,{\rm cris},\l}_{\brho}$) for $R^{\Box,\ide,\l}_{\brho}$ (resp. $R^{\Box, \psi,\ide,\l}_{\brho}$), and call it framed crystalline deformation ring (with fixed determinant $\psi$) of Hodge-Tate weights $\l.$ If $\l = (a_{\kappa},b_{\kappa})_{\kappa\in \Hom(L,E)}$ with $(a_{\kappa}, b_{\kappa}) = (1,0)$ for all $\kappa\in \Hom(L,E),$ we will abbreviate $R^{\Box,\tau}_{\brho}$ (resp. $R^{\Box, \psi,\tau}_{\brho}$) for $R^{\Box,\tau,(1,0)_{\kappa\in \Hom(L,E)}}_{\brho}$ (resp. $R^{\Box, \psi,\tau,(1,0)_{\kappa\in \Hom(L,E)}}_{\brho}$), and call it framed potentially Barsotti-Tate deformation ring (with fixed determinant $\psi$). If $T$ is a set of inertial types for $L,$ then we let $ R_{\brho}^{\Box, T}$ (resp. $R_{\brho}^{\Box, \psi,T}$) be the quotient of $R^{\Box}_{\brho}$ such that $\Spec R_{\brho}^{\Box, T}$ (resp. $ \Spec R_{\brho}^{\Box, \psi,T}$) is the Zariski closure of $\cup_{\tau\in T}\Spec R^{\Box, \tau}_{\brho}[1/p]$ (resp. $\cup_{\tau\in T}\Spec R^{\Box, \psi,\tau}_{\brho}[1/p]$) in $\Spec R^{\Box}_{\brho}.$

Let $J$ be a subset of $S_{\brho}$ and let $I$ be a subset of $S$ such that $I\cap \{\pm\omega^{(i)}\}$ has size at most one for all $i\in \Z/f\Z.$ \cite{Le} defines  a set $T_{J,I}$ which consists of inertial types $\tau$ such that  $\s(\tau),$ the irreducible finite dimensional $\GL_2(\cO_L)$-representation over $E$ associated to $\tau$ under the inertial local Langlands \cite{Henniart}, is of the form $R_s(\mu_{\brho}-s'\eta)$ (see \cite[Lem.~4.2]{Herzig} for the notation $R_s(\mu),$ $(s,\mu) \in (S_2)^f\times X^*(T)$) subject to the condition that $s,s'\in (S_2)^f$ are given by the following table:

\begin{table}[h]
\begin{tabu}{|c|[1pt]c|c|}
  \hline
   $s_i,s'_i$ & $\omega^{(i)}\notin J$ & $\omega^{(i)}\in J$    \\
   \tabucline[1pt]{-}
  $\{\pm \omega^{(i)}  \} \cap I = \emptyset$ & $s_i=s'_i$  & $s'_i\neq {\rm id}$ \\
   \hline
  $\omega^{(i)} \in I $ & $s_i=s'_i={\rm id}$  & $s_i=s'_i\neq {\rm id}$ \\
   \hline
  $-\omega^{(i)} \in I$ & $s_i=s'_i\neq {\rm id}$  & $s_i={\rm id},~s'_i\neq {\rm id}$ \\
    \hline
\end{tabu}
\end{table}

\cite[Lem.~3.5]{Le} shows under the inertial local Langlands \cite{Henniart}, $T_{J,I}$ corresponds to  the set of Deligne-Lusztig representations $T_{\s_J,w_J(I)}$ defined in {\em loc. cit.}. In particular, if $I$ is the empty set, then
\begin{equation}\label{equ--inj-env}
{\rm Proj}_{\cO[\G]} (\s_J) \otimes_{\cO} E \cong \oplus_{\tau\in T_{J,\emptyset}}\s(\tau),
\end{equation}
where ${\rm Proj}_{\cO[\G]} (\s_J)$ denotes a projective envelope of $\s_{J}$ in the category of $\cO[\G]$-modules.
Recall Theorem 3.6 of \cite{Le} (in the special case for reducible nonsplit strongly generic $\brho$).

\begin{theorem}[\cite{Le}]\label{thm--Le}

There is an isomorphism from $R^{\Box, T_{J,I}}_{\brho}$ to a formal power series ring of relative dimension $4$ over $\cO[\![(X_i,Y_i)_{i\in \Z/f\Z}]\!]/ (g_i(J,I))_{i\in \Z/f\Z},$ where $g_i(J,I)$ is given by the following table:

\begin{table}[h]
\begin{tabu}{|c|[1pt]c|c|c|}
  \hline
   $g_i(J,I)$ &  $\omega^{(f-1-i)}\notin S_{\brho}$ & $\omega^{(f-1-i)}\in S_{\brho}\backslash J$ & $\omega^{(f-1-i)}\in J$   \\
   \tabucline[1pt]{-}
  $\{\pm \omega^{(f-1- i)}  \} \cap I = \emptyset$ & $Y_i(Y_i-p)$  & $Y_i(X_iY_i -p)$ & $X_i(X_iY_i-p)$ \\
   \hline
  $\omega^{(f-1-i)} \in I $ & $Y_i$  & $Y_i$ & $X_iY_i -p$ \\
   \hline
  $-\omega^{(f-1-i)} \in I$ & $Y_i-p$  & $X_i Y_i - p$ & $X_i$ \\
    \hline
\end{tabu}
\end{table}
If $I\subseteq I',$ then $g_i(J,I')| g_i(J,I)$ for all $i$ and $R^{\Box, T_{J,I'}}_{\brho}$ is the quotient of $R^{\Box, T_{J,I}}_{\brho}$ by the ideal $(g_i(J,I'))_i.$ Analogous results hold for $R^{\Box, \psi, T_{J,I}}_{\brho}$ provided $\psi$ is chosen so that $R^{\Box, \psi, T_{J,I}}_{\brho}$ is nonzero for some $I.$
\end{theorem}

We record the following result of independent interest. It will not be used for the rest of the paper.

\begin{corollary}
We have a natural isomorphism
\[
\Hom_{\cO-{\rm alg}}(R^{\Box, T_{J,\emptyset}}_{\brho},\F[\epsilon]/\epsilon^2)\cong\Hom_{\cO-{\rm alg}}(R^{\Box, T_{\emptyset,\emptyset}}_{\brho},\F[\epsilon]/\epsilon^2),~ \forall J\subseteq S_{\brho}.
\]

\end{corollary}

\begin{proof}
By Theorem \ref{thm--Le}, $R^{\Box, T_{J,\emptyset}}_{\brho}$ is a formal power series ring over \[\cO[\![X_i,Y_i]\!]_{i\in \Z/f\Z}/(g_i(J,\emptyset))_{i\in \Z/f\Z},\] and $R^{\Box, T_{J,J}}_{\brho}$ is isomorphic to the quotient $R^{\Box, T_{J,\emptyset}}_{\brho}/(g_i(J,J))_{i\in \Z/f\Z}.$ From the description of the ideals $(g_i(J,\emptyset))_i$ and $(g_i(J,J))_i$, one checks directly  that \begin{multline*}\dim_{\F}\Hom_{\cO-{\rm alg}}(R^{\Box, T_{J,\emptyset}}_{\brho},\F[\epsilon]/\epsilon^2)
=\dim_{\F}\Hom_{\cO-{\rm alg}}(R^{\Box, T_{J,J}}_{\brho},\F[\epsilon]/\epsilon^2)=3+(2f+1).\end{multline*} Hence the injection
\[
\Hom_{\cO{\textrm -}{\rm alg}}(R^{\Box, T_{J,J}}_{\brho},\F[\epsilon]/\epsilon^2)\to \Hom_{\cO-{\rm alg}}(R^{\Box, T_{J,\emptyset}}_{\brho},\F[\epsilon]/\epsilon^2)
\]
induced by the projection $R^{\Box, T_{J,\emptyset}}_{\brho} \onto R^{\Box, T_{J,J}}_{\brho}$ is an isomorphism.

Let $-J$ denote the set $\{-\omega^{(i)}~|~\omega^{(i)}\in J\}.$ Then similarly the projection
$
R_{\brho}^{\Box, T_{\emptyset,\emptyset}} \to R_{\brho}^{\Box, T_{\emptyset,-J}}
$
induces an isomorphism
\[
\Hom_{\cO-{\rm alg}}( R_{\brho}^{\Box, T_{\emptyset,-J}},\F[\epsilon]/\epsilon^2)\cong \Hom_{\cO-{\rm alg}}(R_{\brho}^{\Box, T_{\emptyset,\emptyset}},\F[\epsilon]/\epsilon^2).
\]
Finally, one checks directly $ T_{\emptyset,-J} = T_{J,J} = T_{\emptyset,\emptyset}\cap T_{J,\emptyset} $ by noticing $J \subseteq \{\o^{(i)}~|~ i\in \Z/f\Z\}.$
\end{proof}

\begin{proposition}\label{prop-tangent}
Let $t\in\Hom_{\cO-{\rm alg}}(R^{\Box, T_{\emptyset,\emptyset}}_{\brho},\F[\epsilon]/\epsilon^2).$ Then $t$ factors through $R_{\brho}^{\Box, \rm red}$ if and only if $t(Y_i)=0$ for all $i\in \Z/f\Z.$ In particular, we have
\[
\dim_{\F} \left( \Hom_{\cO-{\rm alg}}(R^{\Box, T_{\emptyset,\emptyset}}_{\brho},\F[\epsilon]/\epsilon^2) \cap \Hom_{\cO-{\rm alg}}(R^{\Box, {\rm red}}_{\brho},\F[\epsilon]/\epsilon^2) \right) = 4+f,
\]
where the intersection is taken inside $\Hom_{\cO-{\rm alg}}(R^{\Box}_{\brho},\F[\epsilon]/\epsilon^2).$
An analogous equality holds for fixed determinant deformation rings, i.e.
\[
\dim_{\F} \left( \Hom_{\cO-{\rm alg}}(R^{\Box, \psi, T_{\emptyset,\emptyset}}_{\brho},\F[\epsilon]/\epsilon^2) \cap \Hom_{\cO-{\rm alg}}(R^{\Box, \psi, {\rm red}}_{\brho},\F[\epsilon]/\epsilon^2) \right) = 3+f.
\]
\end{proposition}

\begin{proof}
The fixed determinant case follows from the unfixed determinant case by \cite[Lem.~4.3.1]{EG}. For the unfixed determinant case, we first recall the construction of $R_{\brho}^{\Box, T_{\emptyset,\emptyset}}$ in \cite{Le}.
Let $R = \cO[\![(X_i,Y_i)_{0\leq i\leq f-1},X_{\alpha},X_{\alpha'}]\!]/(g_i(\emptyset,\emptyset))_i.$
The universal \'etale $\varphi$-module $\cM_R$ over $R$ admits the following description, see \cite[Thm.~3.6]{Le},
\begin{align*}
\omega^{(f-i)}\notin S_{\brho}:~~&\left\{ \begin{array}{ll}
\varphi(\mathfrak{e}^{i-1}) & = v^{c_{f-i}-1}(v+p-Y_{i-1})\mathfrak{e}^i+(X_{i-1}+[a_{i-1}])v^{c_{f-i}}\mathfrak{f}^i\\
\varphi(\mathfrak{f}^{i-1}) & = -Y_{i-1}(X_{i-1}+[a_{i-1}])^{-1}\mathfrak{e}^i+v\mathfrak{f}^i
\end{array}\right.\\
\omega^{(f-i)}\in S_{\brho}:~~&\left\{ \begin{array}{ll}
\varphi(\mathfrak{e}^{i-1}) & = v^{c_{f-i}-1}(v+p-X_{i-1}Y_{i-1})\mathfrak{e}^i+X_{i-1}v^{c_{f-i}}\mathfrak{f}^i\\
\varphi(\mathfrak{f}^{i-1}) & = -Y_{i-1}\mathfrak{e}^i+ v\mathfrak{f}^i
\end{array}\right.
\end{align*}
for $i\neq 0,$ and for $i=0$ one needs to modify by the multiplication by the matrix
\[D(\alpha,\alpha')=\left(
                   \begin{array}{cc}
                     X_{\alpha}+[\alpha] & 0 \\
                     0 &  X_{\alpha'}+[\alpha']\\
                   \end{array}
                 \right).
                 \]
Let $R^\Box$ be the ring which represents the functor sending a complete noetherian local  $\cO$-algebra $A$ to the set of isomorphism classes of $\{f:R\to A, b_A\}$ where $b_A$ is a basis for the free rank two $A$-module $\mathbb{V}^*(f^*(\cM_R))$ whose reduction modulo $\frak{m}_A$ gives $\brho.$ Then $R^\Box$ is formally smooth over $R$ of relative dimension $4$. The universal lifting ring $R_{\brho}^{\Box, T_{\emptyset,\emptyset}}$ is then a quotient of $R^\Box$ by a $\widehat{\mathbb{G}}_m^2$ action, and is a power series ring over $R$ of relative dimension $2$,  see \cite[Thm.~3.6]{Le}.

We have the following commutative diagram
\[
\xymatrix{  \Hom_{\cO-{\rm alg}}(R^{\Box, T_{\emptyset,\emptyset}}_{\brho},\F[\epsilon]/\epsilon^2)\ar@{^{(}->}[r]^{\ \ \ \ \ A}\ar@{^{(}->}[d]^{C}& \Ext^1_{G_K}(\brho,\brho) \ar@{^{(}->}[r]^{B} &\Ext^1_{G_{K_\infty}}(\brho,\brho)\ar[d]_{D}^{\cong}  \\
\Hom_{\cO-{\rm alg}}(R^{\Box},\F[\epsilon]/\epsilon^2)\ar[r]\ar@{->>}[d]&\Ext^1 (\cM_{R^{\Box}},\cM_{R^{\Box}})\ar[r]&\Ext^1(\cM,\cM) \ar@{=}[d]\\
\Hom_{\cO-{\rm alg}}(R ,\F[\epsilon]/\epsilon^2)\ar[r]&\Ext^1 (\cM_{R },\cM_{R })\ar[r]&\Ext^1(\cM,\cM) }
\]
where the extension of $\cM_{R^{\Box}}$ (resp. $\cM$) by $\cM_{R^{\Box}}$ (resp. $\cM$) is taken in the corresponding category of \'etale $\varphi$-modules.

The map $A$ is given by the deformation theory, and it is injective. The map $C$ is injective by \cite[Lem.~2.2.7]{CDM}. The restriction map $B$ is injective by \cite[Lem.~2.2.9]{CDM} if $\brho$ is not isomorphic to $\brho(1)$ (which holds when $\brho$ is strongly generic). Since the category of \'etale $\varphi$-modules over $\F(\!(v)\!)$ is anti-equivalent to the category of continuous representations of $G_{L_{\infty}}$ over $\F,$ $D$ is an isomorphism.

For any $t\in\Hom_{\cO-{\rm alg}}(R^{\Box, T_{\emptyset,\emptyset}}_{\brho},\F[\epsilon]/\epsilon^2),$ we let $\cM_t$ denote the  image of $t$ in $\Ext^1(\cM,\cM)$  under the composition $D\circ B\circ A$.
Then $\cM_t$ over $R$ admits the following description
\begin{align*}
\omega^{(f-i)}\notin S_{\brho}:~&\left\{ \begin{array}{ll}
\varphi(\mathfrak{e}^{i-1}) & = v^{c_{f-i}}\frak{e}^i-t(Y_{i-1})v^{c_{f-i}-1}\mathfrak{e}^i+(t(X_{i-1})+a_{i-1})v^{c_{f-i}}\mathfrak{f}^i\\
\varphi(\mathfrak{f}^{i-1}) & = - a_{i-1}^{-1} t(Y_{i-1})\mathfrak{e}^i+v\mathfrak{f}^i
\end{array}\right.\\
\omega^{(f-i)}\in S_{\brho}:~&\left\{ \begin{array}{ll}
\varphi(\mathfrak{e}^{i-1}) & = v^{c_{f-i}}\mathfrak{e}^i + t(X_{i-1})v^{c_{f-i}}\mathfrak{f}^i\\
\varphi(\mathfrak{f}^{i-1}) & = - t(Y_{i-1})\mathfrak{e}^i+ v\mathfrak{f}^i
\end{array}\right. 
\end{align*}
for $i\neq 0,$ and with the usual modification by $\a,\a'$ when $i=0.$

Let $\cN_t = \prod_{i\in \Z/f\Z} (\F[\epsilon]/\epsilon^2)(\!(v)\!) \frak{f}'^i$ be a rank one \'etale $\varphi$-submodule of $\cM_t.$ Since $\brho$ is reducible nonsplit, $\cN\subset \cM$ is the unique \'etale $\varphi$-submodule of $\cM.$ Then $\cN_t\pmod \epsilon = \cN .$ Up to an element in $\F^{\times},$ we may assume $\frak{f}'^i = \frak{f}^i +\epsilon(x_i\frak{e}^i +y_i\frak{f}^i )$ with $x_i,y_i\in \F$ for all $i.$ If $\omega^{(f-i)}\in S_{\brho},$ then
\begin{align*}
\varphi(\frak{f}'^{i-1}) & = \varphi(\frak{f}^{i-1}) + \epsilon ( \varphi(x_{i-1}) \varphi(\frak{e}^{i-1}) + \varphi(y_{i-1}) \varphi(\frak{f}^{i-1}) )
\\
& = - t(Y_{i-1})\mathfrak{e}^i+ v\mathfrak{f}^i\\& \quad + \epsilon \big( \varphi(x_{i-1}) (v^{c_{f-i}}\frak{e}^i + t(X_{i-1})v^{c_{f-i}}\mathfrak{f}^i) +  \varphi(y_{i-1}) (-t(Y_{i-1})\mathfrak{e}^i+v\mathfrak{f}^i) \big)
\\
& = a \frak{f}'^{i} = a (\frak{f}^i +\epsilon(x_i\frak{e}^i +y_i\frak{f}^i )),
\end{align*}
for some $a \in (\F[\epsilon] / \epsilon^2)(\!(v)\!)$ and $a \equiv v \pmod \epsilon.$ Since $t(Y_{i-1})\in \F \epsilon,$ by comparing the $\epsilon\frak{e}^i$ terms and noticing $c_{f-i} \geq 4$, we see $t(Y_{i-1})=0.$ The case $\omega^{(f-i)}\notin S_{\brho}$ can be done in the same way.
\end{proof}

\subsection{Crystalline deformation rings}

From now on we only consider fixed determinant deformation rings. Analogous results hold for deformation rings without the determinant condition. Let $\s$ be a Serre weight given by
\[
\sigma \cong \bigotimes_{\kappa\in \Hom(\F_q, \F)} ( \Sym^{r_{\kappa}} \F_q^2 \otimes \det{}^{t_{\kappa}} )\otimes_{\F_q, \kappa} \F ,
\]
where $0\leq r_{\kappa},t_{\kappa} \leq p-1$ and not all $t_{\kappa}$ are equal to $p-1.$
We identify $\Hom(L,E)$ with $\Hom(\F_q,\F)$ by the natural reduction map. Let $R_{\brho}^{\Box, \psi, \mathrm{cris}, \s}$ denote $R_{\brho}^{\Box, \psi, \mathrm{cris}, \l }$ for $\l= (r_{\kappa}+t_{\kappa}+1, t_{\kappa})_{\kappa\in \Hom(L,E)}.$ If $\s\in \Serre$, $R_{\brho}^{\Box, \psi, \mathrm{cris}, \s}$ is a regular local ring of relative dimension $f+3$ over $\cO$. 

\begin{proposition}\label{prop::red--crys}
The universal deformation of $\brho $ over $R_{\brho}^{\Box,\psi,\mathrm{cris},\s_{\emptyset}} \otimes_{\cO } \F$ is reducible. In particular, $R_{\brho}^{\Box,\psi,\mathrm{cris},\s_{\emptyset}} \otimes_{\cO} \F$ is a quotient of $R^{\Box, \psi, {\rm red}}_{\brho} \otimes_{\cO} \F.$
\end{proposition}

\begin{proof}
By \cite[Thm.~7.2.1]{EGS} and Theorem \ref{thm--Le} (which is \cite[Thm.~3.6]{Le}), $R_{\brho}^{\Box,\psi,\mathrm{cris},\s_{\emptyset}} \otimes_{\cO } \F$ is a quotient of $R_{\brho}^{\Box, T_{\emptyset,\emptyset}}\otimes_{\cO} \F$ by the ideal $(Y_i)_{i\in \Z/f\Z}$. Hence by the form of the universal \'etale $\varphi$-module recalled in the proof of Proposition \ref{prop-tangent}, the universal \'etale $\varphi$-module over $R_{\brho}^{\Box,\psi,\mathrm{cris},\s_{\emptyset}} \otimes_{\cO } \F$ has the following form
\begin{align*}
\omega^{(f-i)}\notin S_{\brho}:~~&\left\{ \begin{array}{ll}
\varphi(\mathfrak{e}^{i-1}) & = v^{c_{f-i}}\mathfrak{e}^i+(X_{i-1}+a_{i-1})v^{c_{f-i}}\mathfrak{f}^i\\
\varphi(\mathfrak{f}^{i-1}) & =  v\mathfrak{f}^i
\end{array}\right.\\
\omega^{(f-i)}\in S_{\brho}:~~&\left\{ \begin{array}{ll}
\varphi(\mathfrak{e}^{i-1}) & = v^{c_{f-i}} \mathfrak{e}^i+X_{i-1}v^{c_{f-i}}\mathfrak{f}^i\\
\varphi(\mathfrak{f}^{i-1}) & =  v\mathfrak{f}^i
\end{array}\right.
\end{align*}
for $i\neq 0,$ and with the usual modification for $i=0.$ The $(\frak{f}^i)_{i\in \Z/f\Z}$ clearly gives a rank one \'etale $\varphi$-submodule.
\end{proof}

\begin{corollary}\label{cor::intersect-tang-space}
We have
\begin{multline*}
 \Hom_{\cO-{\rm alg}}(R^{\Box, \psi, T_{\emptyset,\emptyset}}_{\brho},\F[\epsilon]/\epsilon^2) \cap \Hom_{\cO-{\rm alg}}(R^{\Box, \psi, {\rm red}}_{\brho},\F[\epsilon]/\epsilon^2)  =  \Hom_{\cO-{\rm alg}}(R_{\brho}^{\Box,\psi,\mathrm{cris},\s_{\emptyset}},\F[\epsilon]/\epsilon^2),
\end{multline*}
where the intersection is taken inside $ \Hom_{\cO-{\rm alg}}(R^{\Box, \psi}_{\brho},\F[\epsilon]/\epsilon^2).$
\end{corollary}

\begin{proof}
It follows from \cite[Thm.~7.2.1]{EGS} and Theorem \ref{thm--Le} (which is \cite[Thm.~3.6]{Le}) that $R_{\brho}^{\Box,\psi,\mathrm{cris},\s_{\emptyset}}\otimes_{\cO} \F$ is a quotient of  $R^{\Box, \psi, T_{\emptyset,\emptyset}}_{\brho}\otimes_{\cO} \F$ by the ideal  $(Y_i)_{i\in \Z/f\Z}.$ Then the equality follows from Proposition \ref{prop-tangent} and Proposition \ref{prop::red--crys}.
\end{proof}

If $A$ is a regular local ring, recall that a \emph{regular system of parameters} of   $A$ is defined as any system of parameters of $A$ which generates the maximal ideal of $A$, see \cite[\S 14]{Mat}.

\begin{lemma}\label{lemma:Uj}
Let $R=\F[\![(X_j,Y_j)_{j\in \cJ}]\!]/(X_jY_j)_{j\in \cJ}$ for some finite set $\cJ$. Then $R$ is Cohen-Macaulay of dimension $|\cJ|$ and there exists a regular sequence  $\{U_j,\ j\in \cJ\}$ in $R$  such that for any minimal prime ideal $\p$ of $R$, $\{U_j\ \mathrm{mod}\ \p,\ j\in \cJ\}$  forms a regular system of parameters of $R/\p$.
\end{lemma}
\begin{proof}
It suffices to take $U_j=X_j+Y_j$ for all $j\in \cJ$. Note that for any minimal prime $\p$ of $R$, $R/\p$ is a regular local ring.
\end{proof}

\begin{proposition}\label{prop:Tj}
There is a sequence $(T_1,\dots,T_{f+3})$ in $R_{\brho}^{\Box,\psi}$, which is part of a regular system of parameters and such that for any $\sigma\in\mathscr{D}(\brho)$, the sequence $\{\varpi, p_{\s} (T_j), 1\leq j\leq f+3\}$ forms a regular system of parameters of   $R_{\brho}^{\Box, \psi, \mathrm{cris}, \s}.$ Here, $p_{\sigma}$ denotes the natural quotient map $R_{\brho}^{\Box,\psi}\ra R_{\brho}^{\Box,\psi,\mathrm{cris},\s}$.

\end{proposition}
\begin{proof}
We choose a tame inertial type $\tau$ so that $\Serre \subseteq \JH( \overline{\s(\tau)}),$ this is possible by \cite{Diamond} (see \cite[Prop.~3.5.2]{EGS}) and the genericity of $\brho.$ By \cite[\S 4]{Br14} and \cite[Thm.~7.2.1(1)]{EGS} there are subsets $J^{\min} \subseteq J^{\max} \subseteq \mathcal{S}$   such that $\Serre = \{ \overline{\s}(\tau)_J ~|~ J^{\min} \subseteq J \subseteq J^{\max}\}$ where $\overline{\s}(\tau)_J $ is as in \cite[\S 3.2, 3.3]{EGS}). By \cite[Thm.~7.2.1(2)]{EGS}, $R^{\Box, \psi, \tau}_{\brho} \otimes_{\cO} \F$ has  dimension $f+3$ and is a formal power series ring over
\[
  \F[\![(X_j, Y_j)_{j \in \cJ}]\!] / (X_j Y_j )_{j \in \cJ}
\]
where $\cJ = J
^{\rm max}\backslash J^{\rm min}$. Moreover, for  $\sigma = \overline{\s}(\tau)_J \in \mathscr{D}(\brho)$,  $R^{\Box,\psi,\mathrm{cris},\s}_{\brho}\otimes_{\cO}\F$ can be obtained as   the quotient of $R_{\brho}^{\Box,\psi,\tau}\otimes_{\cO}\F$ by some minimal prime ideal. By Lemma \ref{lemma:Uj} and taking into account of the formal variables, we may find $\{U_j, 1\leq j\leq f+3\}$ in $R_{\brho}^{\Box,\psi, \tau}\otimes_{\cO}\F$ such that their images in $R^{\Box,\psi,\mathrm{cris},\s}_{\brho}\otimes_{\cO}\F$ form a regular system of parameters for any $\sigma \in \mathscr{D}(\brho)$. Choosing a lift $T_j\in R_{\brho}^{\Box,\psi}$ of $U_j$ for each $j$, it is easy to see that   $\{ T_j, 1\leq j\leq f+3\}$ satisfies the required properties.
\end{proof}

\section{$P$-ordinary automorphic representations, Local-global compatibility}
\label{Section::automorphic}

We recall some results of $P$-ordinary automorphic representations and the relevant local-global compatibility proved in \cite[\S 6.3 and \S 7.1]{Breuil-Ding}.

Let $F$ be a totally real extension of $\QM$ in which $p$ is unramified, and $\OC_F$ be its ring of integers. Let $S_p$ denote the set of places of $F$ dividing $p$ and $S_{\infty}$  the set of infinite places of $F.$ For any place $w$ of $F,$ let $F_w$ denote the completion of $F$ at $w$ with ring of integers $\OC_{F_w},$ uniformiser $\varpi_w$ and residue field $k_{F_w}.$ The cardinality of $k_{F_w}$ is denoted by $q_w.$  Let $\AM_{F,f}$ denote the ring of finite ad\`eles of $F.$ If $S$ is a finite set of finite places of $F,$ let $\AM^S_{F,f}$ denote the finite ad\`eles outside $S.$

Let $D$ be a quaternion algebra with center $F.$ Let $S_D$ be the set of ramified places of $D.$ Assume $S_D$ is disjoint from $ S_p .$ Let $(\cO_{D})_w$ denote $\cO_D\otimes_{\cO_F} \cO_{F_w}.$ For $w\notin S_D\cup S_{\infty},$ we identify $ (D\otimes_F F_w)^{\times}$ with $ \GL_2(F_w)$ so that $(\cO_{D})_w^{\times}$ is identified with $\GL_2(\cO_{F_w}).$ In the following, we assume $D$ is either {\em definite}, i.e. $S_{\infty}\subseteq S_D,$ or {\em indefinite}, i.e. $|S_{\infty} \backslash S_D|=1.$ In the definite case, we furthermore assume $(F,D) \neq (\Q, \GL_2)$ (our main result is already known in the case $(F,D) = (\Q, \GL_2)$ by \cite{Em3}).

We fix a place $v|p$ and denote by $L \defn F_{v}$ which is unramified of degree $f \defn [L:\Q_p]$ over $\Q_p.$

\subsection{$p$-adic completed cohomology}

Let $D$ be either definite or indefinite. Let $U$ be an open compact subgroup of $(D\otimes_F \AM_{F,f})^\times.$ If $D$ is definite, we denote by $Y^D_U$ the finite set $D^\times\backslash (D\otimes_F \AM_{F,f})^\times/U;$ If $D$ is indefinite, let $Y^D_U$ denote the quotient of $X^D_U$ by the action of the finite group $\A_{F,f}^{\times} / (F^{\times} (\A_{F,f}^{\times} \cap U))$, where $X_U^D$ is the associated projective Shimura curve as in \cite{BDJ}, see \cite[Rem.~8.1.2(iii)]{BHHMS}. Note that we will follow the convention in \cite{BDJ} which is different from the convention used in \cite{Breuil-Ding} and  \cite{BreuilDiamond} (see \cite[\S 3.1]{BreuilDiamond}).

Fix $U^p = \prod_{w\nmid p} U_w$ an open compact subgroup of $(D\otimes_F \A^p_{F,f})^{\times}.$ For an open compact subgroup $U_p \subset (D\otimes_{\Q} \Q_p)^\times $ and $i,s\in \N,$ let
\[
H^i (U^pU_p, \cO/\varpi^s) \defn   H^i_{\textrm{\'et}} (Y^D_{U^p U_p,\overline{F}} , \cO/\varpi^s ).
\]
If $D$ is definite, $H^i (U^pU_p, \cO/\varpi^s) = 0$ if $i\geq 1$, and $H^0 (U^pU_p, \cO/\varpi^s)$ can be identified with the set of functions $f:D^\times\backslash (D\otimes_F \AM_{F,f})^\times/U^p U_p\to \cO/\varpi^s.$ If $D$ is indefinite, $H^i (U^pU_p, \cO/\varpi^s) = 0$ if $i\geq 3.$ Set
\[
\wt{H}^i (U^{p}, \cO)\defn \plim_{s} \varinjlim_{U_{p}} H^i (U^{p}U_{p}, \cO/\varpi^s).
\]

Let $\psi:F^\times\backslash \AM_{F,f}^\times\to \cO^\times$ be a locally constant character.
For each place $w \in S_p \backslash \{v\}$, let $W_w$ be an irreducible algebraic representation of $\Res_{F_w/\Q_p} \GL_2$ over $E$ with central character $\psi^{-1}|_{F_w^{\times}}$.  Let  $W = \prod_{ w \in S_p \backslash \{v\} } W_w.$ Then $W$ is an irreducible algebraic representation of $(D\otimes_{\Q} \Q_p)^\times$ via
\[
(D\otimes_{\Q} \Q_p)^\times \onto \prod_{w \in S_p \backslash \{v\} }\GL_2(F_w) = \prod_ {w \in S_p \backslash \{v\}}( \Res_{F_w/\Q_p} \GL_2)(\Q_p).
\]
Let $\W$ be an $\cO$-lattice of $W$ stable under $U^{v}_p = \prod_{w|p,w\neq v}U_w.$ We denote by $U^{v}= U^p U_p^{v} \subset (D\otimes_F \A^{\{v\}}_{F,f})^{\times}.$ Then $\W$ admits an action of $U^v$ via the projection $U^v \onto U^{v}_p.$ We extend this action to $U^v (\A_{F,f}^{\times})$ by letting $\A_{F,f}^{\times}$ act by $\psi^{-1}.$ Assume $U_v$ is an open compact subgroup of $\GL_2(\cO_{L}) $ such that $\psi|_{U_{v}\cap \cO_{F_v}^\times}=1.$ Then $\W$ admits an action of $U(\A_{F,f}^{\times})$ by letting $U_v$ act trivially.
For $s\in \N,$ let $\mathcal{V}_{\W/\varpi^s}$ be the local system over $Y_{U^{v}U_{v}}^{D}$ associated to the algebraic representation $\W/\varpi^s,$ see \cite{EmertonInvent}. We define
\begin{align*}
&H^i (U^{v} U_{v}, \W/\varpi^s) \defn H^i(U^{v}U_{v}, \mathcal{V}_{\W/\varpi^s})
\\
&
H^i (U^{v}, \W)\defn \varinjlim_{U_v} \plim_{s} H^i (U^{v} U_{v}, \W/\varpi^s)
\\
&
\wt{H}^i (U^{v}, \W)\defn \plim_{s} \varinjlim_{U_{v}} H^i (U^{v}U_{v}, \W/\varpi^s).
\end{align*}
All these spaces carry compatible actions of the group $(D\otimes_{F} F_{v})^\times = \GL_2(L).$ If $D$ is definite, $\wt{H}^0 (U^{v}, \W)$ is identified with the space of {\em continuous} functions $f: D^\times\backslash (D\otimes_F \AM_{F,f})^\times \to \W$ such that $f(g u) = u^{-1} f(g)$, $\forall g\in (D\otimes_F \AM_{F,f})^\times$ and  $\forall u \in U (\A_{F,f}^{\times}),$ and it is the $\varpi$-adic completion of $H^0 (U^{v}, \W).$ $\wt{H}^0 (U^{v},\W)\otimes_{\cO} E$ is an admissible unitary Banach representation of $\GL_2(L)$ with the norm defined by the (complete) $\cO$-lattice $\wt{H}^0 (U^{v}, \W).$ If $D$ is indefinite and $i = 0,1,2$, $\wt{H}^i (U^{v}, \W)$ is the $\varpi$-adic completion of $H^i (U^{v}, \W),$ and is the gauge lattice for the admissible unitary Banach $\GL_2(L)$-representation $\wt{H}^i (U^{v},\W)\otimes_{\cO} E,$ see \cite{EmertonInvent}.
 
\subsection{$p$-adic automorphic forms}

We consider the space of algebraic automorphic forms over $D$ with the fixed central character $\psi$. Let $U^v, U_v , \W$ be as above. For $A \in \{ \W, \W/\varpi^s\}$ where $s\in \N,$ let
\begin{equation}
S^D_{\psi}(U^{v} U_{v}, A )\defn H^0(U^{v} U_{v}, A)
\end{equation}
if $D$ is definite; let
\begin{equation}
S^D_{\psi}(U^{v} U_{v}, A )\defn H^1(U^{v} U_{v}, A)
\end{equation}
if $D$ is indefinite. Set
\[
\wt{S}^D_{\psi}(U^{v}, \W) \defn  \plim_{s} \varinjlim_{U_{v}} S^D_{\psi}(U^{v} U_{v}, \W/\varpi^s).
\]
Then $ \wt{S}^D_{\psi}(U^{v}, \W)$  is  the $\varpi$-adic completion of $S^D_{\psi}(U^{v}, \W).$

Let $S$ be a set of places of $F$ containing all places in $S_{\infty} \cup S_D\cup S_p,$ all places where $\psi$ is ramified, and all places $w$ such that $U_w$ is not $(\cO_{D})_w^{\times}.$ Let $\TM^S \defn \cO[T_w,S_w^{\pm 1}~|~ w\notin S]$ be the commutative $\cO$-algebra generated by the formal variables $T_w,S_w$ and $S_w^{-1}$ for all $w\notin S.$ Then $\TM^S$ acts on $S^D_{\psi}(U^{v} U_{v}, A )$ for $A \in \{\W, \W /\varpi^s\}$ by letting $T_w$ act by the double coset operator corresponding to
\[
\GL_2(\cO_{F_w})  \left(
                   \begin{array}{cc}
                     \varpi_w & 0 \\
                     0 & 1 \\
                   \end{array}
                 \right)
\GL_2(\cO_{F_w})
\]
and $S_w$ act by the double coset operator corresponding to
\[
\GL_2(\cO_{F_w}) \left(
                   \begin{array}{cc}
                     \varpi_w & 0 \\
                     0 &  \varpi_w\\
                   \end{array}
                 \right)\GL_2(\cO_{F_w}).
\]
This induces an action of $\TM^S$ on $S^D_{\psi}(U^{v}, A)$ and on $\wt{S}^D_{\psi}(U^{v}, A).$ The action of $\TM^S$ commutes with the action of $ \GL_2(L),$ and we have $\GL_2(L) \times \TM^S$-equivariant isomorphisms
\[
\wt{S}^D_{\psi}(U^{v}, \W)/\varpi^s \cong S^D_{\psi}(U^{v}, \W/\varpi^s) \cong S^D_{\psi}(U^{v}, \W)/\varpi^s.
\]
Let $\TM^S(U^{v} U_{v}, A)$ denote the image of the homomorphism
\[
\TM^S \to \End_A(S^D_{\psi}(U^{v} U_{v}, A)),~~ A\in \{ \W, \W/\varpi^s\}.
\]
Then
\[
\TM^S(U^{v} U_{v}, \W) \cong \plim_s \TM^S(U^{v} U_{v}, \W/\varpi^s).
\]
If $U_{v}' \subseteq U_{v}$ is an inclusion of open compact subgroups of $\GL_2(\cO_L),$ we have a natural surjection
$
\TM^S(U^{v} U_{v}', \W) \onto \TM^S(U^{v} U_{v}, \W).
$ 
We then define
\[
\wt{\TM}^S (U^{v}) \defn \plim_{U_{v}} \TM^S(U^{v} U_{v}, \W).
\]
The $\cO$-algebra $\wt{\TM}^S (U^{v})$ is reduced, and $\wt{S}^D_{\psi}(U^{v} , \W)$ is a faithful $\wt{\TM}^S (U^{v})$-module (the definite case follows from   \cite[Lem.~6.3]{Breuil-Ding}; the indefinite case is similar).

Let $\overline{r}:G_{F}\to \GL_2(\FM)$ be a two dimensional continuous totally odd Galois representation. Assume $\overline{r}$ is unramified outside $S.$ We associate to $\overline{r}$ a maximal ideal $\mathfrak{m} = \frak{m}_{\overline{r}}$ of $\TM^S$ of residue field $\F,$ generated by $T_w - S_w \tr(\overline{r}(\Frob_w))$ and $q_w - S_w \det(\overline{r}(\Frob_w))$ for $w\notin S.$ We say $\frak{m}$ is {\em non-Eisenstein} if $\overline{r}$ is absolutely irreducible.

Assume $\overline{r}$ is absolutely irreducible. We say that $\overline{r}$ (and $\frak{m}$) is {\em $(U^{v},\W)$-automorphic} (with respect to $D$) if $S^D_{\psi}(U^{v} U_{v } , \W )_{\frak{m}}$ is nonzero for some $U_{v}$ (equivalently $S^D_{\psi}(U^{v} U_{v } , \W )_{\frak{m}}[\frak{m}] \neq 0$). In this case, $\psi \circ \Art_F^{-1}$ is necessarily a lift of $\o^{-1} (\det\overline{r})^{-1}.$ By abuse of notation we also write $\psi$ for $\psi\circ \Art_F^{-1}.$ We say $\overline{r}$ (and $\frak{m}$) is {\em automorphic} if it is $(U^{v},\W)$-automorphic for some $(U^{v},\W).$ If this is the case then $\frak{m}$ gives rise to a maximal ideal of $\wt{\TM}^S(U^{v})$ (via the projection $\wt{\TM}^S(U^{v}) \onto \TM^S(U^{v} U_{v}, \W)$) which is also denoted by $\frak{m}.$

\begin{lemma}
Let $\frak{m}$ be  non-Eisenstein and $(U^{v},\W)$-automorphic. Then

(i) the $\cO$-algebra $\wt{\TM}^S (U^{v})_{\frak{m}}$ is reduced.

(ii) $\wt{S}^D_{\psi}(U^{v} , \W)_{\frak{m}}$ is a faithful $\wt{\TM}^S (U^{v})_{\frak{m}}$-module.
\end{lemma}
\begin{proof}
The definite case follows from \cite[Lem.~6.6]{Breuil-Ding}. The indefinite case is proved similarly.
\end{proof}

\subsection{$P$-Ordinary automorphic forms}

Let $T$ (resp. $P$) be the subgroup of diagonal torus (resp. upper triangular matrices) of the algebraic group $\GL_2.$ Let $\Ord_P$ denote the {\em ordinary parts} functor \cite[Def.~3.1.9]{EmertonOrd1} from the category of smooth representations of $\GL_2(L)$ on $\cO$-torsion modules to the category of smooth representations of $T(L)$ on $\cO$-torsion modules. For $V$ a $\varpi$-adically continuous representation of $\GL_2(L)$ over $\cO$ (in the sense of \cite[Def.~2.4.1]{EmertonOrd1}), define $\Ord_P(V) \defn  \plim_n \Ord_P(V/\varpi^n V)$ following \cite[Def.~3.4.1]{EmertonOrd1}. For $V^{0}$ a $\GL_2(L)$-stable $\cO$-lattice in a smooth representation $V$ of $\GL_2(L)$ over $E,$ define $\Ord_P(V^{0})$ as in \cite[(4.15)]{Breuil-Ding}. We define the quotient $\wt{\TM}^S(U^{v})^{P-{\rm ord}}_{\frak{m}}$ of $\wt{\TM}^S (U^{v})_{ \frak{m}}$ in the same way as in \cite[\S6.3]{Breuil-Ding}. We record some results of {\em loc. cit.} (and obvious generalizations to the indefinite case).

\begin{lemma}\label{lemma-ordinary-1}
(i) The $\cO$-algebra $\wt{\TM}^S(U^{v})^{P-{\rm ord}} _{\frak{m}}$ is reduced. The $\wt{\TM}^S(U^{v})^{P-{\rm ord}} _{\frak{m}}$-module $\Ord_P(S^D_{\psi}(U^{v} , \W)_{\frak{m}})$ is faithful.

(ii) Assume $U^p$ is {\em sufficiently small} in the sense of \cite[\S 3.3]{CHT}. Then $\Ord_P(S^D_{\psi}(U^{v}, \W)_{\frak{m}})$ is dense in $\Ord_P(\wt{S}^D_{\psi}(U^{v}, \W)_{\frak{m}})$ for the $\varpi$-adic topology. As a consequence, $\Ord_P(\wt{S}^D_{\psi}(U^{v}, \W)_{\frak{m}})$ has a natural $\wt{\TM}^S(U^{v})^{P-{\rm ord}} _{\frak{m}}$-module structure, and is faithful over $\wt{\TM}^S(U^{v})^{P-{\rm ord}}_{\frak{m}}.$

(iii) Let $\cC(T(\cO_{L}), \cO)$ denote the set of continuous functions from $T(\cO_{L})$ to $\cO.$ It is a $T(\cO_{L})$-module via right translation. Let $\cC_{\psi}(T(\cO_{L}), \cO)$ denote the submodule of $\cC(T(\cO_{L}), \cO)$ which consists of functions such that $\cO_{L}^{\times} (\into T(\cO_L))$ acts by $\psi|_{\cO_{L}^{\times}}.$  Assume $U^p$ is sufficiently small. Then there is an integer $r\geq 1$ such that the $T(\cO_{L})$-module $\Ord_P(\wt{S}^D_{\psi}(U^{v}, \W)_{\frak{m}})|_{T(\cO_{L})}$ is isomorphic to a direct summand of $\cC_{\psi}(T(\cO_{L}), \cO)^{\oplus r}.$

(iv) Assume $U^p$ is sufficiently small. Then   $\Ord_P( \wt{S}^D_{\psi}(U^{v}, \W )_{\frak{m}})$ is a $\varpi$-adically admissible representation of $T(L)$ over $\wt{\TM}^S(U^{v})^{P-{\rm ord}}_{\frak{m}}.$

\end{lemma}

\begin{proof}
If $D$ is definite, the proof is an obvious fixed determinant modification of the proof given in \cite{Breuil-Ding}. (i) is \cite[Lem.~6.7]{Breuil-Ding}. (ii) and (iii) are proved in \cite[Lem.~6.8]{Breuil-Ding}. (iv) is Lemma 6.11 of {\em loc. cit.}. If $D$ is indefinite, (i) is proved similarly. (ii)-(iv) are proved along the same way, but one needs to replace the lemma 6.1 of \cite{Breuil-Ding} by standard generalization of \cite[Cor.~5.3.19]{Em3} to the cohomology of Shimura curves.
\end{proof}

Since we assume $\overline{r}$ is absolutely irreducible, let $R^{\psi^{-1}}_{\overline{r},S}$ denote the universal deformation ring of deformations of $\overline{r}$ which are unramified outside $ S$ with determinant $\psi^{-1}\varepsilon^{-1}.$ Assume $\overline{r}_v \defn  \overline{r}|_{G_{L}}$ is {\em strongly generic reducible nonsplit}. Then $\overline{r}_v$ is strictly $P$-ordinary in the sense of \cite[Def.~5.8]{Breuil-Ding}. Recall the complete noetherian local $\cO$-algebra $R^{\psi^{-1}}_{\overline{r}_v}$ (resp. $R^{\psi^{-1}, {\rm red}}_{\overline{r}_v}$) defined in \S \ref{Section::galois}, which parametrizes deformations (resp. {\em reducible} deformations) of $\overline{r}_v$ with determinant $\psi|_{G_{F_{v}}}^{-1} \e^{-1}$.\footnote{For the fixed determinant deformation ring of $\overline{r}_v,$ the notation used here is slightly different from the notation used in \S \ref{Section::galois}.} We have a homomorphism of $\cO$-algebras
$R^{\psi^{-1}}_{\overline{r}_v} \to R^{\psi^{-1}}_{\overline{r},S}$. By works of several mathematicians (see \cite{Taylor}) there is a surjection of complete $\cO$-algebras
\[
R^{\psi^{-1}}_{\overline{r},S} \onto \wt{\TM}^S(U^{v})_{\frak{m}}.
\]

\begin{proposition}\label{prop--local-global}
The homomorphism $R^{\psi^{-1}}_{\overline{r}_v} \to \End\left( \Ord_P( \wt{S}^D_{\psi}(U^{v}, \W )_{\frak{m}}) \right)$ given by the composition
\[
R^{\psi^{-1}}_{\overline{r}_v} \to R^{\psi^{-1}}_{\overline{r},S} \onto \wt{\TM}^S(U^{v})_{\frak{m}} \onto \wt{\TM}^S(U^{v})^{P-\ord}_{\frak{m}} \to \End\left( \Ord_P( \wt{S}^D_{\psi}(U^{v}, \W )_{\frak{m}})\right)
\]
 factors through $R^{\psi^{-1},{\rm red}}_{\overline{r}_v}.$ Moreover, for any $s\geq 1$ the action of $R^{\psi^{-1}}_{\overline{r}_v} $ on $ \Ord_P( \wt{S}^D_{\psi}(U^{v}, \W/\varpi^s )_{\frak{m}})$ factors through $R^{\psi^{-1},{\rm red}}_{\overline{r}_v}$. 
\end{proposition}
\begin{proof}
The same proof of Theorem 6.12 of \cite{Breuil-Ding} works here. In the proof, one needs to replace the local-global compatibility for automorphic forms on unitary groups by the local-global compatibility for automorphic forms on $(D\otimes_F \A_{F,f})^{\times}$ at the place $v| p,$ which is established in \cite{Saito}. The density of
$\Ord_P(S^D_{\psi}(U^{v}, \W )_{\frak{m}})$ inside $\Ord_P(\wt{S}^D_{\psi}(U^{v}, \W )_{\frak{m}})$ (in the indefinite case) is given by (ii) of Lemma \ref{lemma-ordinary-1}. The last statment follows from  the isomorphisms (see the proof of \cite[Lem.~6.8(1)]{Breuil-Ding})
 \begin{multline*}
 \Ord_P( \wt{S}^D_{\psi}(U^{v}, \W )_{\frak{m}})/\varpi^s \cong  \Ord_P( \wt{S}^D_{\psi}(U^{v}, \W )_{\frak{m}}/\varpi^s)  \cong \Ord_P( \wt{S}^D_{\psi}(U^{v}, \W/\varpi^s )_{\frak{m}}).
 \end{multline*}
\end{proof}

\begin{proposition}\label{prop--ord--semisimple}
Write $\overline{r}^{\vee}_v=\smatr{\chi_1}{*}0{\chi_2}$. Then $\Ord_P \big(S^D_{\psi}(U^{v}, \W/\varpi)_{\frak{m}}[\frak{m}] \big)$ is semisimple and isomorphic to $(\chi_1\omega^{-1}\otimes\chi_2 )^{\oplus s}$ for some $s\geq 1$. Here we view $\chi_i$ as a  character  of $L^{\times}$ via the fixed local Artin map.
\end{proposition}

\begin{proof}
If $D$ is indefinite, $\Ord_P \big(S^D_{\psi}(U^{v} , \W/\varpi)_{\frak{m}}[\frak{m}] \big)$ is semisimple by  \cite[Thm.~4.2]{HuJLMS}.  It suffices to show that if $\chi$ occurs in  $\Ord_P \big(S^D_{\psi}(U^{v} , \W/\varpi)_{\frak{m}}[\frak{m}] \big)$ then $\chi\cong \chi_1\omega^{-1}\otimes\chi_2$. We will use the results of \cite{HuJLMS} (or  \cite{Em3}); note that the convention of Shimura curves in \emph{loc. cit.} is  different from ours, e.g. the Galois representation associated to a cuspidal  automorphic representation $\pi$ (in characteristic zero) has determinant $\psi_{\pi} \e^{-1}$ in \emph{loc. cit.}, where $\psi_{\pi}$ is the central character of $\pi;$ while in our case the Galois representation has determinant  $\psi_{\pi}^{-1} \e^{-1}$. Hence in order to obtain results for our $\overline{r},$ we should apply the results of \cite{HuJLMS} (or \cite{Em3}) to $ \overline{r}^{\vee}(-1).$

In this proof we write $ \brho =  \overline{r}^{\vee}_v(-1) = \smatr{\eta_1}{*}0{\eta_2}$ with $\eta_1 = \chi_1 \o^{-1}$ and $\eta_2 = \chi_2 \o^{-1}.$ Let $S$ be the subtorus of $T$ consisting of matrices $\smatr{a}001$ for $a\in L^{\times}$ and define an anti-diagonal embedding
\begin{equation}\label{eq:S-action}S\hookrightarrow G_{L}^{\rm ab}\times S, \ \ \ s\mapsto (\mathrm{Art}_{L}(s),s^{-1})\end{equation}
where $G_L^{\rm ab}$ denotes the maximal abelian quotient of $G_L$ and $\mathrm{Art}_L$ the local Artin map. By \cite[Thm.~4.1, Lem.~2.10]{HuJLMS},  the action of $G_L$ on $(\brho\otimes \chi)^{\mathrm{ab},S}\hookrightarrow \brho\otimes \chi$ factors through $G_L^{\rm ab}$, see \cite[\S2.4, Def.~1]{HuJLMS} for the definition of $(\brho\otimes\chi)^{\mathrm{ab},S}$. Since $\brho$ is nonsplit and $\eta_1\neq \eta_2$, this implies that $(\brho\otimes \chi)^{\mathrm{ab},S}$ is nonzero and is equal to $\eta_1\otimes \chi$ as a $G_L\times S$-representation.  By a computation using \eqref{eq:S-action}, we obtain that $\chi=\eta_1\otimes \eta_2'$ for some character $\eta_2'$ of $G_L$, and a consideration of central character shows that
\[\eta_1\eta_2'=(\det\brho)\cdot \omega=\eta_1\eta_2\omega,\]
namely $\eta_2'=\eta_2\omega = \chi_2$.

If $D$ is definite, then the result is a special case of \cite[Cor.~7.40]{Breuil-Ding}.  We note that it is assumed in \cite{Breuil-Ding} that $L = \Q_p$ in order to treat the case that the Levi subgroup of $P \subset \GL_n$ is a product of $\GL_1$'s and $\GL_2$'s. Since we assume here $n=2$ and $T = \GL_1 \times \GL_1,$ we don't need this assumption. As we follow the convention of \cite{BDJ} instead of \cite{Breuil-Ding}, we apply \cite[Cor.~7.40]{Breuil-Ding} to $\overline{r}^{\vee}_v(-1)$ with $s_1=0$, $s_2=1$ in \emph{loc. cit.}, and the result follows.
\end{proof}

\section{Global applications} \label{section-patching}

We maintain the notation used in the last section. In particular, $F$ is the totally real field in which $p$ is unramified, $v$ is the fixed place over $p$ such that $F_{v}$ is isomorphic to $L$ the fixed unramified extension of $\Q_p$ of degree $f.$ Let $D$ be a definite or indefinite quaternion algebra over $F.$ Let $S_D$ be the set of ramified places of $D,$ $S_p$ be the set of places above $p.$ The aim of this section is to prove Theorem \ref{thm-main-cyclicity}, Corollary \ref{cor:multione-Iwahori}, Theorem \ref{thm:main-flat} and Corollary \ref{cor:Pi(x)}.

\subsection{The ``big'' patching functors} \label{subsection-big-patching}

In this subsection we recall the global patching setup.  Assume $p> 5$ is an odd prime; $\overline{r} : G_F\to \GL_2(\F)$ is automorphic, and $\overline{r}|_{G_{F(\sqrt[p]{1})}}$ is absolutely irreducible; $\overline{r}|_{I_{F_w}}$ is generic for all places $w|p$ in the sense of \cite[Def.~11.7]{BP}. Denote by $\overline{r}_w \defn \overline{r}|_{G_{F_w}}$ for all finite places $w$ of $F.$

Let $\overline{\psi} \defn \o^{-1} (\det \overline{r})^{-1}.$ Let $\psi: G_F\to \cO^{\times}$ be the Teichm\"uller lift of $\overline{\psi}.$ By abuse of notation, we denote by $\psi$ the character $\psi\circ \Art_F : F^{\times}\backslash \A_{F,f}^{\times} \to \cO^{\times}.$ Let $S$ be the set of finite places of $F$ which consists of $S_D$, $S_p\backslash{\{v\}},$ and the places where $\overline{r}$ ramifies. We assume for $w \in S \backslash S_p$ the framed deformation ring of $\overline{r}_w$ is formally smooth over $\cO$ (cf. \cite[Rk. 8.1.1]{BHHMS}).
We choose a finite place $w_1\notin S$ with the following properties:

$\bullet$ $q_{w_1}\not\equiv 1 \pmod{p}$,

$\bullet$ the ratio of the eigenvalues of $\overline{r}(\Frob_{w_1})$ is not equal to $q_{w_1}^{\pm 1}$,

$\bullet$ the residue characteristic of $w_1$ is sufficiently large that for any nontrivial root of unity $\zeta$ in a quadratic extension of $F,$ $w_1$ does not divide $\zeta+\zeta^{-1}-2.$

Let $U = \prod_w U_w \subset (D\otimes_F \A_{F,f})^{\times}$ be an open compact subgroup satisfying $U_w = (\cO_{D})_w^{\times} = \GL_2(\cO_{F_w})$ for $w \notin S \cup \{w_1\},$ $U_{w_1}$ is contained is the subgroup of $(\cO_D)_{w_1}^{\times} = \GL_2(\cO_{F_{w_1}})$ consisting of matrices that are upper-triangular and unipotent modulo $\varpi_{w_1},$ and $U_w = 1 +\varpi_w M_2(\cO_{F_w})$ for $w \in S_p.$ By the choice of $U_{w_1},$ $U$ is sufficiently small in the sense of \cite[\S 3.3]{CHT}.

We assume $E$ is a sufficiently large finite unramified extension of $\Q_p$.  For each $w \in S_p\backslash \{v\},$ we fix a tame inertial type $\tau_w$ over $E$ such that $\JH( \overline{\s (\tau_w^{*})})$ contains exactly one Serre weight in $\mathscr{D}(\overline{r}^{\vee}_w)$ (\cite[Prop.~3.5.1]{EGS}), where $\tau^{*}_w$ is the $E$-linear dual of $\tau_w$. Let $\s_w$ denote the unique Serre weight in $\JH( \overline{\s (\tau_w^{*})})\cap \mathscr{D}(\overline{r}_w^{\vee})$. We fix a $\GL_2(\cO_{F_w})$-invariant lattice $\s^{\circ}(\tau^{*}_w)$ in $\s(\tau^{*}_w).$ Since $\s_w$ has central character $\overline{\psi}|_{\cO_{F_w}^{\times}}$ and $\tau_w$ is tame, $\s^{\circ}(\tau^{*}_w)$ has central character $\psi|_{I_{F_w}}.$ Let
\begin{equation}\label{eqn::defofW}
\W \defn \bigotimes_{w\in S_p\backslash \{v \}} \s^{\circ}(\tau^{*}_w)^d.
\end{equation}
By \cite[Cor.~3.2.3]{BreuilDiamond}, $\overline{r} : G_F\to \GL_2(\F)$ is $(U^{v},\W)$-automorphic.

Let $Q$ be a finite set of finite places which consists of $w \notin S\cup\{w_1\}$ such that the ratio of the eigenvalues of $\overline{r}(\Frob_w)$ is not in $\{1,q_w,q_w^{-1}\}.$ Let $U_1(Q)^v$ be an open compact subgroup of $U^v$ satisfying $(U_1(Q)^v)_w $ is the subgroup of $(\cO_D)_{w}^{\times} = \GL_2(\cO_{F_w})$ of matrices of the form $\smatr{a}{b}0{a} $ modulo $\varpi_w$ for $w\in Q$ and $(U_1(Q)^v)_w = U_w$ for $w\notin Q.$ In particular $U_1(\emptyset)^v  = U^v.$

The abstract Hecke algebra $\TM^{S \cup Q \cup \{w_1\}}$ acts on $\wt{S}^D_{\psi}(U_1(Q)^{v}, \W)$ such that the action factors through a faithful action of $\wt{\TM}^{S \cup Q \cup \{w_1\}}(U_1(Q)^{v}).$ Let $\frak{m}_Q $ denote the maximal ideal of $\wt{\TM}^{S \cup Q \cup \{w_1\}}(U_1(Q)^{v})$ associated to $\overline{r}.$ Let $r^{\rm univ}_{\frak{m}_{Q}}: G_{F}  \to \GL_2( R^{\psi^{-1}}_{\overline{r},S\cup Q})$ be the universal deformation of $\overline{r}$ over $R^{\psi^{-1}}_{\overline{r},S\cup Q}.$ Let $r_{\frak{m}_{Q}}$ denote the composition $G_F \To{r^{\rm univ}_{\frak{m}_{Q}}} \GL_2( R^{\psi^{-1}}_{\overline{r},S\cup Q}) \to \GL_2(\wt{\TM}^{S \cup Q \cup \{w_1\}}(U_1(Q)^{v})).$ When $Q = \emptyset$ we write $\frak{m}_{\overline{r}}$ for the maximal ideal of $\wt{\TM}^{S \cup \{w_1\}}(U^{v})$ associated to $\overline{r}.$

In the definite case we write
\[
M_{Q} \defn \left(\wt{S}^D_{\psi}(U_1(Q)^{v}, \W)_{\frak{m}_Q} \right)^{d},
 \]
and in the indefinite case we write
\[
M_Q \defn \Hom_{G_{F}}\left( r^{\rm univ}_{\frak{m}_{Q}},   \wt{S}^D_{\psi}(U_1(Q)^{v}, \W)_{\frak{m}_Q}\right)^{d}.
\]

For $w$ a finite place of $F,$ let $R^{\Box}_w$ denote the framed deformation ring for $\overline{r}_w$ over $\OC$ (see \S \ref{section--pBT}). Let $R_{w}^{\Box,\psi^{-1}}$ denote the quotient of $R^{\Box}_{w}$ corresponding to liftings with determinant $\psi|^{-1}_{G_{F_w}}\varepsilon^{-1}.$
Let $R^{\rm loc} = \widehat{\otimes}_{w\in S \cup \{v\}}R_w^{\Box,\psi^{-1}}.$ Note that $R^{\rm loc}$ is formally smooth over $\cO$ by assumption. For $Q$ a finite set of finite places of $F$ disjoint from $S,$ let $R^{\Box,\psi^{-1}}_{\overline{r},S \cup Q}$ be the complete $\cO$-algebra which prorepresents  the functor assigning to a local artinian $\cO$-algebra $A$ with residue field $\F$ the set of equivalence classes of tuples $(r, \{\a_w\}_{w\in S})$ where $r$ is an $A$-lifting of $\overline{r}$ with determinant $\psi^{-1}\varepsilon^{-1}$ (we view $\psi^{-1}\varepsilon^{-1}$ as a character of $G_F$ with values in $A^{\times}$) and is unramified outside $S\cup Q,$ $\a_w \in \Ker (\GL_2(A) \to \GL_2(\F))$ for each $w\in S.$

Let $j = 4 | S| - 1 .$ The Taylor-Wiles-Kisin patching construction in \cite{CEGGPS1} and \cite{Scholze} gives us the following data (see \cite[Thm.~6.1]{Dotto-Le} for an analogous situation):
\begin{itemize}
\item[(1)] positive integers $g,q$ such that $q = g+[F:\Q] - | S| + 1;$

\item[(2)]  $ \cO_{\infty}  = \cO[\![z_1,\ldots,z_q]\!]$, a formal power series ring with a homomorphism
\[
\cO_{\infty} \to  R^{\psi^{-1}}_{\overline{r},S}
\]
which extends to a homomorphism from $S_{\infty} \defn \cO[\![z_1,\ldots,z_q,y_1,\ldots,y_j]\!]$ to $R^{\Box,\psi^{-1}}_{\overline{r},S}$;

\item[(3)] $R_{\infty} \defn R^{\loc}[\![x_1,\ldots,x_g]\!]$, a power series ring in $g$-variables over $R^{\loc}$, with a surjective homomorphism $R_{\infty} \to R^{\psi^{-1}}_{\overline{r},S}$; 

\item[(4)] an $\OC$-algebra homomorphism $S_{\infty}\to R_{\infty}$ whose composition with the homomorphism $R_{\infty} \to R^{\psi^{-1}}_{\overline{r},S}$ in (3) is compatible with the homomorphism $\cO_{\infty} \to R_{\overline{r},S}^{\psi^{-1}}$ in (2);

\item[(5)] a finite Cohen-Macaulay  $R_{\infty}[\![\GL_2(\cO_L)]\!]$-module $M_{\infty},$ which is finite projective over $S_{\infty}[\![\GL_2(\cO_L)]\!],$ together with an isomorphism
\[
M_{\infty}/\frak{a}_{\infty} \cong M_{\emptyset}.
\]
where  $\frak{a}_{\infty} $ denotes the ideal $(z_1,\dots,z_{q},y_1,\ldots,y_j)$ of $S_{\infty}$.
\end{itemize}

We note that $R_{\infty}$ is formally smooth over $\cO$ (as $R^{\rm loc}$ is). Let $\frak{m}_{\infty}$ denote the maximal ideal of $R_{\infty}.$ 
Consider the following  admissible smooth representation of $\GL_2(L)$ over $\F$
\begin{equation}\label{def-rep-pi}
\pi \defn  M_{\infty}^{\vee}[\frak{m}_{\infty} ] = M_{\emptyset}^{\vee} [\frak{m}_{\overline{r}}].
\end{equation}
Then $\pi$ is identified with
\[
 \wt{S}^D_{\psi}(U^{v}, \W)[\frak{m}_{\overline{r}}] =  \Hom_{U_p^v} \left(\otimes_{w \in S_p \backslash \{v\}} \s_w,\wt{S}^D_{\psi}(U^{v}, \F)[\frak{m}_{\overline{r}}] \right)
 \]
if $D$ is definite, and is identified with
\[
\Hom_{G_{F}}\left( \overline{r},   \wt{S}^D_{\psi}(U^{v}, \W)[\frak{m}_{\overline{r}}]\right) = \Hom_{U_p^v} \left(\otimes_{w \in S_p \backslash \{v\}} \s_w  , \Hom_{G_{F}}\left( \overline{r},   \wt{S}^D_{\psi}(U^{v}, \F)[\frak{m}_{\overline{r}}]\right) \right)
\]
if $D$ is indefinite.

Let $\cC_{Z,\psi}$ denote the category of finite $\cO$-modules with a continuous action of $\GL_2(\cO_L)$ such that the $\GL_2(\cO_L)$-action has central character $\psi.$ For $\s \in \cC_{Z,\psi},$ let
\begin{equation} \label{equ-def-functor}
M_{\infty}(\s) \defn \Hom^{\rm cont}_{ \cO[\![\GL_2(\cO_L)]\!]}(M_{\infty }, \s^{\vee})^{\vee}.
\end{equation}
Since $M_{\infty}$ is projective over $\cO[\![\GL_2(\cO_L)]\!]$ and is finitely generated over $R_{\infty}[\![\GL_2(\cO_L)]\!]$, $M_{\infty}(-)$ is an exact covariant functor from $\cC_{Z,\psi}$ to the category of finitely generated $ R_{\infty}$-modules. 
For $\s \in \cC_{Z,\psi}$, we have
\begin{equation}\label{eq:mod-minfty=Hom}
\big(M_{\infty}(\s) / \frak{m}_{\infty} M_{\infty}(\s)\big)^{\vee} \cong \Hom_{\GL_2(\cO_{L})}(\s,\pi ).
\end{equation}

Recall that $\dim_G(\pi)$ denotes the Gelfand-Kirillov dimension of $\pi$, see \eqref{equ::GK-dim-defn} in Appendix.  Gee and Newton made the following important observation (\cite[Appendix A]{Gee-Newton}):
\begin{theorem}\label{thm::GN}
$M_{\infty}$ is a flat $R_{\infty}$-module if and only if $\GKdim (\pi )\leq f.$ If the equivalent conditions hold, we have $\GKdim (\pi ) =  f.$
\end{theorem}

Under this flatness assumption, we prove an essential self-duality for $\pi^{\vee} = M_{\emptyset}/\frak{m}_{\overline{r}}.$

\begin{theorem}\label{thm-selfdual}
Assume $\GKdim (\pi )\leq f$. We have an isomorphism of $\Lambda(G)$-modules (see (\ref{eqn::Lambda(G)}) for the definition of $\Lambda(G)$)
\[
\pi^{\vee} \otimes \overline{\psi}|_{F^{\times}_v} \circ\det \cong \EE^{2f}(\pi^{\vee} ),
\]
where $\EE^i$ is defined in \S\ref{subsection-duality} of the Appendix.
\end{theorem}
\begin{proof}
First assume $D$ is definite. Let $\wt{H}^{0}_{\overline{\psi}} \defn \wt{H}^{0}\left(U^{p}, (\cO/\varpi)_{\overline{\psi}} \right),$ where $(\cO/\varpi)_{\overline{\psi}}$ denotes the constant coefficient sheaf $\cO/\varpi$ on which $\A_{F,f}^{\times}$ acts by  $\overline{\psi}$ and $U^p$ acts trivially. Let $\wt{H}_{0,\overline{\psi}^{-1}}$ denote the dual of $\wt{H}^{0}_{\overline{\psi}}.$ By the Poincar\'e duality spectral sequence \cite[\S2.1.5, \S2.1.7]{EmICM}, we obtain a $U^v_p \times \L(G)\times \wt{\mathbb{T}}^{S \cup \{w_1\}}(U^{p})$-equivariant isomorphism
\[
\EE^0(\wt{H}_{0 , \overline{\psi}})\simto \wt{H}_{0, \overline{\psi}^{-1}} ,
\]
where $\wt{\mathbb{T}}^{S \cup \{w_1\}}(U^{p})$ denotes the image of $\mathbb{T}^{S\cup \{w_1\}}$ in $\End\left(\wt{H}^{0}(U^{p}, \cO) \right).$ We consider the $U^v_p$-action, take $\Hom_{U^v_p}(\W/\varpi,-)$ on both sides, and obtain a $ \L(G)\times \wt{\mathbb{T}}^{S \cup \{w_1\}}(U^{p})$-equivariant isomorphism (using e.g. \cite[Lem.~B.3]{Gee-Newton})
\[
\EE^0\left( \wt{H}^{0}(U^{v}, \W^d/\varpi)^{\vee}\right) \simto  \wt{H}^{0}(U^{v}, \W/\varpi)^{\vee}.
\]
Since $\wt{\mathbb{T}}^{S \cup \{w_1\}}(U^{p})$ is a complete semilocal ring by \cite[Lem.~2.1.14]{Gee-Newton}, we have a decomposition $\wt{\mathbb{T}}^{S \cup \{w_1\}}(U^{p}) = \prod_{\frak{m}} \wt{\mathbb{T}}^{S \cup \{w_1\}}(U^{p})_{\frak{m}}$, where the product is taken over the finitely many maximal ideals of $\wt{\mathbb{T}}^{S \cup \{w_1\}}(U^{p}).$
We then have an isomorphism
\[
\prod_{\frak{m}} \EE^0\left( \wt{H}^{0}(U^{v}, \W^d/\varpi)^{\vee}_{\fm } \right) \simto \prod_{\frak{m}}\wt{H}^{0}(U^{v}, \W/\varpi)^{\vee}_{\fm }.
\]
Looking at the $ \fm_{\overline{r}}$-component on both sides and using the relations on Hecke operators (see for example \cite[\S3, \S9]{Br14}), which reflect the isomorphism $\overline{r}^{\vee} \cong \overline{r} \otimes (\det \overline{r})^{-1},$ we have
\begin{equation}\label{eqn--duality-isom}
\EE^0\left( \wt{H}^{0}(U^{v}, \W^d/\varpi)_{\frak{m}_{\overline{r}\otimes \overline{\psi}}}^{\vee} \right) \simto  \wt{H}^{0}(U^{v}, \W/\varpi)_{\fm_{\overline{r}} }^{\vee}.
\end{equation}
Recall that $\W = \bigotimes_{w\in S_p\backslash \{v \}} \s^{\circ}(\tau^{*}_w)^d$ with $\JH( \overline{\s (\tau_w^{*})})$ containing exactly one Serre weight $\s_w$ in $\mathscr{D}(\overline{r}^{\vee}_w).$ By \cite[Prop.~3.15(1)]{BDJ}, $\sigma_{w}^{\vee}\cong \s_w \otimes \overline{\psi}|_{\cO_{F_w}^{\times}}^{-1}\circ\det$ is the unique Jordan--H\"older factor of $\overline{\s (\tau_w^*)^*}$ which is a Serre weight of $(\overline{r} \otimes \overline{\psi})_w^{\vee}.$  
We then have isomorphisms of $\L(G) \times \wt{\mathbb{T}}^{S \cup \{w_1\}}(U^v)_{\fm_{\overline{r}} }$-modules
\begin{equation}\label{eqn--duality-coeff}
\begin{array}{rll} \wt{H}^{0}(U^{v}, \W^d/\varpi)^{\vee}_{\frak{m}_{\overline{r}\otimes \overline{\psi}}}
  \cong& \Hom_{U^v_p} \Big(\bigotimes\limits_{w\in S_p \backslash \{v\}} \s_w\otimes  \overline{\psi}|_{\cO_{F_w}^{\times}}^{-1}\circ\det ,  \wt{H}^{0}(U^{p}, \F)_{\frak{m}_{\overline{r}\otimes \overline{\psi}}} \Big)^{\vee}\\ 
   \cong &\wt{H}^{0}(U^{v}, \W/\varpi)^{\vee}_{\frak{m}_{\overline{r}}} \otimes  \overline{\psi}|_{F_v^{\times}} \circ \det. \\ 
\end{array}
\end{equation}
Since
$
M_{\emptyset}/\varpi  =  \wt{H}^{0}(U^{v}, \W/\varpi)_{\fm_{\overline{r}} }^{\vee}$ in this case,  
 \eqref{eqn--duality-isom} becomes the following isomorphism of $\L(G) \times \wt{\mathbb{T}}^{S \cup \{w_1\}}(U^v)_{\fm_{\overline{r}} }$-modules
 \[\EE^0\big(M_{\emptyset}/\varpi \otimes \overline{\psi}|_{F^{\times}_v}\circ\det\big)\simto M_{\emptyset}/\varpi \]
 or equivalently
\[
\EE^0(M_{\emptyset}/\varpi )\simto (M_{\emptyset}/\varpi ) \otimes \overline{\psi}|_{F^{\times}_v}\circ\det.
\]
On the other hand, by \cite[Thm.~B]{Gee-Newton}, the assumption on $\GKdim (\pi )$ implies that $M_{\emptyset}$ is a faithfully flat module over  $\wt{\mathbb{T}}^{S \cup \{w_1\}}(U^v)_{\fm_{\overline{r}} }$ which is a local complete intersection (hence Gorenstein) ring.  Note also that as $M_{\emptyset}$ is finite Cohen-Macaulay over $\cO[\![ \GL_2(\cO_L)]\!]$ (see for example \cite[Thm.~A(1)]{Gee-Newton}), $M_{\emptyset} / \varpi$ is Cohen-Macaulay over $\L(G).$ We then conclude by applying Proposition \ref{prop-Appendix-selfduality} to   $M=M_{\emptyset}$ and $A=\wt{\mathbb{T}}^{S \cup \{w_1\}}(U^v)_{\fm_{\overline{r}} }$ (noting that $  M_{\emptyset} / \fm_{\overline{r}}  \cong \pi^{\vee} $).

Assume $D$ is indefinite. Let $\wt{H}^{i}_{ \overline{\psi}}$ denote $\wt{H}^{i}\left(U^{p}, (\cO/\varpi)_{\overline{\psi}} \right)$ and $\wt{H}_{i, \overline{\psi}^{-1}}$ denote the dual of $\wt{H}^{i}_{ \overline{\psi}}.$ The \'etale version of the Poincar\'e duality spectral sequence \cite[\S2.1.5,\S2.1.7]{EmICM} (see also \cite[Thm.~3.5]{Hill} and its second remark) gives a $U^v_p \times \L(G) \times G_F \times\wt{\mathbb{T}}^{S \cup \{w_1\}}(U^{p})$-equivariant exact sequence
\[
0 \to \EE^1 (\wt{H}_{0, \overline{\psi}}) \to \wt{H}_{1, \overline{ \psi}^{-1}} (-1) \to \EE^0(\wt{H}_{1, \overline{\psi}}) \to \EE^2(\wt{H}_{0 , \overline{\psi}}).
\]
As in the definite case, we decompose along the (finitely many) maximal ideals of $\wt{\mathbb{T}}^{S \cup \{w_1\}}(U^{p}),$ and get an exact sequence
\begin{multline*}
 0 \to \prod_{\frak{m}}\EE^1 \left( (\wt{H}_{0, \overline{\psi}})_{\frak{m}} \right)\to \prod_{\fm}\wt{H}_{1, \overline{ \psi}^{-1}} (-1)_{\frak{m}} \to \prod_{\frak{m}} \EE^0\left( (\wt{H}_{1, \overline{\psi}})_{\fm} \right)\to \prod_{\fm} \EE^2\left( (\wt{H}_{0 , \overline{\psi}})_{\fm}\right).
\end{multline*}
Since $\overline{r}$ is absolutely irreducible, we can just consider the localization at {\em non-Eisenstein} maximal ideals. Note that $(\wt{H}_{0, \overline{\psi}})_{\frak{m}}$ is zero at any non-Eisenstein maximal ideal $\frak{m}.$ We have a $U^v_p \times \L(G) \times G_F \times\wt{\mathbb{T}}^{S \cup \{w_1\}}(U^{p})$-equivariant isomorphism
\[
\prod_{\text{$\fm$ non-Eisenstein}}\wt{H}_{1, \overline{ \psi}^{-1}} (-1)_{\frak{m}} \simto \prod_{\text{$\fm$ non-Eisenstein}} \EE^0\left( (\wt{H}_{1, \overline{\psi}})_{\fm} \right).
\]
We consider the $U^v_p$-action, take $\Hom_{U^v_p}(\W/\varpi,-)$ on both sides and look at the $\frak{m}_{\overline{r}}$-components, by taking the limit over the compact open subgroups $K_v \subseteq \GL_2(\cO_L)$ of the isomorphism in \cite[Lem.~9.9]{DPS},
we have an isomorphism
\begin{equation}\label{eqn--duality-H1}
  \wt{H}^{1}(U^{v}, \W/\varpi)_{\fm_{\overline{r}} }^{\vee} \simto \EE^0\left( \wt{H}^{1}(U^{v}, \W^d/\varpi)_{\frak{m}_{\overline{r}\otimes \overline{\psi}}}^{\vee} \right),
\end{equation}
where we have used the relation $\overline{r}^{\vee}(-1) \cong \overline{r} \otimes \overline{\psi}.$ Applying $\Hom_{G_F}(\overline{r}, -)$ to the isomorphism (\ref{eqn--duality-H1}) and using similar arguments as (\ref{eqn--duality-coeff}) in the indefinite case, we obtain a $\L(G) \times \wt{\mathbb{T}}^{S \cup \{w_1\}}(U^{v})_{ \frak{m}_{\overline{r}}}$-equivariant isomorphism
\[
(M_{\emptyset} /\varpi ) \otimes \overline{\psi}|_{F^{\times}_v} \circ\det \simto \EE^0(M_{\emptyset} /\varpi).
\]
The rest of the proof is identical to the definite case.
\end{proof}

\subsection{Local-global compatibility}

We extend Proposition \ref{prop--local-global} to the ``big'' patched module $M_{\infty}.$ Assume in this subsection $\overline{r}_v$ is reducible nonsplit and  strongly generic. Let $R_{v}^{\Box, \psi^{-1}, {\rm red}}$ denote the $\cO$-torsion free quotient of $R_{v}^{\Box,\psi^{-1}}$ which parametrizes reducible liftings of $\overline{r}_v$.  Then by Proposition \ref{prop-red-defring-det}, $R_{v}^{\Box, \psi^{-1}, {\rm red}}$ is a power series ring in $(3+2f)$-variables over $\cO.$ Denote $R^{\rm red}_{\infty} = R^{\Box,\psi^{-1}, {\rm red}}_{v}\otimes_{R_{v}^{\Box,\psi^{-1}}} R_{\infty}.$ Let $\s_J\in \mathscr{D}(\overline{r}^{\vee}_v)$ be the corresponding Serre weight of $\overline{r}^{\vee}_v$ associated to $J \subset S_{\overline{r}^{\vee}_v}$ (see (\ref{equation-J_rho})). If $J = \emptyset,$ we write $\s_0$ for $\s_{\emptyset}.$

\begin{proposition}\label{thm-BD-Rord}
Let $\Theta^{\rm ord}_{\s_0}$ be the finite dimensional representation of $\GL_2(\cO_L)$ over $\F$ defined in \S \ref{subsection:Theta-ord}. Then the morphism
\[
 R_{\infty} \to  \End_{\cO}(M_{\infty}(\Theta^{\rm ord}_{\s_0}))
\]
factors through $R_{\infty}^{\rm red}.$
\end{proposition}
\begin{proof}
We first briefly recall the construction of  $M_{\infty}$ in \cite{Dotto-Le}. For every integer $N\geq 1,$ let $Q_N$ be a set of Taylor-Wiles primes as in \cite[\S 6.2.3]{Dotto-Le}. In particular, $Q_N$ has size $q$ and is disjoint from $S.$ For each $w\in Q_N,$ we have $q_w \equiv 1 \pmod{p^N}.$ Let $k_w^{\times}(p)$ denote the Sylow $p$-subgroup of $k_w^{\times}$ for $w\in Q_N$ and $\cO_N$ denote $\cO[ \prod_{w\in Q_N} k_w^{\times}(p)].$ For each $N\geq 1$ we choose a surjection $\cO_{\infty}= \cO[\![z_1,\ldots, z_q]\!] \onto \cO_N$
whose kernel is contained in the ideal generated by $((1+z_i)^{p^N } - 1)_{i=1}^q.$

We write $K$ for $\GL_2(\cO_L).$  For $U_{v} \subset K$ a compact open subgroup, let
\[
M(U_{v} , N)  = S^D_{\psi}(U_{1}(Q_{N})^{v} U_{v}, \W)^{d}_{\frak{m}_{Q_{N}}}  \otimes_{R^{\psi^{-1}}_{\overline{r}, S\cup Q_N}} R^{\Box, \psi^{-1}}_{\overline{r}, S\cup Q_N}
\]
when $D$ is definite; and let
\[
M(U_{v} ,N) = \Hom_{G_{F}}\left( r_{\frak{m}_{Q_N}} ,  S^D_{\psi}(U_{1}(Q_{N})^{v} U_{v}, \W)_{\frak{m}_{Q_{N}}}  \right)^{d}  \otimes_{R^{\psi^{-1}}_{\overline{r}, S\cup Q_N}} R^{\Box, \psi^{-1}}_{\overline{r}, S\cup Q_N}
\]
when $D$ is indefinite.

Let $J \subset \cO_{\infty}$ be an open ideal. Let $I_J$ denote the cofinite subset of $\N$ which consists of integers $N$ such that $J$ contains $\Ker(\cO_{\infty} \onto \cO_N).$ For $N \in I_J$, let
\[
M(U_{v} , J ,N)  =  M(U_{v} ,N) \otimes_{\cO_{\infty}} \cO_{\infty} /J.
\]
Let $(\cO_{\infty} /J)_{I_J}$ be the product $\prod_{i\in I_J}\cO_{\infty} /J.$ Fix a non-principal ultrafilter $\frak{F}$ on $\N.$ Then $\frak{F}$ gives a point $x \in \Spec (\cO_{\infty} /J)_{I_J}$, see \cite[Lem.~2.2.2]{Gee-Newton}. Let
\[
M(U_{v}, J ,\infty) = \left( \prod_{N\in I_J} M(U_{v} , J ,N) \right) \otimes_{(\cO_{\infty} /J)_{I_J}} (\cO_{\infty} /J)_{I_J, x}.
\]

For open ideals $J'\subset J,$ open compact subgroups $U'_{v} \subset U_{v},$ there is a natural map $M(U'_{v}, J' ,\infty) \to M(U_{v}, J ,\infty)$ (see \cite[Lem.~3.4.11]{Gee-Newton}). Then $M_{\infty}$ is defined as
\[
M_{\infty} \defn \plim_{J, U_{v}} M(U_{v}, J ,\infty).
\]

Now we prove the proposition. We assume $D$ is definite; the indefinite case can be treated in a similar way. By (\ref{equ-def-functor}) we have
\begin{align*}
&M_{\infty}(\Theta^{\rm ord}_{\s_0})  =   \Hom_{K}(\Theta^{\rm ord}_{\s_0}, \ilim_{J, U_{v}} M(U_{v}, J ,\infty)^{\vee} )^{\vee} \\
 & = \ilim_{J,U_{v}}\prod_{N \in I_J}\Hom_{K} (\Theta^{\rm ord}_{\s_0}, M(U_{v} , J ,N) ^{\vee})^{\vee} \otimes_{(\cO_{\infty} /J)_{I_J}} (\cO_{\infty} /J)_{I_J, x} \\
& =  \ilim_{J,U_{v}} \left(\prod_{N \in I_J}\Hom_{K} (\Theta^{\rm ord}_{\s_0},  M(U_{v} , N)^{\vee} )^{\vee}\otimes_{\cO_{\infty}} \cO_{\infty}/J  \right)\otimes_{(\cO_{\infty} /J)_{I_J}} (\cO_{\infty} /J)_{I_J, x}.
\end{align*}

It suffices to show that the action of $R^{\psi^{-1}}_{v}$ on $\Hom_{K} (\Theta^{\rm ord}_{\s_0}, M(U_{v} , N)^{\vee} )$ factors through $R_{v}^{\psi^{-1},{\rm red}}$ for $U_{v}$ sufficiently small. Since
\[
M(U_{v} , N)  = S^D_{\psi}(U_{1}(Q_{N})^{v} U_{v}, \W)^{d}_{\frak{m}_{Q_{N}}}  \otimes_{R^{\psi^{-1}}_{\overline{r}, S\cup Q_N}} R^{\Box, \psi^{-1}}_{\overline{r}, S \cup Q_N},
\]
it suffices to show the same statement for 
\[
\Hom_{K} \left( \Theta^{\rm ord}_{\s_0}, \left(S^D_{\psi}(U_{1}(Q_{N})^{v} U_{v}, \W)^d_{\frak{m}_{Q_{N}}} \right)^{\vee}\right)  = \Hom_{K} \left(\Theta^{\rm ord}_{\s_0}, S^D_{\psi}(U_{1}(Q_{N})^{v} U_{v}, \W/\varpi)_{\frak{m}_{Q_{N}}} \right)
\]
where the equality holds as $\Theta_{\sigma_0}^{\rm ord}$ is $\varpi$-torsion.

For the rest of the proof, we simplify the notation by setting
\[
S_{\psi}^D(\cO/\varpi)_{\frak{m}}  =  \ilim_{U_{v}} S_{\psi}^D(U_1(Q_{N})^{v}U_{v}, \W/\varpi)_{\frak{m}_{Q_{N}}}.
\]
For $U_{v} \subseteq K$ sufficiently small, $U_{v}$ acts trivially on $\Theta^{\rm ord}_{\s_0},$ we then have
\begin{align*}
\Hom_{K} (\Theta^{\rm ord}_{\s_0}, S_{\psi}^D(U_1(Q_{N})^{v}U_{v}, \W/\varpi)_{\frak{m}_{Q_{N}}}) & = \Hom_{K}(\Theta^{\rm ord}_{\s_0}, S_{\psi}^D(\cO/\varpi)^{U_{v}}_{\frak{m}} ) \\
& =  \Hom_{K}(\Theta^{\rm ord}_{\s_0}, S_{\psi}^D(\cO/\varpi)_{\frak{m}} ) .
\end{align*}
We are  reduced to show that the action of $ R^{\psi^{-1}}_{v}$
on $\Hom_{K}(\Theta^{\rm ord}_{\s_0},S_{\psi}^D(\cO/\varpi)_{\frak{m}} )$
 factors through $R_{v}^{\psi^{-1},{\rm red}}.$ By Proposition \ref{prop:Hom-Theta-Ord} and Lemma \ref{lemma:j-canonical},  we have a natural isomorphism
\[
 \Hom_{K}(\Theta^{\rm ord}_{\s_0},\Ind_{\overline{P}}^G\Ord_PS_{\psi}^D(\cO/\varpi)_{\frak{m}})\simto \Hom_{K}(\Theta^{\rm ord}_{\s_0},S_{\psi}^D(\cO/\varpi)_{\frak{m}})
\]
which is compatible with the action of $R^{\psi^{-1}}_{v}$   if $\Ind_{\overline{P}}^G\Ord_PS_{\psi}^D(\cO/\varpi)_{\frak{m}}$ is equipped with the induced action of $R^{\psi^{-1}}_{v}$ through $\Ord_PS_{\psi}^D(\cO/\varpi)_{\frak{m}}$. By Proposition \ref{prop--local-global}, the action of $R^{\psi^{-1}}_{v}$ on $\Ord_PS_{\psi}^D(\cO/\varpi)_{\frak{m}}$  factors through $R_{v}^{\psi^{-1}, {\rm red}}.$ This finishes the proof.
\end{proof}

\subsection{The ``big'' minimal patching functors}

The minimal patching functor that we will use is introduced in \cite{Dotto-Le} following \cite{BreuilDiamond} \cite{EGS}. For $w\in S,$ let $R^{\min}_w$ denote the quotient of $R_{w}^{\Box,\psi^{-1}}$ introduced in \cite[\S 3]{BreuilDiamond} \cite[\S 6.5]{EGS}. By \cite[Lem.~3.4.1]{BreuilDiamond}, we have
\begin{itemize}
\item   $R^{\min}_w$ is a formal power series ring  in $3$-variables over $\cO$ if $w\nmid p$;

\item  $R^{\min}_w$ is a formal power series ring  in $(3+[F_w:\Q_p])$-variables over $\cO$ if $w| p,w\neq v.$
\end{itemize}
Let $R^{\min}$ denote $R_{v}^{\Box,\psi^{-1}} \widehat{\otimes}(\widehat{\otimes}_{w\in S}R^{\min}_w).$  Since $ R_{v}^{\Box,\psi^{-1}}$ is a formal power series ring  in $(3+3f)$-variables over $\cO,$ $R^{\min}$ is a formal power series ring  in $(3(|S|+1)+[F:\Q]+2f)$-variables over $\cO.$ Define $R^{\Box,\min}_{\overline{r},S \cup Q}:=R^{\Box,\psi^{-1}}_{\overline{r},S \cup Q}\otimes_{R^{\loc}}R^{\min}.$ Let $R^{\min}_{\overline{r},S\cup Q}$ denote the image of $R^{\psi^{-1}}_{\overline{r},S\cup Q}$ in $R^{\Box,\min}_{\overline{r},S\cup Q}.$ Let $R^{\min}_{\infty} \defn R^{\min}[\![x_1,\ldots,x_g]\!],$ a power series ring in $g$-variables over $R^{\min},$ with a surjective homomorphism
\[
R_{\infty}^{\min} \to R^{\min}_{\overline{r},S}.
\]
Let $\frak{m}^{\min}_{\infty}$ be the maximal ideal of $R^{\min}_{\infty}$. We have an $\OC$-algebra homomorphism $S_{\infty}\to R^{\min}_{\infty}$ such that $R^{\min}_{\infty}/\frak{a}_{\infty}\cong R^{\rm min}_{\overline{r},S}.$ 

The big patched module in the minimal case (constructed in \cite[\S 6]{Dotto-Le}) is a finite Cohen-Macaulay $R_{\infty}[\![\GL_2(\cO_L)]\!]$-module $M^{\min}_{\infty},$ which is finite projective over $S_{\infty}[\![\GL_2(\cO_L)]\!],$ such that the smooth admissible representation of $\GL_2(L)$ over $\F$ given by\footnote{$\pi_v^D(\overline{r})$  is denoted by $\pi_{\rm glob}(\brho)$ in \cite[\S 6]{Dotto-Le} for $\brho = \overline{r}_v^{\vee}$.}
\begin{equation}\label{eq:piv}
\pi^D_{v}(\overline{ r} ) \defn  (M^{\min}_{\infty})^{\vee}[ \frak{m}^{\min}_{\infty}]
 \end{equation}
has multiplicity free $\GL_2(\cO_L)$-socle, namely 
$
\soc_K(\pi^D_{v}(\overline{r})) = \bigoplus_{\sigma \in  \mathscr{D}(\overline{r}^{\vee}_v)} \sigma
$,
see Proposition \ref{thm-Le} below.

For $\s \in \cC_{Z,\psi},$ let
\begin{equation}
M^{\min}_{\infty}(\s) \defn \Hom^{\rm cont}_{ \GL_2(\cO_L)}(M^{\min}_{\infty }, \s^{\vee})^{\vee}.
\end{equation}
Then $M^{\min}_{\infty}$ is an exact covariant functor from $\cC_{Z,\psi}$ to the category of finitely generated $ R^{\min}_{\infty}$-modules. As in \eqref{eq:mod-minfty=Hom}, we have
\begin{equation}
(M^{\min}_{\infty}(\s) / \frak{m}^{\min}_{\infty} M^{\min}_{\infty}(\s))^{\vee} \cong \Hom_{\GL_2(\cO_L)}(\s,\pi^D_{v}(\overline{r})).
\end{equation}

It is expected that $\pi^{D}_{v}(\overline{r})$ should realize the hypothetical mod $p$ local Langlands correspondence for $\brho = \overline{r}_v^{\vee}$. Thus it is  important to understand the precise structure of $\pi^D_{v}(\overline{r})$. The following conjecture is taken from \cite{BP}.

\begin{conjecture}[\cite{BP}] \label{conj:BP}
Assume $\brho$ is generic in the sense of \cite[Def.~11.7]{BP}. Then $\pi_v^D(\overline{r})$ has finite length.  More precisely,
\begin{itemize}
\item[(i)] if $\brho$ is irreducible, then $\pi^D_{v}(\overline{r})$ is irreducible;
\item[(ii)] if $\brho$ is reducible, then $\pi^D_{v}(\overline{r})$ has length $f$, admitting a unique Jordan--H\"older filtration  as follows:
\[\pi_0\ \ligne\ \pi_1\ \ligne\ \cdots\ \ligne\  \pi_{f-1} \ \ligne\ \pi_f \]
where $\pi_0$ and $\pi_f$ are principal series explicitly determined by $\brho$, and $\pi_i$ is supersingular for $1\leq i\leq f-1$. Moreover,  $\pi_{v}^{D}(\overline{r}^{\rm ss} )= \pi_{v}^{D}(\overline{r})^{\rm ss}$.
\end{itemize}
\end{conjecture}

One of our main results  is to prove Conjecture \ref{conj:BP}(ii) in the case $f=2$ and $\brho$ is  nonsplit and strongly generic, see Theorem \ref{thm-main-f=2}.

In this minimal case, $\pi^D_{v}(\overline{ r} ) $ also satisfies a self-duality under the assumption of its Gelfand-Kirillov dimension.
\begin{theorem}\label{thm-selfdual-minimal}
Assume $\GKdim (\pi^D_{v}(\overline{ r} ) )\leq f$. We have an isomorphism of $\Lambda(G)$-modules \[
\pi^D_{v}(\overline{ r} )^{\vee} \otimes \overline{\psi}|_{F^{\times}_v} \circ\det \cong \EE^{2f}(\pi^D_{v}(\overline{ r} )^{\vee}).
\]
\end{theorem}
\begin{proof}
The proof of this result is similar to the proof of Theorem \ref{thm-selfdual}. Instead of copying the argument, we just point out how to modify the proof of Theorem \ref{thm-selfdual} in the minimal case.

To define $\pi^D_{v}(\overline{ r} ), $ \cite{Dotto-Le} uses some finite free $\cO$-module $V_w$ (which is denoted by $L_w$ in \cite[\S 6.5]{EGS} and whose reduction modulo $\varpi$ is denoted by $\overline{M}_w$ in \cite[\S 3.3]{BreuilDiamond}) with a smooth $U_w$-action at each $w\in S\backslash S_p.$  Moreover, $V_w$ is of $\cO$-rank one unless $\overline{r}|_{G_{F_w}}$ is irreducible. As in the proof of Theorem \ref{thm-selfdual}, for example to get an analogous isomorphism as \eqref{eqn--duality-coeff}, the point is to check  that $V_{w}$ (or even $\overline{V_{w}}$) is essentially self-dual, i.e. $(V_w)^{d} \cong V_w \otimes \chi$ for some character $\chi$.\footnote{In the case $\overline{r}|_{G_{F_w}}$ is irreducible and $D$ ramifies at $w$, we actually need a variant of this statement.} 
This is clear if $V_w$ is a free rank one $\cO$-module.
 
Assume for the rest of the proof that $\overline{r}|_{G_{F_w}}$ is irreducible. Then $V_w$ is defined as an $(\cO_D)_w^{\times}$-stable lattice of a $K$-type for the isomorphic class of $(r_w|_{I_{F_w}},0)$ (in the sense of \cite[\S 3.2]{BreuilDiamond}), where $r_w: G_{F_w} \to \GL_2(E)$ is any lift of $\overline{r}|_{G_{F_w}}.$ The reduction $\overline{V_w}$ is always irreducible in this case, see \cite[\S 3.3 Cas IV]{BreuilDiamond}. We note that $r_w$ is always essentially self-dual, i.e. $\Hom_E(r_w,E) \cong r_w\otimes (\det r_w)^{-1}.$ If $D$ splits at $w,$ by the uniqueness of $K$-type in this case (see \cite{Henniart}),   the $K$-type $V_w[1/p]$ is essentially self-dual. We deduce that $V_w$ is also essentially self-dual by the irreducibility of $\overline{V_w}.$ 

If $D$ ramifies at $w,$ we let ${\rm LL}({\rm WD}(r_w))$ denote the smooth admissible representation of $\GL_2(F_w)$ over $E$ which is associated to the Weil-Deligne representation of $r_w$ by the local Langlands correspondence. Let $\Pi_{D_w}$ be the smooth admissible representation of $D^{\times}_w$ over $E$ associated to  ${\rm LL}({\rm WD}(r_w))$ by the Jacquet-Langlands correspondence. We have either $(\Pi_{D_w})|_{F_w^{\times } (\cO_{D})^{\times}_w}$ is irreducible, or $(\Pi_{D_w})|_{F_w^{\times } (\cO_{D})^{\times}_w} = V_1 \oplus V_2$ is a direct sum of two irreducible $F_w^{\times } (\cO_{D})^{\times}_w$-representations such that $V_2$ is conjugate to $V_1$ by a uniformizer of $(D_w)^{\times},$ see for example \cite[\S 5.1.2]{Gee-Kisin}. A $K$-type $V_w[1/p]$ for $(r_w|_{I_{F_w}},0)$ in this case is a choice of an irreducible constituent of $(\Pi_{D_w})|_{F_w^{\times } (\cO_{D})^{\times}_w}.$ As recalled above, the reduction modulo $\varpi$ of any $(\cO_D)_w^{\times}$-stable lattice of a $K$-type is irreducible. So the $(\cO_D)_w^{\times}$-stable lattice of a $K$-type is unique up to homothety and up to  conjugacy by a uniformizer of $(D_w)^{\times}.$ It follows that the $\GL_2(L)$-representation $\pi_v^D(\overline{r})$ does not depend on the choice of $V_w.$
Since the Jacquet-Langlands correspondence is compatible with taking the contragredient, $(V_w[1/p])^{*}$ is a twist by $\chi\circ \det$ of a $K$-type for $(r_w|_{I_{F_w}},0),$ where $\chi: F_w^{\times } \to E^{\times}$ is some character and $\det$ denotes the reduced norm of $(D_w)^{\times}$, and $(V_w)^d$ is  a twist by $\chi^{\circ} \circ \det$ of an $(\cO_D)_w^{\times}$-stable lattice of a $K$-type for $(r_w|_{I_{F_w}},0)$ for some  $\chi^{\circ}: F_w^{\times } \to \cO^{\times}$. 

From the above discussion, it is easy to deduce an analogous isomorphism as (\ref{eqn--duality-coeff}). We then take the eigenspace on which the Hecke operator $T_w $ acts by $\a_w\in \F$ for $w\in S',$ where $S',$ $T_w$ and $\a_w$ are all introduced in \cite[\S 3.3]{BreuilDiamond}. The rest of the proof of Theorem \ref{thm-selfdual} goes through.
\end{proof}

\subsection{Main results in the minimal case}
\label{subsection:cyclic}

From now on we only consider minimal patching functors, so we drop the superscript $\min$ and write $(M_{\infty} , R_{\infty} ,\frak{m}_{\infty} )$ for $(M^{\min}_{\infty} , R^{\min}_{\infty}, \frak{m}^{\min}_{\infty}  )$ for the rest of this section. 

Let $G  = \GL_2(L),$ $K = \GL_2(\cO_L) \supset K_1  = 1+ p \mathrm{M}_2(\cO_L)$, and $Z_1$ be the center of $K_1$. Let $I$ (resp. $I_1$) denote the upper-triangular Iwahori (resp. upper-triangular pro-$p$-Iwahori) subgroup of $G.$ Also let $\Gamma=\F[\![K/Z_1]\!]/\fm_{K_1/Z_1}$ and $\tGamma=\F[\![K/Z_1]\!]/\fm_{K_1/Z_1}^2$.

Let $\brho \defn \overline{r}_v^{\vee}.$ We assume $\overline{r}_{v}$ (or equivalently $\brho$) is reducible nonsplit and strongly generic in the sense of Definition \ref{defn:strong-generic}. Let $R_{\brho}^{\Box, \psi}$ denote the framed deformation ring of $\brho$ with fixed determinant $\psi \e.$ By taking the dual, we have isomorphisms $R_{\brho}^{\Box, \psi} \cong R_v^{\Box, \psi^{-1}}.$ We now switch to the notation for various deformation rings of $\brho$ in \S \ref{section--pBT}. Let $J$ be a subset of $S_{\brho}.$   We write $R^{{\rm tame},\s_J}_{\infty} = R_{\brho}^{\Box, \psi, T_{J,\emptyset}}\otimes_{R_{\brho}^{\Box,\psi}} R_{\infty},$ where $\s_J \in \Serre$ is the Serre weight associated to $J$ and $T_{J,\emptyset}$ is defined in (\ref{equ--inj-env}). Denote by $R^{{\rm cris},\s_J}_{\infty} = R^{\Box, \psi, {\rm cris},\s_J}_{\brho}\otimes_{R_{\brho}^{\Box,\psi}} R_{\infty}.$ Let $\bar{\cI}^{\s_J}$ (resp. $\bar{\cI}^{{\rm tame},\s_J}$, resp. $\bar{\cI}^{\rm red}$) denote the defining ideal of $R^{{\rm cris},\s_J}_{\infty}/\varpi$ (resp. $R^{{\rm tame},\s_J}_{\infty}/\varpi$, resp. $R_{\infty}^{\rm red}/\varpi$) inside $R_{\infty}/\varpi.$

\begin{proposition}\label{prop-relation-ideals}
Let $\s_0  $ denote $\s_{\emptyset}.$ We have $\bar{\cI}^{{\rm tame},\s_0}+\bar{\cI}^{\rm red} = \bar{\cI}^{\s_0}.$
\end{proposition}
\begin{proof}
By Proposition \ref{prop::red--crys}, $\bar{\cI}^{\rm red} \subset  \bar{\cI}^{\s_0}.$ By  \cite[Thm.~7.2.1(4)]{EGS}, $\bar{\cI}^{{\rm tame}, \s_0} \subset \bar{\cI}^{\s_0}.$ Then applying Corollary \ref{cor::intersect-tang-space}, the conclusion follows from Lemma \ref{lemma::tang-ideal-relation} below by taking $R = R_{\infty}/\varpi,$ $\cI_0 = \bar{\cI}^{\s_0},$ $\cI_1 = \bar{\cI}^{{\rm tame} , \s_0}$ and $\cI_2 = \bar{\cI}^{\rm red}.$
\end{proof}

\begin{lemma}\label{lemma::tang-ideal-relation}
 Let $(R,\fm)$ be a noetherian local  $\F$-algebra. Let $\cI_0\subset \frak{m}$ be an ideal of $R$ such that $R/\cI_0$ is regular. Let $\cI_1,\cI_2\subseteq \cI_0$ be ideals of $R.$ Then
\[
\left(\frac{\frak{m}/\cI_1}{(\frak{m}/\cI_1 )^2}\right)^{\vee}\bigcap \left(\frac{\frak{m}/\cI_2}{(\frak{m}/\cI_2 )^2}\right)^{\vee} = \left( \frac{\frak{m}/\cI_0}{(\frak{m}/\cI_0 )^2}\right)^{\vee}
\]
if and only if  $\cI_1+\cI_2=\cI_0.$

\end{lemma}

\begin{proof}
The surjective homomorphism $R\to R/\cI_i$ induces an injection of $\F$-vector spaces
\[
\left(\frac{\frak{m}/\cI_i}{(\frak{m}/\cI_i )^2}\right)^{\vee}\into \left(\frak{m}/\frak{m}^2\right)^{\vee}
\]
with intersection
\[
\left(\frac{\frak{m}/\cI_1}{(\frak{m}/\cI_1 )^2}\right)^{\vee}\bigcap \left(\frac{\frak{m}/\cI_2}{(\frak{m}/\cI_2 )^2}\right)^{\vee}=\left(\frac{\frak{m}/(\cI_1+\cI_2)}{(\frak{m}/(\cI_1 +\cI_2) ^2}\right)^{\vee}.
\]
So if $\cI_1+\cI_2=\cI_0,$ this intersection is identical to $ \left( \frac{\frak{m}/\cI_0}{(\frak{m}/\cI_0 )^2}\right)^{\vee}.$

Conversely, we see that $R/\cI_0$ is a quotient ring of $R/(\cI_1+\cI_2)$ with the same embedding dimension, say $d$.  By (the proof of) \cite[Thm.~29.4(ii)]{Mat}, $R/(\cI_1+\cI_2)$ is a quotient of a regular local ring of Krull dimension $d$ over $\F$, say $A$. In particular, $A$ is a domain.  Since $R/\cI_0$  itself has Krull dimension $d$, the (surjective) composite  map
\[A\twoheadrightarrow R/(\cI_1+\cI_2)\twoheadrightarrow R/\cI_0\]
  has to be an isomorphism. In particular, we deduce $R/(\cI_1+\cI_2)\cong R/\cI_0$ and $\cI_1+\cI_2=\cI_0$.
\end{proof}

The main result of this section is the following.

\begin{theorem}\label{thm-main-cyclicity}
Assume $\brho$ is strongly generic. Then for any $\s \in \Serre$ the $R_{\infty}$-module $M_{\infty}({\PtG\sigma})$ is a (nonzero) cyclic $R_{\infty}$-module.
\end{theorem}

To prove the theorem we first recall some known results.

\begin{proposition}\label{thm-Le}
(i) If $\sigma\notin \mathscr{D}(\brho)$ then $M_{\infty}(\sigma)=0$. If $\s\in \Serre,$ the homomorphism $R_{\infty}\to \End_{\cO}( M_{\infty}(\s))$ factors through $R^{{\rm cris},\s}_{\infty}/\varpi,$ and $M_{\infty}(\s)$ is  free of rank one over $R_{\infty}^{{\rm cris},\sigma}/\varpi.$

(ii) For any $\s\in \Serre,$ the homomorphism $R_{\infty}\to \End_{\cO}( M_{\infty}(\PG\s))$ factors through $R^{\rm tame,\s}_{\infty}/\varpi,$ and $M_{\infty}(\PG\s)$ is a cyclic $R_{\infty}$-module.
\end{proposition}

\begin{proof}
(i) The first statement is the main result of \cite{Gee} (see also \cite[Cor.~5.4.5]{Gee-Kisin}). The cyclicity of $M_{\infty}(\sigma)$ for $\sigma\in\mathscr{D}(\brho)$ follows from \cite[Thm.~10.2.1]{EGS}.  Note that $M_{\infty}(\sigma)$ has Krull dimension $q+j$, which is equal to the dimension of $R_{\infty}^{{\rm cris},\sigma}/\varpi$. Since $R_{\infty}^{{\rm cris},\sigma}/\varpi$ is a domain, $M_{\infty}(\s)$ is a faithful $R_{\infty}^{{\rm cris},\sigma}/\varpi$-module, giving the result.

(ii) is \cite[Thm.~5.1]{Le}.
\end{proof}

\begin{proposition}\label{lem-M(ord)-cyclic}
Let $\s_0$ denote $\s_{\emptyset}.$ Then $M_{\infty}(\Theta^{\rm ord}_{\s_0})$ is a cyclic $R_{\infty}$-module.
\end{proposition}

\begin{proof}
By Nakayama's lemma, it suffices to show the dual of $M_{\infty}(\Theta^{\rm ord}_{\s_0})/\frak{m}_{\infty},$
\[
(M_{\infty}(\Theta^{\rm ord}_{\s_0})/\frak{m}_{\infty})^{\vee} = \Hom_K(\Theta^{\rm ord}_{\s_0}, \pi^D_{v}(\overline{r}))
\]
is of dimension one over $\F.$

By Proposition \ref{prop--ord--semisimple}, $\Ord_P\pi_v^D(\overline{r})$ is a semisimple $T$-representation. On the other hand, the description of $\mathscr{D}(\brho)$ (see \cite[\S11]{BP}) implies that $\JH(\Ind_I^K\chi_{\sigma_0})\cap\mathscr{D}(\brho)=\{\sigma_0\}$.   The result then follows from Proposition \ref{prop:Ord-semisimple} combined with Proposition \ref{thm-Le}(i).
\end{proof}

\begin{proof}[Proof of Theorem \ref{thm-main-cyclicity}]
By Nakayama's lemma, it is equivalent to show that $\dim_{\F}\Hom_{K}(\Proj_{\tGamma}\sigma,\pi_v^D(\overline{r}))=1 $ for any $\sigma\in\mathscr{D}(\brho)$.
We will check the conditions (a), (b), (c) of  Theorem \ref{thm-criterion}
 for $\pi=\pi_{v}^D(\overline{r})$, from which the result follows.

The condition (a) is  a consequence of Proposition \ref{thm-Le}(ii), together with general property of $D_0(\brho)$. For the condition (b), we use  the exact sequence
\[0\ra \pi_v^D(\overline{r})\ra M_{\infty}^{\vee}\To{\times (x_i)_{i\in I}} \bigoplus_{i\in I} M_{\infty}^{\vee}\]
where $(x_i)_{i\in I}$ is any finite set of generators of $\fm_{\infty}$. Since $M_{\infty}^{\vee}$ is injective in the category of smooth representations of $K/Z_1$ on $\cO$-torsion modules,  we deduce that $\Ext^1_{K/Z_1}(\sigma,\pi_v^D(\overline{r}))\neq0 $ only if $\Hom_{K}(\sigma,M_{\infty}^{\vee})\neq0$. Hence the condition (b) follows by Proposition \ref{thm-Le}(i).

It remains to check the condition (c) for $\s = \sigma_0$. By Proposition \ref{prop:Theta-twoparts} (with the notation therein)  and the  exactness of $M_{\infty}(-)$, we have a short exact sequence
\begin{equation}\label{equ-cyclicity-exact-seq}
0 \ra M_{\infty}(\Theta_{\sigma_0})\ra  M_{\infty}(\Theta_{\sigma_0}^{\rm ord}) \oplus M_{\infty} ( (\Theta_{\s_0})_{K_1})\ra M_{\infty}((\Theta_{\sigma_0}^{\rm ord})_{K_1})\ra0.
 \end{equation}
 Since $(\Theta_{\sigma_0}^{\rm ord})_{K_1}$ is a quotient of $\Ind_I^K \chi_{\s_0}$ (by Lemma \ref{lemma:Theta-ord-K1}) and $\JH(\Ind_I^K \chi_{\s_0}) \cap \Serre = \{\s_0\}$, we obtain isomorphisms
\[M_{\infty}\big((\Theta_{\sigma_0}^{\rm ord})_{K_1}\big) \xleftarrow{\sim} M_{\infty}(\Ind_I^K \chi_{\s_0}) \simto
  M_{\infty}(\s_0).\]
Note that $(\Theta_{\s_0})_{K_1}$ is a quotient of $\Proj_{\G} (\s_0).$ So by Proposition \ref{thm-Le}(ii), $M_{\infty}((\Theta_{\s_0})_{K_1})$ is cyclic over $R_{\infty}$ and the  ideal $\bar{\cI}\defn \Ann_{R_{\infty}/\varpi}(M_{\infty}((\Theta_{\s_0})_{K_1})) $ satisfies
\[
\bar{\cI}^{{\rm tame},\s_0} \subseteq \bar{\cI} \subseteq \bar{\cI}^{\s_0}.
\]
By Proposition \ref{lem-M(ord)-cyclic}, $M_{\infty}(\Theta^{\rm ord}_{\s_0})$ is cyclic over $R_{\infty}.$ Let $\bar{\cI}^{\ord,\s_0}$ denote the ideal $\Ann_{R_{\infty}/\varpi} \left( M_{\infty}(\Theta^{\rm ord}_{\s_0}) \right).$ Then it follows from Proposition \ref{thm-BD-Rord} (in the minimal case) and the structure of $\Theta^{\rm ord}_{\s_0}$ that
\[
\bar{\cI}^{\rm red} \subseteq \bar{\cI}^{\ord,\s_0} \subseteq \bar{\cI}^{\s_0}.
\]
By Proposition \ref{prop-relation-ideals}, we get
\[
\bar{\cI} + \bar{\cI}^{\rm ord,\s_0} = \bar{\cI}^{\s_0},
\]
so $M_{\infty}(\Theta_{\sigma_0})$ is cyclic over $R_{\infty}$ by applying  Lemma \ref{lemma-cyclic-CA} below to   (\ref{equ-cyclicity-exact-seq}).
\end{proof}

\begin{lemma}\label{lemma-cyclic-CA}
Let $(R,\fm)$ be a commutative noetherian local ring with $k=R/\fm$. Let   $\cI_0, \cI_1,\cI_2$ be ideals of $R$ such that $\cI_1,\cI_2\subset \cI_0\subset \fm$. Consider the natural surjective homomorphism
$R/\cI_1\oplus R/\cI_2\twoheadrightarrow R/\cI_0$. Then $\Ker(R/\cI_1\oplus R/\cI_2\twoheadrightarrow R/\cI_0)$ is a cyclic $R$-module if and only if $\cI_1+\cI_2=\cI_0.$
\end{lemma}
\begin{proof}
Let $M$ denote the $R$-module $\Ker(R/\cI_1\oplus R/\cI_2\twoheadrightarrow R/\cI_0).$ The short exact sequence
\[
0\to M\to R/\cI_1\oplus R/\cI_2 \to R/\cI_0 \to 0
\]
gives a long exact sequence
\begin{multline*}
 \Tor_1^R(R/\cI_1, k)\oplus \Tor_1^R(R/\cI_2,k)\To{\a} \Tor_1^R(R/\cI_0,k) \to M\otimes k\\
 \to (R/\cI_1\otimes k)\oplus (R/\cI_2\otimes k) \to R/\cI_0 \otimes k \to 0.
\end{multline*}
Hence $\dim_k M\otimes k = 1$ if and only if $\a$ is surjective. The assumption $\cI_i\subset \fm$ implies that there is a natural isomorphism $\Tor_1^R(R/\cI_i,k)\cong \cI_i\otimes_{R}k$, for $i\in \{0,1,2\}.$ Thus, $\alpha$ is surjective if and only if the natural morphism
\[(\cI_1\oplus \cI_2)\otimes_{R}k\ra \cI_0\otimes_R k\]
is surjective. By Nakayama's lemma, this is equivalent to $\cI_1+\cI_2=\cI_0.$
\end{proof}

Recall the $\F[\![I/Z_1]\!]$-module $W_{\chi,3} = \Proj_{I/Z_1}\chi/\fm_{I_1/Z_1}^3$ introduced at the beginning of \S\ref{section:Rep-II}. We have the following  consequence of Theorem  \ref{thm-main-cyclicity}.

\begin{corollary}\label{cor:multione-Iwahori}
(i) For any $\sigma\in\mathscr{D}(\brho)$, $\dim_{\F}\Hom_{K}(\Proj_{\tGamma}\sigma,\pi^D_{v}(\overline{r}))=1$.

(ii) For any $\chi\in \JH(\pi^D_{v}(\overline{r})^{I_1})$, $\dim_{\F}\Hom_I(W_{\chi,3},\pi_{v}^D(\overline{r}))=1$.
\end{corollary}
\begin{proof}
(i) is an equivalent statement of Theorem \ref{thm-main-cyclicity} and (ii) follows from (i) by using Proposition \ref{prop-dim=1-equivalent}.
\end{proof}

\begin{corollary}\label{cor::m^2-torsion}
We have $\pi_v^D (\overline{r})[\frak{m}_{K_1/Z_1}^2] = \wt{D}_0 (\brho).$
\end{corollary}
\begin{proof}
By Proposition \ref{prop-BP-13.1}, we have \[\mathscr{D}(\brho)\cap\JH\big(\wt{D}_0(\brho)/\soc_K\wt{D}_0(\brho)\big)=\emptyset.\] By the proof of \cite[Lem.~9.2]{Br14},  the inclusion $\soc_K\wt{D}_0(\brho)=\oplus_{\sigma\in \mathscr{D}(\brho)}\sigma\subseteq \pi_v^D(\overline{r})$ extends to an inclusion $\wt{D}_0 (\brho) \subseteq \pi_v^D (\overline{r}),$ hence an inclusion $\wt{D}_0 (\brho) \subseteq \pi_v^D (\overline{r})[\frak{m}_{K_1/Z_1}^2].$  Alternatively, we may argue as in the proof of Lemma \ref{lemma:breuil} below, using Proposition \ref{prop-reducible-dimofExt}(ii).
  On the other hand, by Proposition \ref{prop-BP-13.1} and Corollary \ref{cor:multione-Iwahori}(i), the latter inclusion must be an equality.
\end{proof}

Recall from \cite[Cor.~5.3.5]{BHHMS} the following important control theorem of Gelfand-Kirillov dimension.

\begin{theorem}\label{thm:criterion-BHHMS}
Let $\pi$ be a smooth admissible representation of $I/Z_1$ over $\F$. If for each character $\chi$ such that $\Hom_I(\chi,\pi)\neq0$, the natural morphism
\[\Hom_{I}(\chi,\pi)\ra \Hom_I(W_{\chi,3},\pi)\]
is an isomorphism. Then the Gelfand-Kirillov dimension of $\pi$ is at most $f.$
\end{theorem}

Combining Theorem \ref{thm:criterion-BHHMS} with Corollary \ref{cor:multione-Iwahori}, Theorem \ref{thm::GN} and Theorem \ref{thm-selfdual-minimal}, we deduce the following result.

\begin{theorem}\label{thm:main-flat}
We make the following assumptions on $\overline{r}:$
\begin{enumerate}
\item[(a)]$\overline{r}|_{G_{F(\sqrt[p]{1})}}$ is absolutely irreducible;  
\item[(b)] for $w\in S\backslash S_p,$ the framed deformation ring $R^{\Box, \psi^{-1}}_w$ is formally smooth; 
\item[(c)] for $w\in S_p \backslash \{v \},$ $\overline{r}|_{I_{F_w}}$ is generic   in the sense of \cite[Def.~11.7]{BP};
\item[(d)] $ \overline{r}_v$ is reducible nonsplit and strongly generic in the sense of Definition \ref{defn:strong-generic}.
\end{enumerate}
Then the following statements hold:
\begin{enumerate}
\item[(i)] $\GKdim (\pi_{v}^{D}(\overline{r}))= f$ and $M_{\infty}$ is a flat $R_{\infty}$-module; 
\item[(ii)] There is an isomorphism of $\Lambda(G)$-modules $(\pi_v^D(\overline{r}))^{\vee} \otimes \psi|_{F_v^{\times}} \circ\det \cong \EE^{2f}\big((\pi_{v}^D(\overline{r}))^{\vee} \big).$
\end{enumerate}
\end{theorem}

\begin{remark}
We expect that the analog of Theorem \ref{thm:main-flat} remains true in the non-minimal case, i.e. $\dim_G(\pi)=f$ for the representation $\pi$ defined in (\ref{def-rep-pi}). This is the case when $\brho$ is semisimple and sufficiently generic, see \cite[\S8]{BHHMS}.
\end{remark}

We record the following (well-known) consequence of Theorem \ref{thm:main-flat}.  Let $x:R_{\infty}\ra \cO'$  be a local morphism of $\cO$-algebras, where $\cO'$ is the ring of integers of a finite extension $E'$ over $E$. Set
\[\Pi(x)^{0}\defn\Hom_{\cO'}^{\rm cont}(M_{\infty}\otimes_{R_{\infty},x}\cO',\cO')\]
and $\Pi(x)\defn\Pi(x)^0\otimes_{\cO'} E'$.
\begin{corollary}\label{cor:Pi(x)}
$\Pi(x)$ is a nonzero admissible unitary Banach representation of $G$ over $E'$ with $G$-invariant unit ball $\Pi(x)^{0}$ which lifts $\pi_{v}^D(\overline{r})\otimes_{\F}\F'$, where $\F'$ denotes the residue field of $\cO'$.
\end{corollary}
\begin{proof}
Since $M_{\infty}$ is flat over $R_{\infty}$ by Theorem \ref{thm:main-flat}(i), $M_{\infty}\otimes_{R_{\infty},x}\cO'$ is $\cO'$-flat by base change. So $\Hom_{\cO'}^{\rm cont}(M_{\infty}\otimes_{R_{\infty},x}\cO',\cO')$ is nonzero and $\cO'$-torsion free, hence  $\Pi(x)$ is nonzero by \cite[Thm.~1.2]{ST}.  The last assertion follows easily from \cite[Prop.~2.9]{Pa15}.
\end{proof}

\section{Homological algebra}\label{section-HA}
In this and the next section, we prove our second main result, namely with the notation in \S\ref{section-patching}, the $\GL_2(L)$-representation $\pi_v^D(\overline{r})$   (in the minimal case) is finitely generated by its $K_1$-invariants and, if $f=2$, has length $3$ as in Conjecture \ref{conj:BP}(ii).
This section contains some preliminary results.

\subsection{An enveloping algebra}
\label{subsection:envelope} 
Let $\overline{\mathfrak{g}}$ be the graded Lie algebra (labelled by $\Z_{\geq0}$) defined as follows:
\[\overline{\mathfrak{g}}=\F e\oplus \F f\oplus \F h\]
with $e$, $f$ in degree $1$, $h$ in degree $2$ and relations
\begin{equation}\label{eq:lie-relation}[e,f]=h,\ \ [h,e]=[h,f]=0. \end{equation}
Let $U(\overline{\mathfrak{g}})=U_{\F_p}(\overline{\mathfrak{g}})$ denote the universal enveloping algebra of $\overline{\fg}$. It is a graded algebra, with the induced degree function from above. As a consequence of the Poincar\'e-Birkhoff-Witt theorem, $\Ug$ is a domain.

The following lemma is obvious.
\begin{lemma}\label{lemma:Ug/h}
$h$ lies in the center of $\Ug$ and $\Ug/(h)$ is   isomorphic to $\F[e,f]$, the commutative polynomial ring with   variables $e, f$.
\end{lemma}

\begin{lemma}\label{lemma:Ug=regular}
$\Ug$ is a regular algebra  of global dimension $3$ in the sense of \cite[Eq. (0.1)]{AS}. In particular, $\Ug$ is an Auslander regular algebra.
\end{lemma}
\begin{proof}
For the first assertion, see  \cite[Eq. (0.3)]{AS}.  The second assertion is a consequence of the first, see e.g. \cite[Cor.~6.2]{Lev}; alternatively, it is a special case of  \cite[Thm.~III.3.4.6(6)]{LiO}.
\end{proof}
As a consequence of Lemma \ref{lemma:Ug=regular}, the concepts introduced in Appendix \S\ref{section:appendix} are applicable.

If $M=\oplus_{n\geq 0}M_n$ is a finitely generated graded $\Ug$-module, the \emph{Hilbert series} of $M$ is by definition the series
\[h_M(t):=\sum_{n\geq 0}(\dim_{\F}M_n)t^n.\]
This is an additive function on the Grothendieck group of finitely generated graded $\Ug$-modules.
We denote by $M(a)$ the shifted graded module defined by $M(a)_n=M_{n+a}$, with the convention $M_n=0$ if $n<0$. It is clear that $h_{M(a)}(t)=t^{-a}h_{M}(t)$.  In \cite[p. 342]{ATV}, the order of pole of $h_M(t)$ at $t=1$ is called the $gk$-dimension of $M$. It follows from
\cite[Thm.~4.1]{ATV} 
that this notion coincides with our notion of Gelfand-Kirillov dimension in Appendix \S\ref{subsection-duality}.

\begin{lemma}\label{lem:h-Ug}
We have $h_{\Ug}(t)=\frac{1}{(1-t)^2(1-t^2)}$.
\end{lemma}
\begin{proof}
It is a special case of \cite[(2.8)]{ATV}.
\end{proof}

\begin{lemma}\label{lemma:criterion-exact}
Let $G_{\bullet}$ be a chain complex of free $\Ug$-modules of length $n$,
\[G_{\bullet}:\ \  0\ra G_n\ra \cdots \ra G_1\ra G_0\ra0,\]
where $n\in\{2,3\}$.
Assume  the following conditions hold:
\begin{enumerate}
\item[(a)] $H_i(G_{\bullet})=0$ for $i\neq 0,1$;
\item[(b)] the Gelfand-Kirillov dimension of $H_0(G_{\bullet})$ is equal to $3-n$;
\item[(c)] the order of pole of $ \sum_{i=0}^{n}(-1)^ih_{G_i}$ at $t=1$ is  equal to $3-n$.
\end{enumerate}
Then $H_1(G_{\bullet})=0$ and $G_{\bullet}$ is a resolution of $H_0(G_{\bullet})$.
\end{lemma}
\begin{proof}
Assume $H_1(G_{\bullet})\neq 0$. Then $H_1(G_{\bullet})$ has  projective dimension $\leq n-1$, and  so  $H_1(G_{\bullet})$ has grade $\leq n-1$; see \eqref{eq:app-grade} for the notion of grade.  Hence, $H_1(G_{\bullet})$ has Gelfand-Kirillov dimension  $\geq 3-(n-1)=4-n$.
 On the other hand, we have an equality
\[\sum_{i}(-1)^ih_{G_i}=\sum_{i}(-1)^ih_{H_i(G_{\bullet})}.\]
By (a), (b) and the discussion before Lemma \ref{lem:h-Ug},
 the order of pole of RHS at $t=1$ is equal to $4-n$, while the one of LHS is equal to $3-n$ by (c),   a contradiction. This implies  $H_1(G_{\bullet})=0$  and  so $G_{\bullet}$ is a resolution of $H_0(G_{\bullet})$.
\end{proof}

The following lemma is an easy computation.
\begin{lemma}\label{lemma:comm-relation}
The following relations hold in $U(\overline{\mathfrak{g}})$:
\begin{align}
&\label{eq:comm-relation0} (ef)(fe)=(fe)(ef)\\
&\label{eq:comm-relation1} e^3f- fe^3=3e^2h\\
&\label{eq:comm-relation2}  ef^3-f^3e=3f^2h.
\end{align}
\begin{proof}
 Since $eh=he$ and $fh=hf$, we have $(ef)h=h(ef)$, and the equality  \eqref{eq:comm-relation0} follows.

Prove \eqref{eq:comm-relation1}. We have
$e^2f=e(fe+h)=efe+eh$, so that (using $eh=he$)
\[e^3f=e^2fe+e^2h= (efe+eh)e+e^2h=(ef+2h)e^2=(fe+3h)e^2=fe^3+3e^2h\]
giving the result. The equality \eqref{eq:comm-relation2} is checked in a similar way.
\end{proof}

\end{lemma}

\textbf{Convention}: In the statements below, our convention is that the differential map $d_i$ sends an element  of $G_{i}$, say $\underline{v}=(v_1,...,v_r)$ (with $r=\mathrm{rank}_{\Ug}(G_{i})$) to  $\underline{v}$ multiplied by the matrix of $d_i$ from the right.  In this way, $d_i$ is a morphism of  left $\Ug$-modules. As a consequence, the composition of differentials, say $d_i\circ d_{i+1}$, sends $\overline{v}$ to $(\un{v} A_{i+1})A_i$, where $A_i$ is the matrix form of $d_i$.

\begin{lemma}\label{lemma:complex-Koszul}
There exist  chain complexes of graded $\Ug$-modules
\begin{align}
&\label{eq:complex-Ug-Koszul-e}0\ra \Ug(-3)\overset{(-h,e)}{\lra}\Ug(-1)\oplus \Ug(-2)\overset{\binom{e}{h}}{\lra}\Ug\ra0\\
& \label{eq:complex-Ug-Koszul-f}0\ra \Ug(-3)\overset{(-h,f)}{\lra}\Ug(-1)\oplus \Ug(-2)\overset{\binom{f}{h}}{\lra} \Ug\ra0\\
&\label{eq:complex-Ug-Koszul-0}0\ra \Ug(-4)\overset{(-fe,ef)}{\lra} \Ug(-2)\oplus \Ug(-2)\overset{\binom{ef}{fe}}{\lra} \Ug\ra0.\end{align}
Moreover, in each case the complex defines a minimal free resolution of $H_0$ of the complex. Here, a minimal resolution means that the differential maps are zero mod $\Ug_{\geq 1}$.
\end{lemma}
\begin{proof}
It is clear from \eqref{eq:lie-relation} and \eqref{eq:comm-relation0} that the complexes in the lemma are well-defined and minimal. We need to show that they are exact at degrees $1,2$ (with the term $\Ug$ in degree $0$). To do this it is enough to check the assumptions of Lemma \ref{lemma:criterion-exact} with $n=2$. We do it for the complex \eqref{eq:complex-Ug-Koszul-e}, the other cases being analogous. Let $G_{\bullet}$ denote the complex \eqref{eq:complex-Ug-Koszul-e}. It is clear that $H_2(G_{\bullet})=0$,  because $\Ug$ is a domain. Using Lemma \ref{lemma:Ug/h}, we see that  $H_0(G_{\bullet})=\Ug/(e,h)$ is commutative and isomorphic to $\F[f]$ (polynomial ring in one variable $f$), hence it has Gelfand-Kirillov dimension $1$.
Finally, it is easy to compute using Lemma \ref{lem:h-Ug}
\[\sum_{i=0}^2(-1)^ih_{G_i}(t)=\frac{1-t-t^2+t^3}{(1-t)^2(1-t^2)}=\frac{1}{1-t},\]
whose order of pole at $t=1$ is equal to $1$.  We then conclude by Lemma \ref{lemma:criterion-exact}.
\end{proof}

\begin{remark}
Note that the complexes in Lemma \ref{lemma:complex-Koszul} are all of \emph{Koszul type}, see \S\ref{subsection:Koszul}.
\end{remark}

Next, we construct free resolutions of some $\Ug$-modules of finite length.

\begin{lemma}\label{lemma:complex-type-e}
Let $\fa$ be the left ideal of $\Ug$ generated by $e,h,f^3$. Then $\Ug/\fa$ is a $\Ug$-module of length $3$, spanned over $\F$ by $1,f,f^2$. It admits a minimal free resolution $G_{\bullet}\ra \Ug/\fa\ra0$, where
\begin{equation}\label{eq:complex-type-e}G_{\bullet}:\ \ \ 0\lra G_3\overset{d_3}{\lra} G_2\overset{d_2}{\lra} G_1\overset{d_1}{\lra} G_0\lra0\end{equation}
with \[G_0=\Ug,\ \ G_3=\Ug(-6),\]
\[  G_1=\Ug(-1)\oplus \Ug(-2)\oplus\Ug(-3),\]
\[G_2=\Ug(-3)\oplus \Ug(-4)\oplus \Ug(-5),  \]
and the differentials $d_i$ are described as follows
\[d_3=\left(\begin{array}{ccc}-f^3 & h & e\end{array}\right),\ \ \
d_2=\left(\begin{array}{ccc}-h & e & 0 \\ -f^3 & -3f^2 & e \\0 & f^3 & -h\end{array}\right),\ \ \
d_1=\left(\begin{array}{c}e \\ h \\f^3\end{array}\right).\]

Moreover, the complex \eqref{eq:complex-Ug-Koszul-e} is a subcomplex of $G_{\bullet}$ and each term is a direct summand of $G_i$.
\end{lemma}

\begin{proof}
Using \eqref{eq:lie-relation} and Lemma \ref{lemma:comm-relation}, it is direct to check that $d_1\circ d_2=d_2\circ d_3=0$, i.e. $(G_{\bullet})$ is a complex. Write $G_{\bullet}'$ for the complex \eqref{eq:complex-Ug-Koszul-e}. It is clear that $G_{\bullet}'$ is a subcomplex of $G_{\bullet}$ and that $G_{i}'$ is a direct summand of $G_{i}$ for $0\leq i\leq 3$ (here we set $G_{3}':=0$).  To see that $G_{\bullet}$ is acyclic (except at degree $0$), we apply Lemma \ref{lemma:criterion-exact}  to the quotient complex $G_{\bullet}'':=G_{\bullet}/G_{\bullet}'$, with degrees being induced from $G_{\bullet}$, i.e.
\[0\ra G_{3}''\ra G_2''\ra G_1''\ra G_0''=0\ra 0.\]  We may check as in the proof of Lemma \ref{lemma:complex-Koszul} that $H_{3}(G_{\bullet}'')=H_2(G_{\bullet}'')=0$. Combined with Lemma \ref{lemma:complex-Koszul}, the long exact sequence associated to $0\ra G_{\bullet}'\ra G_{\bullet}\ra G_{\bullet}''\ra0$ then implies that
$H_3(G_{\bullet})=H_2(G_{\bullet})=0$. Since $H_0(G_{\bullet})=\Ug/\fa$ has Gelfand-Kirillov dimension $0$ and the order of pole of $\sum_{i=0}^3(-1)^ih_{G_{i}}(t)$  at $t=1$ is also equal to $0$ (by a direct computation), we conclude by Lemma \ref{lemma:criterion-exact} (with $n=3$).
\end{proof}

In a completely similar way, we have the following lemma.

\begin{lemma}\label{lemma:complex-type-f}
Let $\fa$ be the left ideal of $\Ug$ generated by $f,h,e^3$. Then $\Ug/\fa$ is a $\Ug$-module of length $3$, spanned over $\F$ by $1,e, e^2$. It admits a minimal free resolution $G_{\bullet}\ra \Ug/\fa\ra0$, where
\begin{equation}\label{eq:complex-type-f}
G_{\bullet}: \ \  0\lra G_3\overset{d_3}{\lra} G_2\overset{d_2}{\lra} G_1\overset{d_1}{\lra} G_0\lra0\end{equation}
with 
\[G_0=\Ug,\ \ \ G_3=\Ug(-6)\]
 \[G_1=\Ug(-1)\oplus \Ug(-2)\oplus\Ug(-3)\]
\[G_2=\Ug(-3)\oplus \Ug(-4)\oplus \Ug(-5)  \]
and the differentials $d_i$ are described as follows
 \[d_3=\left(\begin{array}{ccc}-e^3 & h & f\end{array}\right),\ \ \
d_2=\left(\begin{array}{ccc}-h & f & 0 \\ -e^3 & 3e^2 & f \\0 & e^3 & -h\end{array}\right),\ \ \
d_1=\left(\begin{array}{c}f \\ h \\e^3\end{array}\right).\]

Moreover, the complex \eqref{eq:complex-Ug-Koszul-f} is a subcomplex of $G_{\bullet}$ and each term is a direct summand of $G_i$.
\end{lemma}

Next, we consider another   quotient of $\Ug$ of finite length.

\begin{lemma}\label{lemma:complex-type-0}
Let $\fa$ be the left ideal of $\Ug$ generated by $e^3,ef,fe,f^3$. Then $\Ug/\fa$ is a $\Ug$-module of length $5$, spanned by $1,e,e^2,f,f^2$. It admits a minimal free resolution $G_{\bullet}\ra \Ug/\fa\ra0$, where
\begin{equation}\label{eq:complex-type-0}G_{\bullet}:\ \ \ 0\lra G_3\overset{d_3}{\lra} G_2\overset{d_2}{\lra} G_1\overset{d_1}{\lra} G_0\lra0\end{equation}
with \[G_0=\Ug,\ \ \ G_3=\Ug(-6)\oplus \Ug(-6)\]
\[ G_1=\Ug(-3)\oplus\Ug(-2)\oplus \Ug(-2)\oplus\Ug(-3)\]
\[G_2=\Ug(-5)\oplus\Ug(-4)\oplus \Ug(-4)\oplus\Ug(-4)\oplus \Ug(-5),  \]
and the differentials $d_i$ are described as follows
\[ {\small d_3=\left(\begin{array}{ccccc}f & -h & -e^2 & 0 & 0 \\0 & 0 & -f^2 & -h & e\end{array}\right)\!,\  
d_2=\left(\begin{array}{cccc}h & -e^3 & e^3 & 0 \\f & 2e^2 & -3e^2 & 0 \\ 0 & -fe & ef & 0 \\0 & -3f^2 & 2f^2 & e \\ 0 & -f^3 & f^3 & h\end{array}\right)\!,\  
d_1=\left(\begin{array}{c}e^3 \\ ef \\fe \\f^3\end{array}\right)\!. }\]

Moreover, the complex \eqref{eq:complex-Ug-Koszul-0} is a subcomplex of $G_{\bullet}$ and each term is  a direct summand of $G_i$.
\end{lemma}
\begin{proof}
The proof is similar to that of Lemma \ref{lemma:complex-type-e}.
\end{proof}
\begin{remark}
It is easy to see that there is a short exact sequence of left $\Ug$-modules
\[0\ra \Ug/(e^3,ef,fe,f^3)\ra \Ug/(e,h,f^3)\oplus \Ug/(f,h,e^3)\ra \F\ra0.\]
\end{remark}

\subsubsection{$H$-actions} \label{subsubsection-H-action}
Recall that
\[H=\bigg\{\matr{a}00{d},\ a,d\in \F_{p^f}^{\times}\bigg\},\]
and for $i\in\cS\defn\{0,\cdots,f-1\}$, $\alpha_i:H\ra \F^{\times}$ denotes the character  sending $\smatr{a}00d$ to $(ad^{-1})^{p^i}$.
  Assume that $\Ug$ is equipped with an action of $H$ such that  for $g\in H$:
\[ge=\alpha_i(g) e, \ \ g f=\alpha_i^{-1}(g)f,\ \ g h=h.\]
Then the differential maps in the complexes  \eqref{eq:complex-Ug-Koszul-e}, \eqref{eq:complex-Ug-Koszul-f}, \eqref{eq:complex-Ug-Koszul-0}, \eqref{eq:complex-type-e}, \eqref{eq:complex-type-f} and \eqref{eq:complex-type-0} are actually $H$-equivariant.

For example, if we write $\Ug_{\chi}$ for $\Ug$ twisted by a character $\chi$ of $H$, then the complexes \eqref{eq:complex-Ug-Koszul-e}, \eqref{eq:complex-Ug-Koszul-f}, \eqref{eq:complex-Ug-Koszul-0} become
\begin{align*}&0\ra \Ug_{\alpha_i}(-3)\ra \Ug_{\alpha_i}(-1)\oplus\Ug_{\ide}(-2)\ra\Ug_{\ide}\ra0\\
&0\ra \Ug_{\alpha_i^{-1}}(-3)\ra\Ug_{\alpha_i^{-1}}(-1)\oplus \Ug_{\ide}(-2)\ra \Ug_{\ide}\ra0\\
&0\ra \Ug_{\ide}(-4) \ra \Ug_{\ide}(-2)\oplus \Ug_{\ide}(-2)\ra \Ug_{\ide}\ra0\end{align*}
while the terms of the complex \eqref{eq:complex-type-e} become
\[G_0=\Ug_{\ide},\ \ \ G_3=\Ug_{\alpha_i^{-2}}(-6),\]
\[G_1=\Ug_{\alpha_i}(-1)\oplus\Ug_{\ide}(-2)\oplus \Ug_{\alpha_i^{-3}}(-3),\]
\[G_2=\Ug_{\alpha_i}(-3)\oplus \Ug_{\alpha_i^{-2}}(-4)\oplus\Ug_{\alpha_i^{-3}}(-5).\]
(We leave to the reader for the complexes \eqref{eq:complex-type-f} and \eqref{eq:complex-type-0}).
It is clear that the embedding from \eqref{eq:complex-Ug-Koszul-e} to \eqref{eq:complex-type-e} is  also $H$-equivariant (see Lemma \ref{lemma:complex-type-e}).

From now on, assume $p>5$. From the above explicit description, we observe the following facts.
\begin{corollary}\label{cor:Min-i}
Let $G_{\bullet}$ be one of the complexes \eqref{eq:complex-Ug-Koszul-e}, \eqref{eq:complex-Ug-Koszul-f}, \eqref{eq:complex-Ug-Koszul-0}, \eqref{eq:complex-type-e}, \eqref{eq:complex-type-f}, \eqref{eq:complex-type-0}.
Then  $G_l$ has the form \[G_l=\bigoplus_{\chi} \big(\Ug_{\chi}(-a_{l,\chi})\big)^{r_{l,\chi}}\]
where $\chi$ runs over the characters of $H$, with the convention $r_{l,\chi}=0$ if $\Ug_{\chi}$ does not appear in the decomposition of $G_l$. Moreover,
for any fixed $\chi$,   if $r_{l,\chi}\neq0$ and $r_{l',\chi}\neq0$, then\footnote{\label{fn:p} Both the characters $\alpha_0^3$ and $\alpha_0^{-3}$ occur in the complex \eqref{eq:complex-type-0}. They are equal if $f=1$ and $p=7$, but one checks that the conclusion is still true in this case.}
\[a_{l',\chi}=a_{l,\chi}+2(l'-l).\]
\end{corollary}

\begin{corollary}\label{cor:G'-G''}
Let  $G_{\bullet}$ be one of the complexes \eqref{eq:complex-type-e}, \eqref{eq:complex-type-f}, \eqref{eq:complex-type-0}, and $G_{\bullet}'$ be one of \eqref{eq:complex-Ug-Koszul-e}, \eqref{eq:complex-Ug-Koszul-f}, \eqref{eq:complex-Ug-Koszul-0}  which embeds in $G_{\bullet}$. Let $G_{\bullet}''$ be the quotient complex $G_{\bullet}/G_{\bullet}'$.
If $\chi'$ (resp. $\chi''$) is a character of $H$ such that $\Ug_{\chi'}(-a')$ for some $a'\in\Z$ (resp. $\Ug_{\chi''}(-a'')$  for $a''\in\Z$)  appears in $G_{\bullet}'$ (resp. $G_{\bullet}''$), then $\chi''\chi'^{-1}\notin \{\ide,\alpha_j^{\pm1}, j\in\cS\}$ (this uses the assumption $p>5$).
\end{corollary}

\begin{corollary}\label{cor:complex-sum}
The direct sum of \eqref{eq:complex-Ug-Koszul-e} with  \eqref{eq:complex-Ug-Koszul-f} twisted by $\alpha_i$ is isomorphic to
\begin{equation}\label{eq:complex-sum}0\ra G(-3)\overset{(-\phi_2,\phi_1)}{\lra} G(-1)\oplus G(-2)\overset{\binom{\phi_1}{\phi_2}}{\lra} G\ra0 \end{equation}
where $G=\Ug_{\ide}\oplus \Ug_{\alpha_i}$, and $\phi_i\in\End_{\Ug}(G)$ are defined by
\[\phi_1=\matr{0}{f}{e}{0},\ \ \phi_2=\matr{h}00h. \]
\end{corollary}

\subsection{The representation $\tau_{\cJ}$}

Recall some notation from \S\ref{section:Rep-II}:  $I$ is the Iwahori subgroup of $K=\GL_2(\cO_L)$, $I_1$ is the pro-$p$ Iwahori subgroup, and $Z_1=Z\cap I_1$. Let $\fm\defn\fm_{I_1/Z_1}$ be the maximal ideal of the Iwasawa algebra $\F[\![I_1/Z_1]\!]$.  Recall that the residue field of $L$ is identified with $\F_q$, where $q=p^{f}$. We assume $p>5$.

We recall some results about the structure of $\gr_{\fm}(\F[\![I_1/Z_1]\!])$ from  \cite[\S5]{BHHMS}, which is based on  \cite{Laz} and \cite{Cl}.  By fixing a saturated $p$-valuation on $I_1/Z_1$ (cf. \cite[\S5.2]{BHHMS}), the associated graded Lie algebra $\gr(I_1/Z_1)$ has a unique structure of a graded $\F_p[\varepsilon]$-Lie algebra, where $\F_p[\varepsilon]$ is the graded polynomial algebra in $\varepsilon$ with $\varepsilon$ in degree $1$. More concretely, $\gr(I_1/Z_1)$ is isomorphic to $\F_q\otimes_{\F_p}\fg$, where
\[\fg\defn \F_p[\varepsilon]e\oplus \F_p[\varepsilon]f\oplus \F_p[\varepsilon]h\]
with $e$ and $f$ in degree $1$, $h$ in degree $2$ and relations
\[[e,f]=h,\ \ [h,e]=2\varepsilon e,\ \ [h,f]=-2\varepsilon f.\]
Consequently, the graded $\F_p$-Lie algebra  $\overline{\gr(I_1/Z_1)}\defn \gr(I_1/Z_1)\otimes_{\F_p[\varepsilon]}\F_p$ is isomorphic to $\F_q\otimes_{\F_p}\overline{\fg}$ where
\[\overline{\fg}\defn\fg\otimes_{\F_p[\varepsilon]}\F_p =\F_pe\oplus \F_pf\oplus \F_ph\]
with $e$ and $f$ in degree $1$, $h$ in degree $2$ and relations $[e,f]=h$, $[h,e]=[h,f]=0$.

Recall that $\F$ is a finite extension of $\F_p$ containing $\F_q$. By fixing an embedding $\kappa_0:\F_q\hookrightarrow \F$ and letting   $\kappa_i=\kappa_0\circ\Fr^i$,  the set of embeddings $\F_q\hookrightarrow \F$  is identified with $\cS=\{0,\dots,f-1\}$.
For $i\in\cS$, we define
$\overline{\fg}_i\defn \F\otimes_{\kappa_i,\F_q}\overline{\gr(I_1/Z_1)}$. Then we have a decomposition
\[\F\otimes_{\F_p}\overline{\gr(I_1/Z_1)}\cong \bigoplus_{i\in\cS}\overline{\fg}_i\]
and a canonical isomorphism $\overline{\fg}_i\cong\F\otimes_{\F_p}\overline{\fg}$.

On the other hand, we have an isomorphism, see \cite[(37)]{BHHMS},
\[\gr_{\fm}(\F_p[\![I_1/Z_1]\!])\cong U_{\F_p}(\F_q\otimes_{\F_p}\overline{\fg})\]
so that
\[\gr_{\fm}(\F[\![I_1/Z_1]\!])\cong \F\otimes_{\F_p}\gr_{\fm}(\F_p[\![I_1/Z_1]\!])\cong\bigotimes_{i\in\cS}U_{\F_p}(\overline{\fg}_i).\]
This isomorphism allows us to apply the results proved in last subsection.

For $i\in\cS$, let $e_i$, $f_i$, $h_i$ be the images of $1\otimes e$, $1\otimes f$, $1\otimes h$ under the isomorphism $\F\otimes_{\F_p}\overline{\mathfrak{g}}\cong \overline{\mathfrak{g}}_i$.
Since $H$ normalizes $I_1$ and $I_1/Z_1$, it acts on $\F[\![I_1/Z_1]\!]$, $\overline{\mathfrak{g}}$, and the elements $e_i$, $f_i$, $h_i$. It is easy to check that  for $g=\smatr{a}00{d}\in H$,
\[ge_i=\alpha_i(g) e_i, \ \ gf_i=\alpha_i^{-1}(g)f_i,\ \ gh_i=h_i,\]
where $\alpha_i$ is the character of $H$ as in \S\ref{subsubsection-H-action}.

\begin{proposition}\label{prop:tauJ-eps}
Given $\cJ\subset \cS$ and $\un{\varepsilon}\in\{\pm1\}^{\cJ}$, there exists an $I$-representation $\tau_{\cJ,\un{\varepsilon}}$ such that
\[\gr_{\fm}(\tau_{\cJ,\un{\varepsilon}}^{\vee})\cong \bigotimes_{i=0}^{f-1}   \Ugi_{\ide}/\fa_i \]
where
\begin{equation}\label{eq:tau-eps} \fa_i=\left\{\begin{array}{lll}
 (e_i^3,e_if_i,f_ie_i,f_i^3) & i\notin \cJ\\
(e_i,h_i,f_i^3) &i\in \cJ\ \mathrm{and}\ \varepsilon_i=+1\\
(f_i,h_i,e_i^3) & i\in \cJ\ \mathrm{and}\ \varepsilon_i=-1.
\end{array}\right.\end{equation}
Here, if $\cJ=\emptyset$, we make the convention that $\{\pm1\}^{\emptyset}=\emptyset$.
\end{proposition}
\begin{proof}
We give a constructive proof.

By \cite[Lem.~2.15(i)]{Hu10}, there exists a unique $I$-representation, denoted by $E_i^-(2)$, which is uniserial of length $3$ and whose socle filtration has graded pieces  $\ide, \alpha^{-1}_i,\alpha_i^{-2}$. By taking a conjugate action of the matrix $\smatr{0}1{p}0$, we obtain a unique $I$-representation $E_i^+(2)$, which is uniserial of length $3$  and whose socle filtration has graded pieces $\ide, \alpha_i,\alpha_i^{2}$.  To make the notation more transparent, we write
\begin{equation}\label{eq:Ei(2)}E_{\ide,\alpha_i,\alpha_i^2}=E_{i}^{+}(2),\ \ E_{\ide,\alpha_i^{-1},\alpha_i^{-2}}=E_{i}^-(2).\end{equation}
It is direct to check that
\[\gr_{\fm}((E_{\ide,\alpha_i,\alpha_i^2})^{\vee})\cong \Ugi_{\ide}/(e_i,h_i,f_i^3),\ \  \gr_{\fm}((E_{\ide,\alpha_i^{-1},\alpha_i^{-2}})^{\vee})\cong \Ugi_{\ide}/(f_i,h_i,e_i^3) \]
as graded $\gr_{\fm}(\F[\![I/Z_1]\!])$-modules. Moreover, taking an amalgam sum $E_{\ide,\alpha_i,\alpha_i^2}\oplus_{\ide}E_{\ide,\alpha_i^{-1},\alpha_i^{-2}}$, defined by:
\[0\ra \ide\ra E_{\ide,\alpha_i,\alpha_i^2}\oplus E_{\ide,\alpha_i^{-1},\alpha_i^{-2}}\ra E_{\ide,\alpha_i,\alpha_i^2}\oplus_{\ide}E_{\ide,\alpha_i^{-1},\alpha_i^{-2}}\ra0,\]
 its dual has graded module isomorphic to $\Ugi_{\ide}/(e_i^3,e_if_i,f_ie_i,f_i^3)$.

Now we let $\tau_{\cJ,\un{\varepsilon}}$ be the tensor product  $\bigotimes_{i\in\cS}\tau_{\cJ,\un{\varepsilon},i}$, where
\[\tau_{\cJ,\un{\varepsilon},i}=\left\{\begin{array}{ll}
E_{\ide,\alpha_i,\alpha_i^2}\oplus_{\ide}E_{\ide,\alpha_i^{-1},\alpha_i^{-2}} &i\notin \cJ\\
E_{\ide,\alpha_i,\alpha_i^2} & i\in \cJ\ \mathrm{and}\ \varepsilon_i=+1\\
E_{\ide,\alpha_i^{-1},\alpha_i^{-2}} &i\in \cJ\ \mathrm{and}\ \varepsilon_i=-1.
\end{array}\right.\]
To conclude, it suffices to prove that
\[\gr_{\fm}(\tau_{\cJ,\un{\varepsilon}}^{\vee})\cong\bigotimes_{i\in\cS}\gr_{\fm}(\tau_{\cJ,\un{\varepsilon},i}^{\vee})\]
which is a special case of  \cite[Lem.~1.1(i)]{AJL}.
\end{proof}

\begin{definition}\label{def:tau-J}
Define an $I$-representation $\tau_{\cJ}$ as
\[\tau_{\cJ}:=\bigoplus_{\un{\varepsilon}\in\{\pm1\}^{\cJ}}\Big(\tau_{\cJ,\un{\varepsilon}}\otimes \big(\prod_{\varepsilon_i=-1}\alpha_i^{-1}\big)\Big),\]
with the convention that $\tau_{\emptyset}:=\tau_{\emptyset,\emptyset}$.\end{definition}

\begin{remark}
The motivation to define $\tau_{\cJ}$   in this way comes from Proposition \ref{prop:tau-embeds} and Theorem  \ref{thm-generation-tau} in next section, which say  that $\pi_v^D(\overline{r})$ contains a suitable twist of $\tau_{\cJ}$ with $\cJ=J_{\brho}$, and is generated by it as a $\GL_2(L)$-representation.
\end{remark}

\begin{lemma}\label{lemma:gr(tauJ)}
  $\gr_{\fm}(\tau_{\cJ}^{\vee})$ has a tensor product decomposition
\[ \Big(\bigotimes_{i\notin \cJ}\Ugi_{\ide}/(e_i^3,e_if_i,f_ie_i,f_i^3)\Big)\bigotimes\Big(\bigotimes_{i\in \cJ}\big(\Ugi_{\ide}/(e_i,h_i,f_i^3)\oplus \Ugi_{\alpha_i}/(f_i,h_i,e_i^3)\big)\Big).\]
\end{lemma}
\begin{proof}
This is an easy check using Definition \ref{def:tau-J} and the proof of Proposition \ref{prop:tauJ-eps}. Note that $\alpha_i^{\vee}=\alpha_i^{-1}$ and vice versa.
\end{proof}

\begin{lemma}\label{lemma:socle-tauJ}
The $I$-socle of $\tau_{\cJ}$ is equal to \[\bigoplus_{J\subset\cJ}\Big(\prod_{i\in J}\alpha_i^{-1}\Big)\cong \bigotimes_{i\in\cJ}\big(\ide\oplus \alpha_{i}^{-1}\big).\] The Jordan--H\"older factors of $\tau_{\cJ}$ are the characters
$\prod_{i\in\cS}\alpha_i^{b_{J,i}}$,
where $J\subset \cJ$ and $(b_{J,i})\in \Z^{\cS}$ runs over
\[
\left\{\begin{array}{cl}
-2\leq b_{J,i}\leq 2& \mathrm{if}\  i\notin \cJ \\
0\leq b_{J,i}\leq 2& \mathrm{if}\ i\in \cJ\backslash J\\
-3\leq b_{J,i}\leq -1& \mathrm{if}\ i\in J\end{array}\right.\]
and  $\tau_{\cJ}$ is multiplicity free.
\end{lemma}
\begin{proof}
By construction and the claim in the proof of Proposition \ref{prop:tauJ-eps}, each $\tau_{\cJ,\un{\varepsilon}}$ is indecomposable with irreducible socle $\ide$ (the trivial character). The first assertion then follows from the definition of $\tau_{\cJ}$, up to a reformulation by setting $J=\{i\in\cJ, \varepsilon_i=-1\}$.

For a fixed $\un{\varepsilon}$, the Jordan--H\"older factors of $\tau_{\cJ,\un{\varepsilon}}$ are the characters
$\prod_{i\in \cS}\alpha_i^{a_i}$
where $a_i$ are integers such that
\[\left\{\begin{array}{cl}
-2\leq a_i\leq 2& \mathrm{if}\ i\notin \cJ \\
0\leq a_i\leq 2& \mathrm{if}\ i\in \cJ, \varepsilon_i=+1\\
-2\leq a_i\leq 0& \mathrm{if}\  i\in \cJ,\varepsilon_i=-1.\end{array}\right.\]
Twisting by $\alpha_{\varepsilon}$, we deduce that the Jordan--H\"older factors of $\tau_{\cJ,\un{\varepsilon}}\otimes(\prod_{\varepsilon_i=-1}\alpha_i^{-1})$ are the characters $\prod_{i\in \cS}\alpha_i^{b_i}$ where $b_i=a_i$ except when $i\in \cJ$ and $\varepsilon_i=-1$ in which case $b_i=a_i-1$. Explicitly, we have
\begin{equation} \label{eq:JH-tauJ-bi}
\left\{\begin{array}{cl}
-2\leq b_i\leq 2& \mathrm{if}\  i\notin \cJ \\
0\leq b_i\leq 2& \mathrm{if}\ i\in \cJ, \varepsilon_i=+1\\
-3\leq b_i\leq -1& \mathrm{if}\ i\in \cJ,\varepsilon_i=-1.\end{array}\right.\end{equation}
This gives the Jordan--H\"older factors of $\tau_{\cJ}$ in the statement (setting $J=\{j\in \cJ, \varepsilon_i=-1\}$).
Finally, the multiplicity freeness of $\tau_{\cJ}$  can be checked directly using \eqref{eq:JH-tauJ-bi}.
\end{proof}

\begin{remark}\label{rem:soc2-tauJ}
With the notation of Lemma \ref{lemma:socle-tauJ}, one checks that the character $\prod_{i\in\cS}\alpha_i^{b_{J,i}}$ lies in $\tau_{\cJ}[\fm^2]$ (where $J\subset \cJ$) if and only if
\[\left\{\begin{array}{cl}|b_{J,i}|\leq 1 &\mathrm{if}\ i\notin J\\  \ -2\leq b_{J,i}\leq -1 &\mathrm{if}\ i\in J\end{array}\right. \]
and there exists at most one $i$ such that $|b_{J,i}|=1$ if $i\notin J$, or $b_{J,i}=-2$ if $i\in J$.
\end{remark}

\begin{proposition}\label{prop:resolution-tauJ}
$\tau_{\cJ}^{\vee}$ admits a length $3f$ minimal resolution by projective $\F[\![I/Z_1]\!]$-modules, \[P_{\cJ,\bullet}\ra \tau_{\cJ}^{\vee}\ra0\]
satisfying the following property:  for each $0\leq l\leq 3f$,
$P_{\cJ,l}$ has a direct sum decomposition  \[P_{\cJ,l}=P_{\cJ,l}'\oplus  P_{\cJ,l}''\]
such that
\begin{enumerate}
\item[(a)] $P_{\cJ,l}'\cong \big(\bigoplus_{\chi}P_{\chi}\big)^{\binom{2f}{l}}$, where $\chi$ runs over the characters of $\mathrm{cosoc}_I(\tau_{\cJ}^{\vee})$; 
\item[(b)]$\Hom_I(P''_{\cJ,l},P_{\chi}/\fm^2)=0$ for any $\chi\in\mathrm{cosoc}_I(\tau_{\cJ}^{\vee})$.
\end{enumerate}
\end{proposition}
\begin{remark}
We don't require that the $P_{\cJ,l}'$ form a subcomplex of $P_{\cJ,\bullet}$.  \end{remark}

\begin{proof}
 First, taking tensor product of the complexes   \eqref{eq:complex-type-e},  \eqref{eq:complex-type-f} (twisted by $\alpha_i$ in this case), \eqref{eq:complex-type-0}, according to $i$ as in \eqref{eq:tau-eps}, we obtain a minimal projective resolution of $\gr_{\fm}(\tau_{\cJ,\un{\varepsilon}}^{\vee}\otimes \prod_{\varepsilon_i=-1}\alpha_i)$ of length $3f$, denoted by $G_{\cJ,\un{\varepsilon},\bullet}$. Using Corollary \ref{cor:Min-i}, one checks that  $G_{\cJ,\un{\varepsilon},\bullet}$ satisfies the following property, denoted by (\textbf{Min}): for each $0\leq k\leq 3f$, $G_{\cJ,\un{\varepsilon},l}$ has a decomposition
\[G_{\cJ,\un{\varepsilon},l}=\bigoplus_{\chi}\Big(\gr_{\fm}(\F[\![I_1/Z_1]\!])_{\chi}(-a_{l,\chi})\Big)^{r_{l,\chi}}\]
and whenever $r_{l,\chi}\neq0$ and $r_{l',\chi}\neq 0$, we have
\[a_{l',\chi}=a_{l,\chi}+2(l'-l).\]
Indeed, this property (\textbf{Min}) is preserved when taking tensor product. Here, we use the property that the characters $\chi$ with $r_{l,\chi}\neq0$ are of the form $\chi=\prod_{i\in\cS}\alpha_i^{b_i}$, with $b_i$ lying in a suitable range  such that $\chi=\chi'$ implies $b_i=b_i'$.\footnote{If $p=7$, it can happen that $\chi=\prod_{i\in\cS}\alpha_i^{3}=\prod_{i\in\cS}\alpha_i^{-3}=\chi'$, but one checks that (\textbf{Min}) still holds.}

Now, applying Lemma \ref{lemma:appendix-lift}, this graded resolution $G_{\cJ,\un{\varepsilon},\bullet}$ can be ``lifted'' to a filt-projective resolution $P_{\cJ,\un{\varepsilon},\bullet}$ of $\tau_{\cJ,\un{\varepsilon}}^{\vee}\otimes \prod_{\varepsilon_i=-1}\alpha_i$, which must be \emph{minimal} by Lemma \ref{lemma:appendix-minimal}.
Define $P_{\cJ,\bullet}$ to be the direct sum of $P_{\cJ,\un{\varepsilon},\bullet}$ over $\un{\varepsilon}\in\{\pm1\}^{\cJ}$, which is a minimal filt-projective resolution of $\tau_{\cJ}^{\vee}$.

To check that $P_{\cJ,l}$ satisfies the required property, we note that the conditions (a), (b) depend only on  $\mathrm{cosoc}_I(P_{\cJ,l})$ (not on the filtration of $P_{\cJ,l}$), so it suffices to prove the corresponding property for $\un{G_{\cJ,l}}$, the underlying $\gr_{\fm}(\F[\![I/Z_1]\!])$-module of $G_{\cJ,l}$ (i.e. forgetting the graded structure). By taking the tensor product of the complexes \eqref{eq:complex-Ug-Koszul-e}, \eqref{eq:complex-Ug-Koszul-f}, \eqref{eq:complex-Ug-Koszul-0}, according to $i$ as in \eqref{eq:tau-eps}, we obtain a subcomplex $G_{\cJ,\un{\varepsilon},\bullet}'$ of $G_{\cJ,\un{\varepsilon},\bullet}$, of length $2f$, such that  $G_{\cJ,\un{\varepsilon},l}'$ is a direct summand of $G_{\cJ,\un{\varepsilon},l}$ for all $0\leq l\leq 3f$, with the convention $G'_{\cJ,\un{\varepsilon},l}:=0$ if $l>2f$.
Taking direct sum over all $\un{\varepsilon}\in\{\pm1\}^{\cJ}$, we obtain a subcomplex $G_{\cJ,\bullet}'$ of $G_{\cJ,\bullet}$ such that $G_{\cJ,l}'$ is a direct summand of $G_{\cJ,l}$. The conditions (a), (b) can be checked directly using Lemma \ref{lemma:socle-tauJ}, Corollary \ref{cor:G'-G''} and Corollary \ref{cor:complex-sum}.
\end{proof}

\subsection{The representation $\lambda_{\chi}$}

Keep the notation of the last subsection.  

\begin{proposition}\label{prop:center}
The center of $\F[\![I/Z_1]\!]/\fm^3$ contains a subring  isomorphic to
\begin{equation}\label{eq:ring-R} \F\big[(x_i,y_i)_{0\leq i\leq f-1}\big]/(x_i,y_i)_{0\leq i\leq f-1}^2.\end{equation}
\end{proposition}

\begin{proof}
Since $e_if_i, f_ie_i\in \fm^2 \F[\![I/Z_1]\!]/\fm^3$, we may view them as elements in $\F[\![I/Z_1]\!]/\fm^3$. Set
\[x_i=e_if_i,\ \ y_i=f_ie_i.\]
It is clear that $x_i, y_i$ lie in the center of $\F[\![I/Z_1]\!]/\fm^3$, and $(x_i,y_i)^2=0$. 
\end{proof}

We denote by $R$ the ring \eqref{eq:ring-R}.
Since it is contained in the center of $\F[\![I/Z_1]\!]/\fm^3$, $R$ acts on any object in the category of $\F[\![I/Z_1]\!]/\fm^3$-modules, and any morphism in this category is $R$-linear.

Recall that for a character $\chi$ of $I$,
$P_{\chi}\defn\Proj_{I/Z_1}\chi$
denotes a projective envelope of $\chi$ in the category of pseudo-compact $\F[\![I/Z_1]\!]$-modules, and $W_{\chi,n}\defn P_{\chi}/\fm^n$ for $n\geq 1$. The structure of $W_{\chi,3}$ has been determined in \S\ref{section:Rep-II}. 
In particular, $[\gr_{\fm}^2W_{\chi,3}:\chi]=2f$.
Also recall that we have denoted by $\overline{W}_{\chi,3}$   the quotient of   $W_{\chi,3}$ by the direct sum of characters in $\gr^2_{\fm}W_{\chi,3}$ which are not isomorphic to $\chi$. Since this representation will be tentatively used in this and next section, we make the following definition.
\begin{definition}\label{defn:lambda}
Define $\lambda_{\chi}=  \overline{W}_{\chi^{\vee},3}$.
\end{definition}
Note that   the cosocle of $\lambda_{\chi}$ is $\chi^{\vee}$ by definition.
Moreover,   $\chi'\in\JH(\lambda_{\chi})$ if and only if $\chi'=\chi^{\vee}$ or $\chi'\in\mathscr{E}(\chi^{\vee})$, and
\[[\lambda_{\chi}:\chi']=\left\{\begin{array}{lll}
2f+1 & \mathrm{if}\ \chi'=\chi^{\vee} \\
1&\mathrm{if}\ \chi'\in\mathscr{E}(\chi^{\vee}).
\end{array}\right.\]   Here,  recall that $\mathscr{E}(\chi^{\vee})$ denotes the set of characters which have nontrivial extensions with $\chi^{\vee}$ (\S\ref{section:Rep-II}).

\begin{lemma}\label{lemma:R=End}
The action of $R$ on $\lambda_{\chi}$ (resp. $\lambda_{\chi}^{\vee}$) induces a ring isomorphism $R\cong \End_I(\lambda_{\chi})$ (resp. $R\cong \End_I(\lambda_{\chi}^{\vee})$) .
\end{lemma}

\begin{proof}
Since  $\dim_{\F}\Hom_I(\lambda_{\chi},\chi')=1$ if and only if $\chi'=\chi^{\vee}$, by d\'evissage we have
\[\dim_{\F}\End_I(\lambda_{\chi})\leq [\lambda_{\chi}:\chi^{\vee}]=2f+1.\] Since $\dim_{\F}R=2f+1$, it suffices to prove that $R\hookrightarrow \End_I(\lambda_{\chi})$, i.e. $R$ acts faithfully on $\lambda_{\chi}$. But this is clear by definition of $R$ and $\lambda_{\chi}$, because $e_if_i$ and $f_ie_i$ (for $i\in\cS$) induce  nonzero endomorphisms of $\lambda_{\chi}$ which are linearly independent over $\F$. The claim about $\End_I(\lambda_{\chi}^{\vee})$ follows from this and the general fact that $\End_I(\lambda_{\chi}^{\vee})\cong R^{op}=R$, where $R^{op}$ denotes the opposite ring of $R$.
\end{proof}

\begin{lemma}\label{lemma:HA-Ext1-vanish}
We have $\Ext^1_{I/Z_1}(\lambda_{\chi},\chi^{\vee})=0$.
\end{lemma}

\begin{proof}
The proof is similar to that of Corollary \ref{cor-Theta-noextension}, using the structure of $\lambda_{\chi}$ in Lemma \ref{lemma-bar-W}.
\end{proof}

Given $P$ a finitely generated $\F[\![I/Z_1]\!]$-module  and $\lambda$ a finite length $\F[\![I/Z_1]\!]$-module, we may consider
\[\Hom_{I}(P,\lambda^{\vee})^{\vee},\]
where $\vee$ denotes Pontryagin dual.  Note that  as a functor  $\Hom_I(-,-^{\vee})^{\vee}$  is   covariant  and right exact in both variables.
We will mostly be interested in the case when $P$ is projective and $\lambda$ is annihilated by $\fm^3$. Typical examples are $\lambda=\chi$ or $\lambda_{\chi}$ for some character $\chi$, in which case
the module $\Hom_I(P,\lambda^{\vee})^{\vee}$ carries naturally an action of $R$ through $\lambda$.
Moreover, if $P\ra P'$ and $\lambda\ra \lambda'$ are  morphisms of $\F[\![I/Z_1]\!]$-modules, then all the morphisms are $R$-linear in the following commutative diagram
\[\xymatrix{\Hom_I(P,\lambda^{\vee})^{\vee}\ar[r]\ar[d]& \Hom_I(P',\lambda^{\vee})^{\vee}\ar[d]\\
\Hom_I(P,\lambda'^{\vee})^{\vee}\ar[r]&\Hom_I(P',\lambda'^{\vee})^{\vee}.}\]

\begin{proposition}\label{prop:module-Pchi}
Let $\chi'$ be a character of $I$. Then as an $R$-module
\[\Hom_I(P_{\chi'},\lambda_{\chi}^{\vee})^{\vee}\cong\left\{\begin{array}{rll}
R & \mathrm{if}\ \chi'=\chi\\
\F & \mathrm{if}\ \chi'\in \mathscr{E}(\chi)\\
0& \mathrm{otherwise}.
\end{array}\right.\]
\end{proposition}
\begin{proof}
We observed that $[\lambda_{\chi}^{\vee}:\chi']=1$ if $\chi'\in \mathscr{E}(\chi)$, so $\dim_{\F}\Hom_I(P_{\chi'},\lambda_{\chi}^{\vee})^{\vee}=1$ by projectivity of $P_{\chi'}$. This treats  the second case. The third case is trivial.

It remains to treat the  case $\chi'=\chi$. The projectivity of $P_{\chi}$ implies that
\[\dim_{\F}\Hom_I(P_{\chi},\lambda_{\chi}^{\vee})^{\vee}=2f+1=\dim_{\F}R.\]
Therefore, it suffices to prove that $\Hom_I(P_{\chi},\lambda_{\chi}^{\vee})^{\vee}$ is a cyclic $R$-module.  By Lemma \ref{lemma-moduletheory} below, applied to
\[S=\F[\![I/Z_1]\!],\ \ R=R, \ \ P=P_{\chi},\ \ \lambda=\lambda_{\chi}\]
it suffices to prove that
\[\dim_{\F}\Hom_I\big(P_{\chi}, (\lambda_{\chi}\otimes_RR/\fm_R)^{\vee}\big)^{\vee}=1\]
or equivalently, $[\lambda_{\chi}\otimes_RR/\fm_R:\chi^{\vee}]=1$. But this is clear because $\lambda_{\chi}\otimes_RR/\fm_R$ is isomorphic to $P_{\chi^{\vee}}/\fm^2$ which is multiplicity free.
\end{proof}

\begin{lemma}\label{lemma-moduletheory}
Let $S, R$ be $\F$-algebras and assume that $R$ is a commutative noetherian local ring with maximal ideal $\fm_R$. Let $P$ be a left $S$-module and $\lambda$ be an $(S,R)$-bimodule. Assume that  $P$ is finite projective and $\lambda$ is finite dimensional over $\F$. Then there is an isomorphism
\[\Hom_{S}(P,\lambda^{\vee})^{\vee}\otimes_RR/\fm_R\simto \Hom_{S}(P,(\lambda\otimes_RR/\fm_R)^{\vee})^{\vee}. \]
\end{lemma}
\begin{proof}
One checks that the natural map $\Hom_{S}(P,\lambda^{\vee})^{\vee}\ra\Hom_{S}(P,(\lambda\otimes_RR/\fm_R)^{\vee})^{\vee}$  factors through
\[\Hom_S(P,\lambda^{\vee})^{\vee}\otimes_RR/\fm_R\ra \Hom_{S}(P,(\lambda\otimes_RR/\fm_R)^{\vee})^{\vee},\]
which  is clearly an isomorphism if $P$ is a finite free $S$-module, hence is also an isomorphism if $P$ is finite projective.  Note that the assumption  $\lambda$ is finite dimensional ensures $(\lambda^{\vee})^{\vee}\cong\lambda$.
\end{proof}

  \begin{proposition}\label{prop:map-Pchi}
 Let $\chi_1,\chi_2$ be characters of $I$. Consider a morphism $\beta:P_{\chi_1}\ra P_{\chi_2}$ and let
\begin{equation}\label{eq:beta-sharp}\beta^{\sharp}_{\chi}: \Hom_I(P_{\chi_1},\lambda_{\chi}^{\vee})^{\vee}\ra \Hom_I(P_{\chi_2},\lambda_{\chi}^{\vee})^{\vee}\end{equation}
be the induced morphism of $R$-modules. Then $\beta_{\chi}^{\sharp}$ has the form
{\begin{center} \label{eq:table}
\begin{tabu}{ |c|[1pt] c|c|c| }
\hline
$\beta^{\sharp}_{\chi}$& $\chi_2=\chi$ &  $\chi_2\in\mathscr{E}(\chi)$ &  \textrm{otherwise}    \\
 \tabucline[1pt]{-}
 $\chi_1=\chi$&$R\ra R$ & $R\ra \F$&$R\ra 0$  \\
\hline
$\chi_1\in\mathscr{E}(\chi)$&$\F\ra R$&$\F\ra \F$&$\F\ra0$ \\
\hline
\textrm{otherwise}&$0\ra R$&$0\ra \F$&$0\ra 0$ \\
\hline
\end{tabu}
\end{center}}
If moreover  $\chi_1,\chi_2\in \{\chi\}\cup\mathscr{E}(\chi)$, then  the following statements hold:

\begin{enumerate}
\item[(i)] if $\chi_2\neq \chi$, then $\beta_{\chi}^{\sharp}$ is nonzero if and only if  $P_{\chi_1}\overset{\beta}{\ra} P_{\chi_2}\twoheadrightarrow P_{\chi_2}/\fm^2$ is nonzero;   
\item[(ii)] if $\chi_2=\chi$, then $\beta_{\chi}^{\sharp}$ is nonzero if and only if $P_{\chi_1}\overset{\beta}{\ra} P_{\chi}\twoheadrightarrow P_{\chi}/\fm^3$ is nonzero; 
\item[(iii)] if $\chi_1,\chi_2\in \mathscr{E}(\chi)$ and $\chi_1\neq \chi_2$, then $\beta_{\chi}^{\sharp}$ is always zero.
\end{enumerate}
\end{proposition}
 \begin{proof}
 The form of $\beta_{\chi}^{\sharp}$ follows immediately from Proposition \ref{prop:module-Pchi}.

(i) By assumption $\chi_2\in \mathscr{E}(\chi)$. We first claim that the natural quotient map $P_{\chi_2}\twoheadrightarrow P_{\chi_2}/\fm^2$ induces an isomorphism
\[\Hom_{I}(P_{\chi_2},\lambda_{\chi}^{\vee})^{\vee}\simto \Hom_I(P_{\chi_2}/\fm^2,\lambda_{\chi}^{\vee})^{\vee}.\]
It is  surjective by right exactness of $\Hom_I(-,\lambda_{\chi}^{\vee})^{\vee}$.  Since $\Hom_{I}(P_{\chi_2},\lambda_{\chi}^{\vee})^{\vee}\cong \F$ by Proposition \ref{prop:module-Pchi}, it is enough to show that $\Hom_I(P_{\chi_2}/\fm^2,\lambda_{\chi}^{\vee})$ is nonzero. By assumption we have $\chi_2\in \mathscr{E}(\chi)$, i.e. $\Ext^1_{I/Z_1}(\chi_2,\chi)\neq0$. Hence, $P_{\chi_2}/\fm^2$ admits a two-dimensional quotient  isomorphic to $E_{\chi,\chi_2}$ which embeds in $\lambda_{\chi}^{\vee}$; here we recall that $E_{\chi,\chi_2}$ denotes the unique nonsplit extension of $\chi_2$ by $\chi$. This proves the claim.

The ``only if'' part follows directly from the claim. To prove the ``if'' part,   assume the composite morphism $P_{\chi_1}\ra P_{\chi_2}\ra P_{\chi_2}/\fm^2$ is nonzero. Then $\chi_1$ occurs in $P_{\chi_2}/\fm^2$ as a subquotient, i.e. either $\chi_1=\chi_2$ or $\chi_1\in\mathscr{E}(\chi_2)$. In the first case, $\beta$ induces a surjection $P_{\chi_1}\ra P_{\chi_1}\ra  P_{\chi_1}/\fm$, hence $\beta$ has to be surjective by Nakayama's lemma and consequently an isomorphism. In particular, $\beta_{\chi}^{\sharp}$ is nonzero. In the second case, we must have $\chi_1=\chi$ (for, otherwise, $\chi\in\mathscr{E}(\chi_1)$ and so $\chi\in\mathscr{E}(\chi_1)\cap \mathscr{E}(\chi_2)$, but this intersection is empty whenever $\chi_1\in\mathscr{E}(\chi_2)$),  and the image of $P_{\chi}\ra P_{\chi_2}\ra P_{\chi_2}/\fm^2$ is isomorphic to $\chi$. By the proof of the claim, we see that the inclusion $\chi\hookrightarrow P_{\chi_2}/\fm^2$ induces an isomorphism $\Hom_I(\chi,\lambda_{\chi}^{\vee})^{\vee}\simto\Hom_I(P_{\chi_2}/\fm^2,\lambda_{\chi}^{\vee})^{\vee}$, hence $\beta_{\chi}^{\sharp}$ is nonzero.

(ii) Since $\lambda_{\chi}^{\vee}$ is annihilated by $\fm^3$, the natural morphism
\[\Hom_I(P_{\chi},\lambda_{\chi}^{\vee})^{\vee}\ra \Hom_I(P_{\chi}/\fm^3,\lambda_{\chi}^{\vee})^{\vee}\]
is an isomorphism, which implies the ``only if'' part.  Moreover, since \[\dim_{\F}\Hom_I(P_{\chi}/\fm^3,\lambda_{\chi}^{\vee})^{\vee}=2f+1=[P_{\chi}/\fm^3:\chi],\] an argument by d\'evissage  shows that for any submodule $W$ of $P_{\chi}/\fm^3$, \[\dim_{\F}\Hom_I(W,\lambda_{\chi}^{\vee})^{\vee}=[W:\chi] \]
and the induced sequence
\[0\ra \Hom_{I}(W,\lambda_{\chi}^{\vee})^{\vee}\ra \Hom_I(P_{\chi}/\fm^3,\lambda_{\chi}^{\vee})^{\vee}\ra \Hom_I\big((P_{\chi}/\fm^3)/W,\lambda_{\chi}^{\vee}\big)^{\vee}\ra0\]
is exact.

Assume that $\overline{\beta}: P_{\chi_1}\ra P_{\chi}\ra P_{\chi}/\fm^3$ is nonzero. By the above discussion, to show that $\beta_{\chi}^{\sharp} $ is nonzero it suffices to show  $[\im(\overline{\beta}):\chi]\neq0$. By assumption, either $\chi_1=\chi$ or $\chi_1\in \mathscr{E}(\chi)$. The result is trivial if $\chi_1=\chi$. If $\chi_1\in\mathscr{E}(\chi)$, then one checks that the unique submodule of $P_{\chi}/\fm^3$ with cosocle $\chi_1$ contains $\chi$ as a submodule.   Indeed, the restriction of $P_{\chi}/\fm^3$ to $I_1$   is isomorphic to $\F[\![I_1/Z_1]\!]/\fm^3$, in which $e_i$ (or $f_i$) generates $f_ie_i$ (resp. $e_if_i$).

(iii) Since $\chi_2\neq\chi$, by (i) it suffices to prove that $P_{\chi_1}\ra P_{\chi_2}\ra P_{\chi_2}/\fm^2$ is zero. This is clear, because the assumption on $\chi_i$ implies that  $\chi_1$ does not occur in $P_{\chi_2}/\fm^2$ as a subquotient.
 \end{proof}

\subsubsection{Tangent space}

Let $\chi'\in\mathscr{E}(\chi)$ and $\beta: P_{\chi'}\ra P_{\chi}$ be a morphism of $\F[\![I/Z_1]\!]$-modules.  The proof of Proposition \ref{prop:map-Pchi}(ii) shows that if  $\overline{\beta}: P_{\chi'}\ra P_{\chi}\ra P_{\chi}/\fm^3$ is nonzero, then $\im(\overline{\beta})$ is nonzero and does not depend on the choice of $\beta$, and therefore neither does $\im(\beta_{\chi}^{\sharp})$, where $\beta_{\chi}^{\sharp}$ is the morphism   \eqref{eq:beta-sharp}.

\begin{definition}\label{def:tchi}
Let $t_{\chi'}$ denote the the image of $\F\ra R$ via the map
\[\beta_{\chi}^{\sharp}:\ \F\cong \Hom_I(P_{\chi'},\lambda_{\chi}^{\vee})^{\vee}\ra \Hom_I(P_{\chi},\lambda_{\chi}^{\vee})^{\vee}\cong R,\]
where $\beta: P_{\chi'}\ra P_{\chi}$ is any morphism which is nonzero when composed with $P_{\chi}\twoheadrightarrow P_{\chi}/\fm^3$.
\end{definition}

\begin{lemma}\label{lemma:tchi-independent}
 $\{t_{\chi'}:\chi'\in\mathscr{E}(\chi)\}$  generate the maximal ideal $\fm_R$. As a consequence,   $\{t_{\chi'}:\chi'\in\mathscr{E}(\chi)\}$ form an $\F$-basis of $\fm_R$.
\end{lemma}
\begin{proof}
There exists an exact sequence
\[\oplus_{\chi'\in\mathscr{E}(\chi)}P_{\chi'}\ra P_{\chi}\ra \chi\ra0.\]
Since $\Hom_{I}(-,\lambda_{\chi}^{\vee})^{\vee}$ is right exact,  the induced sequence
\[\oplus_{\chi'\in\mathscr{E}(\chi)}\Hom_I(P_{\chi'},\lambda_{\chi}^{\vee})^{\vee}\ra \Hom_I(P_{\chi},\lambda_{\chi}^{\vee})^{\vee}\ra \Hom_I(\chi,\lambda_{\chi}^{\vee})^{\vee}\ra0\]
is also exact. By Proposition \ref{prop:map-Pchi} and Definition \ref{def:tchi}, the last sequence is isomorphic to
\[\oplus_{\chi'} \F t_{\chi'}\ra R\ra \F\ra0,\]
proving the first claim.
 Since $\dim_{\F}\fm_R=2f=|\mathscr{E}(\chi)|$, the last claim follows (as $\fm_R^2=0$).
\end{proof}

\begin{lemma}\label{lemma:image-tchi}
Let $P$ be a finitely generated projective $\F[\![I/Z_1]\!]$-module. Let $\chi'\in \mathscr{E}(\chi)$ and $\beta: P_{\chi'}\ra P$ be a morphism such that $\im(\beta)\subset \fm P$. Then the image of $\beta_{\chi}^{\sharp}$ is contained in $t_{\chi'}\Hom_I(P,\lambda_{\chi}^{\vee})^{\vee}$.
\end{lemma}
\begin{proof}
We may assume $P$ is indecomposable, i.e. $P=P_{\chi''}$ for some character $\chi''$. If $\chi''= \chi'$, then   the assumption on $\im(\beta)$ implies that $\im(\beta)\subset \fm^2P_{\chi'}$ because  $\chi'$ does not occur in $\gr^1_{\fm}P_{\chi'}$.
By Proposition \ref{prop:map-Pchi}(i) this implies   $\beta_{\chi}^{\sharp}=0$  and the result trivially holds. In the rest, we assume $\chi''\neq \chi'$.

First assume $\chi''=\chi$, then the result follows from Proposition \ref{prop:map-Pchi}(ii) and Definition \ref{def:tchi}. Finally, if $\chi''\notin \{\chi,\chi'\}$, then the map is always zero by Proposition \ref{prop:map-Pchi}(iii), so the result is also trivial.
\end{proof}

 \subsubsection{Socle}\label{subsection:socle}

 Recall the socle of an $R$-module $M$ from Definition \ref{defn:app-socle}. Since $R$ is local, we have the following equivalent description
\[\soc_R(M):=\{v\in M: rv=0,\ \forall v\in \fm_R\}.\]
For example, since $\fm_R^2=0$, we have $\soc_R(R)=\fm_R$. However, note that $\soc_R(M)\neq \fm_RM$ in general (e.g. take $M=R\oplus \F^m$ for some $m\geq 1$).
\begin{lemma}\label{lemma:forsocle}
The  morphism $\Hom_I(P_{\chi},\lambda_{\chi}^{\vee})^{\vee}\ra \Hom_I(P_{\chi},\chi)^{\vee}$,
induced from  the natural quotient morphism $\lambda_{\chi}\twoheadrightarrow \chi^{\vee}$, is $R$-linear and  identified with $R\twoheadrightarrow \F$.
\end{lemma}
\begin{proof}
Noting that $\Hom_I(P_{\chi},\chi)^{\vee}\cong \F$, the result is clear by Proposition \ref{prop:module-Pchi}.
\end{proof}

\begin{corollary}\label{cor:kernel=socle}
Let $P$ be a finitely generated projective $\F[\![I/Z_1]\!]$-module. The natural  morphism
\[\Hom_I(P,\lambda_{\chi}^{\vee})^{\vee}\ra \Hom_I(P,\chi)^{\vee}\] is $R$-linear whose kernel is identified with the $R$-socle of $\Hom_I(P,\lambda_{\chi}^{\vee})^{\vee}$.
\end{corollary}
\begin{proof}
It is clear that we may assume  $P$  is indecomposable, i.e. $P=P_{\chi'}$ for some character $\chi'$.  If $\chi'=\chi$, then the result follows from Lemma \ref{lemma:forsocle}. If $\chi'\neq \chi$, then $\Hom_I(P_{\chi'},\chi)=0$ and hence the map $\Hom_I(P_{\chi'},\lambda_{\chi}^{\vee})^{\vee}\ra \Hom_I(P_{\chi'},\chi)^{\vee}$ is always zero; since $\Hom_I(P_{\chi'},\lambda_{\chi}^{\vee})^{\vee}$ is itself either $0$ or a simple $R$-module, the result holds trivially.
\end{proof}

If $P$ is a finitely generated $\F[\![I/Z_1]\!]$-module, we denote by
 $\rad_{\chi}(P)$   the largest subobject such that the quotient $P/\rad_{\chi}(P)$ is semisimple and $\chi$-isotypic.
To be explicit, if $P$ can be decomposed as $P_1\oplus P_2$ with $\mathrm{cosoc}(P_1)$ is $\chi$-isotypic and $\Hom_I(P_2,\chi)=0$, then
\[\rad_{\chi}(P)=\rad(P_{1}) \oplus P_2=\fm P_1\oplus P_2.\]
Corollary \ref{cor:kernel=socle} can be restated as follows.

\begin{corollary}\label{cor:rad-chi}
Let $P$ be a finitely generated projective $\F[\![I/Z_1]\!]$-module. There is an exact sequence
\[0\ra \Hom_I(\rad_{\chi}(P),\lambda_{\chi}^{\vee})^{\vee}\ra \Hom_I(P,\lambda_{\chi}^{\vee})^{\vee}\ra\Hom_I(P/\rad_{\chi}(P),\lambda_{\chi}^{\vee})^{\vee}\ra0,\]
which is canonically identified  with
\[0\ra \soc_R(M)\ra M\ra M/\soc_R(M)\ra0\]
where we have written $M=\Hom_I(P,\lambda_{\chi}^{\vee})^{\vee}$.
\end{corollary}
\begin{proof}
Since $P/\rad_{\chi}(P)$ is semisimple and $\chi$-isotypic, the exactness  follows from Lemma \ref{lemma:HA-Ext1-vanish}. The second claim is a reformulation of Corollary \ref{cor:kernel=socle}, noting that there are natural isomorphisms
\[\Hom_I(P,\chi)^{\vee}\simto \Hom_I(P/\rad_{\chi}(P),\chi)^{\vee} \xleftarrow{\sim} \Hom_I(P/\rad_{\chi}(P),\lambda_{\chi}^{\vee})^{\vee}.\]
\end{proof}

\begin{proposition}\label{prop:beta-sharp-minimal}
Let $\beta:P_1\ra P_2$ be a morphism between finitely generated projective $\F[\![I/Z_1]\!]$-modules such that $\im(\beta)\subset \fm P_2$. Write $M_i=\Hom_{I}(P_i,\lambda_{\chi}^{\vee})^{\vee}$  for $i=1,2$, and let $\beta_{\chi}^{\sharp}:M_1\ra M_2$ be the morphism \eqref{eq:beta-sharp}. Then the following statements hold.
\begin{enumerate}
\item[(i)]  $\im(\beta_{\chi}^{\sharp})$ is contained in $\soc_R(M_2)$.
\item[(ii)] Let $\fb\subset R$ be the ideal generated by $t_{\chi'}$, where $\chi'$ runs over the set $\mathscr{E}(\chi)\cap \JH(\mathrm{cosoc}(P_1))$. Then $\beta_{\chi}^{\sharp}$ induces an $R$-linear morphism
$\soc_R(M_1)\ra \fb M_2.$
\end{enumerate}
\end{proposition}
\begin{proof}
(i) We may assume both $P_i$ ($i=1,2$)  are  indecomposable, i.e. $P_i=P_{\chi_i}$ for  characters $\chi_i\in \{\chi\}\cup\mathscr{E}(\chi)$. In view of the table in Proposition \ref{prop:map-Pchi}, the only nontrivial case is when $\chi_1=\chi_2=\chi$. But,  in this case the claim follows directly from Corollary \ref{cor:rad-chi}.

(ii) We may again assume $P_i=P_{\chi_i}$ are indecomposable for $\chi_i\in\{\chi\}\cup\mathscr{E}(\chi)$. If $\chi_1\in \mathscr{E}(\chi)$, the claim is just Lemma \ref{lemma:image-tchi}. 

If $\chi_1=\chi$, then $\soc_R(M_1)=\fm_R$ is identified with $\Hom_I(\fm P_1,\lambda_{\chi}^{\vee})^{\vee}$ by Corollary \ref{cor:rad-chi}. There are two subcases. If $\chi_2\in\mathscr{E}(\chi)$, then $M_2\cong \F$ and $\fb M_2=0$, so the result follows from Proposition \ref{prop:map-Pchi}. Assume $\chi_2=\chi$ (and also $\chi_1=\chi$).
Since $\chi$ does not occur in $\gr^1_{\fm}P_{\chi}$, the assumption $\beta(P_1)\subset \fm P_2$ implies that $\beta(P_1)$ is actually contained in $\fm^2P_2$, and therefore $\beta(\fm P_1)$ is contained in $\fm^3P_2$.  By Proposition \ref{prop:map-Pchi}(ii), we deduce that $\beta_{\chi}^{\sharp}$ is identically zero when restricted to $\Hom_I(\fm P_1,\lambda_{\chi}^{\vee})^{\vee}$. This finishes the proof.
\end{proof}

\subsection{Generalized Koszul complexes}
\label{subsection:Koszul}

Let $R$ be a commutative ring, $M$ be an $R$-module and $\underline{\phi}=(\phi_1,\dots,\phi_n)$ be a family of pairwise commuting $R$-linear endomorphisms of $M$. The \emph{Koszul complex}
\[K_{\bullet}(\underline{\phi},M): \ \ 0\ra K_n\overset{d_n}{\lra} K_{n-1}\lra \cdots\overset{d_1}{\lra} K_0\lra 0 \]
associated to the data $(M,\phi_1,\dots,\phi_n)$ is defined as follows:
\begin{enumerate}
\item[$\bullet$] $K_l=M\otimes_{\Z}\wedge^l(\Z^n)$ for $0\leq l\leq n$; 
\item[$\bullet$] the differential map $d_l:K_{l}\ra K_{l-1}$ (for $1\leq l\leq n$) is defined as
\begin{equation}\label{eq:diff-Koszul}d_l\big(v\otimes(e_{i_1}\wedge \cdots\wedge e_{i_l})\big)=\sum_{r=1}^l(-1)^{r-1}\phi_{i_r}(v)\otimes(e_{i_1}\wedge \cdots\wedge \widehat{e}_{i_r}\wedge\cdots\wedge e_{i_l}),\end{equation}
where $v\in M$, $(e_{1},\dots,e_n)$ is the canonical basis in $\Z^n$ and by $\widehat{e}_{i_r}$ we indicate that $e_{i_r}$ is to be omitted from the exterior product.
\end{enumerate}

For the rest of this subsection, we assume that $R$ is a noetherian local ring and $M$ is finitely generated over $R$.  Let $R'\defn\End_R(M)$, which acts on $M$ from the \emph{right}, sending $(t,\phi)$ to $t\phi\defn \phi(t)$, where $\phi\in R'$ and $t\in M$. In this way, $M$ becomes an $(R,R')$-bimodule. Note that the composition  $\varphi\circ \phi: M\overset{\phi}{\ra} M\overset{\varphi}{\ra} M$  corresponds to the product $\phi\varphi$ in $R'$. This choice of convention is compatible with the one made in \S\ref{subsection:envelope}.

\begin{remark}
Replacing $M$ by $R'$ in the definition of $K_{\bullet}(\un{\phi},M)$, we obtain the Koszul complex $K_{\bullet}(\underline{\phi},R')$, where we view $R'$ as an $(R',R')$-bimodule and $\phi_i$ as an endomorphism of $R'$  sending $f$ to $f\phi_i$. Then we have a canonical isomorphism
\[ M\otimes_{R'}K_{\bullet}(\underline{\phi}, R')  \cong K_{\bullet}(\underline{\phi},M).\]
\end{remark}

Since $R$ is commutative, $R'$ is naturally an $R$-module. For any left ideal $J$ of $R'$, $ MJ$ is an $R$-submodule of $M$. The following result is an analog of  \cite[\S5, Lem.~2]{Se56} (cf. also \cite[\S IV, Appendix I]{Se00}).

\begin{lemma}\label{lemma:Serre-general}
Let $J$ denote the left ideal of $R'$ generated by $\phi_1,\dots,\phi_n$.
Assume that $J$ is a two-sided ideal and the morphism \[ \overline{d}_1=(\overline{\phi}_1,\dots,\overline{\phi}_n):\  K_1/ K_1J\ra  K_0J/K_0J^2\]  is injective. Then for any $1\leq l\leq n$, the morphism
\[\overline{d}_l:  K_l/K_lJ\ra  K_{l-1}J/K_{l-1}J^2\]
is injective.
\end{lemma}

\begin{proof}
By definition, $K_0=M$ and $K_1= M^{n}$. The injectivity of $K_1/ K_1J\ra  K_0J/K_0J^2$ can be restated as follows: if $v_i\in M$ (where $1\leq i\leq n$) satisfy
$\sum_{i=1}^n\phi_i(v_i)\in MJ^2$,
then $v_i\in  MJ$ for all $i$.

  To simplify the notation, we denote $$I_{l}=\{\underline{i}=(i_1,\dots,i_l)|1\leq i_1<\cdots <i_l\leq n\}.$$ For   $\underline{i}=(i_1,\dots,i_l)\in I_l$, set  $e_{\underline{i}}=e_{i_1}\wedge\cdots\wedge e_{i_l}$ and $S_{\underline{i}}=\{i_1,\dots,i_l\}\subset \{1,\dots,n\}=:S.$

Let $v=\sum_{\underline{i}\in I_l} v_{\underline{i}}\otimes e_{\underline{i}}\in K_l$ with $v_{\underline{i}}\in M$, then by \eqref{eq:diff-Koszul}
\[d_l(v)=\sum_{\underline{i}'\in I_{l-1}}f_{\underline{i}'} \otimes e_{\underline{i}'}\in K_{l-1}\]
with $f_{\underline{i}'}$ having the form
\[f_{\underline{i}'}=\sum_{j\in S\backslash S_{\underline{i}'}}\pm\phi_{j}(v_{\underline{i}'\cup \{j\}}).\]
Here, $\underline{i}'\cup \{j\}$ denotes the unique element in $I_{l}$ whose underlying set is $S_{\underline{i}'}\cup\{j\}$.
Now, if $d_l(v)\in  K_{l-1}J^2\cong  MJ^2\otimes_{\Z} \wedge^{l-1}(\Z^n)$, then $f_{\underline{i}'}\in MJ^2$ for all $\underline{i}'\in I_{l-1}$. By assumption, we deduce that
$v_{\underline{i}'\cup\{j\}}\in  MJ$,
and  the result follows.
\end{proof}

\subsection{A typical example}
\label{subsection:example}

In this subsection, we study one typical example of generalized Koszul complexes introduced in last subsection.

Let $R$ be a noetherian local $\F$-algebra, with maximal ideal $\fm_R$ and residue field $\F$. Let $M$ be a finitely generated $R$-module and $R'=\End_R(M)$. Assume that $M$ can be decomposed as $M=\oplus_{i=0}^mM_i$. Then we may represent $R'$ in matrix form
\[R'=\big(\Hom_R(M_i,M_j)\big)_{0\leq i,j\leq m}\]
so that the (right) action of $R'$ on $M=\oplus_{i=0}^{m}M_i$ is given by the matrix multiplication\[M_i\times \Hom_{R}(M_i,M_j)\ra M_j,\ \
(v_i, f_{ij})\mapsto f_{ij}(v_i).\]
The  multiplication in $R'$ is determined by: if $f_{ij}\in \Hom_R(M_i,M_j)$ and $g_{ji}\in\Hom_R(M_j,M_i)$, then
\[ f_{ij}\times  g_{ji} \longmapsto (M_i\overset{f_{ij}}{\ra} M_j\overset{g_{ji}}{\ra} M_i).\]

From now on, we make the following assumptions on $R$ and $M$:
\begin{enumerate}
\item[$\bullet$] $\fm_R\neq 0$  and $\fm_R^2=0$;  
\item[$\bullet$]  $M=R\oplus \F^{m}$
for some $0\leq m\leq n$, where $n\defn\dim_{\F}\fm_R$.
\end{enumerate}

\begin{lemma}\label{lemma:M=v0R'}
Keep the above notation.

(i) We have
\[R'=\left(\begin{array}{cccc}R & \F & \cdots & \F\\ \fm_R &\F & \cdots & \F \\\vdots & \vdots & \vdots & \vdots \\ \fm_R& \F & \cdots & \F\end{array}\right)_{(m+1)\times(m+1)}.\]
(ii) $M$ is a cyclic $R'$-module generated by
$v_0\defn\left(\begin{array}{cccc}1 & 0 & \cdots & 0\end{array}\right).$
\end{lemma}
\begin{proof}
(i) With the notation introduced above, we enumerate $R$ as $M_0$ and $\F^{m}$ as $\oplus_{i=1}^mM_i$.  The result easily follows from what we have recalled, using the isomorphism $\Hom_R(\F,R)\cong\fm_R$ (as $\fm_R^2=0$).

(ii) For $\phi\in R'$, $v_0\phi$ corresponds to the first row of the matrix of $\phi$. The assertion easily follows.
\end{proof}

Note that, when doing matrix multiplication in $R'$, the map $\F\times \fm_R\ra \fm_R$ is the usual multiplication map in $R$, whereas $\fm_R\times \F \ra \F$ is the zero map (because any morphism  $\F\ra R\ra \F$ is zero as $\fm_R\neq0$).

Let $\fb$ be a subspace of $\fm_R$ (we allow the case $\fb=0$). Since $\fm_R^2=0$, $\fb$ can be viewed as an ideal of $R$. Consider the subspace $J_{\mathfrak{b}}$ of $R'$ defined as
\[J_{\fb}:=\left(\begin{array}{cccc}\fm_R & \F & \cdots & \F \\ \fb & 0 & 0 & 0 \\ \vdots &  \vdots & \vdots &  \vdots \\ \fb & 0 & 0 & 0\end{array}\right)_{(m+1)\times (m+1)}.\]
For example, we have
\begin{equation}\label{eq:def-Jb}
J_{\fm_R}:=\left(\begin{array}{cccc}\fm_R & \F & \cdots & \F \\ \fm_R & 0 & 0 & 0 \\ \vdots &  \vdots & \vdots &  \vdots \\ \fm_R & 0 & 0 & 0\end{array}\right),\ \ \ J_{(0)}=\left(\begin{array}{cccc}\fm_R & \F & \cdots & \F\\ 0 & 0 & 0 & 0 \\ \vdots &  \vdots & \vdots &  \vdots \\ 0 & 0 & 0 & 0\end{array}\right).\end{equation}
It is direct to check that $J_{\fb}$ is a two-sided ideal of $R'$.  Recall from \S\ref{subsection:socle} the definition of $\soc_R(M)$.

\begin{lemma}\label{lemma:JM}
(i) We have
\[MJ_{\fb}=\fm_R\oplus\F^m=\soc_R(M),\ \ MJ_{\fb}^2=\fb\oplus (0)^{m}=\fb M.\]

(ii) $M/MJ_{\fb}$ has dimension $1$ over $\F$ and is spanned by the residue class of $v_0$.

(iii)  $\dim_{\F}MJ_{\fb}/MJ_{\fb}^2=n+m-\dim_{\F}\fb$.
\end{lemma}

\begin{proof}
(i) Since $M=v_0R'$ by Lemma \ref{lemma:M=v0R'}(ii), we have $MJ_{\fb}=\langle v_0\cdot J_{\fb}\rangle=\{v_0\phi: \phi\in J_{\fb}\}$ and $MJ_{\fb}^2=\langle v_0\cdot J_{\fb}^2\rangle$. The result is then a direct computation
 (using  $\fm_R^2=0$ for the description of $MJ_{\fb}^2$).
 Finally, (ii) and   (iii)   clearly follow from (i).
\end{proof}

Given $\phi\in J_{\fb}$,  we have a natural $\F$-linear map
\[\overline{\phi}: M/MJ_{\fb}\ra MJ_{\fb}/MJ_{\fb}^2.\]
Since $M/MJ_{\fb}$ is spanned by the residue class of $v_0$, say $\overline{v}_0$, $\overline{\phi}$ is determined by the residue class of $\phi(v_0)$  in $MJ_{\fb}/MJ_{\fb}^2$.

\begin{proposition}\label{prop:dimb=m}
Assume that $J_{\fb}$ can be generated by $n$ elements as a left ideal of $R'$, i.e. there exist $\phi_1,\dots,\phi_n\in J_{\fb}$ such that
$J_{\fb}=\sum_{i=1}^nR'\phi_i.$ 
Then $\phi_1,\dots,\phi_n$ induce a surjection
\[(\overline{\phi}_1,\dots,\overline{\phi}_n): (M/MJ_{\fb})^n\ra MJ_{\fb}/MJ_{\fb}^2.\]
As a consequence, $\dim_{\F}\fb\geq m$ and the equality holds if and only if
$(\overline{\phi}_1,\dots,\overline{\phi}_n)$ is an isomorphism.
\end{proposition}

\begin{proof}
The morphism is well-defined as explained above and is surjective by assumption. The second assertion is  clear  for the reason of dimensions, using Lemma \ref{lemma:JM}(iii).
\end{proof}

The following are criteria for a left ideal of $R'$ to be of the form $J_{\fb}$.

\begin{lemma}\label{lemma:J=Jb-1}
If $J$ is a left ideal of $R'$ such that $J_{(0)}\subset J\subset J_{\fm_R}$, then $J=J_{\fb}$ for some (proper) ideal $\fb$ of $R$.
\end{lemma}
\begin{proof}
We associate to $J$ a subspace $\fb_J$ of $\fm_R$ as follows:
\[\fb_J:=\big\{b_{i0}(\phi):\ \phi\in J,\ 1\leq i\leq m\big\} \]
where   $\phi\in J$ is written in the matrix form $(b_{ij}(\phi))_{0\leq i,j\leq m}$.
 Since $J$ is a left ideal of $R'$, one easily checks  that $J$ contains
\[\left(\begin{array}{cccc}0 &0 & \cdots & 0\\  \fb_J & 0 & 0 & 0 \\ \vdots &  \vdots & \vdots &  \vdots \\  \fb_J & 0 & 0 & 0\end{array}\right),\]
hence also contains $J_{\fb}$ because $J_{(0)}\subset J$ by assumption. On the other hand, since $J\subseteq J_{\fm_R}$, we have  $J\subset J_{\fb_J}$ by definition of $\fb_{J}$.
\end{proof}

\begin{lemma}\label{lemma:J=Jb-2}
Let $J\subset R'$ be a left ideal contained in $J_{\fm_R}$. Assume that $\dim_{\F}M/MJ=1$ and that $J$ can be generated by $n$ elements. Then $J=J_{\fb}$ for some  (proper) ideal $\fb$ of $R$ and $\dim_{\F}\fb\geq m$.
\end{lemma}
\begin{proof}
We know that $\dim_{\F}M/MJ_{\fm_R}=1$ by Lemma \ref{lemma:JM}(i).
Hence the assumptions $J\subset J_{\fm_R}$ and $\dim_{\F}M/MJ=1$ imply that $MJ=MJ_{\fm_R}=\fm_R\oplus \F^m$. Since $J$ is a left ideal and $M=v_0R'$, we have $MJ=\langle v_0\cdot J\rangle$. Since $ v_0\phi$ corresponds to the first row of the matrix of $\phi$,  we deduce that $J_{(0)}\subset J$, see \eqref{eq:def-Jb}. By Lemma \ref{lemma:J=Jb-1}, there exists an ideal $\fb$ of $R$ such that $J=J_{\fb}$ and it follows from  Proposition \ref{prop:dimb=m} that $\dim_{\F}\fb\geq m$.
\end{proof}

We close this subsection with a basic but typical example.

\begin{example}
We use the notation of \S\ref{subsection:envelope}. Let $\lambda=\Ug/((e,f)^3,e^2,f^2)$ and $R= \F[x,y]/(x,y)^2\cong\End_{\Ug}(\lambda)$, by (the graded versions of)  Lemma \ref{lemma:R=End}. Let $H$ act on $\Ug$ as in \S\ref{subsubsection-H-action}. Applying $\Hom_{\Ug}(-,\lambda^{\vee})^{\vee}$ to the complex \eqref{eq:complex-sum} gives   a generalized Koszul complex of $R$-modules
\[0\lra M\overset{(-\phi_2,\phi_1)}{\lra} M\oplus M\overset{\binom{\phi_1}{\phi_2}}{\lra} M\lra0\]
where $M=R\oplus \F$,
\[\phi_1=\matr{0}{f^{\sharp}}{x}0,\ \ \phi_2=\matr{x-y}{0}0{0}\]
where $f^{\sharp}\in\F^{\times}$ is the element induced by $f:\Ug_{\ide}\ra \Ug_{\alpha_i}$. The left ideal of $R'=\End_{R}(M)$ generated by $\phi_1,\phi_2$ is equal to $J_{(x)}$  and
\[MJ_{(x)}/MJ_{(x)}^2=(\fm_R\oplus \F)/((x)\oplus (0))\cong (\fm_R/(x))\oplus \F.\]
It is direct to verify that
 the morphism
\[(\overline{\phi}_1,\overline{\phi}_2): \ M/MJ_{(x)}\oplus M/MJ_{(x)}\ra MJ_{(x)}/MJ_{(x)}^2\]
is an isomorphism.
\end{example}

\section{Finite generation} \label{section-fg}

Let $F$ and $\overline{r}: G_F\to \GL_2(\F)$ be as in \S\ref{section-patching} and  $\pi_{v}^{D}(\overline{r})$ be the $\GL_2(F_v)$-representation constructed in \eqref{eq:piv}.  In this section we study the representation theoretic property of $\pi_{v}^{D}(\overline{r})$. Write $L\defn F_v$, $\brho\defn \overline{r}^{\vee}|_{G_L}$ as in \S\ref{subsection:cyclic}; recall that $\brho$ is reducible nonsplit and  strongly generic. For convenience,   write
\[ \pi(\brho)\defn\pi_{v}^D(\overline{r}),\]
keeping in mind that,  \emph{a priori},  $\pi(\brho)$ may depend on the global setting.  

Let $G=\GL_2(L)$, $K=\GL_2(\cO_L)$, and keep the notation in \S\ref{section-BP} and \S\ref{section:ordinary}.  

\subsection{A minimal projective resolution}

Recall that   $M_{\infty}$ is a flat $R_{\infty}$-module by Theorem  \ref{thm:main-flat}. Since $R_{\infty}$ is a regular local ring, by choosing a minimal set of generators of $\fm_{\infty}:=\fm_{R_{\infty}}$ we obtain a Koszul type resolution of $M_{\infty}/\fm_{\infty}=\pi(\brho)^{\vee}$. Although $M_{\infty}$ is projective as a pseudo-compact $\FKZ$-module, it is not finitely generated and the resolution is not minimal.

The first step to study $\pi(\brho)$, equivalently $\pi(\brho)^{\vee}$,  is to construct a \emph{minimal} projective  resolution of $\pi(\brho)^{\vee}|_{K}$, as follows.  In the proof of next proposition we will use  the notation of \S\ref{section-patching}.

\begin{proposition}\label{prop:Rv-Mv}
There exists a quotient ring of $R_{\infty}$, denoted by $R_{v}$, such that
\begin{enumerate}
\item[(i)] $R_v$ is a regular local ring  over $\F$ of Krull dimension $2f$; 
\item[(ii)] $M_{\infty}\otimes_{R_{\infty}}R_{v}$ is isomorphic to $\bigoplus_{\sigma\in\mathscr{D}(\brho)}\Proj_{K/Z_1}\sigma^{\vee}$ as an $\FKZ$-module.
\end{enumerate}
\end{proposition}
\begin{proof}
By construction $R_{\infty} $ is a power series ring in $(q+j- f-3)$-variables over $R_{\brho}^{\Box , \psi}.$  Recall that we have constructed elements $T_1,\dots,T_{f+3}$ in $R_{\brho}^{\Box,\psi}$ such that the  images of  $\{\varpi,T_1,\dots,T_{f+3}\}$ in $R_{\brho}^{\Box,\psi,\rm{cris},\sigma}$ form a regular system of parameters for any $\sigma\in\mathscr{D}(\brho)$, see Proposition \ref{prop:Tj}. Together with the $q+j-f-3$ formal variables just mentioned (and the uniformizer $\varpi$), we obtain a part of a regular system of parameters of $R_{\infty}$, say $\{\varpi, T_1,\dots,T_{q+j}\}$, such that their images in $R_{\infty}^{\rm{cris},\sigma}$ form a regular system of parameters for any $\sigma \in \mathscr{D}(\brho)$. We claim that
\[R_v\defn R_{\infty}/(\varpi, T_1,\dots,T_{q+j})\]
satisfies the required conditions. Condition (i) is clear by construction.

Prove (ii).  Recall that if $\sigma$ is a Serre weight, then $M_{\infty}(\s)$ is nonzero if and only if $\sigma\in\mathscr{D}(\brho)$, in which case $M_{\infty}(\s)$ is free of rank one over $R_{\infty}^{\mathrm{cris}, \s}\otimes_{\cO} \F$. In particular, if $\sigma\in\mathscr{D}(\brho)$, then $\{T_1,\dots,T_{q+j}\}$ is a regular sequence for $M_{\infty}(\sigma)$ and
\begin{equation}\label{eq:cosoc=dim1}M_{\infty}(\sigma)/(T_1,\dots,T_{q+j})\cong \F.\end{equation}
On the other hand, we know that $M_{\infty}/\varpi$ is a projective $\FKZ$-module.
Inductively using \cite[Prop.~3.10]{Hu18}, we see that $M_{\infty}\otimes_{R_{\infty}}R_{v}=M_{\infty} / (\varpi, T_1,\ldots, T_{q+j})$ is also projective. To finish the proof it suffices to show that
\[\mathrm{cosoc}_{K}(M_{\infty}\otimes_{R_{\infty}}R_v)\cong \mathrm{cosoc}_K(\pi(\brho)^{\vee})\cong\oplus_{\sigma\in\mathscr{D}(\brho)}\sigma^{\vee},\]
which is a direct consequence of \eqref{eq:cosoc=dim1}.
\end{proof}

In the following, we fix a quotient ring $R_{v}$ of $R_{\infty}$ as in Proposition \ref{prop:Rv-Mv} and also an isomorphism
\[R_v\cong \F[\![X_1,\dots,X_{2f}]\!].\]
Let
\[M_v\defn M_{\infty}\otimes_{R_{\infty}}R_v.\]
Then the Koszul complex $K_{\bullet}(\underline{X},M_{v})$ where $\underline{X}=(X_1,\dots,X_{2f})$,
\begin{equation}\label{equation-Q-projresolution}0\ra M_v \ra M_v^{\oplus 2f}\ra \cdots \ra M_v^{\oplus 2f}\ra M_v \ra0\end{equation}
  is a  resolution of $\pi(\brho)^{\vee}\cong M_v\otimes_{R_v}\F$, whose terms are finite projective when viewed as $\FKZ$-modules.
 Dually, letting \[\Omega_{v}\defn (M_{v})^{\vee},\]  we obtain a resolution of $\pi(\brho)$
\begin{equation}\label{equation-Q-injresolution}
0\ra \pi(\brho)\ra \Omega_{v}\ra \Omega_{v}^{\oplus 2f}\ra \cdots \ra \Omega_{v}\ra0.\end{equation}
We still denote by $X_i: \Omega_{v}\ra \Omega_{v}$ the endomorphism induced from $X_i: M_{v}\ra M_{v}$.
 Since $\Omega_{v}$ is an injective object in the category  $\mathrm{Rep}_{\F}(K/Z_1)$, \eqref{equation-Q-injresolution} is a resolution of $\pi(\brho)$ by injective $\FKZ$-modules. It will play a crucial role later on.

\begin{proposition}\label{prop:K=minimal}
The resolution \eqref{equation-Q-projresolution} is minimal in the sense that  the differential map sends $K_{l}(\un{X},M_v)$ to $\rad_K(K_{l-1}(\un{X},M_v))$. 
\end{proposition}

\begin{proof}
If $\sigma$ is a Serre weight, write $P_{\sigma}\defn \Proj_{K/Z_1}\sigma$ with the $\fm_{K_1}$-adic topology. We first prove the following general fact: if $P=\bigoplus_{i=1}^nP_{\sigma_i}$ with $\sigma_1\neq \sigma_i$ for any $i\geq 2$,  and if $\phi: P\ra P$ is a topologically nilpotent continuous $K$-equivariant endomorphism, i.e. $\cap_{k\geq 1}\im(\phi^k)=0$,  then $\phi(P_{\sigma_1})\subset \rad(P)$. Indeed, let $\phi_{ij}$ denote the composite map
\[P_{\sigma_i}\hookrightarrow P\overset{\phi}{\ra} P\twoheadrightarrow P_{\sigma_j}\]
where the first map is the natural inclusion and the last one is the projection. Then $\phi$ is determined by the matrix  $(\phi_{ij})_{1\leq i,j\leq n}$, see \S\ref{subsection:example}.
Note that $\phi(P_{\sigma_1})\subset \rad(P)$ if and only if $\overline{\phi}: \mathrm{cosoc}(P_{\sigma_1})\ra \mathrm{cosoc}(P)$ is the zero map,  if and only if $\overline{\phi}_{1j}=0$ for all  $1\leq j\leq n$.  If $j\neq 1$ then  $\sigma_1\neq \sigma_j$ and we always have $\overline{\phi}_{1j}=0$. In other words, the matrix $(\overline{\phi}_{ij})$ is a $(1,n-1)$-block lower triangular matrix.  If $\overline{\phi}_{11}\neq 0$, then $\overline{\phi}_{11}^k\neq0$, and also $\overline{\phi}^k\neq0 $ for any $k\geq 1$. This contradicts the assumption that $\phi$ is topologically nilpotent.

Now we prove the lemma. By the construction of Koszul complexes, $K_{\bullet}(\un{X}, M_{v})$ is minimal if and only if each endomorphism $X_i:M_v\ra M_v$ has image contained in $\rad_K(M_{v})$. However, $X_i$ is topologically nilpotent, and since $\mathrm{cosoc}_K(M_v)=\mathrm{cosoc}_K(\pi(\brho)^{\vee})$ is multiplicity free, we conclude by the above fact.
\end{proof}

\begin{remark}
A topologically nilpotent endomorphism of $P$ need not   have image contained in $\rad(P)$. Example: $P=P_{\sigma}\oplus P_{\sigma}$ with $\phi:P\ra P$ given by $\matr{0}{\id}00$.
\end{remark}

Despite Proposition \ref{prop:K=minimal}, the complex \eqref{equation-Q-projresolution} is \emph{not} minimal when viewed as a complex of $\FIwZ$-modules. The next step is to remedy this problem.

Recall that by \cite[Lem.~14.1]{BP} the set of Jordan--H\"older factors of $D_1(\brho)=\pi(\brho)^{I_1}$ is, ignoring multiplicities, the same as that of $\big(\bigoplus_{\sigma\in\mathscr{D}(\brho)}\rInj_{\Gamma}\sigma\big)^{I_1}$. On the other hand, by \cite[Prop.~4.3]{Br14} the set $\JH(D_1(\brho))$ is parametrized by the set $\mathscr{PD}(x_0,\cdots,x_{f-1})$, whose definition is recalled in the proof of Lemma \ref{lemma-PD-set}.  Define a subset of $\mathscr{PD}(x_0,\cdots,x_{f-1}) $ as follows:
\begin{multline}\label{eq:def-PD'} 
\mathscr{PD}^{\dag}(x_0,\cdots,x_{f-1})\defn\big\{\lambda\in\mathscr{PD}(x_0,\cdots,x_{f-1}),  \lambda_i(x_i)\in\{x_i,x_i+2,p-1-x_i,p-3-x_i\}\big\}, \end{multline}
and let $\mathscr{PD}^{\ddag}(x_0,\cdots,x_{f-1})$ be its complement.

The following result  gives a refinement of \cite[Lem.~14.1]{BP}.

\begin{lemma}\label{lemma:BP-14.1}
For any character $\chi$ of $I$, let $n_{\chi}\in\Z_{\geq 0}$ such that \[\Big(\bigoplus_{\sigma\in\mathscr{D}(\brho)}\rInj_{\Gamma}\sigma\Big)^{I_1}\cong\bigoplus_{\chi}\chi^{n_{\chi}}.\]
Then $n_{\chi}\neq0$ if and only if $\chi\in\JH(D_1(\brho))$. If $\chi$ corresponds to $\lambda\in \mathscr{PD}^{\dag}(x_0,\cdots,x_{f-1})$, then  $n_{\chi}=1$.
\end{lemma}
\begin{proof}
The first assertion is just \cite[Lem.~14.1]{BP}.  The second one is a consequence of \cite[Prop.~2.1]{HuJLMS}, noting that $n_{\chi}=1$ if and only if both $\chi$ and $\chi^s$ occur in  $D_{0,\sigma}(\brho)^{I_1}$ for the same $\sigma\in\mathscr{D}(\brho)$.
\end{proof}

\begin{lemma} \label{lemma:decomp-M}
For any $0\leq l\leq 2f$, $M_v|_{I}$ has a direct sum decomposition $M_v^{\dag}\oplus M_{v}^{\ddag}$ satisfying the following properties:
\begin{enumerate}
\item[(a)] $M_v^{\dag}\cong \bigoplus_{\chi}P_{\chi^{\vee}}$, where $\chi$ runs over the characters corresponding to $\lambda\in \mathscr{PD}^{\dag}(x_0,\cdots,x_{f-1})$;

\item[(b)]$\Hom_I(M_v^{\ddag},P_{\chi^{\vee}}/\fm^2)=0$ for any $\chi$ as in (a) (recall $\fm=\fm_{I_1/Z_1}$).
\end{enumerate}
\end{lemma}
\begin{proof}
Dually we  work with $\Omega_v$. Since $\Omega_v|_{K}$ is isomorphic to $\bigoplus_{\sigma\in\mathscr{D}(\brho)}\rInj_{K/Z_1}\sigma$,  with the notation of Lemma \ref{lemma:BP-14.1}, we get
\begin{multline*}\Omega_v|_I\cong \bigoplus_{\chi\in \JH(D_1(\brho))}\big(\rInj_{I/Z_1}\chi\big)^{n_{\chi}} =\Big(\bigoplus_{\chi\in\mathscr{PD}^{\dag}}\rInj_{I/Z_1}\chi\Big)\bigoplus\Big(\bigoplus_{\chi\in\mathscr{PD}^{\ddag}}(\rInj_{I/Z_1}\chi)^{n_{\chi}}\Big)\end{multline*}
where we simply write $\chi\in\mathscr{PD}^{\dag}$ (resp. $\chi\in\mathscr{PD}^{\ddag}$) to mean that $\chi$ corresponds to an element in $\mathscr{PD}^{\dag}(x_0,\cdots,x_{f-1})$ (resp. $\mathscr{PD}^{\ddag}(x_0,\cdots,x_{f-1})$).
We  define $\Omega_v^{\dag}$ to be the first summand and $\Omega_{v}^{\ddag}$ to be the second; let $M_v^{\dag}=(\Omega_v^{\dag})^{\vee}$ and $M_{v}^{\ddag}=(\Omega_{v}^{\ddag})^{\vee}$, so that $M_v=M_{v}^{\dag}\oplus M_v^{\ddag}$.

Condition (a) is immediate from the definition.  To check  (b), it suffices to check that if $\chi\in\mathscr{PD}^{\dag}$ and $\chi'\in\mathscr{PD}^{\ddag}$, then $\Ext^1_{I/Z_1}(\chi,\chi')=0$. But this is a direct check using Lemma \ref{lemma:Ext1-chi}.
\end{proof}

The same argument of  Proposition \ref{prop:K=minimal} proves the following variant,  which is a remedy for the failure of minimality of \eqref{equation-Q-projresolution} as a complex of $\FIwZ$-modules.

\begin{proposition}\label{prop:K=minimal-I}
As a complex of  $\FIwZ$-modules, the resolution \eqref{equation-Q-projresolution} is partially minimal relative to $\mathscr{PD}^{\dag}$ in the sense that for any $\chi\in\mathscr{PD}^{\dag}$, the morphism
\[\Hom_I\big(K_{l-1}(\un{X},M_v),\chi^{\vee}\big)\ra \Hom_{I}\big(K_l(\un{X},M_v),\chi^{\vee}\big)\]
is zero.
\end{proposition}

\subsection{Cohomological invariants of $\pi(\brho)$}
 Write $\brho=\smatr{\chi_1}{*}0{\chi_2}$  and define
\begin{equation}\label{equation-reducible-pi0-pif}\pi_0\defn\Ind_{\overline{P}}^G\chi_0\cong \Ind_{P}^G\chi_0^s,  \ \ \  \pi_f\defn\Ind_{\overline{P}}^G\chi_f=\Ind_{P}^G\chi_f^s \end{equation}
where
\begin{equation}\label{eq:chi-0-f}\chi_0\defn\chi_1\omega^{-1}\otimes\chi_2,\ \ \ \chi_f\defn\chi_2\omega^{-1}\otimes\chi_1.\end{equation}
Write $\brho|_{I(\bQp/L)}$ in the form (1) as at the beginning of \S\ref{section-BP} and assume $\brho$ is \emph{strongly generic} in the sense of Definition \ref{defn:strong-generic}, so that the results of
\S\ref{subsection:cyclic} are applicable.

We refer to Appendix \S\ref{section:appendix} for the functor $\EE^{i}(-)$ with respect to $\Lambda=\F[\![K/Z_1]\!]$ and relevant properties. Recall from  \S\ref{subsection:ordinary}   that $\alpha_P$ denotes the character $\omega\otimes\omega^{-1}:T\ra \F^{\times}$; we view it as a character of $P$ by inflation. Let $\zeta$ denote  the central character of $\pi_0$ (and of $\pi_f$); explicitly  $\zeta=\chi_1\chi_2\omega^{-1}$.

\begin{lemma}\label{lemma:Kolh}
Let $\chi:P\ra \F^{\times}$ be a smooth character. Then $(\Ind_P^G\chi)^{\vee}$ is Cohen-Macaulay of grade $2f$  and
\[\EE^{2f}\big((\Ind_P^G\chi)^{\vee}\big)\cong \big(\Ind_P^G(\chi^{-1}\alpha_P)\big)^{\vee}.\]
In particular, $\EE^{2f}(\pi_0^{\vee})\cong \pi_f^{\vee}\otimes \zeta\circ\det$  and the double duality map $\pi_0^{\vee}\ra \EE^{2f}\EE^{2f}(\pi_0^{\vee})$ is an isomorphism. A similar statement holds exchanging $\pi_{0}$ and $\pi_f$. 
\end{lemma}
\begin{proof}
It is a special case of \cite[Prop.~5.4]{Ko}. 
\end{proof}

\begin{proposition}\label{prop-cosocle-pi(rho)}
The $G$-socle (resp. $G$-cosocle) of $\pi(\brho)$ is isomorphic to $\pi_0$ (resp. $\pi_f$).
\end{proposition}

\begin{proof}
First determine the $G$-socle of $\pi(\brho)$. It   is proved in \cite[Lem.~3.1]{HuJLMS} that $\soc_{G}\pi(\brho)$ is an irreducible principal series, say $\soc_G(\pi(\brho))\cong \Ind_{\overline{P}}^G\psi$. We need to show $\psi=\chi_0$. By Proposition \ref{prop-Ord}(iii) it is equivalent to show $\Ord_P(\soc_G\pi(\brho))\cong \chi_0$, which  follows from  Proposition  \ref{prop--ord--semisimple}.

Taking  Pontryagin dual, the $G$-cosocle of  $\pi(\brho)^{\vee}$ is isomorphic to $\pi_0^{\vee}$, see \S\ref{subsection-appendix-socle}.
By Theorem \ref{thm:main-flat}(ii) the double duality map  $\pi(\brho)^{\vee} \ra \EE^{2f}\EE^{2f}(\pi(\brho)^{\vee})$ is an isomorphism and similarly for  $\pi_0^{\vee}$ by Lemma \ref{lemma:Kolh}. Hence, Proposition \ref{prop-EE-essential} implies that the induced inclusion $\EE^{2f}(\pi_0^{\vee})\hookrightarrow \EE^{2f}(\pi(\brho)^{\vee})$ is essential.   By Theorem \ref{thm:main-flat}(ii) and Lemma \ref{lemma:Kolh} again and twisting suitably,  this gives  an essential inclusion
$\pi_f^{\vee} \hookrightarrow \pi(\brho)^{\vee}$.
Moreover, since $\pi_f^{\vee}$ is irreducible, it is exactly the $G$-socle of $\pi(\brho)^{\vee}$.
Dualizing back we obtain the result.
\end{proof}

\begin{remark}
We can deduce from Proposition \ref{prop-cosocle-pi(rho)} that $\pi(\brho)$ is finitely generated as a $G$-represen\-tation.  More precisely, one can prove the following result: if $\pi$ is an admissible  smooth representation of $G$ over $\F$ whose cosocle is nonzero and of finite presentation,
then $\pi$ is finitely generated. We don't pursue this because we will prove a stronger result below, see Theorem \ref{thm-generation-D0}.
\end{remark}

The genericity condition on $\brho$ implies that both $\pi_0$ and $\pi_f$ have an irreducible $K$-socle, and one easily checks that
\[\soc_K(\pi_0)=\sigma_0\]
where $\sigma_0\defn(r_0,\cdots,r_{f-1})$ is the ``ordinary'' Serre weight in $\mathscr{D}(\brho)$.
The $K$-socle of $\pi_f$, denoted by $\sigma_f$, is equal to
\begin{equation}\label{eq:def-sigmaf}(p-3-r_0,\cdots,p-3-r_{f-1})\otimes{\det}^{\sum_{i=0}^{f-1}p^i(r_i+1)}.\end{equation}
 Denote also by $\chi_0:I\ra \F^{\times}$ the character obtained by first restricting $\chi_0$ to $H$ and then inflating to $I$. It is direct to check that $\chi_0$ is exactly  the character of $I$ acting on $\sigma_0^{I_1}$, which explains our choice of convention \eqref{eq:chi-0-f}. Similarly, we have the character $\chi_f$ of $I$ which gives the acton of $I$ on $\sigma_f^{I_1}$.

\begin{proposition}\label{prop-reducible-dimofExt}
For any $i\geq 0$, the following statements hold:
\begin{enumerate}
\item[(i)] $\Ext^{i}_{I/Z_1}(\chi,\pi(\brho))=0$, for any $\chi\notin \JH(D_1(\brho))
$; 
\item[(ii)] $\Ext^i_{K/Z_1}(\sigma,\pi(\brho))=0$, for any $\sigma\notin \mathscr{D}(\brho)$;

\item[(iii)]  $\dim_{\F} \Ext^i_{K/Z_1}(\sigma,\pi(\brho))=\binom{2f}{i}$, for any $\sigma\in\mathscr{D}(\brho)$; 

\item[(iv)] $\dim_{\F}\Ext^i_{G,\zeta}(\pi_0,\pi(\brho))=\binom{2f+1}{i}$.
\end{enumerate}
\end{proposition}
\begin{proof}
(i) The restriction of \eqref{equation-Q-injresolution} to $I/Z_1$ remains an injective resolution, each term being a sum of copies of $\Omega_v$. Since $\Omega_v\cong\bigoplus_{\sigma\in\mathscr{D}(\brho)}\rInj_{K/Z_1}\sigma$, the assertion  follows from Lemma \ref{lemma:BP-14.1}.

(ii) (iii) follow directly from the resolution \eqref{equation-Q-injresolution},  using Proposition \ref{prop:K=minimal}.

(iv) First, by \cite{BL}, there is a short exact sequence (for a suitable $\lambda_0\in\F^{\times}$)
 \begin{equation}\label{eq:BL-pi0}0\ra \cInd_{\mathfrak{R}_0}^G\sigma_0\overset{T-\lambda_0}\ra\cInd_{\mathfrak{R}_0}^G\sigma_0\ra \pi_0\ra0, \end{equation}
 where $\mathfrak{R}_0 =KZ$ and we let $Z$ act on $\sigma_0$ via $\zeta$.
Since $\Omega_{v}$ is injective as a $K/Z_1$-representation and has $G$-socle isomorphic to $\pi_0$ (a consequence of Proposition \ref{prop-cosocle-pi(rho)}), the same proof as in \cite[Prop.~5.1]{Pa15} shows that
\begin{equation}\label{eq:Vytas-Duke-trick}\dim_{\F}\Hom_{G}(\pi_0,\Omega_{v})=\dim\Ext^1_{G,\zeta}(\pi_0,\Omega_{v})=1, \ \ \ \Ext^i_{G,\zeta}(\pi_0,\Omega_{v})=0,\ \ \forall i\geq 2.\end{equation}
In fact, we have isomorphisms induced by \eqref{eq:BL-pi0}
 \begin{equation}\label{eq:E1-j=0}\Hom_G(\pi_0,\Omega_{v})\simto \Hom_K(\sigma_0,\Omega_{v}),\ \ \Hom_{K}(\sigma_0,\Omega_v)\simto\Ext^1_{G,\zeta}(\pi_0,\Omega_v). \end{equation}

On the other hand, applying $\Hom_G(\pi_0,-)$ to \eqref{equation-Q-injresolution} induces a convergent spectral sequence
\begin{equation}\label{eq:ss-E1}E_1^{i,j}=\Ext^j_{G,\zeta}(\pi_0,I^i)\Rightarrow \Ext^{i+j}_{G,\zeta}(\pi_0,\pi(\brho)),\end{equation}
where $I^i:=\Omega_{v}^{\oplus\binom{2f}{i}}$ denotes the degree $i$ term of the complex \eqref{equation-Q-injresolution}. By \eqref{eq:Vytas-Duke-trick}, $E_1^{i,j}=0$ for $j\geq 2$.
 We claim that the  morphisms $E_1^{i,j}\ra E_1^{i+1,j}$ are zero for $j\in\{0,1\}$ and all $i$.  Indeed,  this  is an easy consequence of the minimality in Proposition \ref{prop:K=minimal} using \eqref{eq:E1-j=0}.

By the claim, the  spectral sequence \eqref{eq:ss-E1} degenerates at $E_1$ and we obtain  an exact sequence for any $i\geq 1$:
\[0\ra \Hom_G(\pi_0,I^i)\ra \Ext^{i}_{G,\zeta}(\pi_0,\pi(\brho))\ra \Ext^{1}_{G,\zeta}(\pi_0,I^{i-1})\ra0.\]
The dimension formula then follows from \eqref{eq:Vytas-Duke-trick} and an elementary binomial identity.
\end{proof}

\begin{corollary}\label{cor:Exti-chi-pi1}
Let $\chi$ be a character of $I$ and assume
\begin{equation}\label{eq:inter=oneweight}|\JH(\Ind_I^K\chi)\cap \mathscr{D}(\brho)|=1.\end{equation}
Then $\dim_{\F}\Ext^{i}_{I/Z_1}(\chi,\pi(\brho))=\binom{2f}{i}$ for $i\geq 0$.
\end{corollary}

\begin{proof}
This is a direct consequence of Proposition \ref{prop-reducible-dimofExt}.  In fact, if $\sigma$ denotes the unique Serre weight in $\JH(\Ind_I^K\chi)\cap \mathscr{D}(\brho)$, then using Proposition \ref{prop-reducible-dimofExt}(ii) and by d\'evissage there is an isomorphism
\[\Ext^{i}_{K/Z_1}(\Ind_I^K\chi,\pi(\brho))\cong \Ext^i_{K/Z_1}(\sigma,\pi(\brho)).\]
The result follows from Proposition \ref{prop-reducible-dimofExt}(iii) via Shapiro's lemma.
\end{proof}

\begin{corollary}\label{cor:Exti-chi-pi2}
Let $\chi\in \JH(\pi(\brho)^{I_1})$ and assume it corresponds to an element in $\mathscr{PD}^{\dag}(x_0,\cdots,x_{f-1})$ defined in \eqref{eq:def-PD'}. Then $\dim_{\F}\Ext^{i}_{I/Z_1}(\chi,\pi(\brho))=\binom{2f}{i}$ for any $i\geq 0$.
\end{corollary}
\begin{proof}
 This is a direct consequence of Corollary \ref{cor:Exti-chi-pi1} and Lemma \ref{lemma:BP-14.1}, noting that the condition \eqref{eq:inter=oneweight} is equivalent to $n_{\chi}=1$ in the notation of Lemma \ref{lemma:BP-14.1}.
\end{proof}
 
Next, we determine the derived ordinary parts of $\pi(\brho)$. Recall from \S\ref{subsection:ordinary}  the functors $R^i\Ord_P$.

\begin{proposition}\label{prop-ROrd-pi(rho)}
We have $R^i\Ord_P\pi(\brho)\cong \chi_0^{\oplus n_i}$, where $n_i=\binom{f}{i}$.
\end{proposition}

\begin{proof}
First note that $\Ord_P\pi(\brho)\cong \chi_0$
by Proposition \ref{prop--ord--semisimple}.

 The action of $R_{v}$ on $\Omega_{v}$ induces   morphisms of local rings  
\begin{equation}\label{eq:Rv-res}R_{v}\ra \End_{T}((\Ord_P\Omega_{v})^{\vee})\overset{\rm res}{\ra} \End_{T_0}((\Ord_P\Omega_{v})^{\vee}|_{T_0}) \end{equation}
where $T_0\defn T\cap K$ and the second map is the restriction map.
We claim that the composition is surjective.
Indeed, it suffices to show \[\End_{T_0}((\Ord_P\Omega_{v})^{\vee}|_{T_0})/(\fm_{v})\cong\F,\] where $\fm_{v}$ denotes the maximal ideal of $R_{v}$ and $(\fm_{v})$ the extended ideal in $\End_{T_0}((\Ord_P\Omega_{v})^{\vee}|_{T_0})$.  Since  the actions of $R_{v}$ and $G$ commute with each other, we have
\[(\Ord_P \Omega_{v})^{\vee}/(\fm_{v})\cong \big(\Ord_P(\Omega_{v}[\fm_{v}])\big)^{\vee}=\big(\Ord_P\pi(\brho)\big)^{\vee}\]
which is one-dimensional over $\F$ (isomorphic to $\chi_0^{\vee}$), as seen above. This proves the claim.  As a byproduct, since the restriction map  in \eqref{eq:Rv-res} is clearly injective, it  is actually an isomorphism.

By Corollary \ref{cor:End-T0} and the claim, we have  $\End_{T}((\Ord_P\Omega_{v})^{\vee})\cong \F[\![S_1,\dots,S_{f}]\!]$. Since \eqref{eq:Rv-res} is surjective, we may choose lifts of $S_i$ (for $1\leq i\leq f$) in $R_{v}$, say $Y_i$. Then $Y_i$ are linearly independent in $\fm_{v}/\fm_{v}^2$, and  can be extended to a minimal set of generators of $\fm_v$, say by $(Z_1,\dots,Z_f)$; here we recall that $\dim_{\F}\fm_{v}/\fm_{v}^2=2f$.
Set \[\underline{Y}=(Y_1,\dots,Y_f),\ \ \underline{Z}=(Z_1,\dots,Z_f),\ \ \underline{S}=(S_1,\dots,S_f).\]
Since $R_{v}$ is a regular local ring, $(\un{Y},\un{Z})$ necessarily form a regular sequence in $R_{v}$, which is also $M_{v}$-regular because $M_{v}$ is $R_{v}$-flat. In particular, $\un{Y}$ is an $M_{v}$-regular sequence and it defines a Koszul complex $K_{\bullet}(\un{Y},M_{v})$, which is a projective resolution of
$M_{v}/(\un{Y})$ in the category of pseudo-compact $\F[\![K/Z_1]\!]$-modules. Dually, we obtain
an injective resolution of $\widetilde{\pi}(\brho)\defn(M_{v}/(\underline{Y}))^{\vee}$:
\[0\ra \widetilde{\pi}(\brho)\ra K_{\bullet}(\un{Y},\Omega_{v}).\]
By Proposition \ref{prop-BD-ordinary=injective}(ii), $\Omega_v$ is $\Ord_P$-acyclic and so
\begin{equation}\label{eq:Ri=Hi}R^i\Ord_P\widetilde{\pi}(\brho)\cong H^i\big(\Ord_P(K_{\bullet}(\un{Y},\Omega_{v}))\big). \end{equation}
Since $\un{Y}$ acts on $\Ord_P\Omega_{v}$ via $\un{S}$, we have
\[\Ord_P(K_{\bullet}(\un{Y},\Omega_{v}))\cong K_{\bullet}(\un{S},(\Ord_P\Omega_{v})).\]
Since $\un{S}$ is a regular sequence for $(\Ord_P\Omega_{v})^{\vee}$, $K_{\bullet}(\un{S},(\Ord_P\Omega_{v}))$ is an acyclic complex with $H^0$ isomorphic to $\chi_0$.  Combining with \eqref{eq:Ri=Hi} we deduce that
\begin{equation}\label{eq:Ord-tildepi}\Ord_P\widetilde{\pi}(\brho)\cong \chi_0,\ \ \ R^i\Ord_P\widetilde{\pi}(\brho)=0,\ \ \forall i\geq 1.\end{equation}
In particular, $\widetilde{\pi}(\brho)$ is also $\Ord_P$-acyclic.

Next, we consider the action of $\underline{Z}=(Z_1,\cdots,Z_f)$ on $M_v/(\un{Y})$. By construction, $\un{Z}$ is a regular sequence for $M_{v}/(\un{Y})$, hence gives rise to  a Koszul complex $K_{\bullet}(\underline{Z},M_{v}/(\un{Y}))$ which is a resolution of $\pi(\brho)^{\vee}=M_{v}/(\un{Y},\un{Z})$. Dually we obtain a resolution of $\pi(\brho)$ of Koszul type
\[0\ra \pi(\brho)\ra K_{\bullet}(\un{Z},\widetilde{\pi}(\brho)).\]
 Moreover, since $\widetilde{\pi}(\brho)$ is $\Ord_P$-acyclic, we can calculate $R^i\Ord_P\pi(\brho)$ by taking the cohomology of the complex $\Ord_P\big(K_{\bullet}(\un{Z},\widetilde{\pi}(\brho))\big)=K_{\bullet}(\underline{Z},\Ord_P\widetilde{\pi}(\brho))$. 
In particular, we deduce from \eqref{eq:Ord-tildepi} that $R^i\Ord_P\pi(\brho)$ is semisimple and isomorphic to $\chi_0^{\oplus n_i}$ with $n_i\leq \binom{f}{i}$.

It remains  to prove the equality $n_i=\binom{f}{i}$. The spectral sequence  \eqref{equation-Emerton-SS} shows that
\[\dim_{\F}\Ext^{n}_{G,\zeta}(\pi_0,\pi(\brho))\leq \sum_{i+j=n}\dim_{\F}\Ext^j_{T,\zeta}\big(\chi_0,R^i\Ord_P\pi(\brho)\big).\]
Since $\dim_{\F}\Ext^j_{T,\zeta}(\chi_0,\chi_0)=\binom{f+1}{j}$, this inequality  translates to (by Proposition \ref{prop-reducible-dimofExt}(iv))
\[\binom{2f+1}{n}\leq\sum_{i+j=n}\binom{f+1}{j}\cdot n_i. \]
Recalling $n_i\leq \binom{f}{i}$,   Vandermonde's identity
$\binom{2f+1}{n}=\sum_{i+j=n}\binom{f+1}{j}\binom{f}{i}$ then forces $n_i=\binom{f}{i}$.
\end{proof}

\subsection{A criterion}
\label{subsection-D0}

In this subsection, we devise a criterion for which type of subspaces of $\pi(\brho)$ can generate it as a  $G$-representation.  

\begin{lemma}\label{lemma-dim-Ext2f=1}
We have $\dim_{\F}\Ext^{2f}_{I/Z_1}(\chi_0^s,\pi_f)=1$.
\end{lemma}
\begin{proof}
Recall that $\pi_f=\Ind_P^G\chi_f^s$. Restricting to $I$, we obtain a decomposition (by Mackey's theorem)
\[\pi_f|_{I}\cong \Ind_{I\cap P}^I\chi_f^s\oplus \Ind_{I\cap \overline{P}}^I\chi_f.\]
 By Shapiro's lemma, we have
 \[\Ext^{2f}_{I/Z_1}(\chi_0^s,\pi_f)\cong H^{2f}\big((I\cap P)/Z_1,(\chi_0^s)^{-1}\chi_f^s\big)\oplus H^{2f}\big((I\cap\overline{P})/Z_1,(\chi_0^s)^{-1}\chi_f\big).\]
We need to prove that only one summand of the last term is nonzero and it has dimension $1$.

Since $(I\cap P)/Z_1$ is a Poincar\'e duality group at $p$ of dimension $2f$ (see \cite[Chap.~I, Appendix 1]{Se97} or \cite[(3.4.6)]{NSW}),
  the Poincar\'e duality implies  that
\[\dim_{\F}H^{2f}((I\cap P)/Z_1, \chi)= \dim_{\F}H^0((I\cap P)/Z_1,\chi^*)\]
for any character $\chi$ of $(I\cap P)/Z_1$, where
$\chi^{*}\defn\Hom(\chi,\F)$
is the dualizing module of $\chi$ with $\F$ being endowed with an action of $(I\cap P)/Z_1$ via the usual modulus character
by \cite[Example, p. 42]{Se97}.  Explicitly, this modulus character is equal to (restriction of) $\alpha_{P}=\omega\otimes\omega^{-1}$, hence ${\chi}^*\cong \chi^{-1}(\omega\otimes \omega^{-1})$ as characters of $I\cap P$.
Using \eqref{eq:chi-0-f}, it is direct to check that
\[((\chi_0^s)^{-1}\chi_f^s)^*=(\chi_1^{-1}\chi_2\omega)\otimes
(\chi_1\chi_2^{-1}\omega^{-1}).\]
The genericity condition on $\brho$ implies that this is a nontrivial character of $I\cap P$, and so
\[H^0\big((I\cap P)/Z_1, ((\chi_0^s)^{-1}\chi_f^s )^{*}\big)=0.\]
  In a similar way, one checks that  $H^{2f}\big((I\cap \overline{P})/Z_1,(\chi_0^s)^{-1}\chi_f\big)$ has dimension $1$ (notice that the modulus character associated to $(I\cap \overline{P})/Z_1$ is $\omega^{-1}\otimes \omega$). This finishes the proof.
 \end{proof}

 \begin{lemma} \label{lemma-all-dim1}
The natural morphism $\pi(\brho)\twoheadrightarrow \pi_f$ induces the following isomorphisms
\begin{enumerate}
\item[(i)] $\Ext^{2f+1}_{G,\zeta}(\pi_0,\pi(\brho))\simto\Ext^{2f+1}_{G,\zeta}(\pi_0,\pi_f)$; 
\item[(ii)] $\Ext^{2f}_{K/Z_1}(\sigma_0,\pi(\brho))\simto \Ext^{2f}_{K/Z_1}(\sigma_0,\pi_f)$;
\item[(iii)] $\Ext^{2f}_{I/Z_1}(\chi_0^s,\pi(\brho))\simto\Ext^{2f}_{I/Z_1}(\chi_0^s,\pi_f)$.
\end{enumerate}
Moreover, all these spaces have dimension $1$ over $\F$.
\end{lemma}
\begin{proof}
Since $R^{f+1}\Ord_P=0$, the morphism $R^f\Ord_P\pi(\brho)\ra R^f\Ord_P\pi_f$ is surjective, hence is an isomorphism for the reason of dimensions using Proposition \ref{prop-Ord}(iii) and Proposition \ref{prop-ROrd-pi(rho)}. Using  Corollary \ref{cor-Ext-2f+1}  we deduce an  isomorphism
\[ \Ext^{2f+1}_{G,\zeta}(\pi_0,\pi(\brho))\simto\Ext^{2f+1}_{G,\zeta}(\pi_0,\pi_f),\]
and both the spaces have dimension $1$ because $\Ext^{f+1}_{T,\zeta}(\chi_0,\chi_0)$ has dimension $1$. This proves (i).

Recall the presentation of $\pi_0$ in \eqref{eq:BL-pi0}
\[0\lra \cInd_{\mathfrak{R}_0}^G\sigma_0\overset{T-\lambda_0}{\lra} \cInd_{\mathfrak{R}_0}^G\sigma_0\lra\pi_0\lra0.\]
  Using Frobenius reciprocity, it induces a morphism
\[\partial:\Ext^{2f}_{K/Z_1}(\sigma_0,\pi(\brho))\ra \Ext^{2f+1}_{G,\zeta}(\pi_0,\pi(\brho))\]
which is surjective as $\Ext^{2f+1}_{K/Z_1}(\sigma_0,\pi(\brho))=0$, hence is an isomorphism for the reason of dimensions, see Proposition \ref{prop-reducible-dimofExt}. Similarly we have a morphism
\[\partial': \Ext^{2f}_{K/Z_1}(\sigma_0,\pi_f)\ra \Ext^{2f+1}_{G,\zeta}(\pi_0,\pi_f)\]
which is also surjective using Lemma \ref{lemma:Kolh}.

We have the following commutative diagram
\[{\scriptsize
\xymatrix{\Ext^{2f}_{I/Z_1}(\chi_0^s,\pi(\brho))\ar^{\cong\ \ \ }[r]\ar^{\beta}[d]&\Ext^{2f}_{K/Z_1}\!\big(\Ind_I^K\!\chi_0^s,\pi(\brho)\!\big)\ar^{\ \ \iota}[r]\ar^{\beta'}[d]&\Ext^{2f}_{K/Z_1}\!(\sigma_0,\pi(\brho)\!)\ar^{\gamma}[d]\ar^{\partial}_{\cong}[r]&\Ext^{2f+1}_{G,\zeta}(\!\pi_0,\pi(\brho)\!)\ar^{\delta}[d]\\
\Ext^{2f}_{I/Z_1}(\chi_0^s,\!\pi_f)\ar^{\cong\ \ \ }[r]&\Ext^{2f}_{K/Z_1}\!\big(\Ind_I^K\!\chi_0^s,\pi_f\!\big)\ar^{\ \ \iota'}[r]& \Ext^{2f}_{K/Z_1}(\sigma_0,\pi_f)\ar^{\partial'}@{->>}[r]&\Ext^{2f+1}_{G,\zeta}\!(\pi_0,\pi_f\!),} }\]
where the two horizontal isomorphisms in the leftmost square are given by Shapiro's lemma, and $\iota$ (resp. $\iota'$) is induced by the inclusion $\sigma_0\hookrightarrow\Ind_I^K\chi_0^s$. Moreover, we have
\begin{enumerate}
\item[$\bullet$] $\delta$ is an isomorphism by (i); 
\item[$\bullet$] $\iota$ (resp. $\iota'$) is surjective because $\pi(\brho)|_K$ (resp. $\pi_f|_K$) has injective dimension $2f$.
\end{enumerate}
In particular, all horizontal morphisms are surjective. All the spaces in  the top row have dimension $1$ over $\F$ by Proposition \ref{prop-reducible-dimofExt} and Corollary \ref{cor:Exti-chi-pi2}, and   $\dim_{\F}\Ext^{2f}_{I/Z_1}(\chi_0^s,\pi_f)=1$ by Lemma \ref{lemma-dim-Ext2f=1}. It is then easy to deduce that all the spaces in the bottom row have dimension $1$ as well and that $\beta$, $\gamma$ are both isomorphisms.  This proves (ii) and (iii).
\end{proof}

Now we are ready to prove the criterion.

  \begin{proposition}\label{prop-gen-by-W}
If $W$ is an $I$-subrepresentation of $\pi(\brho)$ such that the natural morphism
\begin{equation}\label{equation-W-pi(rho)}
\Ext^{2f}_{I/Z_1}(\chi_0^s,W)\ra\Ext^{2f}_{I/Z_1}(\chi_0^s,\pi(\brho))\end{equation}
is surjective, then $\pi(\brho)$ is generated by $W$ as a $G$-representation.
\end{proposition}
\begin{proof} 
Let $\langle G.W\rangle\subset \pi(\brho)$ be the $G$-subrepresentation generated by $W$. If $\langle G.W\rangle\subsetneq \pi(\brho)$, then $\langle G.W\rangle$ is contained in $V\defn\Ker(\pi(\brho)\twoheadrightarrow \pi_f)$, because $\pi_f$ is the cosocle of $\pi(\brho)$ by Proposition \ref{prop-cosocle-pi(rho)}.  Hence the morphism \eqref{equation-W-pi(rho)} factors through $\Ext^{2f}_{I/Z_1}(\chi_0^s,V)$ as illustrated in the following diagram:
\[\xymatrix{&\Ext^{2f}_{I/Z_1}(\chi_0^s,W)\ar^{\eqref{equation-W-pi(rho)}}@{->>}[d] \ar@{-->}[ld] &
\\
\Ext_{I/Z_1}^{2f}(\chi_0^s,V)\ar[r]&\Ext^{2f}_{I/Z_1}(\chi_0^s,\pi(\brho))\ar_{\cong}^{\beta}[r] & \Ext^{2f}_{I/Z_1}(\chi_0^s,\pi_f)}\]
where $\beta$ is an isomorphism by Lemma \ref{lemma-all-dim1}(iii). But the composition of the two maps in the bottom row is zero, we get a contradiction if \eqref{equation-W-pi(rho)} is surjective.
 \end{proof}

\subsection{The representation $\tau(\brho)$}

 We define
a suitable $I$-representation $\tau(\brho)$ which can be embedded in $\pi(\brho)|_I$. In next subsection, we will show that $\pi(\brho)$ is generated by $\tau(\brho)$ as a $G$-representation, using the criterion Proposition \ref{prop-gen-by-W}.

Recall from \S\ref{section-BP} the subset $J_{\brho}\subset \cS$  attached to $\brho$.

\begin{lemma}\label{lemma:chi-J}
Let  $J\subset \cS$. The character $\chi_0^s\big(\prod_{j\in J}\alpha_j^{-1}\big) $ occurs in  $\pi(\brho)^{I_1}$ if and only if $J\subset J_{\brho}$.
\end{lemma}
\begin{proof}
 We have seen in the proof of Lemma \ref{lemma-PD-set} that the set of characters occurring in $\pi(\brho)^{I_1}$ is in bijection with a certain set of $f$-tuples   $\mathscr{PD}(x_0,\cdots,x_{f-1})$. Recall that, if $\lambda\in\mathscr{PD}(x_0,\cdots,x_{f-1})$ then, among other conditions,
\begin{equation}\label{eq:PD-lambda}
\lambda_i(x_i)\in\{x_i,x_i+1,x_i+2,p-3-x_i,p-2-x_i,p-1-x_i\}\end{equation}
and $\lambda_i(x_i)\in \{p-3-x_i,x_i+2\}$ implies $i\in J_{\brho}$.
Via this bijection, the character $\chi_0^s$ corresponds to  $(p-1-x_0,\cdots,p-1-x_{f-1})$, and $\chi_0^s (\prod_{j\in J}\alpha_j^{-1})$ corresponds to $\lambda_{J}$ where
\[(\lambda_{J})_i(x_i)\defn\left\{\begin{array}{rll}
p-1-x_i& i\notin J\\
p-3-x_i&i\in J.
\end{array}\right.\]
The result follows from this.
\end{proof}

\begin{definition}\label{def:tau-rho}
We define
\[\tau(\brho):=\chi_0^s\otimes \tau_{J_{\brho}},\]
where $\tau_{J_{\brho}}$ is the $I$-representation defined in Definition \ref{def:tau-J} with $\cJ=J_{\brho}$.
\end{definition}

As a direct consequence of Proposition \ref{prop:resolution-tauJ}, we have the following.
\begin{proposition}\label{prop:resolution-tau}
The projective dimension of $\tau(\brho)^{\vee}$ is $3f$.
Moreover, $\tau(\brho)^{\vee}$ admits a  length $3f$ minimal resolution   by projective  $\F[\![I/Z_1]\!]$-modules,
\[P_{\bullet}\ra \tau(\brho)^{\vee}\ra0\]
satisfying the following property:  for each $0\leq l\leq 3f$,
$P_{l}$ has a direct sum decomposition  \[P_{l}=P_{l}'\oplus  P_{l}''\]
such that
\begin{enumerate}
\item[(a)] $P_l'\cong \big(\bigoplus_{\chi}P_{\chi}\big)^{\binom{2f}{l}}$, where $\chi$ runs over the characters of $\mathrm{cosoc}_I(\tau(\brho)^{\vee})$;
\item[(b)]$\Hom_I(P''_l,P_{\chi}/\fm^2)=0$ for any $\chi\in\mathrm{cosoc}_I(\tau(\brho)^{\vee})$.
\end{enumerate}
As a consequence, $\dim_{\F}\Ext^{l}_{I/Z_1}(\tau(\brho)^{\vee},\chi)=\binom{2f}{l}$ for any $\chi$ occurring in $\mathrm{cosoc}_I(\tau(\brho)^{\vee})$.
\end{proposition}

Next, we study the relation between  $\tau(\brho)$ and $\pi(\brho)$.

\begin{proposition}\label{prop:tau-embeds}
There exists an embedding $\tau(\brho)\hookrightarrow \pi(\brho)|_{I}$.
\end{proposition}

We start with a lemma, which is motivated by \cite[Lem.~9.2]{Br14}.

\begin{lemma}\label{lemma:breuil}
If $\tau_1\subset \tau$ are $I$-representations such that $\tau_1\hookrightarrow \pi(\brho)|_{I}$ and
\begin{equation}\label{eq:lemma-breuil}\JH(\tau/\tau_1)\cap \JH(\pi(\brho)^{I_1})=\emptyset,\end{equation}
then the natural restriction map \[{\rm res}:\Hom_{I}(\tau,\pi(\brho))\ra \Hom_{I}(\tau_1,\pi(\brho))\]
is an isomorphism.
\end{lemma}
\begin{proof}
Using Proposition \ref{prop-reducible-dimofExt}(i), the assumption implies that \[\Hom_I(\tau/\tau_1,\pi(\brho))=\Ext^1_{I/Z_1}(\tau/\tau_1,\pi(\brho))=0,\]
from which the result follows.\end{proof}
\begin{proof}[Proof of Proposition \ref{prop:tau-embeds}]
Lemma \ref{lemma:socle-tauJ} implies that $\tau(\brho)^{I_1}$ is isomorphic to the direct sum of $\chi_0^s(\prod_{j\in J}\alpha_j^{-1})$ for all $J\subset J_{\brho}$.  Hence it follows   from Lemma \ref{lemma:chi-J} that $\tau(\brho)^{I_1}$ embeds in $\pi(\brho)^{I_1}$, hence in $\pi(\brho)|_{I}$. By Lemma \ref{lemma:breuil}, it suffices to check the condition   \eqref{eq:lemma-breuil} with $\tau=\tau(\brho)$ and $\tau_1=\tau(\brho)^{I_1}$.

By Lemma \ref{lemma:socle-tauJ} again, $\tau(\brho)$ is multiplicity free and $\JH(\tau(\brho))$ consists of the  characters of the form $\psi=\chi_0^s(\prod_{i\in \cS}\alpha_i^{b_{J,i}})$, where $J\subset J_{\brho}$ and $(b_{J,i})\in \Z^{J_{\brho}}$ satisfy
\begin{equation}\label{eq:JH-tauJ-bJi}
\left\{\begin{array}{cl}
-2\leq b_{J,i}\leq 2& \mathrm{if}\  i\notin J_{\brho} \\
0\leq b_{J,i}\leq 2& \mathrm{if}\ i\in J_{\brho}\backslash J\\
-3\leq b_{J,i}\leq -1& \mathrm{if}\ i\in J.\end{array}\right.\end{equation}
If such a character $\psi$ also occurs in $\pi(\brho)^{I_1}$, it corresponds to an element of $\mathscr{PD}(x_0,\cdots,x_{f-1})$. Using  \eqref{eq:PD-lambda} and the strong genericity of $\brho$, one checks that the only possibility is
\[b_{J,i}=\left\{\begin{array}{cl}0& \mathrm{if}\ i\notin J\\
-1&\mathrm{if}\ i\in J.
\end{array}\right.\]
In other words, $\psi$ occurs in $\tau(\brho)^{I_1}$ by Lemma \ref{lemma:socle-tauJ}. Since $\tau(\brho)$ is multiplicity free, this implies that $\psi$ can not occur in $\tau(\brho)/\tau(\brho)^{I_1}$, thus  \eqref{eq:lemma-breuil} holds.
\end{proof}

From now on, we fix an embedding $\tau(\brho)\hookrightarrow \pi(\brho)|_I$.

\begin{proposition}\label{prop:Ext1=isom}
For any $\psi\in \JH(\tau(\brho)^{I_1})$, the natural morphism
\[\Ext^1_{I/Z_1}(\psi,\tau(\brho))\ra\Ext^1_{I/Z_1}(\psi,\pi(\brho)|_I)\]
is an isomorphism.
\end{proposition}
\begin{proof}
We have a commutative diagram
\[\xymatrix{\Ext^1_{I/Z_1}(\psi,\tau(\brho)[\fm^2])\ar^{\beta'}[r]\ar^{\iota}[d]&\Ext^1_{I/Z_1}(\psi,\pi(\brho)[\fm^2])\ar^{\iota'}[d]\\
\Ext^1_{I/Z_1}(\psi,\tau(\brho))\ar^{\beta}[r]&\Ext^1_{I/Z_1}(\psi,\pi(\brho)|_I)}\]
for which the following statements hold:
\begin{enumerate}
\item[$\bullet$] $\iota'$ is an injection. Indeed, by Corollary \ref{cor:multione-Iwahori}, for any character occurring in $\pi(\brho)^{I_1}$, in particular for $\psi$, we have  $[\pi(\brho)[\fm^3]:\psi]=1$. As a consequence, \[\Hom_{I}\big(\psi,\pi(\brho)/\pi(\brho)[\fm^2]\big)=\Hom_I\big(\psi,\pi(\brho)[\fm^3]/\pi(\brho)[\fm^2]\big)=0,\]
and the injectivity of $\iota'$ follows.  
\item[$\bullet$] $\beta'$ is an injection, by a similar argument as for $\iota'$. 
\item[$\bullet$]  $\iota$ is an isomorphism. The injectivity can be seen as above because  $\tau(\brho)$ is multiplicity free. For the surjectivity it suffices to show  $\Ext^1_{I/Z_1}(\psi,\psi')=0$ for any $\psi'\in \JH(\tau(\brho)/\tau(\brho)[\fm^{2}])$. Since $\tau(\brho)$ is multiplicity free, it suffices to show that if $\psi\in \JH(\tau(\brho)^{I_1})$ such that $\Ext^1_{I/Z_1}(\psi,\psi')\neq0$ then $\psi'\in \tau(\brho)[\fm^2]$.  By Lemma \ref{lemma:socle-tauJ}, we may write $\psi=\chi_0^s(\prod_{j\in J}\alpha_j^{-1})$ for some $J\subset J_{\brho}$. Then by Lemma \ref{lemma:Ext1-chi} there exists $i\in\cS$ such that
\[\psi'\cong\psi\alpha_i^{\pm1}=\chi_0^s\Big(\prod_{j\in J}\alpha_j^{-1}\Big)\alpha_i^{\pm1}.\]
Using Remark \ref{rem:soc2-tauJ}, it is  direct to check that $\psi'\in \tau(\brho)[\fm^2]$.
\end{enumerate}
Putting these statements together, we deduce that  $\beta$ is an injection. But, we have \[\dim_{\F}\Ext^1_{I/Z_1}(\psi,\pi(\brho))=2f=\dim_{\F}\Ext^1_{I/Z_1}(\psi,\tau(\brho))\] by Corollary \ref{cor:Exti-chi-pi2} and Proposition \ref{prop:resolution-tau}, so $\beta$ is actually an isomorphism.
\end{proof}

\subsection{Main results}\label{subsection:main-generation}
The main result of the section is as follows. 

\begin{theorem}\label{thm-generation-tau}
As a $G$-representation, $\pi(\brho)$ is generated by $\tau(\brho)$.
\end{theorem}

\begin{example}\label{example:f=1}
Assume $f=1$, i.e. $L=\Q_p$. Assume $\brho$ is reducible, generic (in the sense of \cite[Def.~11.7]{BP}) and we allow  $\brho$ to be split.  By the local-global compatibility proved by Emerton (\cite{Em3}),   the representation $\pi(\brho)$ of $\GL_2(\Q_p)$ is exactly the one attached to $\brho$ by the mod $p$ local Langlands correspondence (\cite{Br03} or
\cite[Def.~2.2]{BrICM}). Precisely, if $\brho$ is split then $\pi(\brho)\cong \pi_0\oplus \pi_1$ for $\pi_i$ defined in \eqref{equation-reducible-pi0-pif}; if $\brho$ is nonsplit then $\pi(\brho)$ is the unique nonsplit extension of $\pi_1$ by $\pi_0$.

In the case $\brho$ is split,  we have inclusions (see \eqref{eq:Ei(2)} for the notation)
\[\chi_0^s\otimes E_{\ide,\alpha,\alpha^2}\hookrightarrow\pi_0|_{I},\ \ \chi_0^s\otimes E_{\alpha^{-1},\alpha^{-2},\alpha^{-3}}\hookrightarrow \pi_1|_I\]
and  $\pi_0$ (resp. $\pi_1$) is generated by this subspace  as a $\GL_2(\Q_p)$-representation (as $\pi_i$ is irreducible!). In the case $\brho$ is nonsplit, then $\chi_0^s\otimes (E_{\ide,\alpha,\alpha^2}\oplus_{\ide}E_{\ide,\alpha^{-1},\alpha^{-2}})$ embeds in $\pi(\brho)$ 
and generates it as a $\GL_2(\Q_p)$-representation.
\end{example}

\begin{remark}
(i) In view of Example \ref{example:f=1}, $\tau(\brho)$ should be thought of as the tensor product of certain well-chosen local factors for each embedding $\kappa:L\hookrightarrow E$, which depend only on the splitting behavior of $\brho$ at $\kappa$ (cf. \S\ref{Section::Serre-wts}).

(ii) In Example \ref{example:f=1},   if we replace $E_{\ide,\alpha,\alpha^2}$ (resp. $E_{\alpha^{-1},\alpha^{-2},\alpha^{-3}}$) by its (two-dimensional) subrepresentation $E_{\ide,\alpha}$  (resp. $E_{\alpha^{-1},\alpha^{-2}}$), then the statements remain true; cf. the proof of Theorem \ref{thm-generation-D0} below.   However, for technical reasons it is more convenient to look at $\tau(\brho)$:  e.g. the minimal projective resolution of $\tau(\brho)^{\vee}$ enjoys the properties of Proposition \ref{prop:resolution-tau}.
\end{remark}

Before giving the proof of Theorem \ref{thm-generation-tau}, we deduce some consequences. Recall that $\pi(\brho)^{K_1}\cong D_0(\brho)$ by the main result of \cite{Le}, where $D_0(\brho)$ is as in \S\ref{section-BP}.

\begin{theorem}\label{thm-generation-D0}
$\pi(\brho)$ is generated by $D_0(\brho)$ as a $G$-representation.
\end{theorem}
\begin{proof}
 Recall that the $G$-cosocle of $\pi(\brho)$ is isomorphic to $\pi_f$ by Proposition \ref{prop-cosocle-pi(rho)}, thus a subspace $W$ of $\pi(\brho)$ generates $\pi(\brho)$ if and only if the composite map $\iota_{W}: W\hookrightarrow \pi(\brho)\twoheadrightarrow \pi_f$ is nonzero, if and only if $\mathrm{Im}(\iota_W)\cap \pi_f^{I_1}$ is nonzero.
Since $\pi_f$ is a principal series, it is well-known that $\pi_f^{I_1}$ is two-dimensional, see \cite[Lem.~28]{BL}. Explicitly, we have   $\pi_f^{I_1}\cong\chi_f\oplus\chi_f^s$, where $\chi_f=\sigma_f^{I_1}$ with $\sigma_f=(p-3-r_0,\cdots,p-3-r_{f-1})\otimes{\det}^{\sum_{i=0}^{f-1}(r_i+1)p^i}$, see \eqref{eq:def-sigmaf}. One checks that $\chi_f=\chi_0^s(\prod_{j\in\cS}\alpha_j^{-1})$ and $\chi_f^s=\chi_0(\prod_{j\in\cS}\alpha_j)$. Moreover, using \eqref{eq:JH-tauJ-bJi} and the strong genericity of $\brho$, one checks that  $\chi_f$ is a Jordan--H\"older factor of $\tau(\brho)$ but $\chi_f^s$ is not. Indeed,
  $\chi_f$ occurs as a subquotient in the following direct summand of $\tau(\brho)$:
  \[\tau_{J_{\brho},\un{-1}}\otimes\big(\prod_{j\in J_{\brho}}\alpha_j^{-1}\big), \] where $\un{-1}$ denotes the unique element of $\{\pm1\}^{J_{\brho}}$ taking values $-1$ at all $j\in J_{\brho}$ (with the convention $\underline{-1}=\emptyset$ if $J_{\brho}=\emptyset$). 

By Theorem \ref{thm-generation-tau}, $\pi(\brho)$ is generated by $\tau(\brho)$. Since $\tau(\brho)$ is multiplicity free, the  discussion in the last paragraph shows that a subrepresentation $W$ of $\tau(\brho)$ generates $\pi(\brho)$ if and only if $[W:\chi_f]\neq 0$.
In particular, $\pi(\brho)$ is generated by the  unique subrepresentation $W$ of $\tau_{J_{\brho},\un{-1}}\otimes\big(\prod_{j\in J_{\brho}}\alpha_j^{-1}\big)$ with cosocle isomorphic to $\chi_f$.
To finish the proof  it suffices to prove
  that $W$ is contained in $D_0(\brho)|_I$,  equivalently, $K_1$ acts trivially on $W$ because $D_0(\brho)=\pi(\brho)[\fm_{K_1}]$.
This is a consequence of  the structure of $\tau_{J_{\brho},\un{-1}}$, see  Proposition \ref{prop:tauJ-eps}.
\end{proof}

 \begin{corollary}\label{cor:End=F}
 We have $\End_G(\pi(\brho))\cong\F$.
 \end{corollary}
\begin{proof}
Let $D(\brho)=(D_0(\brho), D_1(\brho),\mathrm{can})$ denote the basic $0$-diagram attached to $\pi(\brho)$ in \cite[\S13]{BP}, where $\mathrm{can}:D_1(\brho)\hookrightarrow D_0(\brho)$ is the canonical inclusion.
Any $G$-equivariant endomorphism of $\pi(\brho)$ induces an endomorphism of   $D(\brho)$, i.e. there is a natural morphism of rings
\[\End_{G}(\pi(\brho))\ra \End_{\mathcal{DIAG}}(D(\brho)),\]
where $\mathcal{DIAG}$ denotes the abelian category of diagrams (cf. \cite[\S9]{BP}).
This morphism is injective because $\pi(\brho)$ is generated by ${D}_0(\brho)$ as a $G$-representation by Theorem \ref{thm-generation-D0}. Therefore, it suffices to show that  $\End_{\mathcal{DIAG}}(D(\brho))\cong\F$. Using that $D_0(\brho)$ is multiplicity free, this follows from \cite[Thm.~15.4(i)]{BP} which says that the diagram $D(\brho)$ is indecomposable.
\end{proof}

\begin{remark}
(i) Corollary \ref{cor:End=F} provides an obviously expected property of $\pi(\brho)$ corresponding to the assumption $\End_{G_L}(\brho)\cong\F$. However, as is clear to the reader, the proof is far from obvious.

(ii) The continuous action of $R_{\infty}$ on $M_{\infty}$ induces a morphism of rings $R_{\infty}\ra \End_{G}^{\rm cont}(M_{\infty})$.  A natural question raised in \cite[6.24]{CEGGPS2} is whether this is an isomorphism. The injectivity is known to be true  by \cite[Thm.~1.8]{EP2}. Also note that a local version of this isomorphism in the case of $\GL_2(\Q_p)$ is proved in \cite[Prop.~3.12]{HP} (under mild genericity conditions on $\brho$), based on the theory of Colmez's functor (\cite{Co}, \cite{Pa13}).
\end{remark}
\medskip

\textit{From now on, we turn to the  proof of Theorem \ref{thm-generation-tau}}.

Recall that we have  a partially minimal resolution of $\pi(\brho)^{\vee}$ by  projective $\F[\![I/Z_1]\!]$-modules which is of Koszul type, i.e.
$K_{\bullet}(\underline{X},M_{v})\ra \pi(\brho)^{\vee}\ra0.$
In the rest, we write for simplicity  \[Q_{\bullet}\defn K_{\bullet}(\underline{X},M_{v}).\]

On the other hand, let $P_{\bullet}$ be a minimal projective resolution of $\tau(\brho)^{\vee}$, see Proposition \ref{prop:resolution-tau}.   The (fixed) inclusion $\tau(\brho)\hookrightarrow \pi(\brho)|_I$ induces a
quotient map $\pi(\brho)^{\vee}|_I\twoheadrightarrow \tau(\brho)^{\vee}$, which extends to a morphism of complexes
\[\xymatrix{Q_{\bullet}\ar[r]\ar_{\beta_{\bullet}}@{-->}[d]&\pi(\brho)^{\vee}|_I\ar[d]\ar[r]&0\\
 P_{\bullet}\ar[r]&\tau(\brho)^{\vee}\ar[r]&0.}\]
Let $\chi\in \mathrm{cosoc}_I(\tau(\brho)^{\vee})$ and recall $\lambda_{\chi}\defn\overline{W}_{\chi^{\vee},3}$, see Definition \ref{defn:lambda}.  Applying $\Hom_{I}(-,\lambda_{\chi}^{\vee})^{\vee}$, we obtain a morphism of complexes of $R$-modules
\begin{equation}\label{eq:mapofcomplexes}
\beta_{\chi,\bullet}^{\sharp}: \ \Hom_I(Q_{\bullet},\lambda_{\chi}^{\vee})^{\vee}\ra \Hom_I(P_{\bullet},\lambda_{\chi}^{\vee})^{\vee}.\end{equation}
where $R$ is defined in \eqref{eq:ring-R}, a subring of the center of $\F[\![I/Z_1]\!]/\fm^3$.

To simplify the notation, we write
\[K_{\chi,\bullet}\defn\Hom_I(Q_{\bullet},\lambda_{\chi}^{\vee})^{\vee},\ \ \ C_{\chi,\bullet}\defn\Hom_I(P_{\bullet},\lambda_{\chi}^{\vee})^{\vee}.\]
 Remark that, as a consequence of Proposition \ref{prop:resolution-tau},
$C_{\chi,l} \cong\Hom_I(P_{l}',\lambda_{\chi}^{\vee})^{\vee}$
is nonzero only when $0\leq l\leq 2f$, i.e. $C_{\chi,\bullet}$ has the same length as $K_{\chi,\bullet}$. We will prove inductively on $l$  that   $\beta_{\chi,l}^{\sharp}$ is an isomorphism for any $0\leq l\leq 2f$. By Corollary \ref{cor:kernel=socle}, this will imply that
 \[\overline{\beta}^{\sharp}_{\chi,l}:\ \Hom_{I}(Q_{l},\chi)^{\vee}\ra\Hom_{I}(P_{l},\chi)^{\vee}\]
is also an isomorphism. Since both $Q_{\bullet}$ and $P_{\bullet}$ are minimal resolutions relative to  $\chi$, we deduce an isomorphism
\[\Ext^{2f}_{I/Z_1}(\pi(\brho)^{\vee},\chi)^{\vee}\simto\Ext^{2f}_{I/Z_1}(\tau(\brho)^{\vee},\chi)^{\vee}. \]
Letting $\chi=(\chi_0^s)^{\vee}$ and taking dual, we  conclude the proof by the criterion Proposition \ref{prop-gen-by-W}.

\subsubsection{The complex $K_{\chi,\bullet}$}

\begin{lemma}\label{lemma:decomposition-Q}
For any $0\leq l\leq 2f$, $Q_l$ has a direct sum decomposition $Q'_l\oplus Q''_l$ with the following properties:
\begin{enumerate}
\item[(a)] $Q_l'\cong \big(\bigoplus_{\chi}P_{\chi}\big)^{\binom{2f}{l}}$, where $\chi$ runs over characters in $\mathrm{cosoc}_I(\tau(\brho)^{\vee})$;  
\item[(b)]$\Hom_I(Q''_l,P_{\chi}/\fm^2)=0$ for any $\chi$ in $\mathrm{cosoc}_I(\tau(\brho)^{\vee})$.
\end{enumerate}
Moreover, $Q_{\bullet}$ is partially minimal relative to $\mathrm{cosoc}_I(\tau(\brho)^{\vee})$ in the sense that for any $\chi$ in $\mathrm{cosoc}_I(\tau(\brho)^{\vee})$, the morphism
\[\Hom_I(Q_{l-1},\chi)\ra \Hom_I(Q_l,\chi)\]
is zero.
\end{lemma}
\begin{proof}
Since $Q_{\bullet}$ is a Koszul complex, it suffices to prove such a decomposition for $l=0$, i.e. decompose $M_v=M_v'\oplus M_v''$ in such a way that (a) and (b) are satisfied with $l=0$. Dually we may work with $\Omega_v$.

The construction is similar to  Lemma \ref{lemma:decomp-M}.
Recall that $\mathscr{PD}^{\dag}(x_0,\cdots,x_{f-1})$ defined in \eqref{eq:def-PD'} is a certain subset of $\mathscr{PD}(x_0,\cdots,x_{f-1})$ whose corresponding characters all occur with multiplicity one in $\Omega_v^{I_1}$. By the proof of Lemma \ref{lemma:chi-J},  $\JH(\tau(\brho)^{I_1})$ corresponds to the subset of  $\mathscr{PD}(x_0,\cdots,x_{f-1})$ consisting of $\lambda$ with $\lambda_i(x_i)\in\{p-1-x_i,p-3-x_i\}$ for all $i\in\cS$, which is a subset of $\mathscr{PD}^{\dag}(x_0,\cdots,x_{f-1})$.  We let \[\Omega_v'=\bigoplus_{\chi\in \tau(\brho)^{I_1}}\rInj_{I/Z_1}\chi\] and $\Omega_v''$  be a  complement of $\Omega_v'$ in $\Omega_v$.  Condition (b) can be checked directly, as in the proof of Lemma \ref{lemma:decomp-M}. The last assertion follows  from Proposition \ref{prop:K=minimal-I}.
\end{proof}

Fix a character $\chi$ occurring in $\mathrm{cosoc}_I(\tau(\brho)^{\vee})$.  Set
\begin{equation}\label{eq:def-mchi}m_{\chi}\defn|\JH\big(\mathrm{cosoc}_I(\tau(\brho)^{\vee})\big)\cap \mathscr{E}(\chi)|.\end{equation}
By Lemma \ref{lemma:decomposition-Q}, $m_{\chi}$ is also equal to $|\JH\big(\mathrm{cosoc}_I(\pi(\brho)^{\vee})\big)\cap \mathscr{E}(\chi)|$. 
It is clear that $0\leq m_{\chi}\leq 2f$.  The next  lemma  shows  that we are in the setting of \S\ref{subsection:example}.

\begin{lemma}\label{lemma:def-mchi}
We have $\Hom_I(M_{v},\lambda_{\chi}^{\vee})^{\vee}\cong R\oplus \F^{m_{\chi}}.$
\end{lemma}
\begin{proof}
This is a direct consequence of Proposition \ref{prop:module-Pchi}.
\end{proof}

 Since $X_i$ acts on $M_{v}$, it also acts on $\Hom_{I}(M_{v},\lambda_{\chi}^{\vee})^{\vee}$ and this action commutes with the action of $R$ (via $\lambda_{\chi}^{\vee}$). In other words, $X_i$ induces an $R$-linear endomorphism of $\Hom_{I}(M_{v},\lambda_{\chi}^{\vee})^{\vee}$.  Let \[R_{\chi}'\defn\End_R\big(\Hom_{I}(M_{v},\lambda_{\chi}^{\vee})^{\vee}\big)\] and $\phi_{\chi,i}\in R_{\chi}'$ be the element induced by $X_i$. Also let $J_{\chi}$ be the left ideal of $R'_{\chi}$ generated by $\phi_{\chi,i}$ for $1\leq i\leq 2f$.

On the other hand,
let $\fb_{\chi}$ denote the ideal of $R$ spanned by $t_{\chi'}$ for all $\chi'\in \JH\big(\mathrm{cosoc}_I(\tau(\brho)^{\vee})\big)\cap \mathscr{E}(\chi)$,  where $t_{\chi'}$ is as in Definition \ref{def:tchi}. Then  $\dim_{\F}\fb_{\chi}=m_{\chi}$ by Lemma \ref{lemma:tchi-independent}.
Recall that   we can associate to $\fb_{\chi}$ a two-sided ideal $J_{\fb_{\chi}}$ of $R_{\chi}'$, see \eqref{eq:def-Jb}.

\begin{lemma}\label{lemma:dim=mchi}
With the above notation, we have    $J_{\chi}=J_{\fb_{\chi}}$. In particular, $J_{\chi}$ is a two-sided ideal of $R_{\chi}'$.
\end{lemma}
\begin{proof}
Recall that $\Hom_I(-,\lambda_{\chi}^{\vee})^{\vee}$ is covariant and right exact. From the (right) exact sequence
$\bigoplus_{i=1}^{2f}M_{v}\overset{\oplus_i X_i}{\lra} M_{v}\lra \pi(\brho)^{\vee}\lra0,$
we obtain
\[\bigoplus_{i=1}^{2f}\Hom_{I}(M_{v},\lambda_{\chi}^{\vee})^{\vee}\overset{\oplus_i\phi_{\chi,i}}{\lra} \Hom_{I}(M_{v},\lambda_{\chi}^{\vee})^{\vee}\lra \F\lra0,\]
where we have used the fact $\Hom_{I}(\lambda_{\chi},\pi(\brho))\cong \F$ (a consequence of Corollary \ref{cor:multione-Iwahori}).
 Equivalently, $\Hom_{I}(M_{v},\lambda_{\chi}^{\vee})^{\vee}/\Hom_{I}(M_{v},\lambda_{\chi}^{\vee})^{\vee}J_{\chi}$ is one-dimensional over $\F$.

It follows from Lemma \ref{lemma:image-tchi} that $J_{\chi}$
 is contained in $J_{\fb_{\chi}}$; indeed, recalling $\Hom_{I}(M_v,\lambda_{\chi}^{\vee})^{\vee}\cong R\oplus \F^{m_{\chi}}$, $J_{\chi}$ sends $ \F^{ m_{\chi}}$ to $\fb_{\chi}\oplus (0)^{  m_{\chi}}$  by Lemma \ref{lemma:image-tchi}   and sends $R$ to $\fm_R\oplus \F^{ m_{\chi}}$.  Here we need the partial minimality of $K_{\bullet}(\un{X},M_{v})$ in Lemma \ref{lemma:decomposition-Q} to apply Lemma \ref{lemma:image-tchi}.
 We claim that $J_{\chi}=J_{\fb_{\chi}}$. Indeed,
by Lemma \ref{lemma:J=Jb-2},  $J_{\chi}=J_{\fb}$ for some ideal $\fb$ of $R$ with $\dim_{\F}\fb \geq m_{\chi}$. The inclusion $J_{\fb}\subset J_{\fb_{\chi}}$ and the fact $ \dim_{\F}\fb_{\chi}=m_{\chi}$ then force $\fb=\fb_{\chi}$.
\end{proof}

\begin{corollary}\label{cor:K-injective-A}
The natural morphism
\begin{equation}\label{eq:cor-K-injective}
K_{\chi,l}/K_{\chi,l}J_{\chi}\ra K_{\chi,l-1}J_{\chi}/K_{\chi,l-1}J_{\chi}^2\end{equation}
is injective for any $1\leq l\leq 2f$.
\end{corollary}
\begin{proof}
It follows from Proposition \ref{prop:dimb=m} and Lemma \ref{lemma:dim=mchi}  that \eqref{eq:cor-K-injective} is  injective (actually an isomorphism)  for $l=1$. We conclude by Lemma \ref{lemma:Serre-general}.
\end{proof}

Using Lemma \ref{lemma:JM} and Lemma \ref{lemma:dim=mchi}, the above corollary can be restated as follows.
\begin{corollary}\label{cor:K-injective-B}
The differential map of $K_{\chi,\bullet}$ induces an injection
\[K_{\chi,l}/\soc_R(K_{\chi,l})\ra \soc_R(K_{\chi,l-1})/\fb_{\chi}K_{\chi,l-1}.\]
\end{corollary}
\begin{remark}
The reason to restate Corollary \ref{cor:K-injective-A} in the form of Corollary \ref{cor:K-injective-B} is that the morphism $\beta_{\chi,\bullet}^{\sharp}: K_{\chi,\bullet}\ra C_{\chi,\bullet}$ is only $R$-linear but not $R'$-linear (in fact  $R'$ does not act on $C_{\chi,\bullet}$). See the diagram \eqref{eq:diagram-injective} below.
\end{remark}

\subsubsection{The complex $C_{\chi,\bullet}$}

\begin{lemma}\label{lemma:P-soc-to-b}
The differential maps of $C_{\chi,\bullet}$ induce morphisms
\[C_{\chi,l}\ra \soc_R(C_{\chi,l-1}), \ \ \soc_R(C_{\chi,l})\ra \fb_{\chi}C_{\chi,l-1}.\]
\end{lemma}
\begin{proof}
By construction, $P_{\bullet}$ is a minimal resolution, i.e. $d(P_l)\subset \fm P_{l-1}$.  The result is  a consequence of Proposition \ref{prop:beta-sharp-minimal}, using that the ideal $\fb$ in (ii) of \emph{loc. cit.} is exactly $\fb_{\chi}$ by Proposition \ref{prop:resolution-tau}.
\end{proof}

\subsubsection{A  lemma}

 \begin{lemma}\label{lemma:equivalence}
Fix $0\leq l\leq 2f$. The following conditions are equivalent:

\begin{enumerate}
\item[(i)]  for any $\chi\in \mathrm{cosoc}_I(\tau(\brho)^{\vee})$, $\beta_{\chi,l}^{\sharp}:K_{\chi,l}\ra C_{\chi,l}$ is an isomorphism;
\item[(ii)] for any $\chi\in \mathrm{cosoc}_I(\tau(\brho)^{\vee})$, $\overline{\beta}_{\chi,l}^{\sharp}: K_{\chi,l}/\soc_R(K_{\chi,l})\ra C_{\chi,l}/\soc_R(C_{\chi,l})$ is an isomorphism;
\item[(iii)] for any $\chi\in \mathrm{cosoc}_I(\tau(\brho)^{\vee})$, $\overline{\beta}_{\chi,l}^{\sharp}: K_{\chi,l}/\soc_R(K_{\chi,l})\ra C_{\chi,l}/\soc_R(C_{\chi,l})$ is an injection.
\end{enumerate}
\end{lemma}

\begin{proof}
It is clear that  (i)$\Rightarrow$(ii).

(ii)$\Rightarrow$(i). Recall the decompositions
\[Q_l=Q_{l}'\oplus Q_{l}'',\ \ \ P_{l}=P_l'\oplus P_{l}''\]
from Lemma \ref{lemma:decomposition-Q} and Proposition \ref{prop:resolution-tau}, respectively. By \emph{loc. cit.}, we know that
\[K_{\chi,l}\cong \Hom_I(Q_{l}',\lambda_{\chi}^{\vee})^{\vee},\ \ C_{\chi,l}\cong\Hom_I(P_l',\lambda_{\chi}^{\vee})^{\vee},\]
and  both $Q_{l}'$ and $P_{l}'$ are isomorphic to $\big(\bigoplus_{\chi}P_{\chi}\big)^{\binom{2f}{l}}$,  where $\chi$ runs over  characters in $\mathrm{cosoc}_I(\tau(\brho)^{\vee})$. As a consequence, $K_{\chi,l}\cong C_{\chi,l}$ as $R$-modules.

  Consider  the composite morphism \[\gamma_l: Q_{l}'\hookrightarrow Q_l\overset{\beta_{l}}{\ra} P_l\twoheadrightarrow P_{l}'. \]
Using Corollary \ref{cor:kernel=socle}, Condition (ii) implies that the induced morphism $\Hom_{I}(Q_l',\chi)^{\vee}\ra \Hom_I(P_l',\chi)^{\vee}$ is an isomorphism for \emph{any} $\chi\in \mathrm{cosoc}_I(\tau(\brho)^{\vee})$, meaning that  $\gamma_l$ induces an isomorphism on the cosocles.  Hence, $\gamma_l$ is itself a surjection by Nakayama's lemma. Moreover, since $P_l'$ and $Q_l'$ are isomorphic and finitely generated as $\F[\![I/Z_1]\!]$-modules, $\gamma_l$ must be an isomorphism which implies (i).

(ii)$\Leftrightarrow$(iii) We saw that $\dim_{\F}K_{\chi,l}/\soc_R(K_{\chi,l})=\dim_{\F}C_{\chi,l}/\soc_R(C_{\chi,l})$, so the equivalence is  obvious for the reason of  dimensions.
\end{proof}

Remark that, in general, $\soc_R(M)$ is \emph{not} contained in $\fm_RM$ (see \S\ref{subsection:socle}), so we can not directly apply Nakayama's lemma in Lemma \ref{lemma:equivalence} when deriving (i) from (ii) if we work with a \emph{single} $\chi$.

\subsubsection{End of the proof}

Now we can  complete the proof of Theorem \ref{thm-generation-tau}.
\begin{proof}[Proof of Theorem \ref{thm-generation-tau}]
Recall that we want to prove $\beta_{\chi,l}^{\sharp}$ is an isomorphism for all $\chi$ in $\mathrm{cosoc}_I(\tau(\brho)^{\vee})$ and all $0\leq l\leq 2f$.
First, the statement is obvious if $l=0$.
Also, Proposition \ref{prop:Ext1=isom} combined with Lemma \ref{lemma:equivalence} implies the statement for $l=1$.

Since $\beta_{\chi,l}^{\sharp}$ is $R$-linear, it induces morphisms
\[\soc_R(K_{\chi,l})\ra \soc_R(C_{\chi,l}),\ \ \fb_{\chi}K_{\chi,l}\ra \fb_{\chi}C_{\chi,l}\]
which are isomorphisms whenever $\beta_{\chi,l}^{\sharp}$ is.
By Corollary \ref{cor:K-injective-B} and Lemma \ref{lemma:P-soc-to-b}, we obtain a commutative diagram
\begin{equation}\label{eq:diagram-injective}
\xymatrix{K_{\chi,l}/\soc_R(K_{\chi,l}) \ar[r]\ar^{\overline{\beta}_{\chi,l}^{\sharp}}[d]&\soc_R(K_{\chi,l-1})/\fb_{\chi}K_{\chi,l-1} \ar[d]\\
C_{\chi,l}/\soc_R(C_{\chi,l})\ar[r]&\soc_R(C_{\chi,l-1})/\fb_{\chi}C_{\chi,l-1}.}\end{equation}
By inductive hypothesis, $\beta_{\chi,l-1}^{\sharp}: K_{\chi,l-1}\simto C_{\chi,l-1}$ is an isomorphism, hence the vertical map on the right in \eqref{eq:diagram-injective} is also an isomorphism as explained above.  Since the upper horizontal map is injective by Corollary \ref{cor:K-injective-B},  $\overline{\beta}_{\chi,l}^{\sharp}$ is also injective.  Finally, this being true for any $\chi$ in $\mathrm{cosoc}_I(\tau(\brho)^{\vee})$, we deduce that $\beta_{\chi,l}^{\sharp}$ is an isomorphism by Lemma \ref{lemma:equivalence}, thus finishing the proof by induction.
\end{proof}

\subsection{The case $f=2$}

In this subsection, we will specialize to the  situation when  $f=2$, i.e. $L=\Q_{p^2}$.
 The main result is the following.

\begin{theorem}\label{thm-main-f=2}
 $\pi(\brho)$ has length $3$, with a unique Jordan--H\"older filtration of the form
\[\pi(\brho)=(\pi_0\ \ligne\ \pi_1\ \ligne\ \pi_2),\]
where $\pi_0,\pi_2$ are defined in \eqref{equation-reducible-pi0-pif} and $\pi_1$ is a  supersingular representation.
\end{theorem}
\begin{proof}
We first fix some notation. Let $\brho^{\rm ss}$ denote the semisimplification of $\brho$.  Since $f=2$,  $\mathscr{D}(\brho^{\rm ss})$ consists of $4$ Serre weights, which we enumerate as follows (cf. \cite[\S16, Case (ii)]{BP}): $\mathscr{D}(\brho^{\rm ss})=\{\sigma_0,\sigma_1,\sigma_1^{[s]},\sigma_2\}$, where  (see Definition \ref{def:delta} for the notation $\mu_i^{*}$)
\[\sigma_0 = \soc_{K}(\pi_0),\ \ \sigma_2=\soc_K(\pi_2),\]
\[\sigma_1=\mu_0^+(\sigma_0), \ \ \ \sigma_1^{[s]}=\mu_1^+(\sigma_0).\]
On the other hand, since $\brho$ is assumed to be nonsplit, $\mathscr{D}(\brho)$ is a \emph{proper} subset of $\mathscr{D}(\brho^{\rm ss})$ of cardinality $2^{|J_{\brho}|}$.

We already know $\rsoc_G(\pi(\brho))\cong \pi_0$, see Proposition \ref{prop-cosocle-pi(rho)}. By \cite[Prop.~3.2]{HuJLMS}, $\pi(\brho)/\pi_0$ admits a unique irreducible subrepresentation $\pi_1$ which is supersingular and satisfies
\begin{equation}\label{eq:soc-pi1}\rsoc_{K}(\pi_1)=\sigma_1\oplus\sigma_1^{[s]}.\end{equation}
Let $\kappa\subset \pi(\brho)$ denote the pullback of $\pi_1$.
We need to show $\pi(\brho)/\kappa$ is irreducible, hence it is automatically isomorphic to $\pi_2$ (as its cosocle is isomorphic to $\pi_2$).

By Theorem \ref{thm-generation-D0}, $\pi(\brho)$ is generated by $D_0(\brho)$ as a $G$-representation; in fact, the proof in \emph{loc. cit.} shows that $\pi(\brho)$ can be generated by any $K$-subrepersentation of $D_0(\brho)$ which admits $\sigma_2$ as a subquotient. As a consequence, since $\kappa$ is a proper subrepresentation of $\pi(\brho)$, $\sigma_2$ does not occur in $\kappa\cap D_0(\brho)$. We claim that there exists an embedding
\[\sigma_2\hookrightarrow D_0(\brho)/(\kappa\cap D_0(\brho)).\]
First, it is clear that $\pi_0\cap D_0(\brho)=\pi_0^{K_1}$, so we have an embedding
\[\soc_K(\pi_1)=\sigma_1\oplus\sigma_{1}^{[s]}\hookrightarrow D_0(\brho)/(\pi_0\cap D_0(\brho)).\]
Denote by $D_{\kappa}$ the pullback of $\sigma_1\oplus \sigma_1^{[s]}$ in $D_0(\brho)$. Then $D_{\kappa}$ is contained in $\kappa\cap D_0(\brho)$.  Now  the structure of $D_0(\brho)$, see \cite[\S16]{BP}, implies that $\sigma_2$ occurs in the socle of $D_0(\brho)/D_{\kappa}$.  This gives the claimed morphism
\[\sigma_2\hookrightarrow D_0(\brho)/D_{\kappa}\twoheadrightarrow D_0(\brho)/(\kappa\cap D_0(\brho));\]
it is injective   by the discussion at the beginning of the paragraph.

By the claim, we obtain an embedding $\sigma_2\hookrightarrow (\pi(\brho)/\kappa)|_{\mathfrak{R}_0}$ (here we endow $\sigma_2$ with a compatible action of $Z$), which further induces by Frobenius reciprocity a $G$-equivariant morphism
\[h: \cInd_{\mathfrak{R}_0}^G\sigma_2\ra \pi(\brho)/\kappa.\]
Moreover, since the composition
\[\cInd_{\mathfrak{R}_0}^G\sigma_2\ra \pi(\brho)/\kappa\twoheadrightarrow  \pi_2\]
is surjective,  $h$ is   surjective as well, because $\pi_2$ is the cosocle of $\pi(\brho)$.
By Lemma \ref{lem-quotient-I(sigma)} below, $\pi(\brho)/\kappa$ is isomorphic to $\cInd_{\mathfrak{R}_0}^G\sigma_2/(T-\lambda_2)^n$ for some $n\geq 1$ and suitable $\lambda_2\in\F^{\times}$ (determined by $\pi_2$), thus
$\dim_{\F}R^2\Ord_P(\pi(\brho)/\kappa)=n$
 by \cite[Thm.~30(3)]{BL} and Proposition \ref{prop-Ord}(iii).    However, there is a surjection $R^2\Ord_P\pi(\brho)\twoheadrightarrow R^2\Ord_P(\pi(\brho)/\kappa)$, and we know that $R^2\Ord_P\pi(\brho)$ is isomorphic to $\chi_0$ by Proposition \ref{prop-ROrd-pi(rho)}, so we must have $n=1$.
\end{proof}

\begin{lemma}\label{lem-quotient-I(sigma)}
Let $\sigma$ be a Serre weight and $V$ be an admissible quotient of $I(\sigma)\defn\cInd_{\mathfrak{R}_0}^G\sigma$. Assume that the $G$-cosocle of $V$ is irreducible and isomorphic to $I(\sigma)/(T-\lambda)$ for some $\lambda\in\F^{\times}$. Then $V$ is isomorphic to $I(\sigma)/(T-\lambda)^n$ for some $n\geq 1$. In particular, $V$ has finite length.
\end{lemma}
\begin{remark}
If $L=\Q_p$, then Lemma \ref{lem-quotient-I(sigma)} follows from the work of \cite{BL,Br03}. However, when $L\neq \Q_p$, the quotient $I(\sigma)/T$ has infinite length, and it is not clear whether an arbitrary admissible quotient of $I(\sigma)$ is automatically of finite length.
 \end{remark}

\begin{proof}
Write $\pi=I(\sigma)/(T-\lambda)$ which is irreducible by assumption, and let $V_1$ be the kernel of the natural projection $V\twoheadrightarrow \pi$. Clearly, we may assume $V_1\neq0$.  We  claim that $\Hom_G(\pi,V_1)\neq 0$.  Indeed, applying $\Hom_G(-,V_1)$ to the exact sequence $0\ra I(\sigma)\overset{T-\lambda}{\ra} I(\sigma)\ra \pi\ra0$ we obtain (by Frobenius reciprocity)
\begin{multline*}0\ra \Hom_G(\pi,V_1)\ra \Hom_{\mathfrak{R}_0}(\sigma,V_1)\overset{T-\lambda}{\ra} \Hom_{\mathfrak{R}_0}(\sigma,V_1)\overset{\partial}{\ra} \Ext^1_G(\pi,V_1)\overset{\phi}{\ra} \Ext^1_{\mathfrak{R}_0}(\sigma,V_1).\end{multline*}
If $\Hom_G(\pi,V_1)$ were zero, then $T-\lambda$ would be injective, hence an isomorphism because $\Hom_{\mathfrak{R}_0}(\sigma,V_1)$ is finite dimensional over $\F$ by the admissibility of $V$. This would   imply that $\phi$ is injective. On the other hand, since $\pi$ is the $G$-cosocle of $V$, the extension
\[0\ra V_1\ra V\ra \pi\ra0\]
is nonsplit, which we denote by $c\in \Ext^1_G(\pi,V_1)$. Since $V$ is a quotient of $I(\sigma)$, the  composite morphism (where the first one is induced  from the identity map $I(\sigma)\ra I(\sigma)$ via   Frobenius reciprocity)
\[\sigma\hookrightarrow I(\sigma)|_{\mathfrak{R}_0}\twoheadrightarrow V|_{\mathfrak{R}_0}\twoheadrightarrow \pi|_{\mathfrak{R}_0}\]
is nonzero with image contained in   $\soc_{\mathfrak{R}_0}(\pi)$. This means  $\phi(c)=0$, which contradicts the injectivity of $\phi$.

Let $V_2$ be the maximal subrepresentation of $V_1$ whose irreducible subquotients are all isomorphic to $\pi$, so that $\Hom_G(\pi,V_1/V_2)=0$. If $V_1/V_2\neq0$, then the same argument as in last paragraph  (applied to $V/V_2$), shows that $\Hom_G(\pi,V_1/V_2)\neq0$, a contradiction to the choice of $V_2$.   Therefore, $V_1/V_2=0$ and all Jordan--H\"older factors of $V$ are isomorphic to $\pi=I(\sigma)/(T-\lambda)$.
On the other hand, by \cite[Thm.~19]{BL} the quotient map $I(\sigma)\twoheadrightarrow V$ factors through the quotient $I(\sigma)/f(T)$ for some nonzero polynomial $f(T)\in \F[T]\cong\End_G(I(\sigma))$.  We claim that $f(T)$ can be chosen to be $(T-\lambda)^n$ for some $n\geq 1$; this implies the lemma by choosing    $n$ minimal. Indeed, \cite[Cor.~36]{BL} implies that for any $\lambda'\in\F$ with $\lambda'\neq \lambda$, \[\Hom_G\big(I(\sigma)/(T-\lambda'),\pi\big)=0\] and consequently
$\Hom_G\big(I(\sigma)/(T-\lambda'),V\big)=0$ from which the claim  follows.
\end{proof}

We have the following immediate consequence of Theorem \ref{thm-main-f=2}.
\begin{corollary}
Assume $f=2$. With the notation of Corollary \ref{cor:Pi(x)}, the unitary admissible Banach   representation $\Pi(x)$ of $G$ has length $\leq 3$.
\end{corollary}

\appendix

\section{Non-commutative Iwasawa theory}
\label{section:appendix}

\subsection{Preliminaries}\label{subsection-duality}

We recall results of \cite{Laz}, \cite{Ven} and \cite{Ko}.
Let $R$ be a left and right noetherian ring (not necessarily commutative) and $M$ be a (left)  $R$-module. If $M\neq0$, the \emph{grade} $j_{R}(M)$ of $M$ over $R$ is defined by
\begin{equation}\label{eq:app-grade}j_{R}(M)=\mathrm{inf}\{i \in \N ~|~ \Ext^i_{R}(M,R)\neq0\}.\end{equation}
By convention, $j_R(0)=\infty$. For simplicity, we write $\EE^i(M):=\Ext^i_{R}(M, R)$.

The ring $R$ is called \emph{Auslander-Gorenstein} if it has finite injective dimension and the following Auslander condition holds: for any $R$-module $M$, any integer $m\geq0$ and any $R$-submodule $N$ of $\EE^m(M)$, we have $j_{R}(N)\geq m$. An Auslander-Gorenstein ring is  called \emph{Auslander regular} if it has finite global dimension.

Let $G_0$ be a compact $p$-adic analytic group.  Define the \emph{Iwasawa algebra} of $G_0$ over $\F$ as
\[ \L (G_0) \defn  \F[\![G_0]\!] = \varprojlim_{N\triangleleft G_0}\F[G_0/N].\]
The ring-theoretic properties of $ \L (G_0)$ are established by the fundamental works of Lazard \cite{Laz} and Venjakob \cite{Ven}. In particular, if $G_0$ has no element of order $p$,  then $ \L (G_0)$ is an
Auslander regular ring of dimension $\dim G_0$, where $\dim G_0$ is the dimension of $G_0$ as a $\Q_p$-analytic variety. If $M$ is nonzero, we have
\[
0 \leq j_{ \L (G_0)} (M) \leq \dim G_0.
\]
Define the \emph{dimension} of $M$ over $\Lambda(G_0)$ by
\[
\dim_{\Lambda(G_0)}( M) \defn \dim G_0 - j_{\Lambda(G_0)}(M).
\]

Let $G$ be a $p$-adic analytic group with a fixed open compact subgroup $G_0 \subseteq G.$   Set
\begin{equation}\label{eqn::Lambda(G)}
\L(G) \defn \F[G]\otimes_{\F[G_0]}  \L (G_0).
\end{equation}
As explained in \cite[\S1]{Ko} $\L(G)$ does not depend on the choice of $G_0$.

Let $\Mod_{\Lambda(G)}^{\rm pc}$ be the category of pseudo-compact $\F$-vector spaces $M$ carrying an $\F$-linear action of $G$ such that the map $G\times M\ra M$ is jointly continuous. Let $\mathcal{C}_G$ be the full subcategory of coadmissible objects, i.e. finitely generated as a $ \L (G_0)$-module for the fixed, equivalently any, open compact subgroup $G_0$ of $G$.

It is explained in \cite[\S3]{Ko} that if $M\in \Mod_{\Lambda(G)}^{\rm pc},$ then $\EE^i(M)$ carries naturally a structure of $\Lambda(G)$-module so that $\EE^i(M)\in \Mod_{\Lambda(G)}^{\rm pc}$. Moreover, $\EE^i$ preserves the coadmissibility, i.e. it restricts to a functor $\EE^i:\mathcal{C}_G\ra \mathcal{C}_G$ (see \cite[Cor.~3.3]{Ko}).
By abuse of notation, for $M\in \mathcal{C}_G$ we often write \[j_G(M)=j_{\Lambda(G_0)}(M),\ \ \dim_G(M)=\dim_{\Lambda(G_0)}(M).\]

For any $M\in\mathcal{C}_G$, there is a double duality spectral sequence, see \cite[\S3.1]{Ven}, which implies  that if $M$ is nonzero of grade $c$ then there is a natural nonzero double duality map $\phi_M:M\ra \EE^c\EE^c(M)$. By functoriality $\phi_M$ is a morphism in $\mathcal{C}_G$. 

\begin{lemma} \label{lem-phi-M}
Let $M$ be an object of grade $c$. The double duality map $\phi_M:M\ra \EE^c\EE^c(M)$ is nonzero, and we have a long exact sequence
\[0\ra \Ker(\phi_M)\ra M\overset{\phi_M}{\ra}\EE^c\EE^c(M)\ra \mathrm{Coker}(\phi_M)\ra0. \]
Moreover, $\Ker(\phi_M)$ (resp. $\mathrm{Coker}(\phi_M)$) has grade $\geq c+1$ (resp. $\geq c+2$).
\end{lemma}
\begin{proof}
See \cite[Prop.~3.5(i)]{Ven}.
\end{proof}

Let  $\rRep_{\F}(G)$
(resp. $\rRep_{\F}^{\rm adm}(G)$) denote the category of smooth (resp.  smooth admissible) representations of $G$ on $\F$-vector spaces.

\begin{proposition}\label{prop::appendix-pontryagin-duality}
The Pontryagin dual $V\mapsto V^{\vee}$ establishes an anti-equivalence of categories between $\rRep_{\F}(G)$ (resp. $\rRep_{\F}^{\rm adm}(G)$) and $\Mod_{\Lambda(G)}^{\rm pc}$ (resp. $\mathcal{C}_G$). \end{proposition}
\begin{proof}
See \cite[Thm.~1.5, Cor.~1.8]{Ko}.
\end{proof}

Let $\pi \in \rRep_{\F}^{\rm adm}(G).$ By Proposition \ref{prop::appendix-pontryagin-duality}, $\pi^{\vee} \in \mathcal{C}_G.$ The {\em Gelfand-Kirillov} dimension  of $\pi$ is defined by (see \cite[Rem.~5.1.1]{BHHMS})
\begin{equation}\label{equ::GK-dim-defn}
\dim_G(\pi) \defn \dim_G(\pi^{\vee}) = \dim (G_0) - j_G(\pi^{\vee}).
\end{equation}
\cite[Prop.~2.18]{EP2} provides the following description of $\dim_G(\pi).$ Let $G_0^{p^n}$ be the subgroup of $p^n$-th powers of elements of $G_0.$ Then there exist real numbers $a\geq b\geq \frac{1}{(\dim_G(\pi))!}$ such that
\[
b p^{n \dim_G(\pi)} + O (p^{n(\dim_G(\pi) -1)}) \leq \dim_{\F} ( \pi^{G_0^{p^n}} ) \leq a p^{n \dim_G(\pi)} + O (p^{n(\dim_G(\pi) -1)}) .
\]

\subsection{Socle and cosocle} \label{subsection-appendix-socle}

Let $R$ be a ring with unit and $M$ be a left $R$-module.
\begin{definition}\label{defn:app-socle}
(i) A submodule $N\subseteq M$ is called \emph{essential} if every nonzero submodule of $M$ intersects $N$ nontrivially.

(ii) A submodule $N\subseteq M$ is called \emph{small} if for any submodule $H$ of $M$, $N+H=M$ implies  $H=M$.

(iii) The \emph{socle} of $M$, denoted by $\mathrm{soc}(M)$, is  the sum of all simple submodules of $M$; we set $\rsoc(M)=0$ if there are no simple submodules of $M$.

(iv) The \emph{radical} of $M$, denoted by $\mathrm{rad}(M)$, is the intersection of all the maximal submodules of $M$; we set $\mathrm{rad}(M)=M$ if there are no maximal submodules of $M$.
The \emph{cosocle} of $M$, denoted by $\mathrm{cosoc}(M)$, is defined to be $M/\mathrm{rad}(M)$.
\end{definition}

If $M\neq 0$ is noetherian, then $\mathrm{cosoc}(M)\neq0$ and $\rad(M)\subsetneq M$ is a small submodule. If $M$ is artinian, then $\mathrm{soc}(M)\neq0$ and  $\rsoc(M)\subsetneq M$ is an essential submodule.

\begin{lemma}\label{lemma-appendix-small}
Let $h:M\ra M'$ be a nonzero morphism of $R$-modules. Let $N\subset M$ be a small submodule, then  $h(N)$ is a small submodule of $M'$.
\end{lemma}
\begin{proof}
Let $H'\subset M'$ be a submodule such that $h(N)+H'=M'$. Then  a standard argument shows that $N+h^{-1}(H')=M$, hence $h^{-1}(H')=M$ as $N$ is  small. This implies $H'\supset h(M)$, and so  $H'=M'$.
\end{proof}

From now on, we let $R=\Lambda(G)$.

\begin{example}\label{exam-appendix}
  Since $\Lambda(G_0)$ is noetherian, $\mathcal{C}_G$ is a noetherian category. Hence $\Rep_{\F}^{\rm adm}(G)$ is artinian by Proposition \ref{prop::appendix-pontryagin-duality}. In particular, if $\pi\in \Rep_{\F}^{\rm adm}(G)$ is nonzero, then $\rsoc_G(\pi)$ is a nonzero essential subrepresentation of $\pi$. Moreover, $\mathrm{cosoc}_G(\pi^{\vee})\cong\mathrm{soc}_G(\pi)^{\vee}$.
\end{example}

\begin{proposition}\label{prop-EE-essential}
Let $M$ be an object in $\mathcal{C}_G$ of grade $c$ and let $C$ be its cosocle. Assume that the double duality map $\phi_M$ is an isomorphism and that $C$ has finite length, with all of its Jordan--H\"older factors having grade $c$.
 Then the inclusion $\EE^c(C)\hookrightarrow \EE^c(M)$ is essential.
\end{proposition}
\begin{proof}
Let $N=\rad(M)$ so that $C= M/N$.   If $N=0$, then $M\cong C$ and the result is trivial. So we may assume $N$ is nonzero for the rest of the proof. By functoriality, we have a commutative diagram
 \begin{equation}\label{equation-funct}\xymatrix{ N\ar@{^{(}->}[r]\ar_{\phi_N}[d]&M\ar@{->>}[r]\ar^{\cong}_{\phi_M}[d]& C\ar_{\phi_C}[d]&\\
 \EE^c\EE^c(N)\ar[r]&\EE^c\EE^c(M)\ar[r]&\EE^c\EE^c(C).&}\end{equation}
By Lemma \ref{lem-phi-M}, the assumption on $C$ implies that $\phi_C$ is injective.

Let $S$ be a nonzero subobject of $\EE^c(M)$; we need to show $S\cap \EE^c(C)\neq 0$. Note that  $j_G(S)=c$ by  \cite[Prop.~III.4.2.8, Prop.~III.4.2.9]{LiO}.  If $S\cap \EE^c(C)=0$, then we get an embedding $\iota:\EE^c(C)\oplus S\hookrightarrow \EE^c(M)$, which induces by taking $\EE^c(-)$   
\begin{equation}\label{equation-f(N)}
f: M\overset{\phi_M}{\simto} \EE^c\EE^cM\overset{\iota^{*}}{\ra} \EE^c\EE^c(C)\oplus \EE^c(S). 
\end{equation}
By \eqref{equation-funct},  $f(N)$ is contained in $\EE^c(S)$ (as its projection to $\EE^c\EE^c(C)$ is zero).
Consider  the induced morphism
\[\overline{f}: C\cong M/N\ra \EE^c\EE^c(C)\oplus \EE^c(S)/f(N).\]
Note that $\mathrm{Coker}(\overline{f})\cong \mathrm{Coker}(f)$ by the snake lemma, and  $\mathrm{Coker}(f)=\Coker(\iota^*)$ by \eqref{equation-f(N)}.

The projection of $\overline{f}$ to $\EE^c\EE^c(C)$ is equal to the double duality map $\phi_C$, hence is injective as remarked above. As a consequence, $\overline{f}$ is also injective and
there is an embedding
\begin{equation}\label{eq:app-embeds-coker}\EE^c(S)/f(N)\hookrightarrow \mathrm{Coker}(\overline{f}).\end{equation}
As a part of the long exact sequence associated to $\iota$ we have
\[\EE^c\EE^c(M)\overset{\iota^*}{\ra} \EE^c\EE^c(C)\oplus \EE^c(S)\ra \EE^{c+1}(\Coker(\iota)),\]
thus   $\Coker(\iota^*)$ embeds in $\EE^{c+1}(\Coker(\iota))$. Together with \eqref{eq:app-embeds-coker} and the isomorphism $\Coker(\overline{f})\cong\Coker(\iota^{*})$,  we obtain an embedding $\EE^c(S)/f(N)\hookrightarrow \EE^{c+1}(\Coker (\iota))$. By  the Auslander condition, we deduce
\begin{equation}\label{equation-appendix-geq}j_{G}(\EE^c(S)/f(N))\geq c+1.\end{equation}
On the other hand, by assumption any nonzero quotient of $C$ has grade $c$, so \eqref{equation-appendix-geq} implies \[
\Hom_{\cC_G}(C,\EE^c(S)/f(N))=0, 
\]
and consequently the projection of $\overline{f}$ to $\EE^c(S)/f(N)$ is zero.

To conclude,  consider the composite morphism \[h: M\simto \EE^c\EE^c(M)\overset{\iota^*}{\ra}   \EE^c\EE^c(C)\oplus \EE^c(S)\ra \EE^c(S). \]
It is \emph{nonzero}, because taking $\EE^c$ again and composing with $\phi_S: S\ra\EE^c\EE^c(S)$, this gives back the inclusion $S\hookrightarrow \EE^c(M)$ by functoriality.   All the above shows that $h(N)=h(M)$, which contradicts  Lemma \ref{lemma-appendix-small}, applied to $M'=h(M)$. 
\end{proof}

\subsection{Self-duality}

Let $M\in\cC_G.$  We say $M$ is   \emph{Cohen-Macaulay} if $\EE^i(M)$ is nonzero for exactly one degree $i$. Actually, we must have  $i=j_G(M)$. By \cite[Cor.~6.3]{Ven}, this is equivalent to requiring
\[j_G(M)=\mathrm{pd}(M)\]
where $\mathrm{pd}(M)$ denotes the projective dimension of $M$ as a $\Lambda(G_0)$-module.

\begin{definition}\label{def:selfdual}
Let $M\in\cC_G$ be a Cohen-Macaulay module of grade $c$. We say $M$ is \emph{self-dual} if there is an isomorphism $\EE^c(M)\cong M$ in $\cC_G$. We say $M$ is \emph{essentially self-dual} if there exists a character $\eta:G\ra \F^{\times}$ such that $\EE^c(M)\cong M\otimes\eta$ in $\cC_G$.
\end{definition}

Let $A$ be a (commutative)   noetherian local $\F$-algebra with residue field $\F$.

\begin{proposition}\label{prop-Appendix-selfduality}
Let $M$ be an $A\otimes_{\F}\Lambda(G)$-module. Assume the following conditions hold:
\begin{itemize}
\item[(a)] $A$ is Gorenstein and $M$ is flat as an $A$-module;
\item[(b)] as a $\Lambda(G)$-module, $M\in\mathcal{C}_G$ and is Cohen-Macaulay of grade $c$;
\item[(c)] $M$ is $A$-equivariantly self-dual (resp. essentially self-dual), i.e. there is an $A\otimes_{\F} \Lambda(G)$-equivariant isomorphism $\epsilon: \EE^c(M)\simto M$ (resp. $\epsilon:\EE^c(M)\simto M\otimes\eta$ for some $\eta:G\ra \F^{\times}$).
\end{itemize}
Then $\F\otimes_AM$ is also self-dual (resp. essentially self-dual).
\end{proposition}
\begin{proof}

We give the proof for the self-dual case,  the other case is proved in the same way.

Let $r$ denote the Krull dimension of $A$. Since $A$ is Gorenstein, hence Cohen-Macaulay, we may choose a regular sequence  $(x_1,\dots,x_r)$ in $A$  which is also $M$-regular by the flatness assumption (a). By (the proof of) \cite[Lem.~A.15]{Gee-Newton}, we see that $M/(x_1,\dots,x_r)$ is an $A/(x_1,\dots,x_r)\otimes_{\F}\L(G)$-module which, as a $\Lambda(G)$-module, is Cohen-Macaulay of grade $c+r$  and  self-dual.  Indeed, by induction on $r$ we may assume $r=1$ and write $x=x_1$. Then the proof  in \emph{loc. cit.} shows that $j_{G}(M/xM)=j_
{G}(M)+1=c+1$ and we have an exact sequence
\[0\lra \EE^c(M)\xrightarrow{\times x} \EE^c(M)\lra \EE^{c+1}(M/xM)\lra0. \]
Since the duality isomorphism $\epsilon:\EE^c(M)\ra M$ is assumed to be $A\otimes_{\F}\Lambda(G)$-equivariant, we deduce an isomorphism $\EE^{c+1}(M/xM)\simto M/xM$ which is $A/xA\otimes_{\F}\L(G)$-equivariant. 
 Therefore, we may assume $A$ is artinian (and Gorenstein).

Since $A$ is artinian and $M$ is flat over $A$, $M$ has a finite filtration with graded pieces isomorphic to $\F\otimes_A M$. As a consequence, $j_G(\F\otimes_A M)=j_G (M)=c$, see \cite[Prop.~3.6]{Ven}. Similarly, we also have $\mathrm{pd}(M)=\mathrm{pd}(\F\otimes_AM)$, hence
 $\F\otimes_AM$ is Cohen-Macaulay.   We deduce that $\EE^c(-\otimes_A M)$ is exact on any exact sequence of finitely generated $A$-modules (recall that $A$ is artinian).  Choose a finite presentation of $\F$:
\begin{equation}\label{eq:appen-F}A^n\overset{f}{\ra} A\ra \F\ra 0,\end{equation}
which induces an exact sequence
\[0\ra \EE^c(\F\otimes_AM)\ra \EE^c(A\otimes_AM)\overset{f^{*}}{\ra} \EE^c(A^n\otimes_A M).\]
 It is easy to see that the map $f^*$ is equal to \[(A\overset{f^{T}}{\ra} A^n)\otimes\EE^c(M) \]
 where $f^{T}$ denotes the transpose of $f$.

On the other hand, applying $\Hom_A(-,A)$ to \eqref{eq:appen-F} gives an exact sequence
\[0\ra \Hom_A(\F,A)\ra A\overset{f^T}{\ra} A^n.\]
Noticing that $\Hom_A(\F,A)\cong \soc(A)\cong \F$ (by the assumption that $A$ is Gorenstein) and that $M$ is $A$-flat, we obtain an exact sequence
\[0\ra \F\otimes_A M\ra M\overset{f^T}{\ra} M^n.\]
The explicit description of maps shows that the diagram
\[\xymatrix{\EE^c(M)\ar^{f^*}[r]\ar_{\cong}^{\epsilon}[d]&\EE^c(M^n)\ar^{\epsilon}_{\cong}[d] \\
M\ar^{f^T}[r]&M^n}\]
is commutative, which induces an isomorphism $\EE^c(\F\otimes_A M)\cong \F\otimes_A M$.
\end{proof}

\subsection{Minimal projective resolutions}

We recall some terminology on filtered rings and filtered modules. A ring $R$ is said to be a {\em filtered ring} if there is a descending chain (indexed by $\N$) of additive subgroups of $R$ denoted by $FR = \{F^n R~|~ n\in \N \}$ satisfying $ F^0 R = R,$ $F^{n+1} R \subseteq F^{n} R $ and $(F^n R) (F^m R) \subseteq F^{n+m} R$ for all $m,n \in \N.$ For convenience, we set $F^nR\defn R$ for $n<0$.
A (left) $R$-module $M$ is said to be a {\em filtered module} if there exists a descending chain (indexed by $\Z$) of additive subgroups of $M$ denoted by $FM = \{F^n M~|~ n\in \Z \}$ satisfying $F^{n+1} M \subseteq F^{n} M $ and $(F^n R)(F^m M) \subseteq F^{n+m} M$ for all $m,n \in \Z.$  An $R$-morphism $f: M \to N$ of two filtered $R$-modules is called a {\em filtered morphism of degree $d$} if $f (F^n M) \subseteq F^{n+d} N$ for all $n\in \Z.$ Let $R\filt$ denote the category where the objects are filtered $R$-modules and the morphisms are filtered morphisms of degree zero. For any $M\in R\filt$ and $a\in \Z$, denote by $M(a)\in R\filt$ the $R$-module $M$ filtered by the filtration $F^n M(a) = F^{n+a} M.$ For instance, a free $R$-module of rank $1$ which is generated by an element of degree $a$ is isomorphic to $R(-a)$.

Let $M \in R\filt$. If $M = \cup_{n\in \Z} F^n M$ then $FM$ is called {\em exhaustive}. If $\cap_{n\in \Z} F^n M = 0$ then $FM$ is called {\em separated}. The filtration topology of $M$ is the topology of $M$ such that the sets of the form $x + F^n M$ form a basis. We say $M$ is {\em complete} (with respect to its filtration topology) if $FM$ is separated and every Cauchy sequence converges.

We say $M  \in R\filt$ is {\em filt-free} if it is free as an $R$-module and has a basis $(e_j)_{j\in J}$ consisting of elements with the property that there exists a family of integers $(k_j)_{j\in J}$ such that $e_j \notin F^{k_j + 1}M,$ $j\in J$ and
\[
F^n M = \sum_{j\in J} (F^{n-k_j} R)e_j = \bigoplus_{j\in J} (F^{n-k_j} R) e_j,~ \forall n\in \Z.
\]
We say  $M  \in R\filt$ is {\em filt-projective} if it is a direct summand of a filt-free $R$-module in $R\filt.$

Let $R$ be a filtered ring with filtration $FR.$ Let $\gr R \defn \oplus_{n\in \N} F^nR /F^{n+1}R$ denote the associated graded ring. Let $M = \oplus_{n \in \Z} M_n$ and $N=\oplus_{n\in\Z}N_n$ be graded $\gr R$-modules. A graded morphism $f:M\ra N$ is called \emph{of degree $d$} if $f(M_n) \subseteq N_{n+d},$ $\forall n\in\Z.$ Any $M  \in R\filt$ gives a $\gr R$-module $\gr M \defn \oplus_{n\in \Z} F^n M / F^{n+1 }M$ which is called the associated graded module. It is clear that if $f : M\to N$ is a filtered morphism of degree $d$ then $f$ gives a graded morphism of degree $d,$ $\gr(f) : \gr M \to \gr N.$

We recall the following result of \cite[Thm.~VII.5]{N-O}.

\begin{lemma}\label{lemma:appendix-proj}
Let $R$ be an exhaustive complete filtered ring. Let $P_{g}$ be a finitely generated projective graded $\gr R$-module, then there is a (unique up to isomorphism) filt-projective module $P$ such that $\gr P=P_g$. If $M\in R\filt$ then for any graded morphism $h: P_g\ra \gr M$ of degree $d$, there is a filtered morphism $f: P\ra M$ of degree $d$ such that $h=\gr (f)$.
\end{lemma}

\begin{lemma}\label{lemma:appendix-lift}
Let $M \in R\filt.$ Let
\[G_{\bullet}:\ \ \ 0\ra G_n \ra \cdots \ra G_1\ra G_0\ra \gr(M)\ra 0 \]
be a (degree zero graded morphism) resolution of $\gr(M)$ by graded projective $\gr R$-modules $G_i.$ Then there exists a filt-projective resolution of $M$
\[P_{\bullet}: \ \ \ 0\ra P_n \ra \cdots \ra P_1\ra P_0\ra M\ra0\]
such that $\gr(P_{\bullet})\cong G_{\bullet}$.
\end{lemma}

\begin{proof}
The proof is similar to \cite[Cor.~I.7.2.9]{LiO}, using Lemma \ref{lemma:appendix-proj} as a replacement of \cite[Lem.~I.6.2]{LiO}.
\end{proof}

In general, $P_{\bullet}$ need not be minimal in the sense that the differential maps send $P_i$ to $\rad(P_{i-1})$. Next we give a practical condition so that $P_{\bullet}$ is (partially) minimal in the special case $R  =\FIwZ$. Let $\fm \defn\fm_{I_1/Z_1}$ and equip  $\FIwZ$   with the $\frak{m}$-adic filtration, namely  $F^n\FIwZ = \frak{m}^n \FIwZ$ for $n\geq 0.$ Any $\FIwZ$-module $M$ equipped with the $\frak{m}$-adic filtration, $F^nM=M$ for $n<0$ and $F^n M = \frak{m}^n M$ for  $n\in\N$, is then an object in $ \FIwZ\filt.$

For a character $\chi : I \to \F^{\times}$, let $P_{\chi}=\Proj_{I/Z_1}\chi$ equipped with the $\fm$-adic filtration. Consider
\[P=\oplus_{i=1}^rP_{\chi_i}(-a_i),\ \ Q=\oplus_{j=1}^{s}P_{\chi_j}(-b_j)\]
where $a_i,b_j\in\Z$ and let $d:P\ra Q$ be a filtered  morphism of degree zero. In general, the filtration on $P$ (or $Q$) does not coincide with its $\fm$-adic filtration. Denote by $d_{ij}:P_{\chi_i}\ra P_{\chi_j}$  the induced morphism of $\FIwZ$-modules
\[P_{\chi_i}\hookrightarrow P\overset{d}{\ra} Q\twoheadrightarrow P_{\chi_j}. \]

\begin{lemma}\label{lemma:appendix-minimal}
If  $a_i>b_j$ for any pair $(i,j)$ with $\chi_i=\chi_j,$ then $d(P)\subseteq \fm Q$.
\end{lemma}
\begin{proof}
Fix $i$ and let $x\in P_{\chi_i}$. Since $d$ has degree $0$ and $x\in P_{\chi_i}=F^{a_{i}}(P_{\chi_i}(-a_{i}))$, we have
\[d(x)\in F^{a_{i}}Q=\oplus_{j=1}^sF^{a_{i}-b_j}P_{\chi_j}.\]
We claim that $d_{ij}(x)\in \fm P_{\chi_j}$ for all $1\leq j\leq s$.
If $\chi_j\neq \chi_i$, then any morphism $P_{\chi_i}\ra P_{\chi_j}$ must have image contained in $\fm P_{\chi_j}$. If $\chi_j=\chi_i$, we use the assumption
 $a_i>b_j$ to deduce that $d_{ij}(x)\in \fm^{a_i-b_j} P_{\chi_j}\subset \fm P_{\chi_j}$.
This finishes the proof.
\end{proof}

\section*{Acknowledgements}
We thank Christophe Breuil, Yiwen Ding, Guy Henniart, Vytautas Pa\v{s}k\=unas and Zicheng Qian for several discussions during the preparation of the paper, and Yitong Wang for his comments. Y. H. is very grateful to Christophe Breuil, Florian Herzig, Stefano Morra and Benjamin Schraen for the collaboration \cite{BHHMS} which has inspired him much in the late stages of the project, and both of us  thank them for their comments on a preliminary version of the paper. We are very grateful to the anonymous referee for several helpful corrections and suggestions, especially for pointing out a gap in the proof of Proposition \ref{prop-coker-no-sigma} and suggesting a simpler fix for it.

Y. H. has presented part of this work in the conference on ``The $p$-adic Langlands programme and related topics'' held at King's College London in May 2019 and the Padova school on ``Serre conjectures and $p$-adic Langlands program'' in June 2019. He thanks the organizers for the invitation and the institutes for the hospitality.
Y. H. thanks Ahmed Abbes for inviting him to I.H.\'E.S. for the period of November-December 2019 and I.H.\'E.S. for its hospitality. Part of this work was done during a visit of H. W. to  Morningside Center of Mathematics and he thanks Morningside Center of Mathematics for its hospitality.

Y. H. is partially supported by National Key R$\&$D Program of China 2020YFA0712600, National Natural Science Foundation of China Grants 12288201 and 11971028; National Center for Mathematics and Interdisciplinary Sciences and Hua Loo-Keng Key Laboratory of Mathematics, Chinese Academy of Sciences. H. W. is partially supported by National Natural Science Foundation of China Grants 11971028 and 11901331, Beijing Natural Science Foundation (1204032).


\newcommand{\etalchar}[1]{$^{#1}$}
\providecommand{\bysame}{\leavevmode\hbox to3em{\hrulefill}\thinspace}
\providecommand{\MR}{\relax\ifhmode\unskip\space\fi MR }
\providecommand{\MRhref}[2]{%
  \href{http://www.ams.org/mathscinet-getitem?mr=#1}{#2}
}
\providecommand{\href}[2]{#2}

 \end{document}